%% file: main.tex
\documentclass[10pt,reqno]{amsart}
\usepackage[margin=1in]{geometry}

\usepackage{fancyhdr}
\usepackage[font={small}]{caption}

\usepackage{hyperref}
\usepackage{fancyvrb}
\usepackage{graphicx}	
\usepackage{amssymb}
\usepackage{mathrsfs}
\usepackage{upgreek}
\usepackage{amsthm}
\usepackage{amsmath}
\usepackage{url}
\usepackage{enumerate}
\usepackage{mathtools}
\usepackage{algorithm}
\usepackage{color}
\usepackage{todonotes}
\usepackage{bbm}

\usepackage{makecell, multirow, colortbl}
\colorlet{grey}{gray!45}

\usepackage{amsaddr}
\usepackage{parskip}
\usepackage{lscape}
\usepackage{enumitem}
\usepackage[final]{pdfpages}

\numberwithin{equation}{section}

\newcommand{\Pv}{\mathbb{P}}
\newcommand{\Ev}{\mathbb{E}}
\newcommand*{\be}{\begin{equation}}
\newcommand*{\ee}{\end{equation}}
\newcommand*{\ba}{\begin{aligned}}
	\newcommand*{\ea}{\end{aligned}}
\newcommand*{\barr}{\begin{array}{c}}
	\newcommand*{\earr}{\end{array}}

\newcommand{\BRW}{\mathrm{BRW}_{f,\la}}
\newcommand{\CPf}{\mathrm{CP}_{f,\la}}

\newcommand{\CPfd}{\mathrm{CP}_{f,\la}^\downarrow}

\newcommand*{\ind}{\mathbbm{1}}
\newcommand{\sss}{\scriptscriptstyle}

\newcommand{\Bin}{\mathrm{Bin}}

\newcommand{\CA}{\mathcal {A}}
\newcommand{\CB}{\mathcal {B}}
\newcommand{\CC}{\mathcal {C}}

\newcommand{\CE}{\mathcal {E}}
\newcommand{\CF}{\mathcal {F}}
\newcommand{\CG}{\mathcal {G}}
\newcommand{\CI}{\mathcal {I}}

\newcommand{\CL}{\mathcal {L}}

\newcommand{\CN}{\mathcal {N}}

\newcommand{\CP}{\mathcal {P}}

\newcommand{\CT}{\mathcal {T}}

\newcommand{\CV}{\mathcal {V}}
\newcommand{\CX}{\mathcal {X}}

\newcommand{\CH}{\mathcal {H}}
\newcommand{\N}{\mathbb{N}}

\newcommand*{\wt}{\widetilde}

\newcommand*{\la}{\lambda}

	\newcommand{\Z}{\mathbb{Z}}

\newcommand{\calT}{\mathcal{T}}

\newcommand{\E}{\mathbb{E}}
\newcommand{\e}{\mathrm{e}}
\renewcommand{\P}{\mathbb{P}}

\newcommand{\R}{\mathbb{R}}

\newcommand{\calA}{\mathcal{A}}
\newcommand{\calB}{\mathcal{B}}

\newcommand{\eps}{\varepsilon}
\renewcommand{\epsilon}{\eps}
\newcommand{\Core}{\mathrm{Core}}

\newtheorem{theorem}{Theorem}[section]
\newtheorem{lemma}[theorem]{Lemma}
\newtheorem{corollary}[theorem]{Corollary}
\newtheorem{definition}[theorem]{Definition}
\newtheorem{observation}[theorem]{Observation}
\newtheorem{remark}[theorem]{Remark}
\newtheorem{assumption}[theorem]{Assumption}
\newtheorem{proposition}[theorem]{Proposition}
\newtheorem{claim}[theorem]{Claim}

\newtheorem{example}[theorem]{Example}
\newtheorem{construction}[theorem]{Construction}

\newcommand{\PPP}{\mathrm{PPP}}

\mathcode`e=`e

\begin{document}

\fancyhead[LO]{Degree-penalized contact processes}
\fancyhead[LE]{Bartha, Komj{\'a}thy and Valesin}
\fancyhead[RO]{}
\fancyhead[RE]{}
\title{Degree-penalized contact processes}
\date{\today}
\subjclass[2020]{82C22 (Primary) 60K35, 05C80, 60J85 (Secondary)}
\keywords{contact process, interacting particle systems, random graphs}
\author{Zsolt Bartha$^\star$, J\'ulia Komj\'athy$^\dagger$, Daniel Valesin$^\ddag$}
\address{$^\star$ Alfr\'ed R\'enyi Institute of Mathematics\\
$^\dagger$ Delft Institute of Applied Mathematics, Delft University of Technology\\
$^\ddag$ Department of Statistics,
University of Warwick}
\email{bartha@renyi.hu, j.komjathy@tudelft.nl, daniel.valesin@warwick.ac.uk}

\maketitle

\input{abstract}      
\input{introduction}
\input{preliminaries}

\input{proofs_GW_extinction} 
\input{proofs_CM_extinction}
\input{last_lemma_temp}

\input{proofs_GW-survival}
\input{proofs_CM}

\input{proofs_CM_stars}

{\appendix
\input{appendix}}

\bibliographystyle{IEEEtranS}
\bibliography{literature}

\end{document}

%% file: abstract.tex
\begin{abstract}
In this paper we study degree-penalized contact processes on Galton-Watson trees (GW) and the configuration model. 
The model we consider is a modification of the usual contact process on a graph. In particular, each vertex can be either infected or healthy. When infected, each vertex heals at rate one. Also, when infected, a vertex $v$ with degree $d_v$ infects its neighboring vertex $u$ with degree $d_u$ 
with rate $\la/ f(d_u, d_v)$ for some positive function $f$. In the case $f(d_u, d_v)=\max(d_u, d_v)^\mu$ for some $\mu>0$, the infection is slowed down to and from high degree vertices. This is in line with arguments used in social network science: people with many contacts do not have the time to infect their neighbors at the same rate as people with fewer contacts. 

We show that new phase transitions occur in terms of the parameter $\mu$ (at $1/2$) and the degree distribution $D$ of the GW tree. 
\begin{itemize}
\item When $\mu\ge 1$, the process goes extinct for all distributions $D$ for all sufficiently small $\lambda>0$;
\item When $\mu\in(1/2, 1)$, and the tail of $D$ weakly follows a power law with tail-exponent less than $1-\mu$, the process survives globally but not locally for all $\la$ small enough; 
\item When $\mu\in(1/2, 1)$, and $\Ev[D^{1-\mu}]<\infty$, the process goes extinct almost surely, for all $\la$ small enough;
\item When $\mu<1/2$, and $D$ is 
heavier then stretched exponential with  stretch-exponent $1-2\mu$, the process survives (locally) with  positive probability for all $\la>0$.
\end{itemize}
We also study the product case $f(x,y)=(xy)^\mu$. In that case, the situation for~$\mu < 1/2$ is the same as the one described above, but $\mu\ge 1/2$ always leads to a subcritical contact process for small enough $\lambda>0$ on all graphs. Furthermore, for finite random graphs with prescribed degree sequences, we establish the corresponding phase transitions in terms of the length of survival.
\end{abstract}

%% file: introduction.tex
\section{Introduction}\label{sec:introduction}
The contact process (CP) is a model for epidemics on graphs, described by a continuous-time Markovian dynamics, in which each vertex is in one of two states: infected or healthy. Infected vertices infect each of their healthy neighbors with a constant rate $\lambda$, while also healing at a constant rate $1$.
The model was first introduced by Harris in 1974 \cite{harris1974contact}, who studied it on the integer lattice. Since then, much work has been done to characterize the behavior of the process also on infinite trees and locally tree-like finite graphs. 
The focus of this line of research has been to establish phase transitions in the long-term behavior of the process, as the spreading rate $\lambda$ varies. A series of works \cite{liggett1996multiple, Pem92,stacey1996existence} showed that the process on the infinite $d$-ary tree ($d\ge 2$), with an initial infection at the root, has three possible phases separated by two critical values $0<\lambda_{c,1}<\lambda_{c,2}$: 
when  $\lambda<\lambda_{c,1}$ the process undergoes eventual extinction, when $\lambda\in (\lambda_{c,1}, \lambda_{c,2})$ there is 'global but not local' survival, and when $\lambda>\lambda_{c,2}$ there is 'local' survival of the infection (see Definition \ref{def:survival}).
More recently, studying the process on Galton-Watson trees, the combination of the results in \cite{huang2020contact} and \cite{Sly19} showed that models with exponentially decaying offspring distributions always have an extinction phase ($\lambda_{c,1}>0$), whereas subexponentially decaying offspring distributions lead to local survival for any positive value of $\lambda$ due to the persistence of the infection around high-degree vertices, i.e., $\lambda_{c,1}=\lambda_{c,2}=0$ in this case.

Motivated by the latter results, we introduce a variant of the original contact process, where we slow down the spread of the infection around high-degree vertices in a degree-dependent way, in order not to let 'superspreaders' scale up the infection rate linearly in their degree. 
This choice is inspired by degree-dependent bond percolation~\cite{hooyberghs2010biased}, by topology-biased random walks~\cite{bonaventura2014characteristic, ding2018centrality, lee2009centrality,pu2015epidemic,zlatic2010topologically}, in which the transition probabilities from a vertex depend on the degrees of its neighbors. 
Those works all assume a polynomial dependence on the degrees.  Related is also the recent degree-dependent first passage percolation \cite{komjathy2021penalising, komjathy2023four1, komjathy2023four2}, which uses the same `degree-penalization' that we shall assume, combined with the first passage percolation dynamics where reinfections to a vertex are not possible.

In the degree-dependent contact process, the total infection rate from a high-degree infected vertex shall only grow polynomially with its degree, with an exponent less than one. Gradually increasing the penalty on the infection rate, we prove that the new process qualitatively differs from the classical version.  
In particular, we obtain new phase diagrams for Galton-Watson trees: as soon as the total infection rate from a high-degree vertex scales less than the square root of its degree, high-degree vertices no longer maintain the infection, but their local surroundings heal quickly, and the process shows local extinction for small $\lambda$, yielding $\lambda_{c,2}>0$, on \emph{any} tree in fact (not just Galton-Watson trees).  On Galton-Watson trees, if the offspring distribution is sufficiently heavy tailed (i.e., heavier than $x^{-\alpha_c}$ for some critical $\alpha_c$ depending on the degree-dependent penalty on the infection rate), then the degree-penalized CP survives globally but not locally (i.e., $\lambda_{c,1}=0$ but $\lambda_{c,2}>0$), while if the tail is lighter, i.e., the offspring distribution has finite $\alpha_c$-th moment (with $\alpha_c<1$), then CP has an extinction phase (i.e., $\lambda_{c,1}>0$). Here we find it surprising that subexponential distributions as heavy as infinite mean power laws can also show extinction.
We also establish the corresponding phase diagrams for large finite random graphs with prescribed degree distributions (the configuration model), in terms of the length of time the infection survives on them.
Here, tree-based recursion techniques break down, and we develop new methods to treat the extinction phase when $\lambda_{c,1}>0$, which work as soon as the offspring distribution has finite variance. In the phase when high-degree vertices no longer maintain the infection for a long time, but the Galton-Watson tree show global survival for small $\lambda>0$, we find new structures -- $k$-cores existing on constant degree vertices only --  that maintain the infection globally on the graph for a long time.
All our results are also valid for the corresponding branching random walks as well. See a summary of our main results in Table \ref{table:summary} where we briefly explain the main parameters. We defer mentioning more related work to Section \ref{sec:discussion}.
\begin{table}[t!]
\begin{tabular}{l||l|l}
\cellcolor{grey}\textbf{Product penalty}& \cellcolor{grey}\textbf{Galton-Watson tree} $\CT_D$ & \cellcolor{grey}\textbf{Configuration model} $\mathrm{CM}(\underline d_n)$
\tabularnewline\hline
$\mu<1/2$&\textbf{Local survival} &\textbf{Survival} until $\Theta_{\P}(\exp(Cn))$ time
\tabularnewline
&for any $\lambda>0$&for any $\lambda>0$
\tabularnewline
&for tail heavier than&for tail heavier than
\tabularnewline
&stretched-exponential with $\zeta=1-2\mu$&stretched-exponential with $\zeta=1-2\mu$
\tabularnewline\hline
$\mu\ge 1/2$&\textbf{Extinction}&\textbf{Extinction} in $O_{\P}(\mathrm{poly}(n))$ time
\tabularnewline
&for $\lambda<1$&for $\lambda< 1$
\tabularnewline
&for any graph&whenever $\sum_{i=1}^n d_i^{1-\mu} =O_{\P}(\mathrm{poly}(n))$
\tabularnewline\hline\hline
\cellcolor{grey}\textbf{Max penalty}& \cellcolor{grey}\textbf{Galton-Watson tree} $\CT_D$ &\cellcolor{grey}\textbf{Configuration model} $\mathrm{CM}(\underline d_n)$
\tabularnewline\hline
$\mu<1/2$&\textbf{Local survival} &\textbf{Survival} until $\Theta_{\P}(\exp(Cn))$ time
\tabularnewline
&for any $\lambda>0$&for any $\lambda>0$
\tabularnewline
&for tail heavier than&for tail heavier than
\tabularnewline
&stretched-exponential with $\zeta=1-2\mu$&stretched-exponential with $\zeta=1-2\mu$
\tabularnewline\hline
$\mu\in(1/2,1)$&\textbf{Only global survival} &\textbf{Survival} until $\Theta_{\P}(\exp(Cn))$ time
\tabularnewline
&for $\lambda<1/2$&for any $\lambda>0$
\tabularnewline
&for weak power law&for power-law empirical degrees
\tabularnewline
&with tail-exponent $\alpha<1-\mu$&with $\mu<3-\tau$
\tabularnewline\cline{2-3}
&\textbf{Extinction} &\textbf{Extinction} in $\Theta_{\P}(\log(n))$ time
\tabularnewline
&for small $\lambda$&for small $\lambda$
\tabularnewline
&when $\E[D^{1-\mu}]<\infty$&for power-law empirical degrees
\tabularnewline
&&with $\tau>3$ (or lighter)
\tabularnewline\hline
$\mu\ge 1$&\textbf{Extinction} &\textbf{Extinction} in $O_{\P}(\mathrm{poly}(n))$ time
\tabularnewline
 &for $\lambda<1$&for $\lambda< 1$
\tabularnewline
 &for any graph&whenever $\sum_{i=1}^n d_i^{1-\mu} =O_{\P}(\mathrm{poly}(n))$
 \tabularnewline\hline 
\end{tabular}\\[.3cm]
\caption{ Summary of our main results: phases of degree-dependent contact process. Here, the infection rate across an edge is $\lambda /f(x,y)=\lambda/ (xy)^\mu$ in the case of the product penalty, and $\lambda/f(x,y)=\lambda/\max\{x,y\}^\mu$ in the case of the max penalty. The second column shows the phases when the underlying graph is a Galton-Watson tree with offspring distribution $D$, and initially only the root is infected. Here, $\alpha$ denotes the power-law tail-exponent, i.e., $\P(D\ge z)\asymp z^{-\alpha}$. The third column shows the phases when the underlying graph is a configuration model with degree sequence $\underline{d}_n$, and initially all the vertices are infected. Here, $\tau$ denotes the exponent of the limiting mass function, i.e., $\P(D\ge z)\asymp z^{-(\tau-1)}$. We allow not just pure power laws, see Definitions \ref{def:power-heavy}--\ref{def:stretched-heavy} and Assumptions \ref{assu:regularity}--\ref{assu:empirical-power-law-2} for weaker assumptions. Some technical conditions are omitted in the table. For $\mu\in(1/2, 1) $ on the configuration model, fast extinction occurs when $\tau>3$, including any other lighter tails, not just power laws. }\label{table:summary}
\end{table}
\subsection{Degree-penalized infection processes: main definitions}\label{sec:DefProcess}
We now define the processes considered in this paper. These processes take place on an underlying graph, which is undirected, but not necessarily simple, i.e., we allow multiple edges and loops, see  Section \ref{sec:DefGraphs} for the underlying graphs we use. We use the convention that the \emph{degree} of a vertex is the number of non-loop edges incident to it (counted with multiplicity) plus twice the number of loops incident to it.
More formally, for a graph $G=(V,E)$ we denote by $e(u,v)$ the number of edges between vertices $u,v\in V$, and by $N(v)$ the \emph{neighborhood} of $v\in V$, the set of vertices $u$ for which $e(u,v)\ge 1$. For a vector $\underline x\in \N^V$, we let $|\underline x|:=\sum_{v\in V} x(v)$ be its $1$-norm.
\begin{definition}[Degree-penalized contact process]\label{def:CP}
Consider a graph $G=(V,E)$, with $d_v$ denoting the degree of vertex $v\in V$. Let $f(x,y)>1$ be a function of two variables, $\la>0$, and $\underline \xi_0\in \{0,1\}^V$. For $u,v\in V$ let $r(u,v)=\la \cdot e(u,v)/f(d_u, d_v)$. We define  $\CPf(G,  \underline \xi_0)=(\underline{\xi}_t)_{t\ge 0}=(\xi_t(v))_{v\in V, t\ge0}$ to be the following continuous-time Markov process on the state space $\{0,1\}^V$. The process starts from the state $\underline \xi_0$ at time $t=0$, and evolves according to the following transition rates:
\begin{align}
\underline\xi\longrightarrow \underline\xi-\ind_v&\quad\text{with rate 1;\  for all $v$ with $\xi(v)=1$},\label{eq:healing-CP}\\
\underline\xi\longrightarrow \underline\xi+\ind_v&\quad\text{with rate $\sum_{u\in N(v)}\xi(u) r(u,v)$; for all $v$ with $\xi(v)=0$},\label{eq:infection-rate-CP}
\end{align}
where $\ind_v\in\{0,1\}^V$ denotes the vector with entry $1$ at position $v$, and zero entries at all 
other positions. 
\end{definition}
We refer to vertices $v$ with $\xi_t(v)=1$ as \emph{infected} at time $t$, and to all other vertices as \emph{healthy} at time $t$, and consequently $|\xi_t|$ is the number of infected vertices at time $t$. Describing the process less formally, each infected vertex $u$ heals at rate $1$, and during the time it is infected, it infects each of its healthy neighbors $v$ at rate $r(u,v)=\la \cdot e(u,v)/f(d_u, d_v)$, where $e(u,v)$ is the number of edges between $u$ and $v$. A common choice for $\underline \xi_0$ we take is $\underline{1}_G$, the all-$1$ vector on the vertex set $V$ of $G$.

A process related to the contact process is the branching random walk on the same graph. Branching random walks are known to stochastically dominate the contact process, since they consider the vertices of the graph as locations that infected particles can occupy, and they allow more than one infected particles per vertex. In comparison, in the contact process only one particle per vertex is allowed. In our setting, the 
\emph{degree-penalized branching random walk} turns out to be useful for upper bounds when proving extinction.
 
\begin{definition}[Degree-penalized branching random walk]\label{def:BRW}
Consider a graph $G=(V,E)$, with $d_v$ denoting the degree of vertex $v\in V$ and $e(u,v)$ the number of edges between $u$ and $v$. Let $f(x,y)>1$ be a function of two variables, $\la>0$, and $\underline x_0\in \N^V$. For $u,v\in V$ let $r(u,v)=\la \cdot e(u,v)/f(d_u, d_v)$. We define  $\BRW(G,  \underline x_0)=(\underline x_t)_{t\ge 0}=(x_t(v))_{v\in V, t\ge0}$ to be the following continuous-time Markov process on the state space $\N^V$. The process starts from the state $\underline x_0$ at time $t=0$, and evolves according to the following transition rates:
\begin{align}
\underline x\longrightarrow \underline x-\ind_v&\quad\text{with rate $x(v)$ for all $v\in V$},\label{eq:healing-BRW}\\
\underline x\longrightarrow \underline x+\ind_v&\quad\text{with rate $\sum_{u\in N(v)}x(u) r(u,v)$ for all $v\in V$}.\label{eq:infection-rate-BRW}
\end{align}
\end{definition}

Informally, we think of $x_t(v)$ as the number of particles at location $v$ at time $t$. Then each particle dies at rate $1$, independently of everything else, and each particle located at $u$ reproduces to every neighboring vertex $v$ at rate $r(u,v)=\la \cdot e(u,v)/f(d_u, d_v)$.

In what follows we study the qualitative long-term behavior of the above processes, for small $\la>0$ infection parameters. The following definition summarizes the possible phases that can occur  on graphs, first with (countably) infinitely many vertices, and then on graphs with finitely many vertices. Here, and in the following, $\underline 0$ denotes the all-zero vector (on the relevant index set).
\begin{definition}[Modes of survival]\label{def:survival}
Given a graph $G=(V,E)$, a penalty function $f(x,y)>0$ and some $\lambda>0$, consider either the process $(\underline\xi_t)_{t\ge 0}=\CPf(G,\underline\xi_0)$ or the process $(\underline x_t)_{t\ge 0}=\BRW(G,\underline x_0)$ with respective fixed starting states $\underline\xi_0\in\{0,1\}^V$ and $\underline x_0\in\N^V$. If $|V|=\infty$, we say that the process exhibits
\begin{itemize}
\item[(i)]
\emph{almost sure extinction} if, with probability 1, there exists some $T>0$ such that $\underline\xi_t=\underline 0$ \,(respectively, $\underline x_t=\underline 0$) for all $t\ge T$,
\item[(ii)]
 \emph{global survival} if, with positive probability, $\underline\xi_t\ne\underline 0$\, (respectively $\underline x_t\ne\underline 0$),  for all $t\ge 0$.
\item[(iii)]
 \emph{local survival} if, with positive probability, there exists $v\in V$ such that for any $t\ge 0$ there exists some $s>t$ such that $\xi_s(v)=1$ (respectively, $x_s(v)=1$).
\end{itemize}
For any underlying graph $G$ and respective initial states $\underline \xi_0 \in \{0,1\}^V$ and $\underline x_0 \in \N^V$ of $\CPf$ and $\BRW$, let us define the (possibly infinite) \emph{extinction time}, and for a vertex $v\in G$ the \emph{local extinction time at $v$}
\begin{equation*}
\begin{aligned}
   T_{\mathrm{ext}}^{\mathrm{cp}}(G, \underline \xi_0)&=\inf\{t\ge 0:\ \underline\xi_t=\underline 0\},\quad  T_{\mathrm{ext}}^{\mathrm{cp}}(G, \underline \xi_0,v)&=\inf\{t\ge 0:\ \xi_{t'}(v)=0\  \forall t'\ge t\}, \\
   T_{\mathrm{ext}}^{\mathrm{brw}}(G, \underline x_0)&=\inf\{t\ge 0:\ \underline x_t=\underline 0\},\quad  T_{\mathrm{ext}}^{\mathrm{brw}}(G, \underline x_0,v)&=\inf\{t\ge 0:\ x_{t'}(v)=0\  \forall t'\ge t\}. 
   \end{aligned}
\end{equation*}
\end{definition}

We note some remarks: First, local survival in (iii) implies global survival in (ii). Second, \emph{only global} (but not local) survival means that (ii) holds, whereas for any choice $v\in V$ almost surely there exists some $T_v>0$ such that $\xi_t(v)=0$ (resp., $x_t(v)=0$) for all $t>T_v$. Finally, provided that $0< |\underline \xi_0|<\infty$ (resp., $0< |\underline x_0|<\infty$), and that the graph $G$ is connected,  the phase that occurs among (i)--(iii) does not depend on the initial state $\underline \xi_0$ (resp., $\underline x_0$).

\subsection{Definition of the underlying graphs}\label{sec:DefGraphs}
Next, we define the graph models that we focus on.
\begin{definition}[Galton-Watson tree]
Given a non-negative integer-valued random variable $D$, we define the \emph{Galton-Watson (GW) tree with offspring distribution} $D$ as follows. Let $\varnothing$ be a distinguished vertex, called the \emph{root} of the tree. $\{\varnothing\}$ is generation 0 of the tree, and its cardinality is $Z_0=1$. Let $(D_{i,j})_{i=0,j=1}^\infty$ be an array of iid copies of $D$. Then we recursively define generation $i+1$ of the tree for $i=0,1\ldots$ in the following way. For each vertex $j$ ($j=1,\ldots,Z_i$) of generation $i$ we assign $D_{i,j}$ many offspring, connect them to vertex $j$, forming together generation $i+1$, i.e.,  generation $i+1$ has cardinality $Z_{i+1}=\sum_{j=1}^{Z_i}D_{i,j}$. We call the resulting finite or infinite tree a ralization of the Galton-Watson tree.
\end{definition}
Our results, in an important regime, extend to any random or deterministic tree as well, as long as it grows at most exponentially almost surely, a concept which we define now. 
\begin{definition}[Branching number of a tree]\label{def:branching-number} Let $\calT$ be an infinite tree, and let $Z_N(\calT):=|\mathrm{Gen}_N(\calT)|$ be the size of generation $N$. Then we define the (possibly infinite) `upper' branching number of $T$ as 
\begin{equation}\label{eq:branching-number}
\overline{\mathrm{br}}(T):=\limsup_{N\to\infty} Z_N(\calT)^{1/N}.
\end{equation}
\end{definition}
\begin{definition}[Spherically symmetric tree]\label{def:SST}
Given a positive integer-valued sequence $\underline d:=(d_0, d_1, d_2, \dots)$, we define the \emph{Spherically Symmetric Tree (SST) with degree sequence} $\underline d$, $\mathrm{SST}(\underline d)$ as follows. Let $\varnothing$ be the \emph{root} of the tree having $d_\varnothing:=d_0$ many offspring. Then $\mathrm{SST}(\underline d)$ is the tree where each vertex in generation $i$ has $d_i$ many offspring. 
\end{definition}

The following two definitions describe two important classes of degree distributions that we use for Galton-Watson trees.

\begin{definition}[Weak power-law tails]\label{def:power-heavy}
Consider a  distribution $D$ on $\{0,1, \dots\}$. We say that the tail of $D$ weakly follows a power law with tail-exponent $\alpha>0$ if for all fixed $\varepsilon>0$ there exists a constant $z_0(\varepsilon)>1$, such that whenever $z>z_0(\varepsilon)$,
\begin{equation}\label{eq:heavier-power-law}
\frac{1}{z^{\alpha+\varepsilon}}\le   \P(D \ge z ) \le \frac{1}{z^{\alpha-\varepsilon}}.  
\end{equation}
 \end{definition}
In the numerators in   \eqref{eq:heavier-power-law} we could have allowed a slowly varying function as well, but those can be ignored by adjusting $z_0(\varepsilon)$, due to Potter's theorem \cite{bingham1989regular}, since any slowly varying function $\ell(x)$ satisfies $x^{-\varepsilon}\ll\ell(x)\ll x^{\varepsilon}$ for all $\varepsilon>0$ as $x\to \infty$. Pure power-law distributions satisfy \eqref{eq:heavier-power-law} with $\varepsilon=0$, in this case the constant $1$ in the numerators of the upper and lower bounds may change.
 The next definition considers a similar domination, but now with stretched exponential tails: 
\begin{definition}[Heavier than stretched exponential tails]\label{def:stretched-heavy}
Consider a  distribution $D$ on $\{0,1, \dots\}$. We say that $D$ is heavier than stretched exponential with stretch-exponent $\zeta>0$ if there exists a function $g: \N\to [0, \infty)$ and an infinite sequence 
of nonnegative numbers $z_1<z_2<\dots$ such that for $i\ge 1$,
\begin{equation}\label{eq:stretched-heavy}
    \P\big(D =z_i \big) \ge  \exp( -  g(z_i)z_i^\zeta) \mbox{ such that } g(x)\to 0 \mbox{ as } x\to \infty. 
\end{equation}
 \end{definition}
An equivalent statement to  \eqref{eq:stretched-heavy} is 
\[\liminf_{z\to \infty}  \frac{-\log (\P(D =z ))}{z^\zeta}=0.  \]
We comment that in case of stretched exponential distributions, the tail $\P(D\ge K)$ and the mass function $\P(D=K)$ are a polynomial prefactor away, which can be incorporated in the function $g$.

The next definition gives the \emph{finite} random graph model that we consider in this paper: the configuration model with a given degree sequence \cite{Boll80, MolRee95}.
 \begin{definition}[Configuration model]\label{def:CM}
Given a positive integer $n$, and a sequence $\underline d_n:=(d_1, \dots, d_n)$ of nonnegative integers with $h_n:=\sum_{i=1}^n d_n$ even, we define the \emph{configuration model} $\mathrm{CM}(\underline d_n)$ as a distribution on (multi)graphs constructed as follows. We take $n$ vertices, and assign $d_1,d_2,\ldots,d_n$ `half-edges' to them, respectively.  Then we take a uniformly random pairing of the set of half-edges, and to each such pair we associate an edge in $\mathrm{CM}(\underline d_n)$ between the respective vertices.
 \end{definition}

In Definition \ref{def:CM}, in the degree sequence   $\underline d_n=(d_1^{\sss{(n)}}, d_2^{\sss{(n)}}, \dots, d_n^{\sss{(n)}})$ we allow that the degrees depend on $n$. If it is not confusing we drop the superscript $(n)$ from the degree sequence.
When the degree sequence is random, (e.g. coming from an iid sequence $D_1, D_2, \dots$), then one may add an extra half-edge to $D_n$ when $\sum_{i=1}^nD_i$ is odd. This will not affect the `regularity' assumptions on the degree sequence below.
The configuration model is a locally tree-like graph: its \emph{local weak limit} is a Galton-Watson tree \cite{aldous2004objective, benjamini2011recurrence}. We expect that our results extend to other non-geometric graph models with branching processes as their local weak limit, e.g. the Erd{\H o}s-R\'enyi random graph, the Chung-Lu or Norros-Reitu model, rank-$1$ inhomogeneous random graphs \cite{ER60, Chung_Lu, NorRei04, BolJAnRio07}, and so on. 

We define  the \emph{empirical mass function} $\nu_n$ of the degrees and the corresponding cumulative distribution function (cdf) for all $z\ge 0$ as
\begin{equation}\label{eq:empirical-degree}
  \nu_n(z):=\frac{n_z}{n}=\frac{\sum_{i=1}^n \ind_{\{d_i=z\}}}{n} \quad \mbox{ and }\quad F_n(z)=\nu_n([0,z])=\frac{1}{n}\sum_{i=1}^n \ind_{\{d_i\le z\}}. 
\end{equation}
Let $D_n$ be a random variable with distribution $\nu_n$. To be able to relate different elements of the sequence $\mathrm{CM}(\underline d_n)$ to each other, we pose the following regularity assumption, common in the literature \cite{MolRee95, MolRee98, janson2007simple}.

\begin{assumption}[Regularity assumptions on the degrees] \label{assu:regularity}
Consider the configuration model in Definition \ref{def:CM}. We assume that the sequence $(\underline d_n)_{n\ge 1}=((d_1, d_2, \dots, d_n))_{n\ge 1}$ satisfies the following:
\begin{enumerate}
\setlength\itemsep{0em}
\item[a)] $D_n$ with cdf $F_n(z)$ in \eqref{eq:empirical-degree} converges in distribution to some a.s.\ finite random variable $D$ with $\E[D]\in(0,\infty)$. We denote the cdf of $D$ by $F_D$.
\item[b)] $\lim_{n\to \infty}\E[D_n]=\E[D]$. In particular, for any constant $M\ge 0$,
\[\lim_{n\to \infty}\E[D_n\ind_{\{D_n\ge M\}}]=\E[D\ind_{\{D\ge M\}}].\]
\end{enumerate}
\end{assumption}
Formulating power-law assumptions about a sequence of empirical distributions is slightly different than about a single distribution, since the minimal mass in the model with $n$ vertices is $1/n$ and the maximal degree is $n$-dependent and finite. Hence, we formulate the next assumption, which ensures that the empirical distribution $F_n$ follows a (possibly truncated) weak power law. 
\begin{assumption}[Power-law empirical degrees]\label{assu:empirical-power-law}
We say that the empirical distribution of $(\underline d_n)_{n\ge 1}$ follows a weak (possibly truncated) power law with exponent $\tau> 1$ with exponent-error $\varepsilon\ge0$, if there exist constants $c_\ell, c_u, z_0=z_0(\varepsilon), n_0(\varepsilon)>0$ and a function $z^{\sss{(\ell)}}_{\max}(\varepsilon,n)\to \infty$ as $n\to \infty$ such that for all $n\ge n_0(\varepsilon)$, $F_n(z)$ in \eqref{eq:empirical-degree} satisfies 
\begin{equation}\label{eq:empirical-power-law}
\frac{c_\ell}{z^{(\tau-1)(1+\varepsilon)}}\le 1-F_n(z) \le \frac{c_u}{z^{(\tau-1)(1-\varepsilon)}},
\end{equation}
for all $z\in [z_0,z^{\sss{(\ell)}}_{\max}(\varepsilon,n)]$, while the upper bound holds for all $z\ge z_0$. In this case we call $\tau-1$ the tail-exponent, consistent with Definition \ref{def:power-heavy}.
\end{assumption}
When the degrees are coming from an iid sample of a distribution $D$ that satisfies \eqref{eq:heavier-power-law} with some $\tau,\eps$, then one can use Chernoff bounds to show that  Assumption \ref{assu:empirical-power-law} is also satisfied with a slightly larger $\varepsilon$  and $z_{\max}(\varepsilon,n)$ can be chosen slightly below the typical maximum degree among iid degrees, which is $n^{(1- \varepsilon)/(\tau-1)}$ with high probability. However,  in Assumption \ref{assu:empirical-power-law} we also allow for much lower $z_{\max}^{\sss{(\ell)}}(\varepsilon,n)$. In such cases we talk about truncated power-law degrees. Since the truncation value $z^{\sss{(\ell)}}_{\max}(\varepsilon,n)\to \infty$ as $n\to \infty$, the limiting distribution $D$ satisfies \eqref{eq:empirical-power-law} for all (fixed) $z\ge z_0$. We also comment that if $\varepsilon>0$, by slightly increasing $\varepsilon$ and $z_0$ if necessary, one may choose $c_\ell=c_u=1$. Further, if instead of \eqref{eq:empirical-power-law}, one has the bounds
\begin{equation}\label{eq:slowly-varying}
\ell_1(z) z^{-(\tau-1)} \le 1-F_n(z)\le \ell_2(z) z^{-(\tau-1)}
\end{equation}
for some slowly varying functions $\ell_1, \ell_2$, then \eqref{eq:empirical-power-law} holds for any $\varepsilon>0$, since $z^{-\varepsilon}\ll \ell_1(z)\le \ell_2(z)\ll z^{\varepsilon}$ by Potter's theorem \cite{bingham1989regular}. Then $z_0$ may depend on $\varepsilon$. 
In one of our results below, 
we additionally require the following assumption on the maximum degree and the empirical mass function.
\begin{assumption}\label{assu:empirical-power-law-2}
We assume that there is an $\eps>0$ such that there exists constants $n_0(\eps), z_0(\eps), C_u>0$, such the empirical measure $\nu_n$ in \eqref{eq:empirical-degree} satisfies, for all $n>n_0(\eps)$,
\begin{align}
&\nu_n(z) \le \frac{C_u}{z^{\tau(1-\varepsilon)}} \quad\mbox{ for all } z\ge z_0(\eps),\label{eq:point-mass}\\
&\max_{i\le n} d_i \le C_u n^{1/(\tau(1-\eps)-1)}.\label{eq:max-degree}
\end{align}
\end{assumption}
The first condition implies the upper bound in Assumption \ref{assu:empirical-power-law}, since \eqref{eq:point-mass} implies that $\nu_n((z,\infty))\le \sum_{i\ge z} c_u i^{-\tau(1-\eps)}=c_u' z^{-(\tau-1)+\tau\eps}= c_u' z^{-(\tau-1)(1-\eps')}$ with $\eps':=\eps \tau/(\tau-1)$.  
The second condition is also quite natural, and both conditions hold  for the empirical measure of iid degrees whp, as the following example shows. The proof can be found on page \pageref{s:proof-ex-iid-degrees} in the Appendix.
\begin{example}[Iid degrees]\label{ex:iid-degrees} Suppose $\underline d_n=(D_{n,1}, \dots, D_{n,n}+\ind\{\sum_{i\le n}D_{n,i} \mbox{ odd} \})$ where $(D_{n,i})_{i\le n}$ are iid from a distribution $D$ satisfying Definition \ref{def:power-heavy} with some $\alpha$. Then $(\underline d_n)_{n\ge 1}$ with high probability satisfies Assumptions \ref{assu:regularity}, \ref{assu:empirical-power-law} with $\tau=\alpha+1$ and any $\varepsilon> 0$, and $z_{\max}^{(\ell)}(\eps, n)=n^{1/(\alpha(1+\eps))}$ in Assumption \ref{assu:empirical-power-law}, i.e., with $z_0(\eps/2)$ from Definition \ref{def:power-heavy},
\begin{equation}\label{eq:assu-11-holds}
\begin{aligned}
\P\left(\begin{array}{c}\forall z\ge z_0(\eps/2): 1-F_n(z) \le z^{-\alpha(1-\eps)} \mbox{ and }\\[0.1cm] \forall z\in[z_0(\eps/2), n^{1/(\alpha(1+\eps))}]: 1-F_n(z) \ge z^{-\alpha(1+\eps)} \end{array} \right) \to 1. 
\end{aligned}
\end{equation}
Further, $D$ satisfying Definition \ref{def:power-heavy} for some $\alpha$ implies that \eqref{eq:max-degree} holds whp with $\tau=\alpha+1$ and any $\eps>0$, i.e., $\P(\max D_{n,i} \le n^{1/(\alpha(1-\eps))}) \to 1$.
If $D$ satisfies also that for all $\eps>0$ there exists $z_0(\eps)$, such that for all $z\ge z_0(\eps)$,
\begin{equation}\label{eq:powerlaw-mass}
 \P(D=z) \le z^{-\tau(1-\eps)},
\end{equation}
then the empirical measure $\nu_n(z)$ of $\underline d_n$ also satisfies \eqref{eq:point-mass} with any $\eps>1/\tau$. That is, for all $\eps'>0$,
\begin{equation}\label{eq:iid-with-1/tau}
    \P\Big( \forall z\ge z_0(\eps): \nu_n(z) \le z^{-\tau(1-1/\tau +\eps)} = z^{-(\tau-1+\eps')}  \Big) \to 1.
\end{equation}
Finally, if one considers \emph{truncated} power-law distributions with $\max_{n,i} D_{n,i}=o(n^{1/\tau})$, then for all $\eps>0$ 
\begin{equation}\label{eq:iid-truncated}
    \P\Big( \forall z\ge z_0(\eps): \nu_n(z) \le z^{-\tau(1-\eps)} \Big) \to 1.
\end{equation}
\end{example}
While \eqref{eq:iid-with-1/tau} seems rather weak, it is essentially best possible. Namely, using the lower bound one can show that the vertices with maximal degree are of order $n^{(1+o(1))/(\tau-1)}$, and when there is a single vertex with degree in this range, then the upper bound in \eqref{eq:iid-with-1/tau} can be sharp. Examples on \emph{truncated} power-law degree distributions can be found in \cite[Example 1.20, 1.21]{van2017scale} where graph distances are discussed under truncation.  Here, as soon as the maximal degree is $o(n^{1/\tau})$, the true $\tau$ can be recovered also for point-masses with any $\eps>0$ in \eqref{eq:iid-truncated}.

\section{Results}

We focus on the behavior of degree-penalized CP and BRW for small values of $\lambda>0$. 
Table \ref{table:summary} contains a simplified summary of our results. We first state our results on the \emph{product penalty}, i.e., when $f(x,y)=(xy)^\mu$ for some $\mu>0$ in Definitions \ref{def:CP} and \ref{def:BRW}. We based this choice on a slightly related model, degree-dependent first passage percolation \cite{komjathy2021penalising}, 
where this penalty function is proven to show rich phenomena for first passage percolation.
Some of our results extend to \emph{polynomial} penalty functions as well, see Remark \ref{rem:polynomial-penalties} below. We start with results on Galton-Watson trees. On a Galton-Watson tree, the degree of a non-root vertex $v$ equals its number of offspring plus $1$.
\begin{theorem}[Product penalty with $\mu<1/2$ on Galton-Watson trees]\label{thm:prod_GW} Let $\mathcal{T}$ be an infinite Galton-Watson tree with offspring distribution $D$, so that $p_0=\P(D=0)=0$.
Consider the degree-penalized contact process $\CPf$ and branching random walk $\BRW$ with penalty function $f(x,y)=(xy)^\mu$ in Definitions \ref{def:CP} and \ref{def:BRW} for some $\mu\in(0,1/2)$.

 When the tail of $D$ is heavier than stretched-exponential with stretch-exponent $1-2\mu$ (in the sense of Definition \ref{def:stretched-heavy}), then for all $\la>0$, $\CPf(\CT, \ind_\varnothing)$ and $\BRW(\CT, \ind_\varnothing)$ both show \emph{local survival}, for almost  all realizations  $\CT$ of the Galton-Watson tree.
\end{theorem} 
The counterpart of this theorem for $\mu\ge 1/2$, holds more generally on \emph{any} graph $G$.
\begin{theorem}[Product penalty with $\mu\ge 1/2$]\label{thm:prod-general}
Consider the degree-penalized contact process $\CPf$ and branching random walk $\BRW$ with penalty function $f(x,y)=(xy)^\mu$ 
 in Definitions \ref{def:CP} and \ref{def:BRW} for some $\mu\ge1/2$. Then for all $\la< 1$, $\CPf(G, \underline{\xi}_0)$ and $\BRW(G, \underline{\xi}_0)$ both \emph{go extinct} almost surely on \emph{any} (finite or infinite) graph $G$ whenever $|\underline\xi_0|<\infty$ (respectively, $|\underline x_0|<\infty$) almost surely. Further, 
\begin{equation}\label{eq:mean-ext-time-product}
\mathbb E[T_{\mathrm{ext}}^{\mathrm{cp}}(G, \underline \xi_0)\mid G, \underline \xi_0]\le \mathbb E[T_{\mathrm{ext}}^{\mathrm{brw}}(G, \underline \xi_0)\mid G, \underline \xi_0]\le \sum_{v\in V} \xi_0(v) d_v^{1-\mu}/(1-\lambda)
\end{equation}
and $\P(T_{\mathrm{ext}}^{\mathrm{cp}}(G, \underline \xi_0)>t)$ and $\P(T_{\mathrm{ext}}^{\mathrm{brw}}(G, \underline \xi_0)>t)$ both decay (at least) exponentially in $t$ at rate at least $\la-1$. 
\end{theorem} 

Our next theorem is about the same processes on the configuration model. For the sake of simplicity, we assume that $\min_{i\le n}d_i\ge 3$, ensuring that for all sufficiently large $n$, $\mathrm{CM}(\underline d_n)$ on $n$ vertices has a giant component $\CC_n^{\sss{(1)}}$ containing $n(1-o(1))$ many vertices with probability that tends to $1$ as $n\to \infty$, see \cite{MolRee95, MolRee98}. We use the $O_{\P}, \Theta_{\P}$-notation in the standard way, see notation on page \pageref{sec:notation}. By $\mathrm{poly}(n)$ we denote polynomial functions of $n$ (with an arbitrary but finite exponent).

\begin{theorem}[Product penalty on CM] \label{thm:prod_CM}
Let $G_n:=\mathrm{CM}(\underline d_n)$ be the configuration model in Definition \ref{def:CM}  on the degree sequence 
$\underline d_n=(d_1, \dots, d_n)$. 
Consider the degree-penalized contact process $\CPf$ and branching random walk $\BRW$ with penalty function $f(x,y)=(xy)^\mu$ for some $\mu>0$ from Definition \ref{def:CP} and \ref{def:BRW}.
\begin{itemize}
\item [(a)] Let $\mu<1/2$, and $\underline d_n$ satisfy the regularity assumptions in Assumption \ref{assu:regularity} with $\min_{i\le n}d_i\ge 3$, so that $D$ has heavier tails than stretched-exponential with stretch-exponent $1-2\mu$ (in the sense of Definition \ref{def:stretched-heavy}). Then both $\CPf(G_n, \underline 1_{G_n} )$ and  $\BRW(G_n, \underline 1_{G_n})$  survive until $\Theta_{\P}(\exp(Cn))$ long time. 
\item[(b)] Let $\mu\ge 1/2$. Then for all $\la< 1$, both $\CPf(G_n, \underline 1_{G_n})$ and $\BRW(G_n, \underline 1_{G_n})$ go extinct in $O_{\P}(\mathrm{poly}(n))$ time, whenever it holds for $(\underline d_n)$ that $\sum_{i=1}^n d_i^{1-\mu} =O_{\P}(\mathrm{poly}(n))$.
\end{itemize}
\end{theorem}
Starting from the all-infected state on $G_n$ is not a serious restriction. In part (a), when started from a single vertex, i.e., $\underline \xi_0=\ind_v$, the process has a positive probability of reaching a large pandemic, and the same result -- long survival -- is valid with positive probability. See \cite{Sly19} on how to move between a single vertex and all vertices as starting states. 

\begin{remark}[Polynomial penalties]\label{rem:polynomial-penalties}\normalfont
The proof of Theorems \ref{thm:prod-general} and \ref{thm:prod_CM} (b) also work more generally for any penalty function $f_1(x,y)=x^{\mu}y^{\nu}$ with $\mu+\nu\ge 1$ under the same conditions, i.e., for all graphs $G$, whenever $\lambda<1$ and initial infected set $\xi_0$ is finite. 
It is also straightforward to extend the result from monomials to polynomials of the form 
\[ f_2(x,y)=\sum_{i\in \N} a_i x^{\mu_i}y^{\nu_i}\] with at least one term, say the first one,  satisfying $\mu_1+\nu_1\ge 1$, and all $a_i\ge 0$. In this case we can guarantee extinction whenever $\la<a_1$, using the stochastic domination of $\mathrm{CP}_{f_2, \la}$ by $\mathrm{CP}_{a_1 f_1, \la}=\mathrm{CP}_{f_1, \la/a_1}$, since the penalty is higher in process with $f_2$, leading to smaller infection rates, see  \eqref{eq:domination-between-two-f} below. The proof of Theorem \ref{thm:prod-general} also extends to processes with penalty function
\[ f_3(x,y) := 1\Big/\sum_{i\in \N} a_i x^{-\mu_i}y^{-\nu_i}, \quad \mbox{with}\quad \sum_{i\in \N} a_i<\infty\] whenever $(\mu_i,\nu_i)_{i\in \N}$ are such that and there is a \emph{unique dominant term} (say the first one) in the following sense: $\mu_1\le \mu_i$ and $\nu_1\le \nu_i$ for every $i\in \N$ and $\mu_1+\nu_1\ge 1$.  We then bound the infection rates from above as follows:
\[\la/f_3(d_u,d_v)=\lambda\sum_i a_i d_u^{-\mu_i}d_v^{-\nu_i}\le\lambda\Big(\sum_i a_i\Big) d_u^{-\mu_1}d_v^{-\nu_1}=\lambda\Big(\sum_i a_i\Big)\Big/f_1(d_u, d_v),\]
with $f_1(x, y)=x^{\mu_1}y^{\mu_1}$. So, using stochastic domination,  whenever $\lambda< \left(\sum_i a_i\right)^{-1}$, Theorem \ref{thm:prod-general} is still valid by the first part of the remark.
 \end{remark}

It turns out that --  instead of the product penalty -- switching to a class of penalty functions $f$ that are monomials of  $\max(x,y)$ shows a richer behavior, and we see an extra phase when $\mu$ crosses $1$.
\begin{theorem}[Max penalty on GW trees]
\label{thm:max_GW}
Let $\mathcal{T}$ be an infinite Galton-Watson tree with offspring distribution $D$, so that $\P(D=0)=0$.
Consider the degree-penalized contact process $\CPf$ and branching random walk $\BRW$ with penalty function $f(x,y)=\max(x,y)^\mu$ for some $\mu>0$ in Definitions \ref{def:CP} and \ref{def:BRW}.
\begin{itemize}
\item[(a)] Let $\mu<1/2$, and the tail of  $D$ be heavier than stretched-exponential with stretch-exponent $1-2\mu$, (in the sense of Definition \ref{def:stretched-heavy}). Then for all $\la>0$, the contact process $\CPf(\CT, \ind_\varnothing)$ and $\BRW(\CT, \ind_\varnothing)$ both show \emph{local survival}, for almost  all realizations  $\CT$ of the Galton-Watson tree.
\item[(b)] Let $\mu\in(1/2, 1)$, and for some $\alpha\in (0,1-\mu)$, the tail of $D$ weakly follow a power law with tail-exponent $\alpha$ (in the sense of Definition \ref{def:power-heavy}).  Then for all $\la>0$ small enough, $\CPf(\CT, \ind_\varnothing)$ and $\BRW(\CT, \ind_\varnothing)$ both show \emph{global survival}, for almost all realizations $\CT$ of the Galton-Watson tree. 
\item[(c)] Let $\mu\in(1/2, 1)$, and $\Ev[D^{1-\mu}]<\infty$.  Then for all $\la$ small enough, the processes $\CPf(\CT, \ind_\varnothing)$ and $\BRW(\CT, \ind_\varnothing)$ both \emph{go extinct} almost surely, for almost all realizations $\CT$ of the Galton-Watson tree. 
\end{itemize}
\end{theorem}
Note that $\alpha<1-\mu$ in part (b) means that $\E[D^{1-\mu}]=\infty$, and for power-law degrees with $\alpha>1-\mu$, we have $\E[D^{1-\mu}]<\infty$. In this sense part (b) and (c) are almost matching and we leave out only the case $\alpha=1-\mu$, where the (potentially present) slowly varying function multiplying the power-law decay shall play a decisive role in survival vs extinction (see below \eqref{eq:heavier-power-law}). To avoid technical difficulties of tail-estimates, we decided to leave out this boundary case. Part (c) above is also valid more generally, see Corollary \ref{cor:SST} below. To prove extinction, we develop a new technique that we call \emph{loop erasure of infection paths}, see Section \ref{sec:discussion}. Now we state the missing phases in Theorem \ref{thm:max_GW} above: When $\mu\ge 1/2$, we can show local extinction, and when $\mu\ge 1$, also global extinction much more generally, hence we state them separately as follows.
\begin{theorem}[Max penalty on trees and graphs]\label{thm:max_trees}
Let $\CT$ be any (possibly infinite) rooted tree with root $\varnothing$.
Consider the degree-penalized contact process $\CPf$ and branching random walk $\BRW$ with penalty function $f(x,y)=\max(x,y)^\mu$ for some $\mu>0$.
\begin{itemize}
\item[(a)] Let $\mu\ge 1/2$. Then for all $\la<1/2$, the processes $\CPf(\CT, \underline \xi_0)$ and $\BRW(\CT, \underline x_0)$ both show \emph{local extinction} almost surely, whenever $|\underline \xi_0|<\infty$ (resp., $|\underline x_0|<\infty$) almost surely. In this case we further have that for any $v\in \CT$, the tail-distributions of the local extinction times $T_{\mathrm{ext}}^{\mathrm{cp}}(\CT, \underline \xi_0,v)$, $T_{\mathrm{ext}}^{\mathrm{brw}}(\CT, \underline x_0,v)$  decay exponentially in $t$.
\item[(b)] Let $\mu\ge 1$. Then for all $\la< 1$, the contact process $\CPf(G, \underline \xi_0)$ and $\BRW(G, \underline x_0)$ both \emph{go extinct} almost surely on \emph{any} (finite or infinite) graph $G$  whenever $|\underline \xi_0|<\infty$ (resp., $|\underline x_0|<\infty$) almost surely, hence also on any tree $\CT$.  
Further, the bound \eqref{eq:mean-ext-time-product} is also valid here on the extinction times, which decay at least exponentially in $t$ with rate at least $\la-1$.
\end{itemize}
\end{theorem}   

Here, we prove Theorem \ref{thm:max_trees}(a) using again the \emph{loop erasure of infection paths} technique of Theorem \ref{thm:max_GW}(c).
It follows from the \emph{proof} of Theorems  \ref{thm:max_GW}(c) and Theorem \ref{thm:max_trees}(a) that (local-global) extinction for small $\lambda>0$ happens on any tree with at most exponential growth. Recall the upper branching number $\overline{\mathrm{br}}(\calT)$ from Definition \ref{def:branching-number}.
\begin{corollary}[Trees with finite branching number]\label{cor:SST}
Let $\CT$ be a rooted tree with  $\overline{\mathrm{br}}(\calT):=b<\infty$, and consider $\CPf$ and $\BRW$ on $\CT$ with penalty function $f(x,y)=\max(x,y)^\mu$ with $\mu \ge 1/2$. 
Then for all $\la<b^{-1}/2$, the processes $\CPf(\CT, \ind_\varnothing)$ and $\BRW(\CT, \ind_\varnothing)$ both go extinct almost surely.

Let $\CT$ be a spherically symmetric tree with with degree sequence $\underline d=(d_0, d_1, d_2, \dots)$ satisfying $\overline{\mathrm{br}}(\CT):=b<\infty$. Then 
	 for all~$\lambda < b^{-(1-\mu)}/2$, the processes $\CPf(\CT, \ind_\varnothing)$ and $\BRW(\CT, \ind_\varnothing)$ both go extinct almost surely.
\end{corollary} 
For spherically symmetric trees, finiteness of the upper branching number $\overline{\mathrm{br}}(\CT)$ is equivalent to requiring that $\log \overline{\mathrm{br}}(\CT)$ $=\limsup_{N \to \infty} \frac{1}{N} \sum_{i=1}^N \log(d_i) < \infty$. The requirement on $\lambda$ in Corollary \ref{cor:SST} for SST's is slightly milder than for arbitrary trees with finite upper branching number.
Our last theorems describes the  behavior of degree-penalized processes with maximum penalty on the configuration model. 
\begin{theorem}[Max penalty on CM, long survival regimes] \label{thm:max_CM_survival}
Let $G_n:=\mathrm{CM}(\underline d_n)$ be the configuration model in Definition \ref{def:CM}  on the degree sequence 
$\underline d_n=(d_1, \dots, d_n)$ that satisfies the regularity assumptions in Assumption \ref{assu:regularity}. Consider the degree-penalized contact process $\CPf$ and branching random walk $\BRW$ with penalty function $f(x,y)=\max(x,y)^\mu$.
\begin{itemize}
\item [(a)] Let $\mu<1/2$, and the tail of $D$ be heavier than stretched-exponential with stretch-exponent $1-2\mu$ (in the sense of Definition \ref{def:stretched-heavy}), and $\min_{i\le n}d_i\ge 3$. Then for all $\lambda>0$ the process $\CPf(G_n, \underline 1_{G_n} )$  survives until $\Theta_{\P}(\exp(Cn))$ long time. 
\item[(b)] Let $\mu\in(1/2, 1)$, and $(\underline d_n)_{n\ge 1}$ satisfy the power-law empirical degree Assumption \ref{assu:empirical-power-law} with exponent $\tau$ and exponent-error $\varepsilon\ge 0$, with
\begin{equation}\label{eq:max_CM_survival_error}
\mu<\big(3-\tau -\varepsilon(\tau-1)\big)\cdot\frac{1-\varepsilon}{1+\varepsilon}.
\end{equation}
Then for all $\la>0$ the process $\CPf(G_n, \underline 1_{G_n})$ survives until $\Theta_{\P}(\exp(Cn))$ long time. 
\end{itemize}
\end{theorem}
As the error in the power-law exponent $\varepsilon\downarrow 0$, the condition in \eqref{eq:max_CM_survival_error} simplifies to $\mu<3-\tau$, which is equivalent to the condition that $\alpha:=\tau-2<1-\mu$. Here $\alpha=\tau-2$ is the tail-exponent of the size-biased version of $D$, say $\widetilde D$, which can be shown to weakly follow a power law with $\alpha=\tau-2>0$ in the sense of Definition \ref{def:power-heavy}.  The local weak limit of the configuration model is a Galton-Watson tree with a version of the size-biased degree distribution $\widetilde D$. Theorem \ref{thm:max_GW}(b) describes that when $\mu\in(1/2, 1)$, on a weak power-law GW tree the processes both survive globally exactly when $\alpha<1-\mu$. Hence,    
this theorem reflects  the analogous  Theorem \ref{thm:max_GW}(b) on Galton-Watson trees, showing that global survival (but local extinction) there implies long survival for the corresponding configuration model.
\begin{theorem}[Max penalty on CM, fast extinction regimes]\label{thm:max_CM_extinction}
Consider the configuration model $G_n:=\mathrm{CM}(\underline d_n)$ in Definition \ref{def:CM}  on the degree sequence 
$\underline d_n=(d_1, \dots, d_n)$. Consider the degree-penalized contact process $\CPf$ and branching random walk $\BRW$ with penalty function $f(x,y)=\max(x,y)^\mu$.
\begin{itemize}
\item[(a)] Let $\mu\in(1/2, 1)$, and  $(\underline d_n)_{n\ge 1}$ satisfy the regularity assumptions in Assumption \ref{assu:regularity}, and the power-law empirical degrees of Assumption \ref{assu:empirical-power-law}--\ref{assu:empirical-power-law-2} with exponent $\tau$ and exponent-error $\varepsilon\ge 0$ with $\tau(1-\eps)>3$. Then for all $\la$ small enough the processes $\CPf(G_n, \underline 1_{G_n})$ and $\BRW(G_n, \underline 1_{G_n})$ both go extinct in $\Theta_{\P}(\log n)$ time. 
\item[(b)] Let $\mu\ge 1$. Then for all $\la< 1$, the processes $\CPf(G_n, \underline 1_{G_n})$ and $\BRW(G_n, \underline 1_{G_n})$ both go extinct in $O_{\P}(\mathrm{poly}(n))$ time, whenever it holds for $(\underline d_n)$ that $\sum_{i=1}^n d_i^{1-\mu}=O_\P(\mathrm{poly}(n))$. 
\end{itemize}
\end{theorem}
Theorem \ref{thm:max_CM_extinction}(a) is the counterpart of Theorem \ref{thm:max_CM_survival}(b), i.e., it shows fast extinction on the configuration model with power-law degrees with sufficiently light tail. For long survival, Theorem \ref{thm:max_CM_survival}(b) essentially requires $\mu<3-\tau$, equivalently, $\tau>3-\mu$.  Here in Theorem \ref{thm:max_CM_extinction}(a) to prove extinction we need essentially $\tau>3$, i.e., we leave the cases when $\tau\in( 3-\mu, 3)$ open. The reason for this is a structural difference between configuration models with $\tau\in(2,3)$ vs $\tau>3$: when $\tau\in(2,3)$, the Galton-Watson tree forming the local weak limit of the configuration model grows doubly-exponentially, and can be embedded into the configuration model only until $\Theta(\log \log n)$ generations, and with many surplus edges (i.e., edges beyond the number of vertices$-1$ that form the tree). However, we show that when $\tau>3$, the local weak limit GW tree can be embedded until $\Theta(\log n)$ generations and with only a \emph{bounded number of surplus edges} for all $n$ vertices all-at-once, see Proposition \ref{prop:surplus}, which might be interesting in its own right. We can then relate extinction of the CP/BRW on this new structure using our methodology of loop erasure (see below in Section \ref{sec:discussion}) so that CP/BRW never reaches the last generation. However, for $\tau\in (3-\mu, 3)$, on the one hand the $\Theta(\log \log n)$ generations of the embedding are too short and leave a good probability for CP/BRW to escape the embedded tree, and on the other hand there are too many additional cycles on the embedded tree that might boost the performance of CP/BRW.

\subsection{Background, discussion and overview of proof techniques}\label{sec:discussion}
In the following we highlight our novel proof techniques and their relation to the literature. The overview follows the structure of the rest of the paper.

\emph{Novel methodology: loop erasure in the space of infection paths (Sections \ref{sec:proofsGW} and \ref{sec:proofsCM-extinction}).} In Sections \ref{sec:max_GW_localext}, \ref{sec:max_GW_ext_light} for the proof of Theorems \ref{thm:max_GW}(c) and \ref{thm:max_trees}(a) we develop a new recursive path counting argument on the space of infection paths, where we essentially carry out a (probability-weighted) loop erasure on the set of possible infection paths. Then we relate the probability that $\BRW$ survives on $\CT$ to the product of degrees $\prod_{i=1}^t d_{\pi_i}^{1-\mu}$ summed over non-backtracking paths, called rays $\pi=(\pi_0=\varnothing, \pi_1, \pi_2, \dots, )$ on the tree, i.e., paths that always go downwards.

To extend the same result to the configuration model, i.e. to prove Theorem \ref{thm:max_CM_extinction}(a) (in Section \ref{sec:proofsCM-extinction}), we need to handle loops in the underlying graph. First in Lemma \ref{lemma:moments} we develop a new moment bound for the total size of GW trees with power-law offspring distribution with $n$-dependent maximum degree, (i.e. coming from the empirical degrees of the configuration model) valid for all $\tau>3$. We use this new bound to show that whp the following holds for configuration models with $\tau>3$ on $n$ vertices: for some small $\delta>0$, the $\delta \log n$ graph-neighborhood of every vertex only has at most a constant $\ell$ many \emph{surplus edges}, i.e., upon removing at most $\ell$ vertices the $\delta \log n$ neighborhood becomes a tree. This result, Proposition \ref{prop:surplus}, may be of independent interest. Returning to the degree-penalized contact process on the configuration model, we extend the (probability weighted) loop-erasure method that we developed for trees, to graphs with a bounded number of surplus edges, which is a non-trivial adaptation itself.

\emph{Survival on GW-trees with stretched-exponential-tailed offspring (Section \ref{sec:survival_GW}).} Theorem \ref{thm:prod_GW} is the analog of the result by Huang and Durrett \cite{huang2020contact}, where the authors show that the classical contact process shows local survival on Galton-Watson trees whenever the offspring distribution has no exponential moments, i.e., for all $c>0$, it holds that $\E[\e^{cD}]=\infty$. For the degree-penalized versions, due to the penalties, the same condition is not sufficient for the proofs to carry through. For our proofs to hold, we need that $D$ has heavier tails than stretched exponential with stretch exponent that is strictly less than $1-2\mu$, as in Definition \ref{def:stretched-heavy}. We leave it an open question whether this condition in Theorem \ref{thm:prod_GW} is sharp. For the classical contact process on Galton-Watson trees, the all-exponential-moments-infinite condition is sharp, as shown by Bhamidi, Nam, Nguyen and Sly \cite{Sly19}.

The combination of Theorems \ref{thm:prod_GW} and \ref{thm:prod-general} shows that the product penalty has a phase transition at $\mu=1/2$. The usual argument that star-graph maintain the infection, as introduced by Chatterjee and Durrett \cite{chatterjee2009contact}, gives a back-of-the-envelope calculation that suggests this phase transition.
Namely, a star-graph has a central vertex of degree say $K$, connected to $K$ leaves or very low-degree vertices.  The degree-penalized contact process on this structure survives typically for a time that is $\Omega_\P(\exp(\lambda^2 K^{1-2\mu}))$. Hence, whenever $1-2\mu>0$, star-graphs survive long enough to infect other star-graphs embedded in the graph, provided these stars are not too far away from each other, i.e., at most the logarithm of the survival time, giving at most $o(K^{1-2\mu})$ away. The stretched-exponential condition on the tail of $D$ ensures that we can find stars within this distance of each other. For the infection to be able to pass between the stars, we also need to ensure that the path connecting the stars only contain low-degree vertices, so that the penalty does not hinder the infection from passing. This is new compared to the classical contact process, see Section \ref{sec:prod_GW_strong}.

\emph{Local extinction and global survival for small $\la$ on power-law GW-trees.} The combination of Theorems \ref{thm:max_GW} and \ref{thm:max_trees} shows that for the \emph{max-penalty} when $\mu\in(1/2, 1)$, on a Galton-Watson tree, local extinction but global survival happens for any small $\lambda>0$ and $D$ has a power-law tail with tail-exponent $\alpha<1-\mu$. The behavior for large rates ($\la>1$) may depend on the exact offspring distribution, and the contact process and the branching random walk may differ in behavior, see  the work of Pemantle and Stacey \cite{pemantle2001branching}. 
Comparing Theorems \ref{thm:max_GW} and \ref{thm:max_trees} for the max-penalty with the corresponding Theorems \ref{thm:prod_GW} and \ref{thm:prod-general} for the \emph{product} penalty, we see that the phase of $\mu\ge 1/2$ for the max-penalty is subdivided into three different sub-phases, and the almost-sure extinction on arbitrary graphs requires $\mu\ge1$ for the max-penalty, c.f. $\mu\ge 1/2$ for the product penalty. The subphases of $\mu\in[1/2,1)$ (Theorem \ref{thm:max_GW} part (b)--(c)) are novel, since they provide the first natural static graph model where the contact process on power-law degree graphs \emph{can} be subcritical (c) and show \emph{only global survival} (b); and the exact condition also depends on the exact power-law exponent.  For dynamical graphs a similar phenomenon occurs, see the recent work of Jacob, Linker and M\"orters \cite{jacob2022contact}.

\emph{Survival proofs: k-cores sustain the infection when stars heal quickly (Section \ref{sec:proofsCM}).} When $\mu\ge 1/2$, in the degree-penalized contact process, star-graphs heal essentially immediately and hence the usual arguments that they maintain the infection for a long time break down. In this regime on the GW tree, when the offspring distribution is sufficiently heavy-tailed (so that the $1-\mu-\eps$th moment is infinite for some $\eps>0$), we prove that  contact process shows local extinction but global survival by escaping to infinity, by Theorem \ref{thm:max_trees}(a) and Theorem \ref{thm:max_GW}(b). 

In the configuration model with the same local weak limit, we find a new sub-graph that maintains the infection exponentially long in $n$. This is a $k$-core $H_n\subseteq G_n$ that we show exists on vertices with degree $k^{(1+ \eta)/(3-\tau)}$, with size linear in $n$, for some small $\eta=\eta(\varepsilon)$.
 We prove that such a $k$-core is always present whp whenever $\tau\in(2,3)$, using the results of Janson and Luczak \cite{janson2007simple}.
 The heuristic idea is that within $H_n$, the expected number of  vertices that an infected vertex infects before healing is (ignoring the $\eta$ error in the exponent):
 \[\deg_{H_n}(u) r(u,v) = \deg_{H_n}(u) \la(\deg_{G_n} (u)\vee \deg_{G_n}(v))^{-\mu} \approx k \la  k^{-\mu/(3-\tau)}\approx \la k^{1-\mu/(3-\tau)},\]
which grows with $k$ whenever $\mu<3-\tau$. We then show that  when we choose $k$ a large $\la$-dependent constant, the graph $H_n$ sustains the contact process exponentially long.
As far as we know this is the first model where $k$-cores are directly used to maintain the infection process.

\emph{Long survival on the configuration model with stretched exponential degree distribution (Section \ref{sec:embedded_stars_GW}).} In the regime where $\mu<1/2$, a star-graph of degree $j$ maintains the infection long enough to pass it to a neighboring star-graph if the graph-distance between them is $o(j^{1-2\mu})$. This idea will lead to Theorem \ref{thm:prod_CM} (a) and, as a consequence, Theorem \ref{thm:max_CM_survival} (a). Our proof here is an almost direct adaptation of the argument in \cite{Sly19} where we embed an expander-graph of stars with degree approximately $j$ into the original graph so that each edge of the expander corresponds to a path of length  $o(j^{1-2\mu})$. This leads to the condition of heavier than stretched exponential degree distributions with the exponent at most $1-2\mu$.

\emph{Another model with degree-dependent transmission rates.} 
Wei Su in \cite{su14} studies a degree-penalized contact process and branching random walk with the \emph{asymmetric} penalty function $f(x,y)=x$. This penalty function implies that the total rate of infection from every vertex $v$ is a constant $\la>0$, irrespective of the degree of $v$. In this case, CP can be coupled to a ``usual'' un-penalized BRW on the GW tree with Poisson($\lambda$) total offspring, and finer results can be obtained on Galton-Watson trees, not just the small $\lambda>0$ behavior. For BRW, extinction occurs when $\la<1$, and  local vs.\ only global survival depends on whether $\la>1/ r(\CT)$ or not, where $r(\CT)$ is  the spectral radius of the underlying tree with respect to symmetric random walk. For the contact process, the minimal degree in the Galton-Watson tree is decisive, see \cite[Theorems 3.1, 4.2]{su14}. 

\emph{Further directions.} We believe that most of our results can be relatively easily adapted to graphs with GW trees as local weak limits, e.g.\ the Chung-Lu  or Norros-Reitu models or even to general inhomogeneous random graphs \cite{BolJAnRio07, Chung_Lu, NorRei04}. Our current proof techniques pose the restriction that they all rely on tree-based arguments or ``almost'' tree-based arguments. 
It would be interesting to see how far this can be relaxed. Sparse random intersection graphs \cite{Bloz10, Bloz13, DeiKet09, KarSchSin99, Sing96} or random intersection graphs with communities (where not every community is a complete graph \cite{van2021random, van2022phase}) provide a natural candidate for this. These graphs are no longer locally tree-like, yet there is an embedded tree-like structure formed by the communities \cite{van2021random}. Another interesting direction is to develop robust techniques that can extend our results (beyond the $\mu\ge 1$ case) to spatial graphs with inhomogeneous degree distributions, for instance to geometric inhomogeneous random graphs \cite{bringmann2019geometric}, scale-free percolation \cite{deijfen2013scale},  or the hyperbolic random graph \cite{krioukov2010hyperbolic}. A coupling argument to the related degree-dependent first passage percolation \cite{komjathy2021penalising}, which explodes also exactly when $\alpha:=\tau-2<1-\mu$, indicates that at least Theorem \ref{thm:max_GW}(b) on global survival must carry through for these graphs. Considering the recent growth phases of degree-dependent first passage percolation (1-FPP) in \cite{komjathy2023four1, komjathy2023four2}, it is an intriguing question to ask whether the front of the degree-dependent contact process started from the origin and conditioned to survive, follows  the same universality classes of growth as the 1-FPP spreading process. 

Metastable behavior of the original contact process on finite graphs is a lively field of research starting with \cite{cassandro1984metastable}; see also \cite{durrett1988contact, mountford2016exponential, mountford1999existence,    mountford1993metastable, schapira2017extinction, schonmann1985metastability}. See \cite{berger2005spread, can2015metastability, chatterjee2009contact,  mountford2013metastable, MVY2013} for results on power-law preferential attachment models and configuration models, \cite{linker2021contact} on hyperbolic random graphs, \cite{da2021contact, jacob2019metastability, jacob2022contact} on dynamically evolving graphs, and \cite{gracar2022contact} on spatial random graphs. Further studying metastability of the degree-penalized processes here (for instance, investigating metastable densities) is an interesting future direction.

\textbf{Organization of the rest of the paper}: Before the proofs we introduce some necessary terminology and preliminary facts about the contact process and branching random walks in Section \ref{sec:prelim}. Then, in Section \ref{sec:proofsGW} we give the proofs of Theorems \ref{thm:prod-general}, \ref{thm:prod_CM}(b), \ref{thm:max_GW}(c), \ref{thm:max_trees}(a), (b) and \ref{thm:max_CM_extinction}(c). In Section \ref{sec:proofsCM-extinction} we prove Theorem \ref{thm:max_CM_extinction}(a). Section \ref{sec:survival_GW} contains the proofs of Theorems \ref{thm:prod_GW} and \ref{thm:max_GW}(a), (b). In Section \ref{sec:proofsCM} we provide the proof of Theorem \ref{thm:max_CM_survival}(b). Finally, in Section \ref{sec:embedded_stars_GW} we give a sketch of the proofs of Theorems \ref{thm:prod_CM}(a) and \ref{thm:max_CM_survival}(a).

\textbf{Notation}:\label{sec:notation} When we compare degrees of vertices in graphs on the same vertex set, we use the notation $\deg_G(v)$ for the degree of vertex $v$ within graph $G$. Unless specified, we always think of graphs as undirected.
With a slight abuse of notation, we use $|G|$ as a shorthand for $|V(G)|$, the number of vertices in $G$.

We use the abbreviations `rhs' and `lhs' for `right-hand side' and `left-hand side' (of an equation), `iid' for `independent and identically distributed'
and `whp' for `with high probability', i.e., with probability converging to 1 as 
the size of the underlying graph (the number of its vertices) tends to infinity. 
For a deterministic function $g(n)$, we say that a sequence of random variables $X_n=o_\P(g(n))$,
if the sequence $(X_n/g(n))_{n\ge 1}$ tends to $0$ in probability, 
and we say that  $X_n=O_\P(g(n))$  if $(X_n/g(n))_{n\ge 1}$  is a tight sequence of random variables. 
Similarly,  $X_n=\Omega_\P(g(n))$ if  $(g(n)/X_n)_{n\ge 1}$ is a tight sequence, 
and finally,  we say that $X_n=\Theta_\P(g(n))$ if $X_n=O_\P(g(n))$ and $X_n=\Omega_\P(g(n))$ both hold.

%% file: preliminaries.tex
\section{Preliminaries}\label{sec:prelim}

In this section we describe some basic properties of the contact process and the underlying random graphs that will be used throughout the paper.

\subsection{Graphical representation of the contact process} We briefly discuss the \emph{graphical representation} of the contact process, based on Section 6.2 of \cite{grimmett2018probability}.  The graphical representation is useful for various coupling arguments.  The idea is to record the infection and healing events of the contact process $\CPf(G,\underline\xi_0)$ on the space-time domain $V\times [0,\infty)$. For a Poisson point process $\mathrm{PPP}$ on $[0,\infty)$, we say that $t\in \mathrm{PPP}$ if $t$ is an arrival time (a point) in the given $\mathrm{PPP}$. Further, $\mathrm{PPP}(I)$ denotes the set of points that fall in the set $I\subseteq R$.
\begin{definition}[Graphical representation of CP]\label{def:graphical}
Consider for each $v\in V$ an independent Poisson process $\PPP_v$ with rate 1, and, independently of these, further independent Poisson processes $\PPP_{uv}$ for each $u,v\in V$ with corresponding rate $r(u,v)=\la \cdot e(u,v)/f(d_u,d_v)$. The healing events in \eqref{eq:healing-CP} form a subset of the arrival times of $(\PPP_v)_{v\in V}$, and the infection events in \eqref{eq:infection-rate-CP} form a subset of the arrival times of $(\PPP_{uv})_{\{u,v\}\in E}$ that we describe now.

An \emph{infection path} is a sequence $\{(v_0,t_0),(v_0,t_1),(v_1,t_1),(v_1,t_2),\ldots,(v_k,t_{k+1})\}$ with vertices $v_0,v_1,\ldots,v_k\in V$ and times $t_0\le t_1\le\ldots\le t_{k+1}$ such that
\begin{enumerate}
\item[(i)] $\PPP_{v_i}([t_i,t_{i+1}])=\emptyset$ for each $i\in\{0,\ldots,k\}$, and
\item[(ii)] $t_i\in\PPP_{v_{i-1}v_i}$ for each $i\in\{1,\ldots,k\}$.
\end{enumerate}
Then, a vertex $u\in V$ is infected at time $t\ge 0$ (equivalently, we set $\xi_t(u)=1$), if there is an \emph{infection path} in $V\times[0,\infty)$ from $(v,0)$ to $(u,t)$ for some $v\in\underline\xi_0$.
\end{definition}
It is straightforward to see that this procedure encodes the contact process $\CPf(G,\underline\xi_0)$. This representation is useful for coupling contact processes with different initial conditions and different spreading rates. The following is an easy consequence of the graphical representation.

\begin{corollary}\label{cor:StochDom1}
For two penalty functions $f_1, f_2$ for which $f_1(x,y)\ge f_2(x,y)$ holds for all $x,y \ge 1$, it holds on any graphs $G$ and arbitrary initial starting state $\underline \xi_0$ and any $\la>0$ that 
\begin{equation}\label{eq:domination-between-two-f}
    \mathrm{CP}_{f_1,\la}(G,\underline \xi_0) \  {\buildrel d \over \le }\ \mathrm{CP}_{f_2,\la}(G,\underline \xi_0).
\end{equation}
\end{corollary}
 The stochastic domination in \eqref{eq:domination-between-two-f} is the consequence of a standard coupling argument: construct the graphical representation of $\mathrm{CP}_{f_2,\la}(G,\underline \xi_0)$, i.e., of the process with higher infection rates $(\la/f_2(u,v))_{u,v}$. Then, independently for different pairs $uv$, on $\mathrm{PPP}_{uv}$, keep every infection event (point) with probability $(\la/f_1(u,v)) / (\la/f_2(u,v) )  = f_2(u,v)/f_1(u,v)$, independently across points. The thinned PPP has rate $\la f_1(u,v)$, hence we obtain a graphical representation of $\mathrm{CP}_{f_1,\la}(G,\underline \xi_0)$. This joint realization of the two processes gives a coupling of $\mathrm{CP}_{f_1,\la}(G,\underline \xi_0)$ and $\mathrm{CP}_{f_2,\la}(G,\underline \xi_0)$, so that every infection event in the former process is also an infection event in the latter process. This finishes the proof of \eqref{eq:domination-between-two-f}. 
\vskip1em

\subsection{Genealogic branching random walks} \label{s_gbrw}
We now describe a construction of branching random walks that keeps track not only of the number of particles per site, but also of the genealogy of particles. This will be useful for proofs to show extinction, which are based on counting particles with given genealogies. Recall that for two vertices~$u$ and~$v$ in a graph~$G$, we write~$\mathrm{e}(u,v)$ to denote the number of edges between~$u$ and~$v$.

\begin{definition}[Set of genealogical labels]\label{def:labels} Given a graph $G=(V, E)$, we let~$\mathscr{T} = \mathscr{T}(G)$ be the set
	\[\mathscr{T}:=  \{(u_0,\ldots,u_m): m \in \mathbb{N},\; u_0,\ldots, u_m \in V,\; e(u_i,u_{i+1}) > 0 \text{ for all }i\}.\]
\end{definition}
An element~$\pi = (u_0,\ldots,u_m) \in \mathscr{T}$ will be a genealogical label attributed to certain particles that occupy~$u_m$, the final vertex in the sequence. More specifically, a particle occupying~$u_m$ receives label~$\pi$ if it has the following genealogical history: its oldest ancestor particle (present at time 0) was at~$u_0$ and gave birth to its next ancestor particle at~$u_1$, which then gave birth to its next ancestor particle at~$u_2$, ..., which then gave birth to the particle in question, at~$u_m$.
Hence, the label~$\pi$ lists the vertices occupied by the ancestors of the particle (and the particle itself), in chronological order. In particular, a particle present at vertex~$v$ at time~$0$ receives the label~$(v)$.

For~$\pi =(u_0,\ldots,u_m) \in \mathscr{T}$, we define
\begin{equation}\label{eq:length-pi} \begin{array}{ll}
	\mathfrak{l}(\pi) := m& \text{(length of $\pi$)},\\
	\mathfrak{s}(\pi) := u_m&\text{(end-location of~$\pi$)}.
\end{array}
\end{equation}
In case~$m > 1$, we also let
\[ \begin{array}{ll}
	\mathfrak{p}(\pi) := (u_0,\ldots,u_{m-1})&\text{(parent path of $\pi$)}.
\end{array}\]

\begin{definition}[Degree-penalized genealogic branching random walk]\label{def_gbrw}
	Consider a graph $G=(V,E)$, with $d_v$ denoting the degree of vertex $v\in V$. Let $f(x,y)>1$ be a function of two variables and~$\lambda>0$; for $u,v\in V$ let $r(u,v)=\lambda \cdot \mathrm{e}(u,v)/f(d_u, d_v)$.  Also let~$\underline x_0\in \mathbb{N}^V$.
	We define  $\mathrm{GBRW}_{f,\lambda}(G,\underline{x}_0) = (\underline{y}_t)_{t \ge 0} = ({y}_t(\pi))_{\pi \in \mathscr{T},t \ge 0}$ to be the following continuous-time Markov process on the state space $\mathbb{N}^\mathscr{T}$. The process starts at time~$t=0$ from the state $\underline y_0$ defined by
\[y_0(\pi) = \begin{cases} x_0(\mathfrak{s}(\pi))&\text{if } \mathfrak{l}(\pi) = 0;\\ 0&\text{otherwise,}\end{cases}\]
and evolves according to the following transition rates:
\begin{align}
\underline{y} \longrightarrow \underline{y} - \ind_{\pi} &\quad \text{ with rate } {y}(\pi) \text{ for all } \pi \in \mathscr{T};\label{eq:pi-rate-down}\\
	\underline{y} \longrightarrow \underline{y} + \ind_{\pi} &\quad \text{ with rate } {y}(\mathfrak{p}(\pi))\cdot r(\mathfrak{s}(\mathfrak{p}(\pi)),\mathfrak{s}(\pi)) \text{ for all } \pi \in \mathscr{T} \text{ with } \mathfrak{l}(\pi) \ge 1. \label{eq:pi-rate-up}
\end{align}
\end{definition}

We interpret~$y_t(\pi)$ as the number of particles with label~$\pi$ at time~$t$. Guided by this interpretation, it is easy to see that we can obtain a degree-penalized branching random walk from~$\mathrm{GBRW}_{f,\lambda}$, as stated in the following lemma.
\begin{lemma}\label{lem_corr_gbrw}
Let~$(\underline{y}_t)_{t \ge 0} = \mathrm{GBRW}_{f,\lambda}(G,\underline{x}_0)$, and define 
\begin{equation} \label{eq:projection}{x}_t(v) = \sum_{\pi \in \mathscr{T}:\ \mathfrak{s}(\pi)=v} {y}_t(\pi),\quad t > 0,\;v \in V.\end{equation}
	Then,~$(\underline{x}_t)_{t \ge 0} = (x_t(v))_{v \in V,t \ge 0}$ is a degree-penalized branching random walk on~$G$ with rate~$\lambda$, penalization function~$f$, and initial configuration~$\underline{x}_0$.
\end{lemma}
\begin{proof}
Let~$(\underline{x}_t)_{t \ge 0}$ be the process obtained from~$(\underline{y}_t)_{t \ge 0}$ as in~\eqref{eq:projection}. For each~$v \in V$,  the transition~$\underline{x} \longrightarrow \underline{x} - \ind_v$ occurs with rate
\[\sum_{\pi \in \mathscr{T}:\ \mathfrak{s}(\pi) = v} {y}(\pi) = {x}(v),\]
and the transition~$\underline{x} \longrightarrow \underline{x} + \ind_v$ occurs with rate
\[\sum_{\substack{\pi \in \mathscr{T}:\ \mathfrak{s}(\pi) = v,\\ \mathfrak{l}(\pi)\ge 1}} y(\mathfrak{p}(\pi)) \cdot r(\mathfrak{s}(\mathfrak{p}(\pi)),\mathfrak{s}(\pi)) = \sum_{w \in V}\sum_{\substack{\pi' \in \mathscr{T}:\\\mathfrak{s}(\pi')=  w}} y(\pi')\cdot  r(w,v) = \sum_{w \in V} x(w)\cdot r(w,v).\]
where we used \eqref{eq:projection} to obtain the last equality.
\end{proof}

In the statement of the following lemma, we interpret products of the form~$\prod_{i=0}^{m-1}$ as~$1$ when~$m = 0$.
\begin{lemma}[Expectation formulas for genealogic branching random walks]\label{lem_products}
	Let~$(\underline{y}_t)_{t \ge 0} = \mathrm{GBRW}_{f,\lambda}(G,\underline{x}_0)$, and let~$\pi = (u_0,\ldots,u_m) \in \mathscr{T}$.
	\begin{itemize}
		\item[(a)] For any~$t \ge 0$, we have
			\begin{equation}\label{eq:ypi-expectation}
				\mathbb{E}[y_t(\pi)] =\frac{t^m}{m!}\mathrm{e}^{-t} \cdot\Big(x_0(u_0) \prod_{i=0}^{m-1} r(u_i,u_{i+1})\Big).
			\end{equation}
		\item[(b)] Define
\begin{equation}\label{eq:Zpi-def}
	Z(\pi):= \begin{cases} x_0(u_0)&\text{if } \pi = (u_0);\\[.2cm] \#\{t > 0: y_t(\pi) = y_{t-}(\pi)+1\}&\text{otherwise,}\end{cases}
\end{equation}
			that is, in case~$\pi$ has length zero (so that~$\pi= (u_0)$),~$Z(\pi)$ is the number of initial particles~$x_0(u_0)$, and in case~$m=\mathfrak{l}(\pi) \ge 1$,~$Z(\pi)$ is the number of particles with label~$\pi$ ever born. Then, 
\begin{equation}\label{eq:Zpi-expectation}
\mathbb{E}[Z(\pi)] = z(\pi)= x_0(u_0) \prod_{i=0}^{m-1} r(u_i,u_{i+1}). 
\end{equation}
	\end{itemize}
\end{lemma}
Before the proof we mention that the factor $\mathrm{e}^{-t} t^m/m!$ is the density of a Gamma random variable with parameters $1$ and $m+1$, i.e., the convolution of $m+1$ iid Exp$(1)$ random variables. Intuitively this factor comes from the convolution of the healing times of the $m+1$ vertices on the path $u_0, \dots, u_m$.

\begin{proof}
	 
		Proof of part (a).\  We argue by induction in~$m = \mathfrak{l}(\pi)$. In case~$m=0$, we have~$\pi = (u_0)$, and the process~$(y_s(\pi))_{s \ge 0}$ is a continuous-time Markov chain that starts at~$y_0(\pi) = x_0(u_0)$ at time 0 and can only decrease, doing so with rate~$y_s(\pi)$ at any time~$s \ge 0$. If we interpret the state of this chain as a number of particles, where each particle dies with rate 1 (and no particles are born), then the probability that a particle is still alive at time~$t$ is~$\mathrm{e}^{-t}$, so the expected number of living particles at time~$t$ is~$x_0(u_0)\mathrm{e}^{-t}$, as desired.

			Now assume that~$m \ge 1$ and the statement in \eqref{eq:ypi-expectation} holds for all $\pi'\in \mathscr{T}$ with $\mathfrak{l}(\pi')\le m-1$. Let
			\[\pi_0 = (u_0),\quad \pi_1 = (u_0,u_1),\quad \ldots, \quad  \pi_m = \pi = (u_0,\ldots,u_m),\]
			and let~$\mathcal{F}$ be the~$\sigma$-algebra generated by
			\[(y_s(\pi_i): 1 \le i \le m-1,\; s \ge 0\}.\]
			Conditioned on~$\mathcal{F}$, the process~$(y_s(\pi))_{s \ge 0}$ is an~$\mathbb{N}$-valued (time-inhomogeneous) Markov process that starts at~$0$ at time~$0$ and, at any time~$s \ge 0$, increases by $1$ with rate
   \[y_s(\pi_{m-1})r(\mathfrak{s}(\mathfrak{p}(\pi)),\mathfrak{s}(\pi))  = y_s(\pi_{m-1})r(u_{m-1},u_m),\]
   and decreases by $1$ with rate~$y_s(\pi)$, by \eqref{eq:pi-rate-up} and \eqref{eq:pi-rate-down}. Again seeing this process as counting particles (which as before die with rate 1, but now can also be born with a time-dependent rate), the conditional expectation of the number of particles at time~$t$ is
			\[\mathbb{E}[y_t(\pi)\mid \mathcal{F}] = r(u_{m-1},u_m)\cdot \int_0^t y_s(\pi_{m-1})\cdot  \mathrm{e}^{-(t-s)}\;\mathrm{d}s.\]
			Taking expectation and using Tonelli's theorem, this gives
			\[\mathbb{E}[y_t(\pi)] = r(u_{m-1},u_m)\cdot \int_0^t \mathbb{E}[y_s(\pi_{m-1})]\cdot \mathrm{e}^{-(t-s)}\;\mathrm{d}s.\]
			Using this recursively~$m$ times, and then using the base induction case~$\mathbb{E}[y_s(\pi_0)] = x_0(u_0)\mathrm{e}^{-s}$, we obtain
			\begin{align*}
				&\mathbb{E}[y_t(\pi)] \\&= x_0(u_0) \prod_{i=0}^{m-1}r(u_{i-1},u_i) \int_0^t \int_{s_1}^t \cdots \int_{s_{m-1}}^t \mathrm{e}^{-s_1} \mathrm{e}^{-(s_2-s_1)}\cdots \mathrm{e}^{-(s_m-s_{m-1})} \mathrm{e}^{-(t-s_m)}\;\mathrm{d}s_m\cdots \mathrm{d}s_1\\
				&= x_0(u_0) \prod_{i=0}^{m-1}r(u_{i-1},u_i)\cdot \mathrm{e}^{-t}\cdot \frac{t^m}{m!}.
			\end{align*}
Proof of part (b). \ 
In case~$\mathfrak{l}(\pi) = 0$, the statement is obvious, using the fact that~$y_0((v)) = x_0(v)$ for all~$v$. Assume that~$\mathfrak{l}(\pi) = m \ge 1$, and write~$\pi = (u_0,\ldots,u_{m})$.  Since the transition~$\underline{y} \to \underline{y} + \ind_\pi$ occurs with rate~$\underline{y}(\mathfrak{p}(\pi))\cdot r(u_{m-1},u_{m})$ by \eqref{eq:pi-rate-up}, we have
\begin{align*}
\mathbb{E}[Z(\pi)] &= r(u_{m-1},u_{m})\cdot \mathbb{E}\left[\int_0^\infty y_t(\mathfrak{p}(\pi))\;\mathrm{d}t\right]\\
&= r(u_{m-1},u_{m})\cdot \mathbb{E}\left[\int_0^\infty y_t((u_0,\ldots,u_{m-1}))\;\mathrm{d}t\right].
\end{align*}
Using Tonelli's theorem and~\eqref{eq:ypi-expectation} on the right-hand side, we obtain
\begin{equation*}
\mathbb{E}[Z(\pi)] = r(u_{m-1},u_{m}) \cdot \left(x_0(u_0) \prod_{i=0}^{m-2}r(u_i,u_{i+1}) \right) \cdot \underbrace{\int_0^\infty \frac{t^{m-1}}{(m-1)!}\mathrm{e}^{-t}\;\mathrm{d}t}_{=1}, 
    \end{equation*}
as desired.
\end{proof}

\begin{corollary}\label{cor:localexttime}
Let~$(\underline{y}_t)_{t \ge 0} = \mathrm{GBRW}_{f,\lambda}(G,\underline{x}_0)$, and let~$\pi  \in \mathscr{T}$ with~$m = \mathfrak{l}(\pi)$. Let $X_{m+1}$ be a $\mathrm{Gamma}(1,m+1)$ random variable, i.e., with density $f_m(s)=e^{-s}s^m/m!$. For any~$t \ge 0$, we have
	\begin{equation} \label{eq_in_cor}
		\mathbb{P}(y_s(\pi) > 0 \text{ for some } s \ge t) \le e \cdot \E[Z(\pi)] \cdot   \P(X_{m+1} \ge t).
	\end{equation}
\end{corollary}
\begin{proof}
	Let
	\[\uptau:= \inf\{s \ge t:\;y_s(\pi) > 0\},\]
	so that the left-hand side of~\eqref{eq_in_cor} equals~$\mathbb{P}(\uptau < \infty)$. Next, define the event
	\[\CA:= \{\uptau < \infty,\; y_s(\pi) > 0 \text{ for all } s \in [\uptau,\uptau+1]\}.\]
	It is easy to check that
	\begin{equation}\label{eq_with_A}
		\mathbb{P}(\CA) \ge \mathbb{P}(\uptau < \infty) \cdot \mathrm{e}^{-1}.
	\end{equation}

		Using first Tonelli's theorem and then the fact that~$y_s(\pi)$ is integer-valued, we have
	\begin{align*}
		\int_t^\infty\mathbb{E}[y_s(\pi)]\;\mathrm{d}s = \mathbb{E}\left[\int_t^\infty y_s(\pi)\;\mathrm{d}s \right]\ge \mathbb{E}\left[\int_t^\infty \ind_{\{y_s(\pi) > 0\}}\;\mathrm{d}s \right].
	\end{align*}
	Now, the right-hand side is bounded from below by
	\[\mathbb{E}\left[\ind_{\CA} \cdot \int_\tau^{\tau+1} \ind_{\{y_s(\pi) > 0\}}\;\mathrm{d}s \right] = \mathbb{P}(\CA), \]
where the equality follows from the definition of~$\CA$. Also using~\eqref{eq_with_A}, we have thus obtained
	\[\mathbb{P}(y_s(\pi) > 0 \text{ for some } s \ge t)=\mathbb{P}(\uptau < \infty) \le \mathrm{e} \int_t^\infty \mathbb{E}[y_s(\pi)]\;\mathrm{d}s.\]
	The desired bound in \eqref{eq_in_cor} now follows from \eqref{eq:ypi-expectation} in Lemma~\ref{lem_products} (a).
\end{proof}
The last statement in this section states the stochastic domination between the contact process and branching random walk.
\begin{lemma}[Domination of contact process by branching random walk]\label{lem:StochDom2}
Given any graph $G=(V,E)$, parameters $f:\mathbb R^2\to [0, \infty)$, $\lambda>0$, and starting state $\underline\xi_0\in\{0,1\}^V$, it holds that 
\begin{equation}
\label{eq:stochdom-2}
(\underline\xi_t)_{t\ge 0}=\CPf(G,\underline\xi_0) \ {\buildrel d \over \le } \  \BRW(G,\underline \xi_0)=(\underline x_t)_{t\ge 0}.
\end{equation}
\end{lemma}
This is a well-known result which can be proved either by comparison of transition rates or a coupling using a graphical construction. See \cite[p.34]{liggett1999stochastic} for details of the latter approach; here we omit further details.

%% file: proofs_GW_extinction.tex
\section{Extinction proofs via particle counting and martingales}\label{sec:proofsGW}

In this section we prove several results relating to global, local or fast extinction. We start by showing Theorem \ref{thm:prod-general} on global extinction for the product penalty with $\mu\ge 1/2$ in Section \ref{sec:prod_GW_ext}. Theorem \ref{thm:prod_CM} (in Section \ref{sec:prod_CM_fast}), Theorem \ref{thm:max_trees}(b) (in Section \ref{sec:max_GW_ext}) and Theorem \ref{thm:max_CM_extinction}(b) (in Section \ref{sec:max_CM_fast}) will all be straightforward consequences. Then we establish the other extinction phases for the max-penalty, showing local extinction on all trees  -- Theorem \ref{thm:max_trees}(a) -- for $\mu\ge 1/2$ in Section \ref{sec:max_GW_localext}. We then prove  global extinction on GW trees with finite $(1-\mu)$th moment -- Theorem \ref{thm:max_GW}(c) --  in Section \ref{sec:max_GW_ext_light}. 

\subsection{Product penalty: global extinction for all graphs when \texorpdfstring{$\mu\ge 1/2$}{mu>=1/2} via martingales}\label{sec:prod_GW_ext}
We start by establishing the subcritical phase for the product penalty (Theorem \ref{thm:prod-general}). Here, the result holds generally for any underlying graph, not just a Galton-Watson tree, and any monomial penalty function with polynomial-degree at least $1$:

\begin{claim}[Supermartingale for global extinction]\label{claim:supermg1}
Let $f(x,y)=a x^\mu y^\nu$ for some $a> 0$ and $\mu,\nu\ge 0$ such that $\mu+\nu\ge 1$, and let $G=(V,E)$ be an arbitrary locally finite graph. Consider the process $(x_t(v))_{t\ge0, v\in V}=\BRW(G,  \underline x_0)$ for $\lambda>0$ on $G$, starting from a given state $\underline{x}_0\in\N^{V}$.
  Define, for any $\alpha\in[1-\mu,\nu]$,
\[M_t:=\sum_{v\in V}x_t(v)d_v^\alpha.\]
Then, whenever  $\sum_{v\in V}x_0(v)d_v^\alpha<\infty$, the process $(M_t)_{t\ge 0}$ is a supermartingale with respect to the filtration $\CF_t=\sigma\big((x_s(v))_{v\in V(G),s\le t}\big)$ for all $\lambda \in (0,a]$ and a strict supermartingale when $\lambda\in(0,a)$. 
\end{claim}

\begin{proof} We start by observing that the interval $[1-\mu,\nu]$ is nonempty since $\mu+\nu\ge 1$. To prove the supermartingale property we analyze the expected increments of $(M_t)_{t\ge 0}$, using the definition of $\BRW$ in Def.~\ref{def:BRW}. The change in $M_t$ may come from either  a particle disappearing at $v$ due to a death event, or from a new particle appearing at $v$ due to reproduction events from neighboring particles.  We obtain, using the rates in \eqref{eq:healing-BRW} and \eqref{eq:infection-rate-BRW} with $r(u,v)=\la e(u,v)/f(d_u,d_v)$, that
\begin{align*}
\E[M_{t+\mathrm{d}t}-M_t \mid \CF_t]&=\E\left[\left.\sum_{v\in V(G)} (x_{t+\mathrm{d}t}(v)-x_t(v))d_v^\alpha \;\right|\; \CF_t\right]\\
&=\sum_{v\in V(G)}\Bigg(-x_t(v)+\sum_{u\in N(v)} x_t(u)r(u,v)\Bigg)\mathrm{d}t\cdot d_v^\alpha\\
&=-M_t\mathrm{d}t+\sum_{v\in V(G)}\sum_{u\in N(v)}x_t(u)\cdot [\lambda e(u,v)/f(d_u,d_v)]\cdot d_v^\alpha\mathrm{d}t.
\end{align*}
We substitute $f(d_u,d_v)=a d_u^{\mu}d_v^{\nu}$ in the last line above, and use that $d_v^{\alpha-\nu}\le 1$ by the assumption  that $\alpha\le\nu$ to obtain that:
\begin{align*}
\E[M_{t+\mathrm{d}t}-M_t \mid \CF_t]
&=-M_t\mathrm{d}t+(\la/a)\cdot\sum_{v\in V(G)}\sum_{u\in N(v)}x_t(u) e(u,v)d_u^{-\mu}d_v^{\alpha-\nu}\mathrm{d}t\\
&\le-M_t\mathrm{d}t+(\la/a)\cdot\sum_{v\in V(G)}\sum_{u\in N(v)}x_t(u) e(u,v)d_u^{-\mu}\mathrm{d}t.
\end{align*}
Exchanging the sums and using that $\sum_{v\in V}e(u,v)=d_u$ (see Notation in Section \ref{sec:introduction}), we obtain
\begin{equation*}
\E[M_{t+\mathrm{d}t}-M_t \mid \CF_t]\le-M_t\mathrm{d}t+(\la/a)\cdot\sum_{u\in V(G)} x_t(u) d_u^{1-\mu}\mathrm{d}t.
\end{equation*}
Finally, since $d_u\ge 0$ is an integer, $d_u^{1-\mu}\le d_u^{\alpha}$ holds by the assumption $1-\mu\le\alpha$. Hence,
\begin{equation}\label{eq:submartingale-1}
\E[M_{t+\mathrm{d}t}-M_t \mid \CF_t]\le -M_t\mathrm{d}t+(\la/a)\cdot\sum_{u\in V(G)}  x_t(u) d_u^\alpha\mathrm{d}t=[(\lambda/a)-1]\cdot M_t\mathrm{d}t.
\end{equation}
Since $M_t\ge 0$, for $\lambda\le a$ we obtain the supermartingale property, as $[(\lambda/a)-1]\cdot M_t\mathrm{d}t\le 0$,  with strict inequality when $\lambda<a$. The finiteness of the initial state $M_0$ is ensured by the assumption that $M_0=\sum_{v\in V} x_0(v)d_v^\alpha<\infty$. This finishes the proof. 
\end{proof}

\begin{proof}[Proof of Theorem \ref{thm:prod-general}]
Without loss of generality we may assume that all vertices in $G$ have degree at least $1$. Indeed, if $G$ would contain a (countably infinite or finite) number of vertices with degree $0$, the contact process on those, starting from any $\underline\xi_0$ with finitely many infected vertices, reduces to a pure death process where each particle dies at rate $1$. This is because infection cannot happen to and from these vertices. This process goes almost surely extinct. Hence we assume wlog that $d_v\ge 1$ for all $v\in V$.  

By Lemma \ref{lem:StochDom2}, it is sufficient to prove the almost sure extinction of $\BRW(G, \underline \xi_0)$ for any $\xi_0$ that is almost surely finite, i.e., $\sum_{v\in V} \xi_0(v)<\infty$ almost surely. Fix now any such realization of the initial state. Then, since only finitely many coordinates are non-zero, $\sum_{v\in V} \xi_0(v) d_v^\alpha=\sum_{v\in V} x_0(v) d_v^\alpha<\infty$ also holds for any $\alpha>0$. We assumed $\la\in (0,1)$ also in Theorem \ref{thm:prod-general}. Hence, the conditions of Claim \ref{claim:supermg1} are satisfied with $\nu=\mu\ge 1/2$, and we can set $\alpha=1-\mu$ there to obtain the non-negative (strict) supermartingale $(M_t)_{t\le 0}$, (i.e., not a martingale).  

Apply Doob's martingale convergence theorem for the non-negative supermartingale $(M_t)_{t\ge 0}$. Since $(M_t)_{t\ge 0}$ is integer-valued, its almost sure limit can only be $0$, and it cannot take any value in $(0,1)$. Therefore, almost surely $M_t=0$ for large enough $t$. 
By the coupling between $\CPf$ and $\BRW$ and that $d_v^{1-\mu}\ge 1$ whenever $d_v\ge 1$, we obtain that  \[ 
\sum_{v\in V} \xi_t(v) \le  \sum_{v\in V} x_t(v)d_v^{1-\mu} =M_t \ {\buildrel a.s. \over \longrightarrow}\  0 
\]
implying global extinction. To compute the extinction time, by definition, $T_{\mathrm{ext}}(G, \underline \xi_0)\ge t$ implies the existence of at least one infected particle at time $t$. Since $\alpha>0$ and $d_v\ge1$ for all $v$, the existence of at least one infected particle at time $t$ in turn implies $M_t\ge 1$. By Markov's inequality, and since $M_0=|\underline\xi_0|$, taking expectation of \eqref{eq:submartingale-1} and solving the resulting differential equation for $\mathbb E[M_t \mid  G, \underline \xi_0]$ yields for all $\la\le 1$:
\begin{align*}
\Pv(T_{\mathrm{ext}}(G, \underline \xi_0)  \ge t \mid G, \underline \xi_0) &\le  \Pv(M_t\ge 1\mid G, \underline \xi_0) \le \mathbb E[M_t\mid G, \underline \xi_0]\\
&\le \Big(\sum_{v\in V}\xi_0(v)d_v^\alpha\Big)\exp(-(1-\lambda) t).
\end{align*}
Hence,
\begin{align*}
\Ev[T_{\mathrm{ext}}(G, \underline \xi_0) \mid G, \underline \xi_0]&\le \int_{0}^{\infty} \Big(\sum_{v\in V}\xi_0(v)d_v^\alpha\Big)\exp(-(1-\lambda) t) \mathrm dt\\
&= \Big(\sum_{v\in V}\xi_0(v)d_v^{1-\mu}\Big)/(1-\lambda).
\end{align*}
This finishes the proof. The extensions in Remark \ref{rem:polynomial-penalties} follow immediately by the stochastic domination in \eqref{eq:domination-between-two-f} and then the martingale argument applied to the monomial obtained. 
\end{proof}

\subsection{Product-penalty: fast extinction on the configuration model when \texorpdfstring{$\mu\ge 1/2$}{mu>=1/2}}\label{sec:prod_CM_fast}
We obtain Theorem \ref{thm:prod_CM}(b) as an immediate consequence of Theorem \ref{thm:prod-general}, since it applies for arbitrary finite graphs as well.

\begin{proof}[Proof of Theorem \ref{thm:prod_CM}(b)] The bound in \eqref{eq:mean-ext-time-product} in Theorem \ref{thm:prod-general} applied to the configuration model $G_n$ yields that $\mathbb E[T^{\mathrm{cp}}_{\mathrm{ext}}(G_n, \underline 1_{G_n})| (d_i)_{i\le n}]\le \sum_{i=1}^n d_i^{1-\mu}/(1-\la)$. Fast extinction now follows by using the assumption that $\sum_{i=1}^n d_i^{1-\mu}=O_{\P}(\mathrm{poly}(n))$. Assumption \ref{assu:regularity} implies that $\max_{i\le n} d_i=o(n)$, so then this condition is automatically satisfied,  but it holds even in a much larger class of degree sequences $(\underline d_n)$ that do not grow superpolynomially.
\end{proof}

\subsection{Max-penalty: global extinction for all graphs when \texorpdfstring{$\mu\ge 1$}{mu>=1}}\label{sec:max_GW_ext}
Global extinction in Theorem \ref{thm:max_trees} part (b) is a straightforward consequence of that in Theorem \ref{thm:prod-general}.
\begin{proof}[Proof of Theorem \ref{thm:max_trees}, part (b)]
For all $\mu\ge 0$,
\[f_1(x,y):=\max(x,y)^\mu \ge x^{\mu/2}y^{\mu/2}=:f_2(x,y)\]
holds for all $x,y\ge 1$. 
Hence, the stochastic domination in \eqref{eq:domination-between-two-f} applies and $\mathrm{CP}_{f_1,\la}\ {\buildrel d \over \le }\ \mathrm{CP}_{f_2,\la}$. Since the exponent in $f_2$ is $\mu/2\ge 1/2$ by the assumption that $\mu\ge 1$, Theorem \ref{thm:prod-general} applies for $\mathrm{CP}_{f_2,\la}$, and the process goes extinct for all $\lambda<1$. Hence, so does also $\mathrm{CP}_{f_1,\la}$.
\end{proof}

\subsection{Max-penalty: fast extinction on the configuration model when \texorpdfstring{$\mu\ge 1$}{mu>=1}}\label{sec:max_CM_fast}
Fast extinction in Theorem \ref{thm:max_CM_extinction} part (b) follows from Theorem \ref{thm:prod_CM} part (b) in a similarly straightforward way.
\begin{proof}[Proof of Theorem \ref{thm:max_CM_extinction}, part (b)]
The stochastic domination between the product and max-penalties discussed in the proof of Theorem \ref{thm:max_trees}(b) above implies the result from Theorem \ref{thm:prod_CM}(b).
\end{proof}

\subsection{Max-penalty: Loop erasure in particle counting when \texorpdfstring{$\mu\in (1/2,1)$}{1/2<mu<1}}\label{sec:max_GW_localext}
To prove local extinction, and also global extinction later under the max-penalty, we go back to the construction of genealogic branching random walks from Section~\ref{s_gbrw}. We use Lemma~\ref{lem_products} and bound the number of total particles ever born, decomposed along genealogical paths. We first give some definitions. 

For a graph $G = (V,E)$, we will take throughout this section the infection-rate function to be
\begin{equation}\label{eq_choice_of_r}
r(u,v)= \frac{\lambda\cdot  \mathrm{e}(u,v)}{\max(d_u,d_v)^\mu},\quad u,v \in V.
\end{equation}
Recall that~$\mathscr{T}=\mathscr{T}(G)$ denotes the set of genealogical labels in~$G$, as in Definition~\ref{def:labels}, and $Z(\pi)$ from \eqref{eq:Zpi-def}. We define, for~$\pi = (\pi_0,\ldots,\pi_m) \in \mathscr{T}$, \begin{equation}\label{eq_def_z}
		z(\pi) := \prod_{i=0}^{m-1} r(\pi_i,\pi_{i+1}) = \lambda^{m}\cdot  \prod_{i=0}^{m-1}\frac{\mathrm{e}(\pi_i,\pi_{i+1})}{\max(d_{\pi_i},d_{\pi_{i+1}})^\mu},
	\end{equation}
 with $z(\pi) = 1$ if  the length of the path $\mathfrak{l}(\pi) = 0$.
 Note that, by Lemma~\ref{lem_products}(b), $z(\pi)=\E[Z(\pi)]$, the expected number of particles with label $\pi$ ever born, in a genealogical branching process with birth rate $\lambda$, maximum-penalty function with exponent $\mu$, and started with a single particle with label~$(\pi_0)$.

\begin{definition}[Backtracking steps]\label{def_backtracking}
	Let~$G = (V,E)$ be a graph.  Given a path $\pi = (\pi_1,\ldots,\pi_m) \in \mathscr{T}$ with length~$\mathfrak{l}(\pi) = m \ge 2$, we define
	\begin{equation}\label{eq:tau-def}
		\tau(\pi):= \min\{i \ge 2:\; \pi_i = \pi_{i-2} \neq \pi_{i-1}\}
	\end{equation}
	(with the convention~$\min \varnothing = \infty$). That is,~$\tau(\pi)$ is the first index on the path when~$\pi$ returns to a vertex $u$ right after having jumped away from it to a different vertex $v$. We informally refer to this kind of motion~$u \to v \to u$ (with~$u \neq v$) as a \emph{backtracking step}. 
	For~$\pi$ with~$\tau(\pi) < \infty$, we define
	\[g(\pi):= (\pi_0,\ldots,\pi_{\tau-2},\pi_{\tau+1},\ldots,\pi_{\mathfrak{l}(\pi)}),\]
	that is,~$g(\pi)$ is the path obtained by removing the first backtracking step of~$\pi$. We define $g^{-1}(\pi)=\{\pi': g(\pi')=\pi\}$ the set of paths that map to $\pi$ under $g$.
\end{definition}

We clarify that traversal of self-loops, even multiple times, is not considered a backtracking step for the above definition.

\begin{figure}[htb]
\begin{center}
\setlength\fboxsep{0cm}
\setlength\fboxrule{0cm}
\fbox{\includegraphics[width=\textwidth]{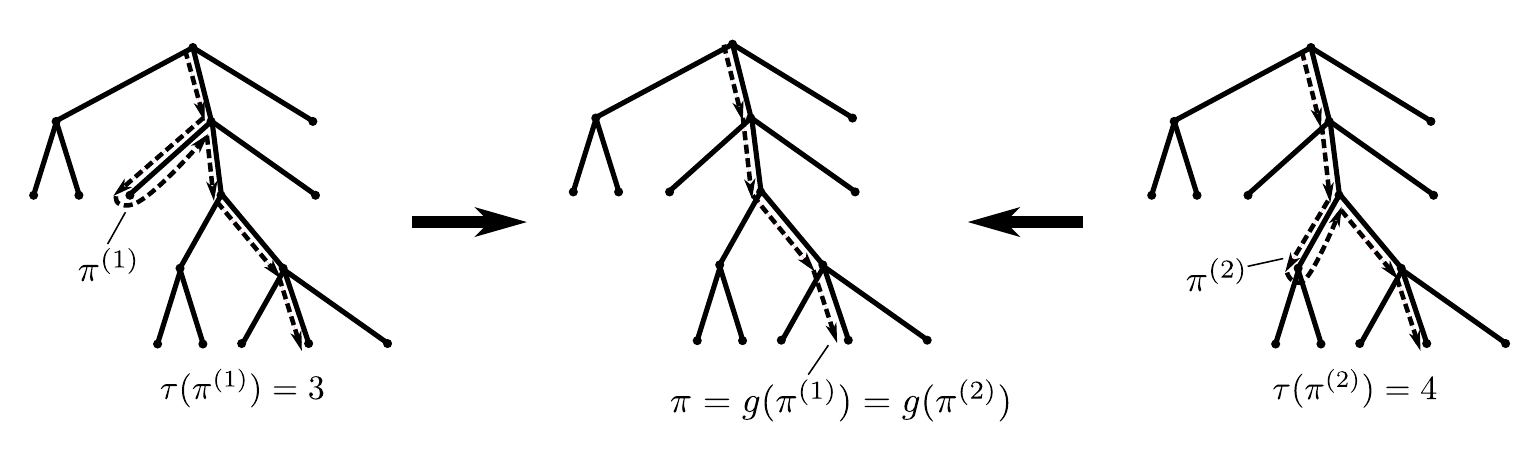}}
\end{center}
	\caption{\label{fig_loop} {Illustration of the definition of $\tau$ and $g$ of Definition~\ref{def_backtracking}.}}
\end{figure}

\begin{claim}[Removal of one backtracking step]\label{claim_remove_one}
	Let~$G=(V,E)$ be a graph,~$\lambda > 0$ and~$\mu \ge 1/2$. Let $z(\cdot)$ be as in~\eqref{eq_def_z}. 
	For any~$\pi \in \mathscr{T}$ and any index~$a \in \N$, we have
	\begin{equation}\label{eq_sum_zpi}
		\sum_{ \pi' \in g^{-1}(\pi):\  \tau(\pi') = a}z(\pi') \le \lambda^2\cdot z(\pi) \max_{v: v\neq \pi_{a-2}}e(\pi_{a-2},v).
	\end{equation}
\end{claim}
\begin{proof} 
	Fix~$\pi$ and~$a$ as in the statement of the lemma. Write~$\pi = (\pi_0,\ldots,\pi_m)$, where~$m = \mathfrak{l}(\pi)$. We assume that the set~$\{\pi' \in g^{-1}(\pi):\; \tau(\pi') = a\}$ is non-empty, as the desired inequality is trivial otherwise. By \eqref{eq:tau-def} we then have~$a \in \{2,\ldots, m+2\}$, and any~$\pi' \in g^{-1}(\pi)$ with~$\tau(\pi') = a$ and $\pi$ are of the form
	\begin{equation} \label{eq_rep_pi}
		\pi' = (\pi_0,\ldots, \pi_{a-3}, u,v,u,\pi_{a-1}, \ldots, \pi_m),\qquad \pi=(\pi_0, \ldots, \pi_{a-3}, u, \pi_{a-1}, \ldots, \pi_m),
	\end{equation}
	where~$u = \pi_{a-2}=\pi'_{a-2}=\pi'_a$ and~$v$ is a neighbor of~$u$ (with~$v \neq u$). 
	Then, by \eqref{eq_def_z},
	\begin{equation*}
		z(\pi') = r(u,v)^2\cdot z(\pi) =  \left(\frac{\lambda\cdot  \mathrm{e}(u,v)}{\max(d_u,d_v)^\mu}\right)^2 \cdot z(\pi) \le \frac{\lambda^2\cdot \mathrm{e}(u,v)^{2}}{(d_u)^{2\mu} } \cdot z(\pi),
	\end{equation*}
	and
	\begin{align*}
		\sum_{\pi' \in g^{-1}(\pi):\  \tau(\pi') = a}z(\pi') &\le \frac{\lambda^2\cdot z(\pi)}{(d_u)^{2\mu}} \sum_{v:v\neq u} \mathrm{e}(u,v)^{2}\le \frac{\lambda^2\cdot z(\pi)}{(d_u)^{2\mu}} \cdot d_u \cdot \max_{v: v\neq u}e(u,v) \\
		&= \lambda^2\cdot z(\pi)\cdot (d_u)^{1-2\mu}  \cdot \max_{v: v\neq u}e(u,v)\le \lambda^2\cdot z(\pi) \cdot  \max_{v: v\neq \pi_{a-2}}e(\pi_{a-2},v) ,
	\end{align*}
where the last inequality follows from~$d_u \ge 1$ and~$\mu \ge 1/2$, and that $u=\pi_{a-2}$. 
\end{proof}

Let us write
\[g^{(0)} := g,\quad g^{(k+1)} := g \circ g^{(k)},\; k \ge 0.\]
In the statement and proof of the following lemma, to avoid summations with long subscripts, for any set~$A$ and any function~$h:A \to \mathbb{R}$, we write~$\sum\{h(x): x\in A\} = \sum_{x \in A} h(x)$
(with the convention that this is zero when~$A$ is empty).
\begin{lemma}[Removal of multiple backtracking steps]\label{lem_remove_many}
	Let~$G$,~$\lambda$,~$\mu$,~$f$ and~$z(\cdot)$ be as in Claim~\ref{claim_remove_one}. Fix~$\pi \in \mathscr{T}$. Then, for any~$k \ge 1$ and any sequence of positive integers~$(a_1,\ldots,a_k)$, we have
	\begin{align*}
		\sum&\left\{ z(\pi'):\begin{array}{l} \;\pi' \in (g^{(k)})^{-1}(\pi),\\[.1cm]
		\tau(\pi') = a_1,\; \tau(g(\pi')) = a_2,\;\ldots,\; \tau(g^{(k-1)}(\pi')) = a_k \end{array}  \right\}\\
		&\le \Big(\max_{\substack{u,v\in G\\ v\neq u}}e(u,v)\Big)^k\lambda^{2k}\cdot z(\pi).
	\end{align*}
\end{lemma}
\begin{proof}
	The proof is by induction on~$k$, the case~$k=1$ being Claim~\ref{claim_remove_one}. Assume the statement has been proved for~$k$, and fix~$\pi \in \mathscr{T}$ and a sequence~$(a_1,\ldots,a_{k+1})$. Then, since $\tau$ gives the location of the first backtracking step,  
	\begin{align}
		&\sum\left\{ z(\pi'): \;\pi' \in (g^{(k+1)})^{-1}(\pi),\; \tau(\pi') = a_1,\; \ldots,\; \tau(g^{(k)}(\pi')) = a_{k+1}   \right\}\nonumber\\[.2cm]
			&= \sum\left\{\sum\left\{ z(\pi'): \begin{array}{l}
				\pi' \in g^{-1}(\pi''),\\[.1cm]\tau(\pi') = a_1 
			\end{array} \right\}: \begin{array}{l}
				\pi'' \in (g^{(k)})^{-1}(\pi),\\[.1cm]
				\tau(\pi'') = a_2,\; \ldots,\; \tau(g^{(k-1)}(\pi'')) = a_k
			\label{eq_remove_many}\end{array}\right\}.
	\end{align}
	By Claim~\ref{claim_remove_one}, for each~$\pi''$, the inner sum above is smaller than
 \[\lambda^2 z(\pi'') \max_{v: v\neq \pi''_{a_1-2}}e(\pi''_{a_1-2},v)\le \lambda^2 z(\pi'') \max_{u,v \in G: v\neq u}e(u,v),\]
 so the double sum in \eqref{eq_remove_many} is smaller than
	\begin{align*}
		\max_{u,v \in G: v\neq u}e(u,v)\lambda^2 \cdot   \sum \left\{z(\pi''):
			\pi'' \in (g^{(k)})^{-1}(\pi),\; \tau(\pi'') = a_2,\; \ldots,\; \tau(g^{k-1}(\pi'')) = a_k
		\right\}.
	\end{align*}
	Using the induction hypothesis, this is smaller than~$\big(\max_{u,v \in G: v\neq u}\big)^{k+1}\lambda^{2(k+1)}z(\pi)$, as required.
\end{proof}
We would now like to use the above lemma to obtain a bound involving all possible sequences~$(a_1,\ldots,a_k)$. Before doing so, we prove the following simple fact.

\begin{claim}\label{claim_mon_tau}
	Let~$G$ be a graph and~$\pi \in \mathscr{T}$ be such that~$\tau(\pi) < \infty$ and~$\tau(g(\pi)) < \infty$. Then,
	\[\tau(g(\pi)) \ge \tau(\pi) - 1.\]
\end{claim}
\begin{proof}
	This follows from the observation that the sub-path~$(\pi_0,\ldots,\pi_{\tau-2})$ remains intact after applying~$g$ to~$\pi$, and this sub-path contains no backtracking steps by the minimality of~$\tau(\pi)$.
\end{proof}

\begin{corollary}\label{cor_rem_mult}
Let~$G$,~$\lambda$,~$\mu$ and~$f$ be as in Claim~\ref{claim_remove_one}. Fix~$\pi \in \mathscr{T}$ and~$k \ge 1$. Then,
	\begin{equation}\label{eq_big_sum}
		\sum_{ \pi' \in (g^{(k)})^{-1}(\pi)} z(\pi') \le 2^{\mathfrak{l}(\pi)}\cdot \Big(4\lambda^2\big(\max_{\substack{u,v\in G\\ v\neq u}}e(u,v)\big)\Big)^k \cdot z(\pi).
	\end{equation}
\end{corollary}
\begin{proof}
	Fix~$\pi$ and~$k$ as in the statement.  Define
	\[\mathcal{A}:= \{(\tau(\pi'),\tau(g(\pi')),\ldots, \tau(g^{(k-1)}(\pi'))):\; \pi' \in (g^{(k)})^{-1}(\pi)\}.\]
	That is, for a single $\pi' \in (g^{(k)})^{-1}(\pi)$, the sequence $(\tau(\pi'),\tau(g(\pi')),\ldots, \tau(g^{(k-1)}(\pi')))$ gives the locations -- i.e., not the vertex but its index on the `current' path -- of loop erasure when we sequentially apply $g$, $k$ times, on the path $\pi'$.
$\mathcal{A}$ is then the set of all sequences of length~$k$ that can be obtained by taking~$\pi' \in (g^{(k)})^{-1}(\pi)$ and applying~$\tau$,~$\tau \circ g$,~$\ldots$,~$\tau \circ g^{(k-1)}$ to~$\pi'$.
	By Lemma~\ref{lem_remove_many}, the left-hand side of~\eqref{eq_big_sum} is smaller than
	\[\sum_{ \pi' \in (g^{(k)})^{-1}(\pi)} z(\pi')\le\Big(\max_{u,v \in G: v\neq u}e(u,v)\Big)^{k}\lambda^{2k}z(\pi)\cdot |\mathcal{A}|.\]
	The desired bound will then follow from the inequality~$|\mathcal{A}| \le 2^{\mathfrak{l}(\pi)+2k}$, which we now prove.

	For each~$\pi' \in (g^{(k)})^{-1}(\pi)$, we add $2(i-1)$ to the location of the $i$th erasure in the sequential application of loop erasure $g$ on $\pi'$, which, by Claim \ref{claim_mon_tau} leads to a   a strictly increasing sequence of numbers,  i.e., we define
	\[ c_i(\pi'):= \tau(g^{(i-1)}(\pi')) + 2(i-1),\quad  i \in \{1,\ldots, k\}\]
	(with~$g^{(0)}(\pi') =\pi'$).
	Note that
	\begin{align*}
		c_k(\pi') &= \tau(g^{(k-1)}(\pi')) + 2(k-1) \le \mathfrak{l}(g^{(k-1)}(\pi')) + 2(k-1) \\
  &= \mathfrak{l}(\pi)+2 + 2(k-1) = \mathfrak{l}(\pi) + 2k.
	\end{align*}
	Moreover, for~$i \in \{1, \ldots, k-1\}$,
	\[c_{i+1}(\pi') - c_i(\pi') = \tau(g^{(i)}(\pi')) - \tau(g^{(i-1)}(\pi')) + 2,\]
	which is positive by Claim~\ref{claim_mon_tau}. These considerations show that~$(c_1(\pi'),\ldots,c_k(\pi'))$ is an increasing sequence in~$\{1,\ldots,\mathfrak{l}(\pi)+2k\}$. Therefore,~$\mathcal{A}$ can be mapped injectively into the set of increasing sequences with~$k$ elements in $\{1,\ldots,\mathfrak{l}(\pi) +2k\}$. It is a combinatorial exercise to show that the number of such sequences is~$\binom{\mathfrak{l}(\pi)+2k} {k} \le 2^{\mathfrak{l}(\pi)+2k}$.
\end{proof}

\begin{proof}[Proof of Theorem \ref{thm:max_trees} part (a)]
	By Lemma \ref{lem:StochDom2} it is enough to prove the result for the branching random walk. Assume that~$\mu \ge 1/2$ and~$\lambda < 1/2$. Let~$\mathcal{T}$ be a tree with a root~$\varnothing$. For each vertex~$u$ of~$\mathcal{T}$, let~$\pi_{\downarrow u}$ denote the geodesic path from~$\varnothing$ to~$u$. Consider the branching random walk on~$\mathcal{T}$ with penalty function~$f(x,y) = \max(x,y)^\mu$, birth rate~$\mu$ and initial configuration consisting of a single particle, located at the root. For this process, let~$Z(\cdot)$ be as in~\eqref{eq:Zpi-def} and~$z(\cdot)$ be as in~\eqref{eq_def_z}; note that by~\eqref{eq:Zpi-expectation}, we have~$\mathbb{E}[Z(\pi)] = z(\pi)$ for any~$\pi \in \mathscr{T}$. Further, let $\mathscr{T}_0=\{\pi\in\mathscr{T}:\ \pi_0=\varnothing\}$ denote the set of paths in $\mathcal{T}$ that start at the root.
Then, since $\CT$ is a tree and $e(u,v)\in\{0,1\}$ for all pairs $u,v\in \CT$, 
	\begin{equation} \label{eq_computation_z}\begin{split}
		\sum_{\pi \in \mathscr{T}_0:\  \mathfrak{s}(\pi) = u}\mathbb{E} \left[Z(\pi) \right]&= \sum_{\pi \in \mathscr{T}_0:\   \mathfrak{s}(\pi) = u}z(\pi) = \sum_{k=0}^\infty \; \;\sum_{\pi \in (g^{(k)})^{-1}(\pi_{\downarrow u})} z(\pi) \\
		&\le z(\pi_{\downarrow u}) \cdot 2^{\mathfrak{l}(\pi_{\downarrow u})}\cdot  \sum_{k=0}^\infty (4\lambda^2)^k = \frac{2^{\mathfrak{l}(\pi_{\downarrow u})}}{1-4\lambda^2}\cdot z(\pi_{\downarrow u}), 
	\end{split}
	\end{equation}
	where the inequality follows from Corollary~\ref{cor_rem_mult}. Since the right-hand side above is finite, we see that the expectation of the number of particles ever born at~$u$ is finite, so this number is almost surely finite. This proves local extinction for the initial configuration in which there is a single particle at the root. As already observed, this implies local extinction for the branching random walk, and also the contact process, started from any finite initial configuration.

To prove the exponential decay of the local extinction time, we will use Corollary \ref{cor:localexttime} to write, for any $t>0$,
\begin{align}
\P\left(T_{\mathrm{ext}}^{\mathrm{brw}}(\CT,\ind_{\varnothing},u)>t\right)
&=\sum_{\pi \in \mathscr{T}_0:\ \mathfrak{s}(\pi) = u}\P(y_s(\pi)>0\text{ for some }s\ge t)\nonumber\\
&\le\sum_{\pi \in \mathscr{T}_0:\ \mathfrak{s}(\pi) = u} e\cdot \E[Z(\pi)]\cdot\P(X_{\mathfrak{l}(\pi)+1}\ge t)\label{eq:localexttime1},
\end{align}
where $X_{m}$ is a Gamma($1,m$) variable for any $m>0$. Let $\alpha\in(0,1)$ be a constant specified later. Further bounding the right-hand side of \eqref{eq:localexttime1}, we write
\begin{equation}\label{eq:localexttime2}
\P\left(T_{\mathrm{ext}}^{\mathrm{brw}}(\CT,\ind_{\varnothing},u)>t\right)\le
\sum_{\substack{ \pi \in \mathscr{T}_0:\ \mathfrak{s}(\pi) = u, \\ \mathfrak{l}(\pi)<\lfloor\alpha t\rfloor}}e\cdot z(\pi)\cdot\P(X_{\mathfrak{l}(\pi)+1}\ge t)+
\sum_{\substack{ \pi \in \mathscr{T}_0:\ \mathfrak{s}(\pi) = u, \\ \mathfrak{l}(\pi)\ge\lfloor\alpha t\rfloor}}e\cdot z(\pi).
\end{equation}
First, we bound the first sum on the right-hand side of \eqref{eq:localexttime2}. Noting that $X_{\lfloor\alpha t\rfloor}$ stochastically dominates $X_{\mathfrak{l}(\pi)+1}$ when $\mathfrak{l}(\pi)<\lfloor\alpha t\rfloor$, and using Corollary \ref{cor_rem_mult} we get
\begin{align}
\sum_{\substack{ \pi \in \mathscr{T}_0:\ \mathfrak{s}(\pi) = u, \\ \mathfrak{l}(\pi)<\lfloor\alpha t\rfloor}}e\cdot z(\pi)&\cdot\P(X_{\mathfrak{l}(\pi)+1}\ge t)\le
e\cdot \P(X_{\lfloor\alpha t\rfloor}\ge t)\cdot\sum_{r=\mathfrak{l}(\pi_{\downarrow}(u))}^{\lfloor\alpha t\rfloor-1}\sum_{\substack{ \pi \in \mathscr{T}_0:\ \mathfrak{s}(\pi) = u, \\ \mathfrak{l}(\pi)=r}}z(\pi)\nonumber\\
&\le
e\cdot \P(X_{\lfloor\alpha t\rfloor}\ge t)\cdot z(\pi_{\downarrow}(u))\cdot2^{\mathfrak{l}(\pi_{\downarrow}(u))}\sum_{k=0}^{(\lfloor\alpha t\rfloor-1-\mathfrak{l}(\pi_{\downarrow}(u)))/2}(4\lambda^2)^k.\label{eq:localexttime3}
\end{align}
Since $\lambda<1/2$, the sum on the right-hand side of \eqref{eq:localexttime3} is bounded by $1/(1-4\lambda^2)$. By \eqref{eq_def_z}, we have
\begin{equation}\label{eq:localexttime_z}
z(\pi_{\downarrow}(u))=\lambda^{\mathfrak{l}(\pi_{\downarrow}(u))}\prod_{i=0}^{\mathfrak{l}(\pi_{\downarrow}(u))-1}(\max(d_{\pi_i},d_{\pi_{i+1}}))^{-\mu}\le(2^{-\mu}\lambda)^{\mathfrak{l}(\pi_{\downarrow}(u))}.
\end{equation}
Combining \eqref{eq:localexttime3} and \eqref{eq:localexttime_z} to further upper bound the right-hand side of \eqref{eq:localexttime3} yields
\begin{equation}\label{eq:localexttime4}
\sum_{\substack{ \pi \in \mathscr{T}_0:\ \mathfrak{s}(\pi) = u, \\ \mathfrak{l}(\pi)<\lfloor\alpha t\rfloor}}e\cdot z(\pi)\cdot\P(X_{\mathfrak{l}(\pi)+1}\ge t)\le \frac{e}{1-4\lambda^2}\cdot \P(X_{\lfloor\alpha t\rfloor}\ge t)\cdot (2^{1-\mu}\lambda)^{\mathfrak{l}(\pi_{\downarrow}(u))}.
\end{equation}
To bound the probabilistic term on the right-hand side of \eqref{eq:localexttime4}, we use the large deviation principle for Gamma variables to write
\begin{equation*}
\P(X_{\lfloor\alpha t\rfloor}\ge t)\le e^{-\lfloor\alpha t\rfloor I_{\mathrm{exp}}(1/\alpha)},
\end{equation*}
where $I_{\mathrm{exp}}$ is the large deviation rate function of the Exponential distribution with parameter $1$, defined as
\begin{equation}\label{eq:LDP_exp}
I_{\mathrm{exp}}(a)=a-1+\log(1/a)
\end{equation}
for $a>1$.
As a result, we get
\begin{equation}\label{eq:localexttime5}
\sum_{\substack{ \pi \in \mathscr{T}_0:\ \mathfrak{s}(\pi) = u, \\ \mathfrak{l}(\pi)<\lfloor\alpha t\rfloor}}e\cdot z(\pi)\cdot\P(X_{\mathfrak{l}(\pi)+1}\ge t)\le \frac{e\cdot (2^{1-\mu}\lambda)^{\mathfrak{l}(\pi_{\downarrow}(u))}}{1-4\lambda^2}\cdot e^{-\lfloor\alpha t\rfloor I_{\mathrm{exp}}(1/\alpha)}.
\end{equation}

Next, we bound the second sum on the right-hand side of \eqref{eq:localexttime2}. Similarly to \eqref{eq:localexttime3}, again using Corollary \ref{cor_rem_mult}, we get
\begin{equation}\label{eq:localexttime6}
\sum_{\substack{ \pi \in \mathscr{T}_0:\ \mathfrak{s}(\pi) = u, \\ \mathfrak{l}(\pi)\ge\lfloor\alpha t\rfloor}}e\cdot z(\pi)\le
e\cdot z(\pi_{\downarrow}(u))\cdot2^{\mathfrak{l}(\pi_{\downarrow}(u))}\sum_{k=(\lfloor\alpha t\rfloor-\mathfrak{l}(\pi_{\downarrow}(u)))/2}^{\infty}(4\lambda^2)^k. 
\end{equation}
Bounding $z(\pi_{\downarrow}(u))$ as in \eqref{eq:localexttime_z}, and evaluating the geometric sum in \eqref{eq:localexttime6} yields
\begin{align}
\sum_{\substack{ \pi \in \mathscr{T}_0:\ \mathfrak{s}(\pi) = u, \\ \mathfrak{l}(\pi)\ge\lfloor\alpha t\rfloor}}e\cdot z(\pi)&\le
e\cdot (2^{1-\mu}\lambda)^{\mathfrak{l}(\pi_{\downarrow}(u))}\cdot\frac{(4\lambda^2)^{(\lfloor\alpha t\rfloor-\mathfrak{l}(\pi_{\downarrow}(u)))/2}}{1-4\lambda^2}\nonumber\\
&=\frac{e\cdot 2^{-\mu\mathfrak{l}(\pi_{\downarrow}(u))}}{1-4\lambda^2}\cdot (2\lambda)^{\lfloor\alpha t\rfloor}\label{eq:localexttime7}.
\end{align}

Substituting the bounds \eqref{eq:localexttime5} and \eqref{eq:localexttime7} into 
\eqref{eq:localexttime2} yields
\begin{equation}
\P\left(T_{\mathrm{ext}}^{\mathrm{brw}}(\CT,\ind_{\varnothing},u)>t\right)\le
\frac{e\cdot (2^{1-\mu}\lambda)^{\mathfrak{l}(\pi_{\downarrow}(u))}}{1-4\lambda^2}\cdot e^{-\lfloor\alpha t\rfloor I_{\mathrm{exp}}(1/\alpha)}+\frac{e\cdot 2^{-\mu\mathfrak{l}(\pi_{\downarrow}(u))}}{1-4\lambda^2}\cdot (2\lambda)^{\lfloor\alpha t\rfloor}.\label{eq:localexttime8}
\end{equation}
For $\lambda<1/2$, \eqref{eq:localexttime8} shows the exponential decay of the local extinction time at $u$. Since the first term on the right-hand side is increasing in $\alpha$, whereas the second term is decreasing, the optimized bound is given by $\alpha=\alpha^\star$, where $\alpha^\star$ is the solution of
\begin{equation}\label{eq:localext_alpha}
e^{-\lfloor\alpha^\star t\rfloor I_{\mathrm{exp}}(1/\alpha^\star)}=(2\lambda)^{\lfloor\alpha^\star t\rfloor}.
\end{equation}
Using \eqref{eq:LDP_exp}, \eqref{eq:localext_alpha} simplifies to
\begin{equation}\label{eq:localexttime_alphaopt}
1/\alpha^\star-1+\log(\alpha^\star)=-\log(2\lambda).
\end{equation}
Since the left-hand side of \eqref{eq:localexttime_alphaopt} is strictly decreasing from $\infty$ to $0$ as $\alpha^\star$ increases from $0$ to $1$, there is exactly one solution $\alpha^\star\in(0,1)$ for any given $\lambda<1/2$. This finishes the proof for $\underline{x}_0=\ind_\varnothing$.

To extend the argument to any starting state $\underline x_0$ with $|\underline x_0|<\infty$, we make two observations. First, since the above argument is valid for \emph{any} tree $\CT$ with \emph{any} fixed root $\varnothing$, (by re-rooting the tree) this implies that
\begin{equation}\label{eq:localexttime_any_vertex}
\P\left(T_{\mathrm{ext}}^{\mathrm{brw}}(\CT,\ind_v,u)>t\right)\le c_1(v)e^{-c_2 t}
\end{equation}
for any $u,v\in\CT$ and $t>0$ (for $\lambda<1/2$). Here, the constant $c_1(v)$ further depends on $u,\lambda,\mu$, while $c_2$ depends on $\lambda$, but, importantly, not on $v$. Second, when $(\underline x_t)_{t\ge 0}=\BRW(\CT,\underline x_0)$, then by the independent behavior of the particles in BRW, we have that
\begin{equation}\label{eq:localexttime_BRWs}
(\underline x_t)_{t\ge 0}\stackrel{d}{=}\left(\sum_{v:x_0(v)>0}\sum_{i=1}^{x_0(v)}\underline x^{(v,i)}_t\right)_{t\ge 0},
\end{equation}
where $(\underline x^{(v,i)}_t)_{v,i}$ are independent realizations of the processes $\BRW(\CT,\ind_v)$. Hence, if $T_\mathrm{ext}^{(v,i,u)}$ denotes the local extinction time of $(\underline x^{(v,i)}_t)_{t\ge 0}$ at $u$, then a union bound combined with \eqref{eq:localexttime_any_vertex} gives
\begin{align*}
\P\left(T_{\mathrm{ext}}^{\mathrm{brw}}(\CT,\underline x_0,u)>t\right)&=\P\left(\max_{v,i} T_{\mathrm{ext}}^{(v,i,u)}>t\right)\\&\le\sum_{v:x_0(v)>0}\sum_{i=1}^{x_0(v)}\P\left(T_{\mathrm{ext}}^{(v,i,u)}>t\right)\le\sum_{v:x_0(v)>0}\sum_{i=1}^{x_0(v)} c_1(v)e^{-c_2 t},
\end{align*}
that is, exponential decay of the distribution of the local extinction time (with the same constant in the exponent for any $|\underline x_0|$). This finishes the proof.
\end{proof}

\subsection{Max-penalty: global extinction on trees when growth is limited}\label{sec:max_GW_ext_light}

In this section, we consider rooted trees. The root will always be denoted by~$\varnothing$. We always assume that trees have no loops or parallel edges.   For any vertex~$u$ of~$\mathcal{T}$, we keep using the notation~$\pi_{\downarrow u}$ for the geodesic from~$\varnothing$ to~$u$. Given~$\mu > 0$, for each vertex~$u$ in~$\mathcal{T}$ we let
\begin{equation}\label{eq_def_zeta}
	\zeta(u):=\prod_{i=0}^{\mathfrak{l}(\pi)-1} {\max(d_{\pi_i},d_{\pi_{i+1}})^{-\mu}},\quad \text{where }\pi = \pi_{\downarrow u},
\end{equation}
so that (recalling~\eqref{eq_def_z}, and recalling that we exclude parallel edges, so that~$\mathrm{e}(\pi_i,\pi_{i+1}) = 1$) we have
\begin{equation}\label{eq_rel_z_zeta}
z(\pi_{\downarrow u}) = \lambda^{\mathfrak{l}(\pi_{\downarrow u})}\cdot \zeta(u).
\end{equation}
We will write~$\mathrm{Gen}_N(\mathcal{T})$ for the set of vertices at graph distance~$N$ from~$\varnothing$, for~$N \in \mathbb{N}$.

\begin{lemma}\label{lem_condition_global_ext}
	Let~$\mathcal{T}$ be a tree with root~$\varnothing$. Fix~$\mu \in (1/2,1)$,~$\lambda > 0$ and assume that
	\begin{equation}
		\label{eq_assumption_beta}	
		\sum_{N=0}^\infty (2\lambda)^N \sum_{u \in \mathrm{Gen}_N(\mathcal{T})} \zeta(u) < \infty.
	\end{equation}
	Then, $\mathrm{BRW}_{f,\lambda}(\CT, \ind_{\varnothing})$ with penalty function~$f(x,y) = \max(x,y)^\mu$ goes extinct globally.
\end{lemma}
\begin{proof}
	We continue using the notation $\mathscr{T}_0=\{\pi\in\mathscr{T}:\ \pi_0=\varnothing\}$ for the set of paths in $\mathcal{T}$ that start at the root. Repeating the estimate in~\eqref{eq_computation_z} and using~\eqref{eq_rel_z_zeta}, for any~$N \in \mathbb{N}$ and any vertex~$u \in \mathrm{Gen}_N(\mathcal{T})$ we have $\mathfrak{l}(\pi_{\downarrow u})=N$, so summing over all infection paths ending at $u$ gives
	\begin{align*}
		\sum_{\pi \in \mathscr{T}_0:\ \mathfrak{s}(\pi) = u}\mathbb{E}[Z(\pi)] =\sum_{\pi \in \mathscr{T}_0:\ \mathfrak{s}(\pi) = u} z(\pi) \le \frac{(2\lambda)^N}{1-4\lambda^2}\cdot \zeta(u).
	\end{align*}
	Then, when summing over all infection paths in the tree, we have 
	\begin{align*}
		\sum_{\pi \in \mathscr{T}_0} \mathbb{E}[Z(\pi)] &= \sum_{N=0}^\infty\; \sum_{u \in \mathrm{Gen}_N(\mathcal{T})}\; \sum_{\pi \in \mathscr{T}_0:\ \mathfrak{s}(\pi) = u}\mathbb{E}[Z(\pi)]\\
  &\le \frac{1}{1-4\lambda^2} \sum_{N=0}^\infty (2\lambda)^N \sum_{u \in \mathrm{Gen}_N(\mathcal{T})} \zeta(u)< \infty
	\end{align*}
	by the assumption. 
	 This shows that, starting from a single particle at the root, the expected number of particles ever born (overall in~$\mathcal{T}$) is finite, so this number is finite almost surely. This implies global extinction.
\end{proof}

In the applications we have in mind, rather than verifying~\eqref{eq_assumption_beta} directly, we will verify that
	\begin{equation}
		\label{eq_tilde_beta}	
		\sum_{N=1}^\infty (2\lambda)^N \sum_{u \in \mathrm{Gen}_N(\mathcal{T})} \tilde{\zeta}(u) < \infty,
	\end{equation}
where~$\tilde{\zeta}(u)$ is defined for all~$u \neq \varnothing$ by
\begin{equation}
\label{eq_def_tilde_zeta}
	\tilde{\zeta}(u):= {(d_{\varnothing})^{-\mu}} \cdot \prod_{i=1}^{\mathfrak{l}(\pi)-1} {(d_{\pi_i}-1)^{-\mu}},\quad \text{where }\pi = \pi_{\downarrow u}
\end{equation}
(we leave~$\tilde{\zeta}$ undefined at the root).
Clearly, by \eqref{eq_def_zeta},~$ \zeta(u)\le\tilde{\zeta}(u) $ for all~$u \neq \varnothing$, so~\eqref{eq_tilde_beta} implies~\eqref{eq_assumption_beta}.

\begin{proof}[Proof of Theorem \ref{thm:max_GW} part (c)]
	We assume that the offspring distribution of the Galton-Watson tree satisfies~$\mathbb{E}[D^{1-\mu}] < \infty$. We claim that, for any~$N \ge 1$,
	\begin{equation}\label{eq_claim_gens}
		\mathbb{E}\left[\sum_{u \in \mathrm{Gen}_N(\mathcal{T})} \tilde{\zeta}(u)\right] = (\mathbb{E}[D^{1-\mu}])^N.
	\end{equation}
	This is obvious in case~$N=1$. Assume that it has been proved for~$N$. For the induction step, by \eqref{eq_def_tilde_zeta}, we note that
	\begin{align}
		\sum_{u \in \mathrm{Gen}_{N+1}(\mathcal{T})} \tilde{\zeta}(u) &= \sum_{v \in \mathrm{Gen}_N(\mathcal{T})} \tilde{\zeta}(v) \sum_{\substack{u \in \mathrm{Gen}_{N+1}(\mathcal{T}):\\u \sim v}} {(d_v-1)^{-\mu}} \nonumber \\
		&=\sum_{v \in \mathrm{Gen}_N(\mathcal{T})} \tilde{\zeta}(v)\cdot  (d_v - 1)\cdot  {(d_v-1)^{-\mu}}  = \sum_{v \in \mathrm{Gen}_N(\mathcal{T})} \tilde{\zeta}(v)\cdot (d_v-1)^{1-\mu}. \label{eq_gens_z}
	\end{align}
	Let~$\mathcal{T}_N$ denote the truncation of~$\mathcal{T}$ at generation~$N$, that is,~$\mathcal{T}_N$ is the subgraph of~$\mathcal{T}$ induced by the set of vertices at graph distance at most~$N$ from~$\varnothing$. Note that~$\mathcal{T}_N$ does not include information about the offsprings of vertices in generation~$N$, and conditioned on~$\mathcal{T}_N$, the sizes of these offsprings are iid, with same law as~$D$.  Taking expectations in~\eqref{eq_gens_z}, we have
\begin{align*}
	\mathbb{E}\left[\sum_{u \in \mathrm{Gen}_{N+1}(\mathcal{T})} \tilde{\zeta}(u) \right]&= \mathbb{E}\left[ \mathbb{E}\left[\left.\sum_{v \in \mathrm{Gen}_N(\mathcal{T})} \tilde{\zeta}(v)\cdot (d_v-1)^{1-\mu}\right| \mathcal{T}_N\right]\right]\\
	&=\mathbb{E}\left[ \sum_{v \in \mathrm{Gen}_N(\mathcal{T})}\tilde{\zeta}(v)\cdot\mathbb{E}\left[ \left.  (d_v-1)^{1-\mu}\right| \mathcal{T}_N\right]\right]\\
	&=\mathbb{E}\left[ \sum_{v \in \mathrm{Gen}_N(\mathcal{T})}\tilde{\zeta}(v)\right]\cdot \mathbb{E}[D^{1-\mu}] = (\mathbb{E}[D^{1-\mu}])^{N+1},
\end{align*}
where the last equality follows from the induction hypothesis. This completes the proof of~\eqref{eq_claim_gens}.

Now, if~$\lambda < (2\mathbb{E}[D^{1-\mu}])^{-1}$, then
\[\mathbb{E}\left[ \sum_{N=1}^\infty (2\lambda)^N \cdot \sum_{u \in \mathrm{Gen}_N(\mathcal{T})} \tilde{\zeta}(u) \right] = \sum_{N=1}^\infty (2\lambda \cdot  \mathbb{E}[D^{1-\mu}])^N < \infty.\]
Hence,~$\sum_{N=1}^\infty (2\lambda)^N \cdot \sum_{u \in \mathrm{Gen}_N(\mathcal{T})} \tilde{\zeta}(u)$ is finite for almost all realizations of~$\mathcal{T}$. It then follows from Lemma~\ref{lem_condition_global_ext} (and the observation following its proof) that there is global extinction of the penalized branching random walk for almost every realization of~$\mathcal{T}$.
\end{proof}

We now see further applications of Lemma~\ref{lem_condition_global_ext}, the proof of Corollary \ref{cor:SST}.
\begin{proof}[Proof of Corollary \ref{cor:SST}]
	The case of trees with finite upper branching number $b$ follows from verifying condition~\eqref{eq_assumption_beta} with the simple bound~$\zeta(u) \le 1$ for all~$u$.
	For the case of spherically symmetric trees, we can verify condition~\eqref{eq_tilde_beta} directly instead of working with the branching number. Note that, for any~$N \ge 1$, we have
	\[\tilde{\zeta}(u) = (d_0)^{-\mu} \prod_{i=1}^{N-1} (d_i-1)^{-\mu}  \quad \text{for any } u \in\mathrm{Gen}_N(\mathcal{T}),\]
	so
	\[\sum_{u \in \mathrm{Gen}_N(\mathcal{T})} \tilde{\zeta}(u) = (d_0)^{-\mu} \prod_{i=1}^{N-1} (d_i-1)^{-\mu} \cdot |\mathrm{Gen}_N(\mathcal{T})| = (d_0)^{1-\mu}\prod_{i=1}^{N-1}(d_i -1)^{1-\mu},\]
	and then
	\begin{align*}
		&(2\lambda)^N\sum_{u \in \mathrm{Gen}_N(\mathcal{T})} \tilde{\zeta}(u)\\
 &= \exp\left\{N \left(\log(2) + \log(\lambda) + \frac{(1-\mu)(\log d_0)}{N} + \frac{1-\mu}{N} \sum_{i=1}^{N-1} \log(d_i-1) \right) \right\}.
	\end{align*}
	Now, it is easy to check that $\limsup 1/N\cdot \sum_{i=1}^{N-1} \log(d_i-1)\le \log \overline{\mathrm{br}}(\CT)$, so  if~$\lambda < \mathrm{e}^{-(1-\mu)\log \overline{\mathrm{br}}(\CT)}/2$, then there exists~$c < 0$ such that the expression inside parenthesis above is smaller than~$c$ for~$N$ large enough. It readily follows that~\eqref{eq_tilde_beta} is satisfied, so global extinction follows from Lemma~\ref{lem_condition_global_ext}.
\end{proof}

\subsection{Max-penalty: fast extinction when \texorpdfstring{$\mu\in (1/2,1)$}{1/2<mu<1}}\label{sec:max_overall_fast}

We close this section by proving a result that bounds the survival of $\BRW$ for $f(x,y)=\max(x,y)^\mu$, $\mu\in(1/2,1)$ on any graph, both in space and in time. We will use this result in Section~\ref{sec:proofsCM-extinction} to prove Theorem~\ref{thm:max_CM_extinction} part (a), stating that the max-penalty contact process goes quickly  extinct on the configuration model whenever $\tau>3$.

We again go back to the genealogic branching random walk construction of Section~\ref{s_gbrw}. For a graph~$G=(V,E)$, recall the definition of the set of genealogical labels~$\mathscr{T}$ from Definition~\ref{def:labels}, the notations~$\mathfrak{l}(\pi)$ and~$\mathfrak{s}(\pi)$, the construction of~$(\underline{y}_t)_{t \ge 0}$ in Definition~\ref{def_gbrw} and its relation to the branching random walk~$(\underline{x}_t)_{t \ge 0}$ given in Lemma~\ref{lem_corr_gbrw}. Here we will take these processes with birth rate~$\lambda$ and max-penalty function with exponent~$\mu$,~$f(x,y) = \max(x,y)^\mu$, so that~$r(\cdot,\cdot)$ is as in~\eqref{eq_choice_of_r}.  As before, for~$\pi$ with~$\mathfrak{l}(\pi) \ge 1$,~$Z(\pi)$ denotes the number of particles with label~$\pi$ born in the whole history of the process. We let~$z(\cdot)$ be as in~\eqref{eq_def_z}. Finally, recall the first backtracking index~$\tau(\cdot)$ and the backtracking erasure function~$g(\cdot)$ from Definition~\ref{def_backtracking}.
\begin{lemma}\label{lem_overall_fast}
	Let~$\mu \in (1/2,1)$ and~$G= (V,E)$ be a graph with a distinguished vertex~$\bar{v}$. Assume that for some constant $\ell>0$~$\mathrm{e}(u,v) \le \ell$ for any~$u,v \in V$. Fix~$N \ge 2$ and let~$b_N$ denote the number of non-backtracking paths of length at most~$N$ started at~$\bar{v}$,
	\begin{equation}\label{eq_assumption_growth}
		b_N:=|\{\pi \in \mathscr{T}:\; \pi_0 = \bar{v},\; \mathfrak{l}(\pi) \le N,\; \tau(\pi) = \infty\}|.
	\end{equation}
	Consider the penalized branching random walk~$(\underline{x}_t^{(\bar v)})_{t \ge 0}$ on~$G$ with penalization function~$f(x,y)= \max(x,y)^\mu$, birth rate~$\lambda<1/(4\ell)$ and started from a single particle, located at~$\bar{v}$. Then, for any fixed constant $C>1$,
	\begin{equation}\label{eq:overall-fast}
	\begin{aligned}
	\mathbb{P}&\left( \begin{array}{l}
			(\underline{x}^{(\bar v)}_t) \text{ dies before time $CN$, and never reaches }\\
		\text{ any vertex at graph distance $N$ from $\bar{v}$}\end{array} \right)\\
		&> 1 - 2b_N\Big (\mathrm e \ell\cdot (4\ell\lambda)^N  + \mathrm{e}^{-N(C-1)^2/(2C)}\Big).
	\end{aligned}
	\end{equation}
\end{lemma}
\begin{proof}
	Let~$(\underline{y}_t)_{t \ge 0}$ be the genealogic branching random walk corresponding to~$(\underline{x}_t)_{t \ge 0}$ as in Lemma~\ref{lem_corr_gbrw}; in particular,~${y}_0((\bar{v})) = 1$ and~${y}_0(\pi) = 0$ for any~$\pi \neq \bar{v}$. We note that
	\begin{align*}
		&\left\{ (\underline{x}_t) \text{ is alive at time $CN$, or reaches some vertex at distance $N$ from $\bar{v}$} \right\} \\
			&\subset \{y_{CN}(\pi) > 0 \text{ for some } \pi \in \mathscr{T} \text{ with } \pi_0 = \bar{v},\; \mathfrak{l}(\pi) <N \}\\
			&\quad\cup \{y_t(\pi) > 0 \text{ for some } \pi \in \mathscr{T} \text{ with } \pi_0 = \bar{v},\; \mathfrak{l}(\pi) = N \text{ and some } t > 0\}.
	\end{align*}
	Using a union bound and the inequalities~$\mathbb{P}(y_{CN}(\pi) > 0) \le \mathbb{E}[y_{CN}(\pi)]$ and~$\mathbb{P}(y_t(\pi) > 0 \text{ for some } t) \le \P\mathrm{e} \cdot \mathbb{E}[Z(\pi)] = \mathrm e \cdot z(\pi)$ from Corollary \ref{cor:localexttime}, we have
	\begin{align} \label{eq_union_bound_pis}
		\mathbb{P}\left(\begin{array}{l}
			(\underline{x}_t) \text{ is alive at time $CN$, or reaches }\\
			\text{some vertex at distance~$N$ from~$\bar{v}$} 
		\end{array}\right) \le \sum_{\substack{\pi \in \mathscr{T}:\\\pi_0 = \bar{v},\\\mathfrak{l}(\pi) < N }} \mathbb{E}[y_{CN}(\pi)] + \mathrm{e}\cdot  \sum_{\substack{\pi \in \mathscr{T}:\\ \pi_0 = \bar{v},\\ \mathfrak{l}(\pi) = N}}z(\pi).
	\end{align}
	We bound the two sums in the rhs separately. Using \eqref{eq_def_z} the following bound holds for any path:
	\begin{equation} \label{eq_simple_bound_z}z(\pi) \le (\ell\lambda)^{\mathfrak{l}(\pi)},\end{equation} which follows from~$\max(d_u,d_v)^\mu \ge 1$ and the assumption that~$\mathrm{e}(u,v) \le \ell$.

	We first deal with the second sum in~\eqref{eq_union_bound_pis}. Recall that if~$\pi' \in (g^{(k)})^{-1}(\pi)$, then~$\mathfrak{l}(\pi') = \mathfrak{l}(\pi)+2k$. Then, we break the sum as follows:
	\begin{align*}
		\sum_{\substack{\pi : \pi_0 = \bar{v},\\ \mathfrak{l}(\pi) = N}} z(\pi) &= \sum_{\substack{(m,k):\\ m+2k = N}}  \; \sum_{\substack{ \pi: \pi_0 = \bar{v},\\ \mathfrak{l}(\pi) = m,\\ \tau(\pi) = \infty}} \;\sum_{\pi' \in (g^{(k)})^{-1}(\pi)}\; z(\pi').
	\end{align*}
	Using~\eqref{eq_big_sum} in Corollary \ref{cor_rem_mult}, the right-hand side is at most
	\[ \sum_{\substack{\pi : \pi_0 = \bar{v},\\ \mathfrak{l}(\pi) = N}} z(\pi)\le  \sum_{\substack{(m,k):\\ m+2k = N}}  \; \sum_{\substack{ \pi: \pi_0 = \bar{v},\\ \mathfrak{l}(\pi) = m,\\ \tau(\pi) = \infty}} \;2^m\cdot (4\lambda^2\ell)^k \cdot z(\pi).\]
	Using~\eqref{eq_simple_bound_z} and $b_N$ from \eqref{eq_assumption_growth}, this is at most
	\begin{align*}\sum_{\substack{(m,k):\\ m+2k = N}}  \; &\sum_{\substack{ \pi: \pi_0 = \bar{v},\\ \mathfrak{l}(\pi) = m,\\ \tau(\pi) = \infty}} \;2^m\cdot (4\ell\lambda^2)^k \cdot (\ell\lambda)^m\\
  &= \sum_{\substack{(m,k):\\ m+2k = N}} (4\ell)^{m+k}\cdot \lambda^{m+2k} \cdot |\{\pi:\; \pi_0 = \bar{v},\; \mathfrak{l}(\pi) = m,\; \tau(\pi) = \infty\}|\\[.2cm]
		&\le b_N\cdot \sum_{\substack{(m,k):\\ m+2k = N}} (4\ell)^{m+k}\cdot \lambda^{m+2k}. 
	\end{align*}
	Using that $m+2k=N$ implies that $m+k=(N+m)/2$ for each $m\in{0, \dots, N}$, the above sum is at most
	\begin{align}
	\sum_{\substack{\pi : \pi_0 = \bar{v},\\ \mathfrak{l}(\pi) = N}} z(\pi)&\le b_N\cdot  \sum_{m=0}^N (4\ell)^{(N+m)/2}\cdot \lambda^N  = (2\ell^{1/2}\lambda)^N\cdot b_N\cdot  \sum_{m=0}^N (2\ell^{1/2})^m \nonumber \\
	&\le 2\ell^{1/2}(4\ell\lambda)^N\cdot b_N. \label{eq:first-term-to-handle}
	\end{align}

	We now turn to the first term in~\eqref{eq_union_bound_pis}. Using~\eqref{eq:ypi-expectation}, we have
	\begin{align}\label{eq_max_and_z}
		\sum_{\substack{\pi: \pi_0 = \bar{v},\\\mathfrak{l}(\pi) < N}} \mathbb{E}[y_{CN}(\pi)] \le \left( \max_{0 \le m < N} \frac{(CN)^m}{m!} \mathrm{e}^{-CN}\right) \cdot  \sum_{\substack{\pi: \pi_0 = \bar{v},\\ \mathfrak{l}(\pi) < N}} z(\pi).
	\end{align}
Let us bound the sum in the right-hand side using \eqref{eq_big_sum} with $\max e(u,v)\le \ell$ and then \eqref{eq_simple_bound_z} as
	\begin{align*}
		\sum_{\substack{\pi: \pi_0 = \bar{v},\\ \mathfrak{l}(\pi) < N}} z(\pi) 
		&\le \sum_{\substack{\pi: \pi_0= \bar{v},\\ \mathfrak{l}(\pi) <N,\\ \tau(\pi) = \infty}} \;\sum_{k=0}^\infty \;\sum_{\pi' \in (g^{(k)})^{-1}(\pi)} z(\pi')\\
  &\stackrel{\eqref{eq_big_sum}}{\le} \frac{1}{1-4\ell\lambda^2}\sum_{\substack{\pi: \pi_0= \bar{v},\\ \mathfrak{l}(\pi) <N,\\ \tau(\pi) = \infty}}2^{\mathfrak{l}(\pi)} z(\pi) 
	\stackrel{\eqref{eq_simple_bound_z}}{\le} \frac{1}{1-4\ell\lambda^2} \sum_{\substack{\pi: \pi_0= \bar{v},\\ \mathfrak{l}(\pi) <N,\\ \tau(\pi) = \infty}} (2\ell\lambda)^{\mathfrak{l}(\pi)}.	\end{align*}
	Since~$\lambda < 1/(4\ell)$ with $\ell\ge 1$, we have~$\frac{1}{1-4\ell\lambda^2} < 2$ and~$2\ell\lambda < 1/2$, so the last factor in \eqref{eq_max_and_z} is smaller than
	\begin{equation}\label{eq:first-second-factor} 2|\{\pi: \pi_0 = \bar{v},\; \mathfrak{l}(\pi) < N,\;\tau(\pi) = \infty\}| \le 2b_N.\end{equation}
	Next, the expression inside the maximum in~\eqref{eq_max_and_z} equals~$\mathbb{P}(W = m)$ for $W$ having the $\mathrm{Poisson}(CN)$ distribution. We bound
	\[\max_{0 \le m < N} \mathbb{P}(W = m) \le \mathbb{P}(W \le N).\]
	We use a Chernoff bound for Poisson random variables: for~$X \sim \mathrm{Poisson}(\nu)$ we have~$\mathbb{P}(X \le \nu - t) \le \mathrm{e}^{-t^2/(2\nu)}$, see~\cite[Exercise 2.21]{van2016random}. This gives
	\[\mathbb{P}(W \le N) \le \exp\left\{- \frac{(CN-N)^2}{2CN} \right\} = \exp\left\{-\frac{(C-1)^2}{2C}\cdot N\right\}.\]
Combining this with \eqref{eq:first-second-factor} in \eqref{eq_max_and_z} and \eqref{eq:first-term-to-handle} completes the proof of \eqref{eq:overall-fast}.
\end{proof}

%% file: proofs_CM_extinction.tex
\section{The configuration model: fast extinction via loop erasure }\label{sec:proofsCM-extinction}
In this section we prove Theorem \ref{thm:max_CM_extinction} part (a). This theorem says that the contact process $\CPf$ and the branching random walk $\BRW$ go extinct quickly for small $\lambda>0$ on the configuration model when $f(x,y)=\max(x,y)^{\mu}$ with $\mu\in (1/2, 1)$ and the degree distribution is lighter than a power-law with exponent $\tau>3$. The proof idea is the following. Fixing a large constant $\ell$, first, we show that with probability $1-o(1/n)$, there are at most $\ell$ surplus edges in the $r$-neighborhood $B_{r}(u_n)$ of a uniformly chosen vertex $u_n$ with  $r=\delta \log n$ for some small $\delta>0$.  That is, one can remove at most $\ell$ edges from $B_{\delta \log n}(u_n)$ to obtain a tree. Then, we apply Lemma \ref{lem_overall_fast} to show that the expected number of particles of $\BRW$ on infection paths in $B_{\delta \log n}(u_n)$ that reach the boundary $\partial B_{\delta \log n}(u_n)$ decays exponentially for small $\lambda$. This implies that $\BRW$ dies out inside $B_{\delta \log n}(u_n)$ before reaching $\partial B_{\delta \log n}(u_n)$ with probability at least $1-o(1/n)$. A union bound over the $n$ vertices then finishes the proof.

Our first goal is to prove a statement about   the surplus edges of $B_{\delta \log n}(u_n)$, and then we move on to the analysis of infection paths of $\BRW$. The number of surplus edges of a (sub)graph $H=(V_H, E_H)$ is given by $|E_H| - (|V_H|-1)$.
Recall the configuration model from Definition \ref{def:CM} and that $e(u,v)$ denotes the number of edges between vertices $u,v$.
\begin{proposition}\label{prop:surplus}
Consider the configuration model with degree sequence $\underline d_n$ satisfying Assumption \ref{assu:regularity}, and Assumptions \ref{assu:empirical-power-law} and \ref{assu:empirical-power-law-2} with some $\tau, \varepsilon, c_u, z_0$ (for all sufficiently large $n$) with $\tau(1-\eps)>3$.
Fix some $\delta>0$. Let $u_n$ be a uniformly chosen vertex in $[n]$ and let $\mathrm{Surp}_{\delta \log n}(u_n)$ denote the number of surplus edges in $B_{\delta \log n}(u_n)$. 
Then, for all $\eps'>0$ there exists $\delta>0$ and $\delta'>0$ so that for any  $\ell>(\tau(1-\eps)-1)/(\tau(1-\eps)-3-2\eps')$ 
\begin{equation}\label{eq:size}
    \P(|B_{\delta\log n}(u_n)| \ge n^{(1+\eps')/(\tau(1-\eps)-1)} \text{ and }\mathrm{Surp}_{\delta \log n}(u_n) \ge \ell) \le n^{-1-\delta'}
\end{equation} 
Finally, for any $\ell> 3 \vee (\tau(1-\eps)-1)/(\tau(1-\eps)-3)$, there exists some $\delta'>0$ that
\begin{equation}\label{eq:multiple-edges}
\P(\max_{u,v\in[n]} e(u,v) \ge \ell) \le n^{-1-\delta'}.
\end{equation}
\end{proposition}
Observe that with probability $1/n$ the root's degree is the maximal degree in the graph, which can be as high as $O(n^{1/(\tau(1-\eps)-1})$, so $\eps'>0$ in \eqref{eq:size} is necessary for the bound to be true. If one aims to bound the maximal multiplicity of edges inside $B_{\delta \log n}$, the inequality \eqref{eq:size} also includes that, since multiple edges also count as surplus edges. For generality we include the stronger result in \eqref{eq:multiple-edges} here.

The proof is based on a breadth-first-search exploration process of $B_{\delta \log n}(u_n)$, and a coupling to a (power-law) branching process tree $\CT^{\#}_{\delta \log n}$ so that the tree contains $B_{\delta \log n}(u_n)$. First we give a good bound on the size of the tree that holds with probability $1-o(1/n)$. When the offspring distribution decays exponentially, this is fairly easy, but when it follows for instance a power law, we need to develop some new bounds.  

Hence, the next lemma bounds the $k$th moment of the size of (truncated) power-law BP trees, but before that, we give some definitions.
Let $(\zeta_n)_{n\ge 1}$ be a sequence of discrete measures on $\N$ that satisfies 
\begin{align}
\zeta_n(z)&\le c_u z^{-(\tau'-1)},\qquad 
M_n:=\max \mathrm{support}(\zeta_n) \le C_u n^{1/(\tau'-1)}\label{eq:zeta-support},\\
&\tau' \notin \N.     \label{eq:not-integer}
\end{align}
Usually $\zeta_n$ is the size-biased measure of an empirical degree sequence $\underline{d}_n$ satisfying Assumptions \ref{assu:regularity} and \eqref{assu:empirical-power-law-2}. 
For each integer $k \ge 1$, there exists~${C}_k > 0$ such that, if $n$ is large enough, domination by an integral of the rhs of \eqref{eq:widehat-nu} yields that the $k$th moment
\begin{equation}\label{eq_moments_empirical}
    \sum_{z=1}^{M_n} z^k \cdot \zeta_n(z) \le c_u+\sum_{z=1}^{M_n} c_u z^{k-(\tau'-1)}\le  {C}_k \cdot n^{h_k} 
    \end{equation}
    where
    \begin{equation}\label{eq:hk-def}
    h_k:= \max\left\{(k+1)/(\tau'-1)-1,\;0\right\}.  
     \end{equation}
Whenever $\tau'>2$, the coefficient of $k+1$ in $h_k$ is positive but less than $1$. Thus,~$k \mapsto h_k$ is non-decreasing and, due to the additive term $-1$, for any~$k,\ell$, the super-additivity property holds:
\begin{equation}\label{eq_super_additivity}
   h_k+ h_{\ell} \le h_{k+\ell}.
\end{equation}
\begin{lemma}\label{lemma:moments}
Let $\mathcal{T}^{\#}$ be a Galton-Watson tree with offspring distribution~$\zeta_n$ satisfying \eqref{eq:zeta-support} and \eqref{eq:not-integer} with $\tau'>3$, and for each~$r$, let~$Z_r$ be the size of its generation~$r$. 
    For any integer~$k \ge 1$, there exists $\mathfrak{C}_k > 0$ such that the following holds for all sufficiently large~$n$:
\begin{equation}\label{eq_bound_high_moment_coll}
    \mathbb{E}[(Z_r)^k] \le \mathfrak{C}_k \cdot n^{h_k} \cdot \mathrm{e}^{\mathfrak{C}_k r}\qquad \text{for all } r \ge 0.
\end{equation}
\end{lemma}
The criterion $\tau'>3$ is important: this guarantees that the mean offspring $\E[\CX]=\E[\CX_n]$ does not grow with $n$. BPs with $\tau'\in(2,3)$ grow doubly-exponentially, and \eqref{eq_bound_high_moment_coll} does not hold for them. The importance here is that the rhs of~\eqref{eq_bound_high_moment_coll} only depends on the generation number $r$ exponentially, i.e., the constant $\mathfrak{C}_k$ in the exponential growth does not depend on~$n$. This is non-trivial, since the $k$-th moment of the offspring distribution itself does, but it only enters the bound \emph{once}, as the prefactor $n^{h_k}$. 
\begin{proof}[Proof of Lemma \ref{lemma:moments}]
We will argue by induction over~$k$. Let~$\mathcal{X}$ be a random variable distributed as~$\zeta_n$ (we will generally omit the dependence on~$n$).

For the base case~$k = 1$, recalling~\eqref{eq_bound_high_moment_coll}, note that~$h_1 = 0$ since~$\tau' > 2$; hence,~$\mathbb{E}[\mathcal{X}]$ is bounded by the constant~${C}_1$ which does not depend on~$n$ (equivalently, in $h_k$ the maximum is at $0$ in \eqref{eq:hk-def}). The right-hand side of~\eqref{eq_bound_high_moment_coll} is safisfied in this case since 
\[\mathbb{E}[Z_r] = \mathbb{E}[\mathcal{X}]^r \le ({C}_1)^r = n^{h_1} \cdot \mathrm{e}^{\log(C_1) r}.\]
Now assume that we have proved~\eqref{eq_bound_high_moment_coll} for~$j=1,\ldots,k-1$, that is, assume that we have already found constants~$\mathfrak{C}_1,\ldots,\mathfrak{C}_{k-1}$ such that
\begin{equation} \label{eq_step_first_induction}
    \E[(Z_r)^j] \le \mathfrak{C}_{j} \cdot n^{h_{j}} \cdot \mathrm{e}^{\mathfrak{C}_{j} r} \qquad \text{for all } j \in \{1,\ldots, k-1\} \text{ and all } r \ge 0,
\end{equation}
and we want to find $\mathfrak{C}_{k}$ so that \eqref{eq_bound_high_moment_coll} holds.
Let~$\mathfrak{f}(s)$ denote the probability-generating function of~$\mathcal{X}$,
\[\mathfrak{f}(s):= \sum_{z \ge 1} s^z \cdot \mathbb{P}(\mathcal{X} = z) = \sum_{z \ge 1} s^z \cdot \widehat{\nu}_n(z), \qquad s \in \mathbb{R}. \]
Since~$\zeta_n$ has finite support,~$\mathfrak{f}$ is well defined for any~$s$; it is also infinitely differentiable, with derivative of order~$m$ at~$s=1$ satisfying
\[\mathfrak{f}^{(m)}(1) = \mathbb{E}[\mathcal{X} (\mathcal{X}-1) \cdots (\mathcal{X}-m+1)].\]
For any~$r \in \N$, let~$\mathfrak{f}_r$ denote the~$r$-fold composition of~$\mathfrak{f}$ with itself (i.e.,~$\mathfrak{f}_0$ is the identity function,~$\mathfrak{f}_1 = \mathfrak{f}$ and~$\mathfrak{f}_r = \mathfrak{f} \circ \mathfrak{f}_{r-1}$ for~$r >1$). It is well-known that~$\mathfrak{f}_r$ is the probability-generating function of~$Z_r$, which is again well defined and infinitely differentiable for all~$s$,
\[\mathfrak{f}_r(s) = \sum_{z =1}^\infty s^z \cdot \mathbb{P}(Z_r = z), \quad s \in \mathbb{R}, \quad \mbox{ and }  \quad \mathfrak{f}_r^{(m)}(1) = \mathbb{E}[Z_r(Z_r-1)\cdots (Z_r -m+1)].
\]
We claim that there exists~$\mathfrak{C}_{k}'>0$ such that
\begin{equation} \label{eq_want_from_induction2}
    \mathfrak{f}_r^{(k)}(1) \le \mathfrak{C}_{k}' \cdot n^{h_{k}} \cdot \mathrm{e}^{\mathfrak{C}_{k}' r} \qquad \text{for all }r \ge 0.
\end{equation}
Before proving this, let us show how to use it together with the induction hypothesis to obtain~\eqref{eq_bound_high_moment_coll} (with a constant~$\mathfrak{C}_{k}$ that is possibly different from~$\mathfrak{C}_{k}'$). We bound
\begin{align*}
    \E[(Z_r)^{k}] \le &\E[Z_r (Z_r-1) \cdots (Z_r -(k-1))]\\
    &+ | \E[(Z_r)^{k}]- \E[Z_r (Z_r-1) \cdots (Z_r -(k-1))]|\\
    \le &\mathfrak{f}_r^{(k)}(1) + \sum_{j=1}^{k-1} |a_{k-1,j}| \cdot \E[Z_r^j], 
\end{align*}
where~$a_{k-1,j}$ is the coefficient of~$x^j$ in the polynomial~$x(x-1)\cdots (x-(k-1))$. By~\eqref{eq_want_from_induction2} and the induction hypothesis, the right-hand side above is smaller than
\[\E[(Z_r)^{k}]\le \mathfrak{C}_{k}' \cdot  n^{h_{k}} \cdot \mathrm{e}^{\mathfrak{C}_{k}' r} + \sum_{j=1}^{k-1} |a_{k-1,j}| \cdot \mathfrak{C}_j \cdot n^{h_j} \cdot \mathrm{e}^{\mathfrak{C}_j r}.\]
Since~$j \mapsto h_j$ is increasing, we can choose~$\mathfrak{C}_{k}$ (not depending on~$n$ or~$r$) such that the above expression is smaller than~$\mathfrak{C}_{k} \cdot n^{h_{k}} \cdot \mathrm{e}^{\mathfrak{C}_{k}r}$ for all~$r$. This proves~\eqref{eq_bound_high_moment_coll} once \eqref{eq_want_from_induction2} is proved. To prove~\eqref{eq_want_from_induction2}, fix~$r \ge 1$. We start by writing
\begin{equation}\label{eq_derivative_comp}\mathfrak{f}_{r}^{(k)}(1) = (\mathfrak{f}\circ \mathfrak{f}_{r-1})^{(k)}(1).\end{equation}

We will use the chain rule for higher-order derivatives (also known as Fa\`a di Bruno's formula); let us briefly state it. 
Let~$f,g: I \to \mathbb{R}$ be functions defined in an open interval~$I$ containing~$s \in \mathbb{R}$. Fix~$k \in \mathbb{N}$ and assume that~$f$ and~$g$ are~$k$ times differentiable in~$s$. Let $\CP_k$ denote the set of partitions of $\{1,\ldots, k\}$. For some $\CP=\{B_1, \dots, B_\ell\}\in \CP_k$, we let $|\CP|=\ell$ the number of blocks in $\CP$, and for $B\in \CP$ similarly we write $|B|$ the number of elements in $B$. Let then $\CP_{k,\ell}\subset \CP_k$ be the set of partitions containing $\ell$ blocks. Then,
\begin{align*}(f \circ g)^{(k)}(s) &= \sum_{\CP \in \CP_k} f^{(|\mathcal{P}|)}(g(s)) \cdot \prod_{B \in \mathcal{P}} g^{(|B|)}(s)\\
&=\sum_{\ell=1}^k \sum_{\{B_1, \dots, B_\ell\} \in \CP_{m,\ell}} f^{(\ell)}(g(s)) \cdot \prod_{j=1}^\ell g^{(|B_j|)}(s),
\end{align*}
Using this formula with $f=\mathfrak f$ and $g=\mathfrak f_{r-1}$ (together with~$\mathfrak{f}_{r-1}(1)=1$) in \eqref{eq_derivative_comp}, we have
\begin{equation} \label{eq_faa}
   \mathfrak{f}_{r}^{(k)}(1)= (\mathfrak{f}\circ \mathfrak{f}_{r-1})^{(k)}(1) = \sum_{\ell=1}^k \sum_{\{B_1, \dots, B_\ell\} \in \CP_{k,\ell}} \mathfrak f^{(\ell)}(1) \cdot \prod_{j=1}^\ell \mathfrak f_{r-1}^{(|B_j|)}(1). \end{equation} 
We now inspect each term in \eqref{eq_faa}. 
The value $\ell=1$ gives the trivial partition which consists of a single block~$\{1,\ldots,k\}$. The corresponding term is
\begin{equation} \label{eq_trivial_partition} \mathfrak{f}'(1) \cdot \mathfrak{f}_{r-1}^{(k)}(1) = \E[\mathcal{X}]  \cdot \mathfrak{f}_{r-1}^{(k)}(1) \qquad \mbox{when } \ell=1.\end{equation}
Now fix a partition~$\mathcal{P} = \{B_1,\ldots, B_\ell\}$ with $\ell\ge 2$. The corresponding term in \eqref{eq_faa} equals
\begin{align*}
\mathfrak{f}^{(\ell)}(1) \cdot \prod_{j=1}^\ell \mathfrak{f}_{r-1}^{(|B_j|)}(1)&= \E[\mathcal{X} (\mathcal{X}-1)\cdots (\mathcal{X}-\ell+1)]\\
&\quad \cdot \prod_{j\le\ell} \E[Z_{r-1}(Z_{r-1}-1) \cdots (Z_{r-1}-|B_j|+1)]\\
&\le \E[\mathcal{X}^\ell] \cdot \prod_{j\le \ell} \E[(Z_{r-1})^{|B_j|}] \stackrel{\eqref{eq_moments_empirical}}{\le} C_\ell \cdot n^{h_\ell} \cdot \prod_{j\le \ell} \E[(Z_{r-1})^{|B_j|}].
\end{align*}
Since~$\ell\ge 2$, each block has size~$|B_j| < k$. We thus use the induction hypothesis~\eqref{eq_step_first_induction} to bound the rhs as
\begin{align}
\nonumber
\mathfrak{f}^{(\ell)}(1) \cdot \prod_{j=1}^\ell \mathfrak{f}_{r-1}^{(|B_j|)}(1)&\le  C_\ell \cdot n^{h_\ell} \cdot \prod_{j=1}^\ell \mathfrak{C}_{|B_j|} \cdot n^{h_{|B_j|}} \cdot \mathrm{e}^{\mathfrak{C}_{|B_j|} (r-1) } \\
\nonumber &= C_\ell \Big(\prod_{j=1}^\ell \mathfrak{C}_{|B_j|} \cdot \mathrm{e}^{\sum_{j=1}^\ell \mathfrak{C}_{|B_j|} \cdot (r-1)}\Big) \cdot n^{h_\ell + \sum_{j=1}^\ell h_{|B_j|}}\\
&\le c'\mathrm{e}^{C'(r-1)} \cdot n^{h_\ell + \sum_{j=1}^\ell h_{|B_j|}}, \label{eq_power_of_n}
\end{align}
where $c', C'$ are constants that neither depend on $r$ nor on the partition $\CP$, and are given by
\[c':= (\max_{i \le k} C_i) \cdot (\max_{i\le k-1}  \mathfrak{C}_i)^{k},\qquad C':= k\cdot \max_{i \le k-1} \mathfrak{C}_i. \]
We inspect the exponent of $n$ that appears in \eqref{eq_power_of_n}, and set out to prove the inequality
\begin{equation}\label{eq_claim_with_h}
    h_\ell + \sum_{j\le \ell} h_{|B_j|} \le h_{k}.
\end{equation}
We consider two cases. The first case is when~$h_\ell = 0$. The superadditivity~\eqref{eq_super_additivity} yields that
\[h_\ell + \sum_{j\le \ell}h_{|B_j|} \le h_{\sum |B_j|} = h_{k}.\]
The second case is~$h_\ell > 0$, with a more involved proof. Recall $h_k$ from \eqref{eq:hk-def}. We write~$\alpha:= \frac{1}{\tau'-1}$ and~$\beta := 1- \frac{1}{\tau'-1}$, so that~$h_i = \max(\alpha i - \beta,0)$ for any~$i$, and carry out some formal rearrangements:
\begin{align*}
    h_\ell + \sum_{j\le \ell} h_{|B_j|} &= \alpha \ell - \beta + \sum_{j:h_{|B_j|} >0} (\alpha |B_j| - \beta)= -\beta+ \sum_{j=1}^\ell \alpha +\sum_{j:h_{|B_j|} >0} (\alpha |B_j| - \beta)\\
       &= -\beta+ \sum_{j:h_{|B_j|} = 0} \alpha + \sum_{j:h_{|B_j|} > 0} (\alpha|B_j| + \alpha - \beta)\\
       &\le -\beta+\sum_{j:h_{|B_j|} = 0} \alpha|B_j| + \sum_{j:h_{|B_j|} > 0} (\alpha|B_j| + \alpha - \beta).
\end{align*}
By the assumption in the lemma that $\tau'>3$,~$\alpha-\beta < 0$. Since $\{B_1, \dots, B_\ell\} \in \CP_{k,\ell}$, i.e., the blocks  partition $\{1,\dots, k\}$, $\sum_{j\le \ell}|B_j| =k$ holds which gives that 
\[h_\ell + \sum_{j\le \ell} h_{|B_j|}\le  -\beta+\sum_{j:h_{|B_j|} = 0} \alpha|B_j| + \sum_{j:h_{|B_j|} > 0} \alpha|B_j|  = \alpha k - \beta = h_{k}.\]
This completes the proof of~\eqref{eq_claim_with_h}. We substitute it as an upper bound in \eqref{eq_power_of_n} to obtain that for any~$\mathcal{P}\in \CP_{k,\ell}$ for any $\ell\ge 2$,
\[\mathfrak{f}^{(\ell)}(1) \cdot \prod_{B \in \mathcal{P}} \mathfrak{f}_{r-1}^{(|B|)}(1) \le c' \mathrm{e}^{C'(r-1)} \cdot n^{h_{k}}.\]
Substituting this bound into~\eqref{eq_faa} and using~\eqref{eq_trivial_partition} for $\ell=1$, we arrive at
\[\mathfrak{f}_{r}^{(k)}(1)  \le c'' \mathrm{e}^{C'(r-1)} \cdot n^{h_{k}}+ \E[\mathcal{X}] \cdot \mathfrak{f}^{(k)}_{r-1}(1),\]
where~$c'':= c' \cdot |\CP_k|=c'2^k$. This bound can now be used recursively: the same inequality (with~$(r,r-1)$ replaced by~$(r-1,r-2)$) can be used to bound~$\mathfrak{f}^{(k)}_{r-1}(1)$ on the right-hand side, and then further. This gives
\begin{align*}
\mathfrak{f}_{r}^{(k)}(1)  &\le c'' \mathrm{e}^{C'(r-1)} \cdot n^{h_{k}} \cdot (1 + \E[\mathcal{X}] + \E[\mathcal{X}]^2+\cdots + \E[\mathcal{X}]^r)\\
&\le c'' \mathrm{e}^{C'(r-1)} \cdot \mathbb{E}[\mathcal{X}]^{r+1} \cdot n^{h_{k}}.
\end{align*}
Now, we can choose~$\mathfrak{C}_{k}' > 0$ such that the right-hand side above is smaller than~$\mathfrak{C}_{k}' \mathrm{e}^{\mathfrak{C}_{k}' r} \cdot n^{h_{k}}$ for all~$r$. This completes the proof of~\eqref{eq_want_from_induction2}.
\end{proof}
We now proceed to embed $B_r(u_n)$ in Proposition \ref{prop:surplus} to a branching process that satisfies the conditions of Lemma \ref{lemma:moments}.
Recall $\nu_n(z)=n_z/n$ from \eqref{eq:empirical-degree}.
Define the \emph{size-biased version} and the \emph{down-shifted size-biased version} of $\nu_n$ as
\begin{equation}\label{eq:size-biased}
\nu_n^\star(z):=\frac{z\nu_n(z)}{\E[D_n]}, \qquad\mbox{and}\qquad  \widetilde\nu_n(z):=\frac{(z+1) \nu_n(z+1)}{\E[D_n]} = \frac{(z+1)n_{z+1}}{\sum_{i\le n}d_i}.
\end{equation}
If $D_n \sim \nu_n, D_n^\star\sim \nu_n^\star$ then $D_n^\star-1\sim \widetilde \nu_n$. It is well-known that $\nu_n \ {\buildrel d\over \le }\  \nu_n^\star$, i.e., the size-biased version of a random variable on $\N$ stochastically dominates the original measure. This follows from Harris' inequality: $\P(D_n^\star> z) \E[D_n]= \E[D_n \ind_{\{D_n> z\}}]\ge \E[D_n]\P(D_n> z)$ for any $z>0$. The next definition makes the tail of any starting distribution $\nu$ having a $q>1$ moment slightly heavier so that it also stochastically dominates $\nu$.
\begin{definition}[$\eta$-heavier-transformation of a probability measure]\label{def:eta-heavy}
    Let $\nu$ be a probability measure so that $\nu(z)\le z^{-\tau'}$ holds for all sufficiently large $z>0$. Let $\eta$ satisfy that $\tau'(1-\eta)>1$, and given a distribution $\nu$, let  $z_0^{\#}\ge 1$ be the smallest integer that satisfy the following:
    \begin{equation}\label{eq:z0-crit}
    \begin{aligned}
       \min_{z\ge z_0^{\#}: \nu(i)\neq 0} \nu(i)^{-\eta} &\ge 8/7, \qquad
        \sum_{z\ge z_0^{\#}} \nu(i)^{1-\eta} < 7/8.
        \end{aligned}
    \end{equation}
   Choose a normalising factor $Z:=Z(\eta, \nu)$ so that the following measure is a probability measure: 
\begin{equation}\label{eq:nu-hash}
 \nu^{\#}(z):=\begin{cases} 0 \quad&\mbox{if } z\le z_0^{\#}, \\
\nu(z)^{1-\eta}/Z \qquad &\mbox{if } z> z_0^{\#}\end{cases}   
\end{equation}
\end{definition}
The choice $7/8$ is quite arbitrary in \eqref{eq:z0-crit}, any number strictly less than $1$ would serve our purposes. 
\begin{claim}[Stochastic domination between $\nu$ and $\nu^{\#}$]\label{claim:hash-domin}
Let $\nu$ be a probability measure so that for some $\tau'>1$, $\nu(z)\le z^{-\tau'}$ holds for all sufficiently large $z>0$.
Then the measure $ \nu^{\#}$ exists and  stochastically dominates $\nu$ for all $\eta$ satisfying $\tau'(1-\eta)>1$, and has finite $q$-th moment for all $q<\tau'(1-\eta)-1$. Finally, $Z<7/8$.
\end{claim}
\begin{proof}
Suppose the measure exists. Then $Z<7/8$ follows from the second criterion in \eqref{eq:z0-crit} since $Z=\sum_{i>z_0^{\#}} \nu(i)^{1-\eta}\le 7/8$.
For $z\le z_0^{\#}$, $\nu^{\#}([0,z])\le \nu([0,z])$ is immediate from the first row in \eqref{eq:nu-hash}. For $z>z_0^{\#}$, we aim to show
$\nu((z,\infty))\le \nu^{\#}((z,\infty))$, which is equivalent to
\[
Z\sum_{i>z} \nu(i) \le \sum_{i>z} \nu(i)^{1-\eta},
\]
which holds since $Z<7/8$ and $\nu(i)<1$ implies that $\nu(i)\le\nu(i)^{1-\eta}$ for each $i>z$.
To see the moment conditions, for all $z\ge 1$ it holds that $Z\nu^{\#}(z)\le \nu(z)^{1-\eta}$, and so the $q$th moment is finite whenever $ \sum_{z\ge z_0^\#} z^q \nu(z)^{1-\eta}<\infty$, which in turn is at most $\sum_{z\ge z_0^\#} z^q z^{-\tau'(1-\eta)}$. This sum is convergent if $q-\tau'(1-\eta)<-1$, equivalently if $q<\tau'(1-\eta)-1$. This also gives with $q=0$ that $\tau'(1-\eta)>1$ is indeed sufficient for $z_0^{\#}$ in \eqref{eq:z0-crit} to exist and the normalising factor $Z$ to be finite. 
\end{proof}

The following exploration process gradually constructs the configuration model by matching half-edges sequentially in a way that reveals the graph neighborhood of a vertex $v_0$, $B_r(v_0)$, in a breadth-first search manner. The exploration also immediately couples the $r$-neighborhood $B_r(v_0)$ to the first $r$ generations of a random rooted tree $\CT^{\mathrm{expl}}_r$ so that  $B_r(v_0)\subseteq \CT^{\mathrm{expl}}_r$ holds a.s. under the coupling.

\begin{construction}[Exploration of the neighborhood of a vertex]\label{cons:exploration-1}\normalfont We take as input a degree sequence $\underline d_n$, a starting vertex $v_0$, a target radius $r$, and an additional offspring distribution $\zeta$. The coupled exploration of $B_r(0)$ in the configuration model $\mathrm{CM}(\underline d_n)$ is then as follows:

\emph{Step 0. Initialization.} To initialize, we set $v_0$ active and reveal its half-edges (say $h_1, \dots h_{d_{v_0}}$) and set also all of its half-edges active. We introduce the list of the active vertices $A_v(0):=\{v_0\}$ and of the active half-edges $A_h(0):=\{h_1, \dots, h_{d_{v_0}}\}$, and we set $\mathrm{Ex}_v(0):=\emptyset, \mathrm{Ex}_h(0):=\emptyset$ for the list of explored vertices and half-edges, respectively.  

\emph{Step $s$. Exploring a half-edge.}    In each discrete step $s\ge 1$ we take the first half-edge $h_s$ from $A_h(s-1)$, in a first-in-first-out (breadth-first search) order, and reveal the half-edge $m(h_s)$ it is matched to. We then append $h_s$ and $m(h_s)$ to the end of the list of explored half-edges $\mathrm{Ex}_h(s-1)$, obtaining $\mathrm{Ex}_h(s)$, and we remove $h_s$ from the active half-edges $A_h(s-1)$, and also remove $m(h_s)$ from it if it happened to belong to $A_h(s-1)$. Then we carry out three more substeps:

\emph{Substep s.(i): Adding newly discovered vertices.} If the vertex $v(m(h_s))$ that $m(h_s)$ is attached to is a new vertex, i.e., not in $A_v(s-1)$, then we append $v(m(h_s))$ to the end of the list $A_v(s-1)$, obtaining $A_v(s)$, and we append the remaining $X_s^{\sss{(n)}}$ many half-edges of $v(m(h_s))$ to the end of the active half-edge list, obtaining $A_h(s)$. We call $X_s^{\sss{(n)}}$ the forward degree of the vertex discovered in step $s$.

\emph{Substep s.(ii) Handling loops and creating ghost subtrees.} If, however, the half-edge $m(h_s)$ is already active and it is attached to an active vertex $v(m(h_s))$, then we call this a \emph{collision} at step $s$. This creates a loop and hence a surplus edge in $B_r(v_0)$. We then do the following: in $B_r(v_0)$ we create the loop formed by $(h_s, m(h_s))$, and in $\CT^{\mathrm{expl}}_r$ we create two `ghost' subtrees as follows. Let $r_1:=d_G(v_0, v(h_s)), r_2:=d_G(v_0, v(m(h_s)))$, respectively.  We then sample two independent branching processes, $\CT^{\#, s1}_{r -r_1}$ and $\CT^{\#,s2}_{r -r_2}$ with offspring distribution $\zeta$, (the first one has depth $r-r_1$ while the second one has depth $r-r_2$) and attach their root to the half-edges $h_s$ and $m(h_s)$ respectively, and add these ghost-subtrees to $\CT_r^{\mathrm{expl}}$. 

\emph{Substep s.(iii): Checking for vertices being fully explored.} If the half-edges of the vertices $v(h_s)$ and/or $v(m(h_s))$ are \emph{all} explored after substep s.(ii), then we append $v(h_s)$ and/or $v(m(h_s))$ also to the set of explored vertices $\mathrm{Ex}_v(s)$, otherwise we keep them active.

\emph{Stopping condition.} The exploration stops when we have matched all half-edges belonging to vertices at graph distance $r -1$ from $v_0$. We denote the number of needed steps by $t(r)$.

\emph{Output.} The output is the graph $B_r(v_0)$ and the tree $\CT_r^{\mathrm{\mathrm{expl}}}$. We denote the number of half-edges added in step $s$ to the active half-edges by $X_s^{\sss{(n)}}$, giving the random sequence $X_1^{\sss{(n)}}, X_2^{\sss{(n)}}, \dots, X_{t(r)}^{\sss{(n)}}$, with the convention that we set $X_s^{\sss{(n)}}:=0$ if a collision have occurred at step $s$ and no new vertex was added. We denote by $\mathrm{Coll}_r(v_0)$ the number of collisions that occurred during the process. \end{construction}

\begin{observation}\label{obs:last-gen}\normalfont The exploration reveals the whole graph (including all loops) within $B_{r-1}(v_0)$, and also the size of $B_{r}(v_0)$. To see the latter, by the stopping condition, we have explored all vertices in generation $r-1$, and their forward degrees, say $X_{s_{r-1}}^{\sss{(n)}}, \dots, X_{t_{r-1}}^{\sss{(n)}}$ are thus known. Matching then all these half-edges reveals edges between at least one vertex in generation $r-1$, and the other vertex can be either in generation $r-1$ or $r$. For each edge where the other vertex is also in generation $r-1$, a loop between two vertices in generation $r-1$ arises, and the size of $B_{r}(v_0)$ is reduced by $2$ compared to $\sum_{i\in[s_{r-1}, t_{r-1}]} X_{i}^{\sss{(n)}}$. Each collision where two edges lead to the same vertex in generation $r$, reduces the size of $B_{r}(v_0)$ compared to $\sum_{i\in[s_{r-1}, t_{r-1}]} X_{i}^{\sss{(n)}}$ by $1$. Note that $B_r(v_0)\subseteq \CT_r^{\mathrm{expl}}$ for any offspring distribution $\zeta$. 
\end{observation}
\begin{observation}\label{obs:surplus}\normalfont
All surplus edges are either self-loops, multiple edges, or between two vertices, say $v, v'$ so that the distance between $|d_G(u_n, v)-d_G(u_n, v')|\le 1$. Indeed, when a surplus edge is created, the half-edge $h_s$ is matched to an active half-edge in $A_h(s-1)$. All half-edges in $A_h(s-1)$ either belong to the same generation as $v(h_s)$ or they belong to the next generation.
\end{observation}

Recall the size biasing from \eqref{eq:size-biased} and the hash-transformation of a measure from \eqref{eq:nu-hash} in Definition \ref{def:eta-heavy}.
\begin{lemma}\label{lemma:coupled-exploration}
Consider Construction \ref{cons:exploration-1} started from a uniformly chosen vertex $v_0:=u_n$ on the configuration model $\mathrm{CM}(\underline d_n)$ so that $(\underline d_n)_{n\ge 1}$ satisfies Assumptions \ref{assu:regularity} and \ref{assu:empirical-power-law-2} with some $\tau, \varepsilon, c_u, z_0$ for all sufficiently large $n$ so that $\tau(1-\eps)>2$ in \eqref{eq:point-mass}. Let $\eta>0$ be so that $(\tau(1-\eps)-1)(1-\eta)>1$.
Assume that the number of exploration steps $t(r)\le \sum_{i\in[n]}d_i/17$. Then, for all sufficiently large $n$, the forward-degree sequence $(X_s^{\sss{(n)}})_{s\le t(r)}$ is stochastically dominated by an iid sequence $(Y_s)_{s\le t(r)}$
from $(\nu_n^\star)^{\#}$ defined from \eqref{eq:size-biased} and \eqref{eq:nu-hash}. 
Under Assumption \ref{assu:empirical-power-law-2} this measure satisfies for some constant $c_u'<\infty$:
\begin{equation}\label{eq:widehat-nu}
 \widehat\nu_n(z):=(\nu_n^\star)^{\#}(z) \le c_u' z^{-(\tau(1-\eps)-1)(1-\eta)}.
 \end{equation}
As a result, there exists a coupling $B_r(u_n)\subseteq \CT_r^{\mathrm{expl}}\subseteq \CT^{\#}_r$  where $\CT^{\#}_r$ is the first $r$ generations of a branching process having iid offspring from $\widehat\nu_n(z)$.
\end{lemma}
\begin{remark}\normalfont With the same method it could also be proved that $(X_s^{\sss{(n)}})_{s\le t(r)}$ is stochastically dominated by an iid sequence $(Z_s)_{s\le t(r)}$
from $(\widetilde \nu_n)^{\#}$ defined from \eqref{eq:size-biased} and \eqref{eq:nu-hash}, the $\eta$-heavier transformation of the down-shifted size-biased version of $\nu_n$. In that case, however, the root's degree $d_{u_n}$ cannot necessarily be dominated by $(\widetilde \nu_n)^{\#}$. Further, $(\widetilde \nu_n)^{\#}$ and $(\nu_n^\star)^{\#}$ both satisfy the same inequality \eqref{eq:widehat-nu}, so for simplicity we dominate by a `usual' GW tree $\calT_r^{\#}$ where all vertices have the same offspring distribution.
\end{remark}
The proof will follow from the following statement and Construction \ref{cons:exploration-1}.
\begin{claim}[Domination and size-biasing during the exploration]\label{claim:removal-Delta}
Let $\nu_n$ be the empirical measure of $\underline d_n=(d_1, \dots d_n)$ in \eqref{eq:empirical-degree} satisfying that $\nu_n^\star(z)\le c_u z^{-\tau'}$ for all $z\ge z_0$ for some $\tau'>1$ in \eqref{eq:size-biased}. For a subset $\Delta\subset[\sum_{i\le n}d_i]$, remove the half-edges with label in $\Delta$ to obtain a new degree sequence $\underline d_n^{\Delta}:=(d_1'\ind_{d_1'\neq 0}, \dots d_n'\ind_{d_n'\neq 0})$, and 
let $\nu_{n,\Delta}^\star$ denote the size-biased version of the empirical distribution of $\underline d_n^{\Delta}$. Then, for any choice of $\Delta$ with $|\Delta|\le (\sum_{i\in[n]}d_i)/8$, $\nu_{n,\Delta}^\star$ is stochastically dominated by $(\nu_n^\star)^{\#}$ for any $\eta$ so that $\tau'(1-\eta)>1$.
\end{claim}
\begin{proof}We assume here that $\nu_n^\star(z)\le c_u z^{-\tau'}$ for all $z\ge z_0$. 
Then, Claim \ref{claim:hash-domin} gives that $\nu_n^\star$ is stochastically dominated by $(\nu_n^\star)^{\#}$ whenever $\tau'(1-\eta)>1$. So when $\Delta=\emptyset$ then the statement holds. Recall that $\nu_n(z)=n_z/n$, and let $h_n:=\sum_{i\in[n]}d_i$. Then $\sum_{i\in[n]}d_i'\ind_{d_i'\neq 0}= h_n-|\Delta|$ since we removed $|\Delta|$ many half-edges. Recall $z_0^\#$ from \eqref{eq:z0-crit} and \eqref{eq:nu-hash}. Let us first consider any $z<z_0^\#$. Clearly $\nu_{n,\Delta}^\star([0,z])\ge 0$ while $(\nu_{n}^\star)^\#([0,z])=0$ so the criterion for stochastic domination $\nu_{n,\Delta}^\star([0,z])\ge (\nu_{n}^\star)^\#([0,z])$ holds in this case. 
Let now $z\ge z_0^{\#}$. Observe that all degrees can only decrease by removing half-edges, hence writing $n_i':=\sum_{j\in[n]}\ind_{d_j'=i}$ for the number of vertices of degree $i$ after removing the half-edges with label in $\Delta$, it holds that $\sum_{i>z}i  n_{i}'\le \sum_{i>z}i n_{i}$ relating to \eqref{eq:size-biased}.
 Now we look at the upper tail using that $|\Delta|\le n\E[D_n]/8$
\begin{equation*}
\nu_{n,\Delta}^\star((z,\infty))=\frac{\sum_{i>z}i  n_{i}'}{h_n-|\Delta|}\le \frac{\sum_{i> z}i n_{i}}{n\E[D_n](1-1/8)} = \nu_n^\star((z,\infty))\cdot 8/7. 
\end{equation*}
At the same time, using that $Z<7/8$ in Claim \ref{claim:hash-domin}, the tail of $(\nu_n^\star)^\#$ satisfies:
\[ (\nu_n^\star)^{\#}((z,\infty))=\frac1Z\Big(\sum_{i>z}\nu_n^\star(i)^{1-\eta}\Big)\le \nu_n^\star((z,\infty))/Z \ge\nu_n^\star((z,\infty))\cdot 8/7.   \]
Hence the stochastic domination criterion $\nu_{n,\Delta}^\star((z,\infty))\le (\nu_n^\star)^{\#}((z,\infty))$ is satisfied.
\end{proof}

\begin{proof}[Proof of Lemma \ref{lemma:coupled-exploration}]
Let us write $h_n:=\sum_{i\le n}d_i$. 
Consider step (s) of Construction \ref{cons:exploration-1}, when we match half-edge $h_s$. Its pair $m(h_s)$ is chosen uniformly among 
the available $h_n-2s-1$ many half-edges at step $s$. At this point the half-edges \emph{not} available for matching $h_s$ to form the set $\Delta_s:=\mathrm{Ex}_h(s-1)\cup \{h_s\}$.
Consider the `available' degrees at this moment, say $\underline d_n^{(s)}:=(d_1^{(s)}\ind_{d_1^{(s)}\neq 0}, \dots, d_n^{(s)}\ind_{d_n^{(s)}\neq 0})$, where $d_j^{\sss{(s)}}$ is the number of not-matched half-edges of vertex $j$ before step $s$ if $h_s$ is not attached to vertex $j$ and $1$ less if $h_s$ is attached to vertex $j$. Since we choose the half-edge $m(h_s)$ uniformly at random from the currently available half-edges, the vertex $v(m(h_s))$ that $m(h_s)$ is attached to is chosen \emph{size-biasedly} from $\underline d_n^{(s)}$, conditionally independently of previous matchings, i.e., its forward degree then
\[\P(X_s^{\sss{(n})}=i-1)=\P(d_{v(m(h_s))}^{\sss{(s)}}=i) = \frac{ i \sum_{j\in[n]} \ind_{\{d_j^{\sss{(s)}}=i\}}}{h_n-2s-1} =\nu_{n, \Delta_s}^\star (i), \]
with $\Delta_s:=\mathrm{Ex}_h(s-1)\cup \{h_s\}$. In particular $X_s^{\sss{(n)}}+1$ follows the measure $\nu_{n, \Delta_s}^\star (i)$ in Claim \ref{claim:removal-Delta}.
Thus, let us apply Claim \ref{claim:hash-domin} with $\Delta_s:=\mathrm{Ex}_h(s-1)\cup \{h_s\}$, i.e., removing the set of unavailable half-edges. Since $t_r\le h_n/17$, we have $|\Delta_s|\le 2h_n/17+1\le h_n/8$ so Claim \ref{claim:removal-Delta} applies. By Claim \ref{claim:removal-Delta}, the measure $\nu_{n, \Delta_s}^\star(i)$ is stochastically dominated by $(\nu_n^\star)^\#$ for each $s$, so let $Y_s $ be such a random variable. Using the conditional independence of the consecutive matchings, one can thus construct a coupling where $X_s^{\sss{(n)}}\le X_s^{\sss{(n)}}+1\le Y_s$ and $Y_s$ are iid from $(\nu_n^\star)^\#$. Further, since $u_n$ is a vertex chosen uniformly at random, the root's degree $d_{u_n}$ has distribution $\nu_n$. By below \eqref{eq:size-biased}, the measure $\nu_n^\star$ stochastically dominates $\nu_n$. So it holds that 
\[ \nu_n\ {\buildrel d\over \le}\  \nu_n^\star\ {\buildrel d\over \le}\ (\nu_n^\star)^{\#}, \] 
and thus one can construct a coupling where $d_{u_n}\le Y_0$ with $Y_0$ from $(\nu_n^\star)^{\#}$.
To finish, recall that whenever the exploration discovers a loop at some step $s$, it appends two ghost subtrees to the half-edges $h_s$ and $m(h_s)$ exactly so that their last generation ends at distance $r$ from $u_n$. Setting the offspring distribution of these branching processes to be also $\widehat\nu_n=(\nu_n^\star)^{\#}$ gives then a coupling where $B_r(0)$ is embedded in $\CT_r^{\mathrm{expl}}$ which are both embedded in $\CT^{\#}_r$, a branching process where  all vertices have iid degree from $\widehat \nu_n$. 

Using Assumption \ref{assu:empirical-power-law-2} we now bound $\widehat \nu(z)=(\nu_n^\star)^\#(z)$ for all $z\ge z_0^\#\vee z_0$. Since we assumed $\nu_n(z)\le c_u z^{-\tau(1-\eps)}$ with $\tau(1-\eps)>2$, it holds for some finite constant $\overline m$ that $\E[D_n]<\overline m<\infty$ uniformly for all $n$, and Assumption \ref{assu:regularity} also ensures that  $\E[D_n]\ge \underline m$ for some $\underline m$, uniformly for all $n$. Hence $\nu_n^\star(z)\le c_u (z+1) z^{-\tau(1-\eps)}/\underline m$ for all $n$ and all $z\ge z_0$. Finally, for all $z\ge z_0^\#\vee z_0$
\[ \widehat\nu_n(z)=(\nu_n^\star)^{\#}(z)=\frac{1}{Z}\nu_n^\star(z)^{1-\eta}=\frac{c_u^{1-\eta}}{Z \underline m^{1-\eta}}z^{1-\eta}z^{-\tau(1-\eps)(1-\eta)}\le c_u' z^{-(\tau(1-\eps)-1)(1-\eta)},  \]
which proves \eqref{eq:widehat-nu}. The condition $(\tau(1-\eps)-1)(1-\eta)>1$ is necessary for the hash-measure to exist in Claim \ref{claim:hash-domin}. 
\end{proof}
We are ready to prove Proposition \ref{prop:surplus}.
\begin{proof}[Proof of Proposition \ref{prop:surplus}]
We start by applying Lemma \ref{lemma:coupled-exploration}. This gives that $B_r(u_n)$ is contained in a BP tree $\CT_r^{\#}$ as long as the number of half-edges explored is $t(r)<n\E[D_n]/17$, with offspring distribution $\widehat \nu_n$ defined in \eqref{eq:widehat-nu}. Next, we ensure that this measure satisfies the conditions \eqref{eq:zeta-support} and \eqref{eq:not-integer} so that we can use the moment bounds of Lemma \ref{lemma:moments}.
To see \eqref{eq:zeta-support} is satisfied, we observe that $\widehat \nu_n$ has power-law exponent
$\tau'-1:=(\tau(1-\eps)-1)(1-\eta)>2$, i.e., $\tau'>3$, and we can easily ensure that $\tau'\notin \N$ by changing $\eta$ if necessary. The condition on the maximum of the support in \eqref{eq:zeta-support} follows from Assumption \ref{assu:empirical-power-law-2} since the exponent $1/(\tau(1-\eps)-1)$ there is less than $1/(\tau'-1)$ which is allowed in \eqref{eq:zeta-support}. Hence Lemma \ref{lemma:moments} is applicable for the BP tree $\CT_r^{\#}$ in Lemma \ref{lemma:coupled-exploration}.

 By Observation \ref{obs:last-gen}, in order to bound also the surplus edges in $B_{\delta \log n}$ we need to reveal the size of one more generation, and so we set out to bound $|\CT^{\#}_{\delta \log n+1}|$ for some $\delta>0$.  Set $r_n:=\delta \log n+1$. Let $k\in \N$, and $\zeta>0$ to be determined later. We use first the increasing function $x^k$, then Markov's inequality, and then Minkowski's inequality in the second inequality:
 \begin{equation*}
 \P\big(|\CT_{r_n}^{\#}|\ge n^{\zeta}\big)=\P\big((|\CT_{r_n}^{\#}|)^k\ge n^{k\zeta}\big)\le n^{-k\zeta} \E\Big[\Big(\sum_{i\le k_n}Z_i\big)^k\Big] \le n^{-k\zeta} \Big(\sum_{i\le r_n} \E\big[Z_i^k\big]^{1/k} \Big)^k.
\end{equation*}
 We now apply Lemma \ref{lemma:moments} on $\E[Z_i^k]$ for each $i\le r_n$:
  \begin{equation*}
 \P\big((|\CT_{r_n}^{\#}|)^k\ge n^{k\zeta}\big)\le  n^{-k\zeta} \Big(\sum_{i\le r_n} (\mathfrak{C}_k \cdot n^{h_k}\cdot \mathrm{e}^{\mathfrak{C}_k i})^{1/k} \Big)^k = n^{-k\zeta}\cdot   \mathfrak{C}_k \cdot n^{h_k} \Big(\sum_{i\le r_n} \mathrm{e}^{\mathfrak C_k i/k} \Big)^{k}.
 \end{equation*}
 The sum on the rhs is geometric and since $\mathfrak C_k>0$, it is at most $C' \e^{(\mathfrak{C}_k/k) r_n}$ for some constant $C'$, with $r_n=\delta \log n+1$, which gives 
   \begin{equation*}
 \P\big((|\CT_{r_n}^{\#}|)^k\ge n^{k\zeta}\big)\le  n^{-k\zeta}\cdot  \mathfrak{C}_k \cdot n^{h_k} C'^{k} n^{\delta \mathfrak C_k} = C n^{-k \zeta + h_k + \delta \mathfrak C_k}.
 \end{equation*}
 We inspect the exponent of $n$. Recall that $h_k=((k+1)/(\tau'-1)-1) \vee 0$ from \eqref{eq:hk-def}. Since $\tau'-1>2$, we may write
 \begin{equation}\label{eq:exponent} -k \zeta + h_k + \delta \mathfrak C_k = -k(\zeta-\tfrac{1}{\tau'-1}) - (1-\tfrac{1}{\tau'-1})+\delta \mathfrak C_k.\end{equation}
 The exponent of $n$ can be made strictly less than $-1$ for sufficiently large $k$ if $\zeta>1/(\tau'-1)$. Since $\tau'-1=(\tau(1-\eps)-1)(1-\eta)$ with $\eta$ arbitrarily small, this yields the formulation $B_{\delta\log n}> n^{(1+\eps')/(\tau(1-\eps)-1)}$ in \eqref{eq:size} of the proposition. For any such $\zeta$ one can now choose $k\in \N$ so large that the exponent goes below $-1$, in particular any $k$ satisfying  $k> \zeta/(\tau'-1)-1$ is a good choice. Given $k$, one now chooses $\delta$ small enough so that the whole exponent in \eqref{eq:exponent} still stays below $-1$, giving also $\delta'>0$.
 
 By the coupling  $B_{\delta \log n}(u_n)\subseteq B_{\delta \log n+1}(u_n)\subseteq \CT_{\delta \log n+1}^{\#}$, we have just proved
 \begin{equation}\label{eq:delta-logn}
 \P(\CA_{\mathrm{size}}):=\P\big( |B_{\delta \log n+1}(u_n)| \le n^{(1+\eps')/(\tau(1-\eps)-1)}\big) \ge 1-n^{-1-\delta'}, 
 \end{equation}
 and then by monotonicity $\{|B_{\delta \log n}(u_n)| \le n^{(1+\eps')/(\tau(1-\eps)-1)}\}$ also holds with the same error probability. 
 Now we start bounding the surplus edges inside $B_{\delta \log n}(u_n)$. On the event $\calA_{\mathrm{size}}$, the exploration in Construction \ref{cons:exploration-1} finishes in $t(\delta \log n)\le n^{\zeta}$ with $\zeta:=(1+\eps')/(\tau(1-\eps)-1)$ steps, and by Observation \ref{obs:last-gen}, the exploration reveals $B_{\delta \log n}(u_n)$ and all surplus edges inside. We estimate the probability of a collision from above at each step of the exploration. When the exploration is at step $s$, a collision happens if the half-edge $h_s$ is matched to one of the active half-edges in $A_h(s-1)$, see substep s.(ii) in Construction \ref{cons:exploration-1}. The size of $A_h(s-1)$ is at any time no more than the total size of $B_{\delta \log n +1}(u_n)$, i.e., at most  $n^{\zeta}$.
Hence, since $s\le n^\zeta$ on $\calA_{\mathrm{size}}$ also,  and so
\[ \P(\mbox{a surplus edge is created at step } s)\le \frac{n^{\zeta}}{h_n-2s-1}\le \frac{2}{\E[D_n]}n^{\zeta-1}:=cn^{\zeta-1},\]
uniformly for all $s\le n^\zeta$, and conditionally independently of other steps. One can thus dominate the sequence of indicators of whether 
a surplus edge is created at step $s$ by an iid sequence of $n^\zeta$ many Bernoulli random variables with mean $2n^{\zeta-1} \E[D_n]$. Thus,   
the number of collisions is at most $\mathrm{Bin}(n^{\zeta}, c n^{\zeta-1})$. Since $\zeta=(1+\eps')/(\tau(1-\eps)-1)$, and we assumed $\tau(1-\eps)-1>2$, we have $\zeta<1/2$, and 
so  the mean, $\Theta(n^{2\zeta-1})$ tends to zero for small $\eps'$. For some $\ell$ to be chosen later, we bound 
\begin{align*}
\P( \mathrm{Surp}_{\delta \log n}(u_n) \ge \ell \mid \calA_{\mathrm{size}})&\le \P(\mathrm{Bin}(n^{\zeta}, cn^{\zeta-1})\ge \ell)\\
&\le  \sum_{i\ge \ell} \binom{n^\zeta}{i} (cn^{\zeta-1})^i \le  \sum_{i=\ell}^{\infty} (cn^{2\zeta-1})^i \le c' n^{(2\zeta-1)\ell},
\end{align*}
where we used that $\binom{n^\zeta}{i}\le n^{\zeta i}$, and that the geometric sum in the middle has base less than $1$ for all sufficiently large $n$ since $2\zeta-1<0$. 
Choose now $\ell$ so large that the exponent of $n$ on the rhs, $(2\zeta-1)\ell<-1$, i.e., $\ell>1/(1-2\zeta)$.
Then one has for some $\delta'>0$ that
\begin{equation}\label{eq:surplus}\P( \mathrm{Surp}_{\delta \log n}(u_n)\ge \ell \mid \calA_{\mathrm{size}}) \le n^{-1-\delta'}.
 \end{equation}
 One can compute using $\zeta$ that $\ell \ge \tfrac{\tau(1-\eps)-1}{\tau(1-\eps) -3 -2\eps'}$
which also shows that $\tau(1-\eps)>3$ is necessary for the argument to work.
Combining now \eqref{eq:delta-logn} with \eqref{eq:surplus} with a union bound finishes the proof of \eqref{eq:size}.
Finally we estimate the maximal multiplicity of the edges in the whole graph. We introduce $\ind_{u,v}^{(\ell)}:=1$ if there are at least $\ell$ edges between vertex $u$ and $v$.
Then by Markov's inequality, and pairing $\ell$ chosen half-edges from $u$ and from $v$ together yields that
\begin{equation}\label{eq:multiple-1}
\begin{aligned}
 \P\Big(\max_{u,v\in[n]} e(u,v)\ge \ell\Big)&=\P\Big(\sum_{u,v\in[n]} \ind_{u,v}^{(\ell)}\ge 1\Big) \le \E\Big[\sum_{u,v\in[n]} \ind_{u,v}^{(\ell)}\Big] \\
 &\le \sum_{u,v} \frac{d_u^\ell d_v^\ell}{(h_n-2\ell-1)^\ell}\le c n^2 \E[D_n^\ell]^2/n^\ell
 \end{aligned}
 \end{equation}
 for some constant $c>0$. Using \eqref{eq:point-mass} and \eqref{eq:max-degree} in Assumption \ref{assu:empirical-power-law-2}, with $M_n:=C_u n^{1/(\tau(1-\eps)-1)}$ and so one bounds the moment as
\begin{align*}
\E[D_n^\ell] &\le \sum_{z\le M_n} c_u z^{\ell-\tau(1-\eps)} \le c \int_1^{M_n} z^{\ell-\tau(1-\eps)} \mathrm dz\\
&\le C M_n^{(\ell+1-\tau(1-\eps))\vee 0} = n^{(\ell/(\tau(1-\eps)-1) -1)\vee 0},
\end{align*}
similarly to $h_\ell$ in \eqref{eq:hk-def}. If now the maximum is at $0$ in the exponent, one obtains $\ell>3$ is necessary for the exponent to be below $-1$, and if the maximum is at the other term $\ell/(\tau(1-\eps)-1) -1$ then one obtains $\ell>(\tau(1-\eps)-1)/(\tau(1-\eps)-3)$ then the exponent in \eqref{eq:multiple-1} is less than $-1$. Hence 
$\ell > 3\vee (\tau(1-\eps)-1)/(\tau(1-\eps)-3)$ is a sufficient choice, finishing the proof of \eqref{eq:multiple-edges} and thus the proposition. 
\end{proof}

With Proposition \ref{prop:surplus} at hand, we now move on to analyze the contact process on $B_{\delta \log n}(u_n)$.
On the event in \eqref{eq:size}, $B_{\delta \log n}(u_n)$ has at most $\ell$ surplus edges. By Observation \ref{obs:surplus}, all the surplus edges created during the exploration are either self-loops, multiple edges, or the distance between the root $u_n$ and the two end-vertices of the surplus edge differ by at most $1$. We will apply the next lemma to bound the number of non-backtracking infection paths of the contact process on $B_{\delta \log n}(u_n)$.

%% file: last_lemma_temp.tex
Recall from Definition \ref{def:labels} that $\mathscr{T}(G)$ denotes the genealogical label of particles in the contact process, equivalently, the set of possible infection paths $\pi$ on $G$. Recall also that $\mathfrak{l}(\pi)$ is the length of the path (number of edges) from \eqref{eq:length-pi}, while $\tau(\pi)$ in \eqref{eq:tau-def} denotes the location of the first backtracking step on the path, with the convention that $\tau(\pi)=\infty$ if the path is non-backtracking.
\begin{lemma}\label{lemma:suprplus-paths}
	Let~$\mathcal{T} = (V,E)$ be a tree with root~$\varnothing$; assume that~$\mathcal{T}$ has no self-loops or parallel edges. Let~$N, k \in \mathbb{N}$. Let $u_1,v_1,u_2,v_2,\ldots,u_k,v_k\in V$ be (not necessarily distinct) vertices such that for all $i\in\{1,\ldots,k\}$
 \begin{equation}\label{eq_distances1}
 0\le\mathrm{dist}_{\mathcal{T}}(\varnothing,u_i)\le\mathrm{dist}_{\mathcal{T}}(\varnothing,v_i)\le\mathrm{dist}_{\mathcal{T}}(\varnothing,u_i)+1.
 \end{equation}
Consider another graph $\mathcal{T}^{(k)}$ on the same vertex set $V$, with edge set $E':=E\cup\{u_1,v_1\}\cup\ldots\cup\{u_k,v_k\}$.
Let $\CT_N:= \{v \in V:\; \mathrm{dist}_{\mathcal{T}}(\varnothing,v) \le N\}$ as before, and define
\begin{align}
\CB_N:= \{\pi \in \mathscr{T}(\mathcal{T}^{(k)}):\; \pi_0 = \varnothing,\; \mathfrak{l}(\pi) \le N,\; \tau(\pi) = \infty\}. \label{eq_want_count}
\end{align}
Then $|\CB_N|\le (2k+1)^N |\CT_N|$.
\end{lemma}
The lemma allows for self-loops and multiple edges, these also satisfy \eqref{eq_distances1}.
\begin{proof}

 We start by introducing a labelling of the directed edges of any path $\pi\in\calB_N$, describing whether the edge uses a surplus edge in one of the two possible directions, or the edge is not a surplus edge. So introducing the symbol $o$ for the latter,  
we define the set of possible labels $\CL$, and we then introduce $\mathrm{Seq}_N$ as the set of length-$N$ sequences with elements from $\CL$ with a vertex in $\CT$ appended at the end:
\begin{align}
\label{eq_def_E}
\CL&:=\bigcup_{1\le i\le k}\{(u_i,v_i),(v_i,u_i)\}\cup\{o\},
\\
\label{eq_def_S}
\mathrm{Seq}_N&:=\{(s_1,s_2,\ldots,s_N,v):\ s_j\in\CL,v\in V, \mathrm{dist}_{\mathcal{T}}(\varnothing,v)\le N\}.
\end{align}
Observe that $|\CL|\le 2k+1$ (self-loops and multiple edges can make this inequality strict) and thus $|\mathrm{Seq}_N|\le (2k+1)^N |\CT_N|$. Therefore, if we show that there is an injection from $\mathcal{B}_N$ to $\mathrm{Seq}_N$, it will yield
\[|\mathcal{B}_N|\le|\mathrm{Seq}_N|\le(2k+1)^N |\CT_N|,\]
proving the lemma. We now construct this injection.

Fix any $\pi=(\pi_0,\pi_1,\ldots,\pi_m)\in\mathcal{B}_N$, where $m=\mathfrak{l}(\pi)$. We think of this path as the sequence $(e_1, e_2, \ldots, e_m)$ with $e_j=(\pi_{j-1}, \pi_j)$ a directed edge.  
By the definition of $\mathcal{B}_N$ in \eqref{eq_want_count}, $m\le N$. Recalling the labels from \eqref{eq_def_E}, for each $1\le j\le m$ define
\begin{equation*}
s_j:=
\begin{cases}
(u_i,v_i)&\text{if } e_j=(u_i,v_i) \text{ for some $i\le k$},\\
(v_i,u_i)&\text{if } e_j=(v_i,u_i) \text{ for some $i\le k$},\\
o&\text{otherwise}.
\end{cases}
\end{equation*}
Furthermore, define $s_j=o$ for each $m+1\le j\le N$, and finally, let $v=\pi_{m}$. By the condition \eqref{eq_distances1}, each edge in $\pi$ can only change the distance from $\varnothing$ by at most $1$, thus $\mathrm{dist}_{\mathcal{T}}(\varnothing,v)\le N$. Hence, we associate a vector $L(\pi):=(s_1,\ldots,s_N,v)\in\mathrm{Seq}_N$ to each $\pi\in\mathcal{B}_N$. We will show that this mapping is injective, that is, $(s_1,\ldots,s_N,v)$ uniquely encodes the path $\pi$.

For each $1\le j\le N$ the label $s_j$ reveals whether the edge $e_j$ crosses one of the surplus edges $\{u_i,v_i\}$, and if so, in which direction. Between two consecutive crossings,  $\pi$ is a non-backtracking path on the edges of the tree $\mathcal{T}$, hence it is uniquely determined, since in a tree there is a single non-backtracking path between any two vertices: e.g. if $s_j=(u_i,v_i)$ and $s_{j'}=(u_{i'},v_{i'})$ for $j<j'$ and $s_{j+1}=\ldots=s_{j'-1}=o$, then $(\pi_j,\ldots,\pi_{j'-1})$ is the unique geodesic (i.e., non-backtracking shortest path) in $\mathcal{T}$ from $v_i$ to $u_{i'}$. A similar argument shows that if $j_{\mathrm{max}}=\max\{j:s_j\ne o\}$, then $(\pi_{j_{\mathrm{max}}},\ldots,\pi_{\mathfrak{l}(\pi)})$ is the unique geodesic in $\mathcal{T}$ from the endpoint of $s_{j_{\mathrm{max}}}$ to $v$, the endpoint of $\pi$. This shows that the defined map is indeed injective, finishing the proof.
\end{proof}

\begin{proof}[Proof of Theorem \ref{thm:max_CM_extinction}(a)]
Let $G_n$ be a realization of $\mathrm{CM}(\underline d_n)$.
Recalling from Lemma \ref{lem:StochDom2} the stochastic domination between CP and BRW, and that a branching random walk with initial configuration $\underline \xi_0$ can be realized as the union of independent BRWs, each started from a single particle present in $\underline \xi_0$, we obtain that
\[\CPf(G_n, \underline 1_{G_n}) \ {\buildrel d \over \le}\ \BRW(G_n,\underline 1_{G_n} ) = \bigcup_{v\in[n]} \BRW(G_n, \ind_{v})=:\bigcup_{v\in[n]} \underline x_t^{(v)},\]
where the branching random walks $\underline x_t^{(v)}$ are independent given $G_n$.
Let now $T_\mathrm{ext}$ denote the extinction time of $\BRW(G_n,\underline 1_{G_n} )$, and let 
$T_{\mathrm{ext}}^{(v)}$ denote the extinction time of $ \underline x_t^{(v)}$. Then $T_\mathrm{ext}=\max_{v\in[n]} T_{\mathrm{ext}}^{(v)}$.
Hence for any $t>0$, 
\begin{equation}\label{eq:union-bound-brw} \P\big(T_{\mathrm{ext}}>t\big) =\P\big(\exists v\in[n]: T_{\mathrm{ext}}^{(v)}>t\big)\le n \cdot\Big(\frac{1}{n}\sum_{v\in[n]} \P\big(T_{\mathrm{ext}}^{(v)}>t\big)\Big)=n\cdot \P\big( T_{\mathrm{ext}}^{(u_n)}>t\big), \end{equation}
where $u_n$ is a uniformly chosen vertex. 
We will show that for some $C>0$, $\P\big( T_{\mathrm{ext}}^{(u_n)}>C\log n\big)=o(1/n)$.   
which then shows that the extinction time is  $O_\P(\log n)$ by \eqref{eq:union-bound-brw}.

We first apply Proposition \ref{prop:surplus}, which is applicable since its conditions coincide with that of Theorem \ref{thm:max_CM_extinction}(a). Proposition \ref{prop:surplus} then gives constants $\delta, \delta', \eps', \ell>0$ and $\zeta:=(1+\eps')/(\tau(1-\eps)-1)<1/2$ so that  the event 
\begin{equation*}
\CA_{\mathrm{good}}(u_n):=\{\max_{u,v\in[n]} e(u,v)\le \ell\}\cap\Big\{|B_{\delta\log n}(u
_n)|\le n^{\zeta}\Big\}\cap \big\{ \mathrm{Surp}_{\delta \log n}(u_n)\le\ell\big\} 
\end{equation*}
holds with probability $1-2n^{-1-\delta'}$.
On the event $\CA_{\mathrm{good}}(u_n)$, there are at most $\ell$ surplus edges in $B_{\delta\log n}(u_n)$, so we may apply Lemma \ref{lemma:suprplus-paths} to see that the set of non-backtracking infections paths in $B_{\delta \log n}(u_n)=\CT^{(\ell)}$ starting at $u_n$ of length $N=\delta\log n$, defined in \eqref{eq_want_count} satisfies  on the event$ \CA_{\mathrm{good}}$ that
\begin{equation*} |\CB_{\delta \log n}|\le (2\ell+1)^{\delta \log n} |\CT_{\delta\log n} | \le n^{\delta \log(2\ell+1)} |B_{\delta \log n}(u_n)| \le n^{\zeta+\delta \log(2\ell+1)}.
\end{equation*}
Now we apply Lemma \ref{lem_overall_fast}, with $\ell$ as the maximal number of multiple edges and $\bar v:=u_n$. The main result there, \eqref{eq:overall-fast} turns into, with $N=\delta \log n$ and $\lambda<1/(4\ell)$,
\begin{equation}\label{eq:un-event}
\begin{aligned}
\mathbb{P}&\left( \begin{array}{l}
(\underline{x}^{(u_n)}_t) \text{ dies before time $C\delta \log n$, and never reaches }\\
		\text{ any vertex at graph distance $\delta \log n$ from $u_n$}\end{array} \right) \\ &\qquad \qquad \qquad > 1 - 2|\CB_{\delta \log n}|\Big (\mathrm e \ell\cdot (4\ell\lambda)^{\delta \log n}  + \mathrm{e}^{-\delta \log n(C-1)^2/(2C)}\Big)\\
 &\qquad \qquad\qquad \ge 1- 2 n^{\zeta+ \delta \log (2\ell+1)}\Big( \e \ell n^{-\delta |\log(4\ell \lambda)| } +  n^{-\delta (C-1)^{2}/2C}\Big).
  \end{aligned}
  \end{equation}
 Distributing the brackets,  there are two error terms, the first one is 
  \[ 2e\ell \cdot n^{\zeta + \delta \log (2\ell +1) -\delta |\log (4\ell\lambda)|} \le  n^{-1-\delta'}\]
  whenever $4\ell\lambda$ is small enough so that the exponent of $n$ goes below $-1-\delta'$, in particular when
\begin{equation}\label{eq:lambda-crit}\lambda < \frac{1}{4\ell} \exp\Big( -\tfrac{1}{\delta} (1+\delta'+\zeta+\delta \log (2\ell+1))\Big).\end{equation}
  The second error term is
  \[ 2 n^{\zeta + \delta \log (2\ell +1) -\delta (C-1)^2/(2C)}\le n^{-1-\delta'} \]
  whenever $C$ is so large that the exponent of $n$ goes below $-1-\delta'$, in particular using that $(C-1)^2/(2C) > (C-1)/4$ the exponent is below $-1$ whenever 
  \[C >1+ 4 \tfrac{1}{\delta}(1+\delta'+\zeta + \delta \log (2\ell+1)).  \]
  This shows that for all $\lambda$ sufficiently small (satisfying \eqref{eq:lambda-crit}), the event in \eqref{eq:un-event} holds with probability at least $1- n^{-1-\delta'}$.   On this event, the process $\underline x_t^{(u_n)}$ never leaves the ball $B_{\delta \log n}(u_n)$, in particular the process never sees other parts of the graph. In other words, extinction of $\underline x_t^{(u_n)}$ on $B_{\delta \log n}(u_n)$ without reaching the boundary of $B_{\delta \log n}(u_n)$ implies extinction of $\underline x_t^{(u_n)}$ on $G_n$.  Hence,  the event $\{ T_{\mathrm{ext}}^{(u_n)}> C\delta \log n\}$ is covered by the complement of the event in \eqref{eq:un-event}, $\P(T_{\mathrm{ext}}^{(u_n)}> C\delta \log n)\le 2n^{-1-\delta'}$. Substituting this back to \eqref{eq:union-bound-brw} finishes the proof.
\end{proof}

%% file: proofs_GW-survival.tex
\section{Proofs of survival on Galton-Watson trees}\label{sec:survival_GW}
In this section we present proofs of survival regimes. We start with (only) global survival -- Theorem \ref{thm:max_GW}(b), then we prove Theorems \ref{thm:prod_GW} and \ref{thm:max_GW} (a) in Section \ref{sec:prod_GW_strong}.
\subsection{Max-penalty: global survival via infinite infection rays on heavy tailed GW trees}\label{sec:max_GW_weak}
To prove global survival for the max-penalty with $\mu\in(1/2,1)$ on GW trees with sufficiently fat-tailed offspring distributions, i.e., Theorem \ref{thm:max_GW}(b), we will show the existence of a (random) infinite \emph{ray} in the Galton-Watson tree on which the infection survives forever. 

\begin{definition}[Down-directed contact process]
Let $\CT$ be any (given) tree with root $\varnothing$. Consider the directed graph $\CT^{\downarrow} $ where each edge $\{u,v\}$ of $\CT$ is directed away from the root, i.e., from parent to child. Then we denote by $\CPf^{\downarrow}(\CT, \underline \xi_0)=(\underline \xi_t^{\downarrow})_{t\ge 1}$ the degree-penalized contact process in Definition \ref{def:CP} on the \emph{directed} graph $\CT^{\downarrow}$ with initial state $\underline \xi_0$.
\end{definition}
One can obtain the down-directed contact process $\CT^{\downarrow}(\CT, \underline \xi_0)$ from the graphical representation of the original $\CPf(\CT, \underline \xi_0)$ by deleting the Poisson point processes that represent infections from child to parent (i.e., upward in the tree), and leaving only those infection paths intact which only contain parent-to-child infection events. 
Hence, for every given tree $\CT$ and starting state $\underline \xi_0\in \{0,1\}^{V(\CT)}$ it holds that
\begin{equation}\label{eq:domin-downCP-CP}
\CPf^{\downarrow}(\CT, \underline \xi_0)\  {\buildrel d \over \le} \ \CPf(\CT, \underline \xi_0). 
\end{equation}
The next proposition shows that $\CPf^{\downarrow}$ survives globally with  positive probability on a Galton-Watson tree:
\begin{proposition}\label{prop:rays}
Let $\CT$ be a Galton-Watson tree with offspring distribution $D$ satisfying Definition \ref{def:power-heavy} for some $\alpha>0$ and $\Pv(D\ge 1)=1$. Suppose $\mu\in[1/2, 1)$ and $f(x,y)=\max(x,y)^\mu$, and moreover $\mu+\alpha<1$.
Then the down-directed contact process $\CPfd(\CT, \ind_{\varnothing})$  exhibits global survival with positive probability on $\CT$ for any $\lambda>0$, for almost all realizations $\CT$ of the Galton-Watson tree.
\end{proposition}
\begin{proof}[Proof of Theorem \ref{thm:max_GW}(b)]
The result follows from Proposition \ref{prop:rays} by using the stochastic domination in \eqref{eq:domin-downCP-CP}.
\end{proof}
\begin{proof}[Proof of Proposition \ref{prop:rays}]
In this proof we denote by $D_v=d_v-1$ the out-degree (number of children) of the vertex in $\CT^{\downarrow}$.
Let $\CA_{\mathrm{glob}}$ be the event that $\CPf^{\downarrow}(\CT, \ind_{\varnothing})$ survives globally.
Let $\CB_K=\{\exists t_0\ge 0, \exists v\in V(\CT), \deg(v)\ge K: \xi^{\downarrow}_{t_0}(v)=1\}$ be the event that $\CPfd$ ever reaches a vertex with degree at least $K$ for a large enough $K$ decided later. This event has strictly positive probability $p_K$ with lower bound depending only on $K$, since $p_K\ge \P(D_\varnothing \ge K)>0$. 
\begin{equation}\label{eq:E1K}
\P(\CA_{\mathrm{glob}})\ge\P(\CB_K)\P(\CA_{\mathrm{glob}} \mid \CB_K),
\end{equation}
so it is enough to show that $\P(\CA_{\mathrm{glob}}  \mid \CB_K)>0$ for some large enough $K$.
Fix some constants $1<s_1<s_2$ to be chosen later.

Consider a vertex $v$ with degree $D_v= L\ge K$ in the Galton-Watson tree and let $\CN(v, [L^{s_1},L^{s_2}])$ and $N(v, [L^{s_1},L^{s_2}])$ be the set and  number of children of $v$ in $\CT$ with degrees in $[L^{s_1},L^{s_2}]$, respectively. Since the children have iid degrees, $N(v,[L^{s_1},L^{s_2}]))$ is Binomially distributed with parameters  $L$  and $\P(D\in [L^{s_1},L^{s_2}])$. We bound its mean from below using \eqref{eq:heavier-power-law}. Given some $\varepsilon\in[0, \alpha(s_2-s_1)/(s_2+s_1))$, assuming $L>K_0(\varepsilon)$ so that \eqref{eq:heavier-power-law} holds,
\[ 
\begin{aligned}
\E[&N(v, [L^{s_1},L^{s_2}])\mid  D_v= L]= L\big(\P(D\ge L^{s_1})-\P(D\ge L^{s_2})\big)\\
&\ge L \Big(\frac{1}{L^{s_1(\alpha+\varepsilon)}}- \frac{1}{L^{s_2(\alpha-\varepsilon)}}\Big) = L^{1-\alpha s_1-\varepsilon s_1}\left(1- L^{-\alpha(s_2-s_1)+\varepsilon(s_2+s_1)}\right).
\end{aligned}\]
By the assumption that $s_2>s_1$ and $\varepsilon<\alpha(s_2-s_1)/(s_2+s_1)$, we obtain the existence of $K_1(\varepsilon, s_2, s_1, \alpha)$ such that the second factor on the rhs above is at least $1/2$ for all $L>K_1(\varepsilon, s_2, s_1, \alpha)\vee K_0(\varepsilon)$. Hence for all such $L$,
\begin{equation}\label{eq:ENv}
\E\big[N(v, [L^{s_1},L^{s_2}])\mid D_v= L\big) \ge L^{1-\alpha s_1-\varepsilon s_1}/2.
\end{equation}
 We now require that $s_1, \varepsilon$ is such that $1-\alpha s_1 -\varepsilon s_1>0$, then the mean tends to infinity with $L$. Using now Chernoff's bound on this Binomial random variable  we obtain that 
 \begin{equation}
 \begin{aligned}\label{eq:error-1}
 \P(\CA_{1}(v,L) | D_v=L)&:=\P\big( N(v, [L^{s_1},L^{s_2}]) > L^{1-\alpha s_1-\varepsilon s_1}/4 \mid D_v= L \big)\\ &\ge  1-\exp\big( -  L^{1-\alpha s_1-\varepsilon s_1}/48\big)=:1-\mathrm{err}_1(L).
\end{aligned} 
 \end{equation}
 Assume now that $\CPf^{\downarrow}$ has reached vertex $v$ at some time, and that $\CA_{1}(v,L)$ holds for $v$.
 Let now $\CA_2(v,L)$ be the event that $v$ infects at least one of the first $L^{1-\alpha s_1-\varepsilon s_1}/4$ many children within the set $\CN(v, [L^{s_1}, L^{s_2}])$ before healing.
We bound the complement of this event using that the degree of such a child is in the interval $[L^{s_1}, L^{s_2}]$, which gives that the infection rate from $v$ to any child $u\in \CN(v, [L^{s_1}, L^{s_2}])$ is at least $r(v,u)=\la \max(L, D_u)^{-\mu}\ge \la L^{-\mu s_2}$ (since we assumed that $s_2> s_1>1$). We obtain that
\begin{equation}\label{eq:error-2}
\begin{aligned}
    \P&(\neg \CA_2(v,L) \mid v \mbox{ ever infected}, D_v=L, \CA_1(v,L) )\\
    &= \frac{1}{1+\sum_{u_i \in \CN(v, [L^{s_1}, L^{s_2}]), i \le L^{1-\alpha s_1-\varepsilon s_1}/4 }r(v,u_i)} \le \frac{1}{1+\lambda L^{1-\alpha s_1-\mu s_2-\varepsilon s_1}/4}
    \\
    &\le 8\lambda^{-1}L^{-(1-\alpha s_1-\mu s_2-\varepsilon s_1)}=:\mathrm{err}_2(L),
    \end{aligned}
\end{equation}
where we used that $L$ is sufficiently large, and the assumption that $1-\alpha s_1-\mu s_2-\varepsilon s_1>0$ to obtain the last line. This assumption can be satisfied with $s_2>s_1>1$ and $\varepsilon>0$ small enough whenever $1-\alpha-\mu>0$, which is true since we assumed $\alpha+\mu<1$. Also note that it cannot be satisfied when $\alpha+\mu\ge 1$.

We use the error bound in \eqref{eq:error-2} repeatedly. 
Let now $v_0$ be the first vertex reached by $\CPf^{\downarrow}$ with degree at least $K$ in the event $\CB_K$ in \eqref{eq:E1K}, and let $D_{v_0}$ denote its random degree. We now define a random infection ray $(v_0, v_1, \dots, v_m, v_{m+1} \dots)$ recursively. Suppose we already defined $(v_0, \dots, v_m)$ for some $m\ge 0$, and their degrees $(D_{v_0}, \dots, D_{v_m})$. 
We now check whether the event $\CA_1(v_m, D_{v_m})\cap \CA_{2}(v_m, D_{v_m})$ holds, and if so, then we choose any vertex $v_{m+1}\in \CN(v_m, [D_{v_m}^{s_1}, D_{v_m}^{s_2}])$ that is infected by $v_m$ before $v_m$ heals.
We now obtain the existence of an infinite ray by taking the limit of the nested sequence of events: 
\[ 
\begin{aligned}
\P\big( (v_0, \dots, v_m, \dots) \mbox{ exists}\big)&=\lim_{m_0\to \infty}  \P\Big(\cap_{m\le m_0} \{v_{m+1} \mbox{ exists}\}\Big)\\
&=
\lim_{m_0\to \infty} \prod_{m=0}^{m_0} \P\Big( v_{m+1} \mbox{ exists} \mid (v_0, \dots, v_{m}) \mbox{ exists}\Big),
\end{aligned}\] We denote by $\CF_{m}$ the sigma-algebra generated by
\[\cup_{i\le m-1}\{\CA_1(v_i, D_{v_i}), \CA_2(v_i, D_{v_i}), v_i, D_{v_i}\} \cup \{v_{m},D_{v_m}\}.\]
I.e., we reveal the degree and existence of $v_m$, but not whether $\CA_1(v_m, D_{v_m}) \cap \CA_2(v_m, D_{v_m})$ holds since those events already give $v_{m+1}$. Using this sigma-algebra, we can use the Markov property of $\CPf^{\downarrow}$, lower bound the probability of existence of $v_{0}$ by $\P(\CB_K)$, and that of $v_{m+1}$ by the conditional probability of $\CA_1(v_m, D_{v_m})\cap \CA_{2}(v_m, D_{v_m})$ to obtain
\begin{align}
\P\big( (v_0, &\dots, v_m, \dots)\mbox{ exists}\big)\nonumber\\
&\ge \lim_{m_0\to \infty}\P(\CB_K)\E\Bigg[\prod_{m=0}^{m_0} \P\Big(\CA_1(v_{m}, D_{v_m}) \cap \CA_2(v_m, D_{v_m})\mid \CF_{m}\Big)\bigg] \nonumber\\
& \ge\P(\CB_K)\lim_{m_0\to \infty}\Bigg[\prod_{m=1}^{m_0} \P\Big(\CA_1(v_m, D_{v_m}) \cap \CA_2(v_m, D_{v_m})\mid v_{m} \mbox{ ever infected}, D_{v_{m}}\Big)\Bigg]. \label{eq:infinite-ray}
\end{align}
Observe that now the calculations in \eqref{eq:error-1} and \eqref{eq:error-2} apply, and the $m$th factor is, conditionally on $D_{v_m}$, at least 
$1-\mathrm{err}_1(D_{v_m})-\mathrm{err}_2(D_{v_m})$.
We inductively show that the $m$th factor in the product above is at least 
\begin{equation}\label{eq:vm-degree}
 1-\mathrm{err}_1(K^{s_1^m})-\mathrm{err}_2(K^{s_1^m}),
\end{equation}
by showing that $D_{v_m}\ge K^{s_1^m}$ whenever $v_m$ exists. Monotonicity of $\mathrm{err}_1(L)+\mathrm{err}_2(L)$ in $L$ then immediately yields the lower bound \eqref{eq:vm-degree}, as follows. Since we assumed $D_{v_0}\ge K=K^{s_1^0}$, the induction starts. Assume now that $v_{m-1}\ge K^{s_1^{m-1}}$. Then per definition, (see \eqref{eq:error-2}),  $D_{v_m}\in [D_{v_{m-1}}^{s_1}, D_{v_{m-1}}^{s_2}]$. Using now the induction hypothesis immediately gives  \eqref{eq:vm-degree}.
Hence, we return to \eqref{eq:infinite-ray}, for a.e. realization in the conditional expectation the lower bound in  \eqref{eq:vm-degree} holds, hence, 
\begin{equation}\label{eq:infinite-ray-2}
\begin{aligned}
\P\big( (v_0, \dots, v_m, \dots) \mbox{ exists}\big)&\ge \P(\CB_K)\prod_{i=1}^{\infty} ( 1-\mathrm{err}_1(K^{s_1^m})-\mathrm{err}_2(K^{s_1^m}))\\
&\ge \P(\CB_K)\Big( 1- \sum_{m=0}^\infty \mathrm{err}_1(K^{s_1^m})+\mathrm{err}_2(K^{s_1^m}) \Big).
\end{aligned}\end{equation}
Using the values of $\mathrm{err}_1(K^{s_1^m})+\mathrm{err}_2(K^{s_1^m})$ from \eqref{eq:error-1}, \eqref{eq:error-2}, given that 
\begin{equation}\label{eq:s1-s2}
    1<s_1<s_2, \qquad 1-\alpha s_1-\mu s_2-\varepsilon s_1>0, \qquad 1-\alpha s_1 - \varepsilon s_1>0,
\end{equation}
the sum on the right hand side is summable in $m$, and both terms decrease faster then geometrically in $m$, hence they are dominated by a constant times their first term:
\begin{align}
\sum_{m=0}^{\infty}\exp(- K^{s_1^m(1-\alpha s_1 - \varepsilon s_1)}/48)&+\sum_{m=0}^{\infty}8\lambda^{-1} K^{-s_1^m(1-\alpha s_1-\mu s_2-\varepsilon s_1)}\nonumber\\&\le C \exp(-K^{1-\alpha s_1 - \varepsilon s_1}/48) + C \la^{-1} K^{-(1-\alpha s_1-\mu s_2-\varepsilon s_1)}. \label{eq:summed-errors}
\end{align}
One can check that the system of inequalities in \eqref{eq:s1-s2} is solvable whenever $1-\alpha-\mu>0$. Namely, choose first $1<s_1< s_2$ close enough to $1$ so that 
$1-\alpha s_1-\mu s_2>0$ holds. Choose then $\varepsilon>0$  small enough so that \eqref{eq:ENv} and \eqref{eq:s1-s2} hold as well, and finally one can set $K$ sufficiently large so that all inequalities above are valid. 
In particular, given now any $\la>0$ (i.e., small), one can choose $K$ sufficiently large so that the sum in \eqref{eq:summed-errors} is at most $1/2$, and then we obtain in \eqref{eq:infinite-ray-2} that an infinite infection ray exists with probability at least $\P(\CA_K)/2$, which is strictly positive. Hence, global survival occurs with strictly positive probability, whenever $\alpha+\mu<1$, finishing the proof.
\end{proof}

\subsection{Product penalty: local survival using a row of star-graphs when \texorpdfstring{$\mu<1/2$}{mu<1/2}}\label{sec:prod_GW_strong}
We will prove local survival of $\CPf$ (for both product and maximum penalty) when $\mu\in [0, 1/2)$ on the Galton-Watson tree, with at last stretched exponential offspring distributions, i.e., Theorem \ref{thm:prod_GW} in multiple steps. 

The idea is the following: As a direct consequence of known results, in Claim \ref{claim:star} we prove that when $\mu<1/2$, the infection survives on a star-graph of degree $K$, which consist of a degree-$K$ vertex and its degree-$1$ neighbors, for a time $T_K=\exp(\Theta(\lambda^2 K^{1-2\mu}))$ with probability very close to $1$. Moreover, throughout this time the star will be \emph{infested}, by which we mean that a sufficiently high fraction of its vertices are infected. 

We then show that a star-graph that is infested for time $T_K$, sends the infection through a path of length $\ell$ to another such star-graph with probability close to $1$ if and only if $\ell= o(\log T_K)$. Hence we need that $\ell=o(K^{1-2\mu})$ so that the infection successfully infests another star-graph. 

Let $H_{K, \ell(K)}$ be a graph that consists of a one-ended infinite row of star-graphs of degree $K$, $(v_1,v_2,  \dots)$, with paths of length $\ell(K)=o(K^{1-2\mu})$ between two consecutive stars. We show that the degree-penalized contact process survives forever on $H_{K, \ell(K)}$ with positive probability, as long as $K$ is sufficiently large compared to $\lambda$. We do this by mapping the process on $H_{K, \ell(K)}$ to a discrete time analog of the contact process on $\N_+=\{1,2,\dots\}$ corresponding to the infinite row of star-graphs $(v_1, v_2, \dots)$.

We then show that $H_{K, \ell(K)}$ can be embedded almost surely in a Galton Watson tree $\CT_D$ in a way that in the embedding, every vertex in $H_{K, \ell(K)}$ has degree at most $M$ times its degree in $H_{K, \ell(K)}$. This only changes $\lambda$ in the arguments above by a constant factor, i.e, to $\tilde\lambda:=\lambda/M^{2\mu}$, so if $\CPf$ survives on $H_{K, \ell(K)}$ whenever $K$ is sufficiently large, then the same is true for $\mathrm{CP}_{f, \tilde\lambda}$ by increasing $K$ if necessary. For the embedding to be possible, the tail of $D$ must be heavier than stretched exponential with stretch-exponent $1-2\mu$, in the sense of Definition \ref{def:stretched-heavy}, which is the mildest condition possible for this proof to work.

\subsubsection{Embedding stars in the Galton-Watson tree}\label{sec:embedded_stars}
We now make the former outline precise, starting with the definition of the infinite row of star-graphs and the embedding that does not increase degrees too much.  
\begin{definition}[Infinite path of stars and $M$-embedding]\label{def:M-embedding}
Given two integers $K, \ell\ge 1$, let $H=H_{K,\ell}$ be an infinite graph defined as follows: we start by taking an infinite path $(v_1, \CP_1, v_2, \CP_2, \dots, v_i, \CP_i, v_{i+1}, \dots)$, where for all $i\ge 1$ the  paths $\CP_i=(u_1^{\sss{(i)}}, \dots, u_\ell^{\sss{(i)}})$ have length $\ell$, and then to each $v_i, i\in \N$ 
we attach $K$ additional neighbors $w_1^{\sss{(i)}}, \dots, w_{K}^{\sss{(i)}}$, each with $\deg_H(w_j^{\sss{(i)}})=1$, which we call leaves.
We call $K$ the star-degree of $H_{K, \ell}$ and $\ell$ the connecting-path length, which might depend on $K$. See Figure \ref{fig:Hkl}.

We say that $H=H_{K,\ell}$ is (degree-factor) $M$-embedded in a graph $G$ if $G$ contains $H_{K,\ell}$ as subgraph, and for all vertices $v\in H_{K, \ell}\subseteq G$  it holds that 
\begin{equation}\label{eq:factor-M} \frac{\deg_G(v)}{\deg_H(v)}\le M. \end{equation}
\end{definition}

\begin{figure}[ht]
\includegraphics[width=\textwidth]{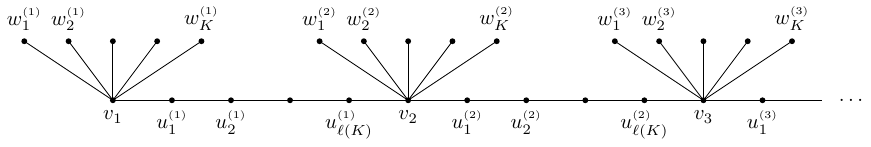}
\caption{The graph $H_{K,\ell(K)}$.}
\label{fig:Hkl}
\end{figure}
 
The next lemma shows that for large $K$, $H_{K,\ell}$ can be $M$-embedded almost surely into a Galton-Watson tree $\CT$ with offspring distribution $D$. The proof reveals that the tail of $D$ determines the minimal $\ell=\ell(K)$ that is possible for the embedding to hold almost surely.
\begin{lemma}\label{lem:embedded_stars}
Let $\CT$ be a Galton-Watson tree with degree distribution $D$ so that the tail of $D$ is heavier  than stretched exponential with stretch-exponent $1-2\mu$, in the sense of Definition \ref{def:stretched-heavy}, along the infinite sequence $(z_i)_{i\ge 1}$, and prefactor $g(z)\to 0$ as $z\to 0$.
 Then there exists a constant $M\ge 1$, such that $H_{K, \ell(K)}$ can be $M$-embedded in $\CT$ for all sufficiently large $K$ such that $2K\in \{z_i, i\ge 1\}$, for almost all realizations of $\CT$, whenever
 \begin{equation}\label{eq:ellK}
 \ell(K)\ge 2^{1-2\mu} \sqrt{g(2K)} K^{1-2\mu} = o(K^{1-2\mu}).
 \end{equation}
\end{lemma}
\begin{proof}
First, fix some small constant $\eps>0$ decided later. Let $\E[D]:=q>1$, and define $D_{M}:=D\ind_{D< M}$, i.e., the distribution where $\P(D_{M}=0)=\P(D=0)+\P(D\ge M)$ and $\P(D_{M}=k)=\P(D=k)$ for all $k\in[1,M)$.
Given $\varepsilon>0$, we choose $M>2$ such that both of the following inequalities hold:
\begin{equation}\label{eq:qM} 
\begin{aligned}
q_M:=\E[D_M]&=\E[D\ind\{D< M\}]\ge \E[D]-\eps>1+2\varepsilon,\\
\P(D < M) &\ge 1-\varepsilon.
\end{aligned}
\end{equation}

It is clear that $D_{M}$ can be coupled to $D$ such that $\P(D_{M} \le D)=1$, and this embedding can be done for each vertex of the original Galton Watson tree $\CT$, obtaining a sub-forest $\CF_{M}$ of $\CT$. The embedding can be done by first sampling $D_v \sim D$ many children for each vertex $v$, and then accepting the number of offspring as it is when $D_v$ is between $0$ and $M-1$, but   setting the degree of $v$ in $\CF_M$ to be $0$ when $D_v\ge M$.  We will denote the distribution of a single tree in $\CF_M$ by $\CT_M$, which is a branching process with offspring distribution $D_M$. 

 Define the event, for $2K\in\{z_i\}_{i\ge 1}$,
\[ \CA_1:=\{ \exists v\in \CT: D_v = 2K\}.\]
Since we assumed $\P(D=0)=0$, $\CT$ survives almost surely and so $\P(\CA_1)=1$.
Take then the vertex $v\in \CT$ that is closest to the root $\varnothing$ and has $D_v =2K$, and set it to $v_1$ in $H_{K, \ell}$ of the embedding. Clearly $v_1$ then satisfies \eqref{eq:factor-M} since its degree in $\CT$ is $2K\le MK$ by our assumption that $M\ge 2$.

 Similarly as in the proof of Proposition \ref{prop:rays} below \eqref{eq:E1K}, let $\CN(v, [a,b]), N(v, [a,b])$ denote the set and number of children of a vertex $v\in \CT$ with offspring in the interval $[a,b]$. 
Consider now the event $\CA_{\text{child}}(v_1):=\{ N(v_1, [0,M) ) \ge K+1\}$. Since $D_{v_1}= 2K$ per assumption, and the children of $v_1$ have iid degrees, using \eqref{eq:qM}, each of these children has offspring less than $M$ with probability at least $1-\varepsilon$. Hence, using the concentration of Binomial random variables (e.g. a Chernoff's bound), whenever $\varepsilon<1/8$ (which we safely assume), for all $K$ sufficiently large,
\begin{equation}\label{eq:M-children}
\begin{aligned}
\P\big(\CA_{\text{child}}(v_1)\big) &=\P\big(N(v_1, [0,M) ) \ge K+1\big)\\
&\ge \P( \ \mathrm{Bin}(2K, 1-\eps) > K) \ge 1-\e^{-K/12}.
\end{aligned}
\end{equation}

On the event $\CA_{\text{child}}(v_1)$, we label by $w_1,w_2,\ldots,w_{K+1}$ the first $K+1$  children in $\CN(v_1, [0,M))$. 
Including the edge towards $v_1$, the total degree of any of these vertices in $\CT$ is at most $M$, satisfying thus the degree factor $M$ in \eqref{eq:factor-M}. So, $v_1$ and any $K$ out of the children $w_1, \dots, w_{K+1}$ may serve as the embedding of $w_1^{\sss{(1)}}, \dots, w_{K}^{\sss{(1)}}$ of $H_{K, \ell}$, and any one of these children may take the role of $u_1^{\sss{(1)}}$ of the path $\mathcal P_1$ in $H_{K, \ell}$.

 From each of these vertices $w_i$ we start the (embedded) branching process $\CT_{M}(w_i)\subseteq \CT(w_i)$ with offspring distribution $D_{M}$. Let the number of descendants of $w_i$ in $\CT_{M}(w_i)$ in generation $\ell$  (that is, of distance $\ell$ from $w_i$) be $Z^{(i)}_{\ell}$ for each $\ell\ge 1$. It is well-known that $W^{(i)}_{\ell}:=Z^{(i)}_{\ell}/q_M^{\ell}$ is a martingale for each $i$ \cite{athreya2004branching}, and that $\lim_{\ell\to\infty}W^{(i)}_{\ell}=W^{(i)}_{\infty}$ exists a.s. Since $\E[D_M]=q_M>1+2\varepsilon$, this branching process is supercritical, and because $D_{M}$ is bounded by $M$, the Kesten-Stigum Theorem gives that $\eta:=\P(W^{(i)}_{\infty}\ne 0)>0$ is the probability that the corresponding branching process $\CT_M$ survives indefinitely. It follows then that, for any $i$,
\[
\lim_{\ell\to\infty}\P(Z^{(i)}_{\ell}\ge (q_M-\eps)^{\ell})=
\lim_{\ell\to\infty}\P\left(\frac{Z^{(i)}_{\ell}}{q_M^{\ell}}\ge \left(\frac{q_M-\eps}{q_M}\right)^{\ell}\right)=\P(W^{(i)}_{\infty}> 0)=\eta.
\]
By \eqref{eq:qM}, $q_M-\varepsilon>1+\varepsilon$ and consequently, there exists a (deterministic) $\ell_0$ only depending on $D_{M}$ (but not on $K$) such that for all $\ell>\ell_0$ we have
\begin{equation}\label{eq:E1}
\P(\CB_i):=\P(Z^{(i)}_{\ell}\ge (q_M-\eps)^{\ell})\ge\eta/2.
\end{equation}
Denote the set of individuals in the  $\ell$-th generation of $w_i$ by $\mathcal{G}^{(i)}_{\ell}$ for each $i=1,2,\ldots,K$, and let $\mathcal{G}_{\ell}=\cup_{i=1}^K\mathcal{G}^{(i)}_{\ell}$. 
Since $(w_i)_{i\le K}$ are siblings, $\CG_\ell$ is embedded in $\CT$ also in the same (possibly other than $\ell$) generation.
We now return to the original branching process $\CT$ for a single generation.
For each $v\in\mathcal{G}_{\ell}$ consider i.i.d. copies $D_v$ of $D$ (that is, without the truncation at $M$ used so far), and define the events for $i\le K$:
\begin{align}
\widetilde\CB_i&:=\{\exists u_{\sss{(i)}}\in\mathcal{G}^{(i)}_{\ell}:\ D_{u_{\sss{(i)}}}= 2K\}\label{eq:Fi}
\end{align}
for each $i=1,\ldots,K+1$. 
By \eqref{eq:E1} we have $\P(\CB_i)\ge \eta/2$.
Furthermore, since on the event $\CB_i$
\begin{align}
\P(\neg \widetilde\CB_i\mid \CB_i)&\le \big(1-\P( D= 2K)\big)^{(q_M-\eps)^{\ell}}\nonumber\\
&\le \exp\Big(-\P(D=2K)(q_M-\eps)^{\ell}\Big).\nonumber
\end{align}
Since we have assumed $2K\!\in\!\{z_i\}_{i\ge 1}$ in Definition \ref{def:stretched-heavy}, we can use the bound \[\P(D\!=\!2K)\ge\exp(-g(2K) (2K)^{1-2\mu})\]
for the function $g(2K)\to 0$ as $K\to 0$ in Definition \ref{def:stretched-heavy}. Hence, $g(2K)=o(\sqrt{g(2K)})$ but at the same time $\sqrt{g(2K)}\to 0$ as $K\to \infty$. We then also use that $q_M-\eps>1+\eps$ by assumption, and so by choosing $\ell=\ell(K)\ge\sqrt{g(2K)} (2K)^{1-2\mu}$, one can compute that $(2K)^{1-2\mu}(\sqrt{g(2K)}\log(q_M-\eps)-g(2K))\to \infty$ and so for all sufficiently large $K$ it holds that 
\begin{equation}\label{eq:FEi}
\begin{aligned}
\P(\neg \widetilde\CB_i\mid \CB_i)&\le\exp\Big(-\e^{-g(2K) (2K)^{1-2\mu}}(q_M-\eps)^{\ell(K)}\Big)\\
&\le \exp\big(-\e^{ (2K)^{1-2\mu} ( \sqrt{g(2K)}\log(q_M-\eps) - g(2K))}\big) \le 1/2.
\end{aligned}
\end{equation}
Combining \eqref{eq:E1} and \eqref{eq:FEi} yields
\begin{align*}
\P(\widetilde\CB_i)\ge\P(\CB_i)\cdot\P(\widetilde\CB_i\mid \CB_i)
\ge
(\eta/2)\cdot (1/2) \ge \eta/4. 
\end{align*}
Now we define the event that at least two events $\widetilde\CB_i, \widetilde\CB_j$ happen for $v_1$:
\begin{align}
\widetilde\CA(v_1)&:=\{\exists i, j:\  i\neq j : \widetilde\CB_i\cap \widetilde\CB_j \mbox{ holds}\}.\label{eq:defF}
\end{align}

Now consider the number of indices $i\le K+1$ for which $\widetilde\CB_i$ holds. 
By \eqref{eq:Fi}, on the event $\CA_{\text{child}}(v_1)$ in \eqref{eq:M-children}, this number stochastically dominates a binomial random variable with parameters $K+1$ and $\eta/4$. Hence, by the definition of $\widetilde\CA(v_1)$ in \eqref{eq:defF}, it holds for some constant $c(\eta)>0$ that
\begin{align*}
\P(\widetilde\CA(v_1)\mid \CA_{\text{child}}(v_1))&\ge \P(\mathrm{Bin}(K+1, \eta/4)\ge 2)\\
&=1-(1-\eta/4)^{K+1}-K(\eta/4)(1-\eta/4)^{K} \ge 1-\e^{-c(\eta) K}.
\end{align*}
Combining this with \eqref{eq:M-children}, we obtain that for all sufficiently large $K$,
\begin{equation}\label{eq:2-children}
\P(\CA_{\text{child}}(v_1)\cap \widetilde\CA(v_1)) \ge 1- \e^{-c(\eta) K} - \e^{-K/12}\ge 1-\varepsilon.
\end{equation}
On the event $\tilde\CA(v_1)\cap \CA_{\text{child}}(v_1)$, there are two vertices $v_{2,1}, v_{2,2}$ such that their most recent common ancestor is the starting vertex $v_1$, and $\deg(v_{2,1}),\deg(v_{2,2})=2K $, and $d_G(v_{1},v_{2,1})=d_G(v_1,v_{2,2})=\ell(K) \ge \sqrt{g(2K)}(2K)^{1-2\mu}$ with $\ell(K)=o(K^{1-2\mu})$,
and the paths $\mathcal{P}_{1,1}, \mathcal P_{1,2}$ joining  $v$ with $v_{2,1}$ and $v_{2,2}$ respectively are edge-disjoint with all internal vertices having degree at most $M$. 
Observe that $(v_1, \mathcal P_{1,1}, v_{2,1})$ and $(v_1,  \mathcal P_{1,2}, v_{2,2})$ both serve as a factor $M$-embedding of the vertices in $(v_1, \mathcal P_1, v_2)$ in $H_{K, \ell(K) }$, hence we may choose any of them for the embedding. Further, the vertices $v_{2,1}$ and $v_{2,2}$ have degree $2K$ in $\CT$, hence, using the argument between \eqref{eq:M-children} and \eqref{eq:defF}, one can repetitively apply the procedure of checking whether the events $\CA_{\text{child}}(\cdot) \cap \tilde\CA(\cdot)$ hold for these vertices, and the vertices then found by either $\CA_{\text{child}}(v_{2,1}) \cap \tilde\CA(v_{2,1})$ or $\CA_{\text{child}}(v_{2,2}) \cap \tilde\CA(v_{2,2})$ may all serve  as the embedding of the path $\mathcal P_2$ and $v_3$, and so on. 

 We thus consider an auxiliary ``renormalised'' branching process. We say that $v_1$ has $2$ children (in this case $v_{2,1}, v_{2,2}$) with probability (at least) $1-\varepsilon$ in \eqref{eq:2-children} and $0$ otherwise. Observe that the path leading to any vertex in generation $j$ of this branching process serves as an $M$-embedding of $(v_1, \CP_1, v_2, \dots, \CP_{j-1}, v_j)$.
 This renormalised branching process is supercritical. Hence, it survives with positive probability, giving that the $M$-embedding of the infinite graph $H_{K, \ell(K)}$ exists in $\CT$, starting from $v_1$, with positive probability. 
Kolmogorov's 0-1 law finishes the proof that $\CT$ then has a proper $M$-embedding of $H_{K, \ell(K)}$ somewhere in $\CT$ with probability $1$.
\end{proof}
We now define star-graphs (subgraphs of $H_{K, \ell}$) and the notion of infested stars.
\begin{definition}
A \emph{star-graph} $S$ of degree $K$ is a graph which consists of one vertex $v$ of degree $\deg_S(v)=K$ (its center) and its $K$ neighbors $(w_i)_{i\le K}$, each of degree $\deg_S(w_i)=1$ that we call leaves. 
Consider the classical contact process with infection rate $r$ on $S$.
We will call such a star $r$-\emph{infested} at some time $t$ by the contact process if at least $r K/(16e^2)$ of its leaves are infected.
\end{definition}

The next claim adapts \cite[Lemma 3.1]{MVY2013} to the degree-penalized contact process on $S$. The claim shows that  starting with only the center infected, a star-graph is $\la K^{-\mu}$-\emph{infested} for a time interval of length $T_K\ge \exp(c r^2 K)= \exp(c\lambda^2 K^{1-2\mu})$ with high probability, and during this time-interval the center vertex $v$ is infected more than half of the time.
Writing $r$ for the rate of infection of the classical contact process on a star-graph, \cite[Lemma 3.1]{MVY2013} holds under the condition that $r^2 K$ is uniformly bounded away from $0$. Since in the degree-penalized CP, the rate across the edges of the star-graph is $r=\lambda K^{-\mu}$, we shall require that $\lambda^2 K^{1-2\mu}$ is uniformly bounded away from $0$.
\begin{claim}[Lemma 3.1 of \cite{MVY2013} adapted]\label{claim:star}
Assume $\mu<1/2$, $\lambda<1$. Consider a star-graph $S$ of degree $K$ with center $v$. 
Let $\xi_t$ denote the contact process $\mathrm{CP}$ on $S$ where   $r(v,u)=r(u,v)=\la /K^\mu$.
 Then there exists a constant $c_1>0$ such that
 \begin{equation}\label{eq:star-i}\P\left(|\underline\xi_1|\ge\lambda K^{1-\mu}/(4e)\mid \xi_0(v)=1\right)\ge (1-e^{-c_1\lambda K^{1-\mu}})/e.
 \end{equation}
Further, if $\lambda^2 K^{1-2\mu}>32e^2$, then \begin{equation}\label{eq:star-ii}
\P\Big(\underline\xi_{\exp\{c_1\lambda^2 K^{1-2\mu}\}}\ne\emptyset\ \Big|\  |\underline\xi_0|\ge\lambda  K^{1-\mu}/(8e)\Big)\ge 1-e^{-c_1\lambda^2 K^{1-2\mu}}=:1-\mathrm{err}_{\lambda,K}.
\end{equation}
Moreover, let $T_K:=\exp(c_1\lambda^2 K^{1-2\mu})$. Then
\begin{align}
\P\Bigg(S&\text{ is $\lambda K^{-\mu}$-infested for all } t\in\left[0,T_K\right] \mbox{and} \int_{0}^{T_K} \xi_t(v) \ge T_K/2\ \Big|\  |\xi_0|\ge\lambda  K^{1-\mu}/(8e)\Bigg)\nonumber\\
&\qquad\ge 1-\mathrm{err}_{\lambda,K}.\label{eq:star-iii}
\end{align}
\end{claim}

The proof of Claim \ref{claim:star} is very similar to \cite[Lemma 3.1]{MVY2013}, therefore we include it in the Appendix.

\subsubsection{Contact process on an infinite line of stars}\label{sec:spread_between_stars}
We continue by studying the spread of the infection on $H_{K, \ell(K)}$. In particular, we prove that the probability that an infested star passes on the infestation to a neighboring star in $H_{K, \ell(K)}$ can be made arbitrarily close to $1$ with the right choice of the parameters.
\begin{claim}\label{claim:spread_between_stars} 
For each fixed small $\lambda>0$ and $\delta>0$ there is a $K_{\lambda, \delta}$ such  that the following holds for all $K\ge K_{\lambda, \delta}$. Consider the degree-penalized contact process $\mathrm{CP}_{f,\lambda}$ on $H_{K, \ell(K)}$ with $f=(xy)^{\mu}$ for some $\mu<1/2$.
Consider two consecutive stars $v_i, v_{i+1}$ in $H_{K, \ell(K)}$ in Definition \ref{def:M-embedding}, with $\ell(K)=o(K^{1-2\mu})$, and let $T_K:=\exp(c_1\lambda^2 K^{1-2\mu})$ from Claim \ref{claim:star}.
Suppose that $v_i$ is $\lambda K^{-\mu}$-infested at some time $t_0$. Then at time $t_0+T_K$, $v_{i+1}$ is $\lambda K^{-\mu}$-infested with probability at least $1-\delta$. 
\end{claim}
\begin{proof}
In Definition \ref{def:M-embedding}, we denoted the vertices on the path $\CP_i$ connecting $v_i$ to $v_{i+1}$ by $u_1^{\sss{(i)}}, u_2^{\sss{(i)}}, \dots u_{\ell}^{\sss{(i)}}$. In this proof we will omit the superscript.  Also, $c>0$ is a constant whose value can be made specific but may vary even within line.  Further, $|\cdot|$ means the Lebesgue measure of a set in $\R$.
   We define the event and bound its probability from below using \eqref{eq:star-iii}:
  \begin{equation}\label{eq:localtime}
\P(\CA_1(v_i)) :=\{ v_i \mbox{ is }\lambda K^{-\mu}\mbox{-infested for all } t\in [t_0, t_0+T_K]\} \ge 1-\e^{c_1 \lambda^2 K^{1-2\mu}}\ge 1-\delta/8,
\end{equation}
whenever $K\ge \log(8/\delta) \lambda^{-2/(1-2\mu)}/c_1=:K_0(\delta)$.
For some $m_K$ and $t_K$ to be determined later, partition the time interval $[t_0, t_0+T_K]$ into  $m_K$ disjoint intervals of length $t_K$, denoted by $J_1, \dots J_{m_K}$, with $m_K= \lfloor T_K/t_K\rfloor$.
Since $v_i$ is infested at time $t\in [t_0, t_0+T_K]$, the proof of Claim \eqref{claim:star} reveals that $\xi_t(v_i)$ is stochastically dominating a two-state Markov chain on $\{0,1\}$ with transition rate $q_{0,1}=\lambda^2K^{1-2\mu}/(16\e^2)$ and $q_{1,0}=1$. 
For each interval $J_j=[J_j^-, J_j^+)$, let $\tau_j$ denote the first time in $J_j$ when $\xi_t(v_i)=1$. Define then the event that 
\begin{equation}\label{eq:JJ}
\CA_2(J_j):=\{\tau_j \le 1\}.
\end{equation}
Then $\P(\CA_2(J_j))\ge 1/2$ for all $J_j$, and the Markov property of the process ensures that $\CA_2(J_j)_{j\le m_k}$ are independent.
 Then Chernoff's bound yields that
\begin{equation}\label{eq:e3}
\begin{aligned}
    \P(\CA_3(v_i))&:=\P\big( \#\{j\le m_K: \CA_2(J_j) \mbox{ holds}\} \ge m_K/4 \big) \\
    &\ge \P\big(\mathrm{Bin}(m_K, 1/2)\ge m_K /4 \big) \ge 1- \e^{- c m_k },
    \end{aligned}
\end{equation}
for $c=1/48$.
Consider now $\{j: \CA_2(J_j)\}$, and for each such $j$, 
   call such $J_j$ \emph{successful} if there is some time $t\in J_j$ when at least $\lambda K^{1-\mu}/(4e)$ many leaves in the star-graph of $v_{i+1}$ are infected.
 We now lower bound the probability of the event that $J_j$ is successful conditioned on $\CA_2(J_j)$, as follows.
 Define a sequence of time-moments  $s_h:=\tau_j+h 4^\mu$ for $h\in\{1, \dots, \ell+1\}$,
and for $h=1,\dots, \ell$ we recursively check whether $u_h$ is infected at time $s_h$, given that $u_{h-1}$ is infected at $s_{h-1}$ (setting $u_0:=v_i$), and that whether $v_{i+1}=:u_{\ell+1}$ is infected at time $s_{\ell+1}$ given that $u_\ell$ is infected at time $s_\ell$. We also set $s_{\ell+2}:=s_{\ell+1}+1$ and check whether at least $\lambda K^{1-\mu}/(4e)$ many leaves in the star of $v_{i+1}$ are infected at time $s_{\ell+2}$, given that $v_{i+1}$ is infected at time $s_{\ell+1}$.  We shall thus bound, for some constant $c$, the time-interval lengths and their number as 
\begin{equation}\label{eq:tk}
t_K:=4^\mu (\ell+2)+2 \le (4^\mu\vee 2) (\ell+3), \qquad m_K=\lfloor T_K/t_K\rfloor \ge c \,T_K /\ell.
\end{equation} 
Returning to an interval $J_j$ being successful, denote the infection status of the set of leaves in the star around $v_{i+1}$ by $\underline \xi^{(i+1)}_t$. Then, using the strong Markov property, we can lower bound
\begin{align}
\P&(J_j \mbox{ successful} \mid \CA_2(J_j
))\ge \P\big(  |\underline \xi^{(i+1)}_{s_{\ell+2}}|\ge \lambda K^{1-\mu}/(4e) \mid \xi_{\tau_j}(v_i)=1 \big)\label{eq:successful-def} \\
&\ge
\P(\xi_{s_1}(u_1)=1\mid \xi_{\tau_j}(v_i)=1 )\prod_{h=2}^{\ell+1}\P\Big( \xi_{s_h}(u_h) = 1\mid \xi_{s_{h-1}}(u_{h-1})=1\Big) \label{eq:product-factors} \\
&\qquad\cdot \P( |\underline \xi^{(i+1)}_{s_{\ell+2}}|\ge \lambda K^{1-\mu}/(4e)\mid \xi_{s_{\ell+1}}(v_{i+1})=1).\label{eq:first-infest}
\end{align}
On the last factor we shall use Claim \ref{claim:star} shortly, but first we bound the probability of each other factor in \eqref{eq:product-factors} from below by requiring that the sender vertex $u_{h-1}$ infects $u_{h}$ during a time interval of length $4^\mu$ from below, and then  $u_h$ stays infected for the rest of the time-interval. More generally, along an edge $(u,v)$, for any two time-moments $t<t'$, with infection rate $r$ along the edge,  
\[
\begin{aligned}
 \P(\xi_{t'}(v) = 1 \mid \xi_{t}(u) = 1  ) \ge \int_{\tau=0}^{t'-t} (\e^{-\tau})(r \e^{- r\tau}) \e^{-((t'-t)-\tau)}\mathrm d\tau  =\e^{-(t'-t)} \Big(1-\e^{-r (t'-t)}\Big).
 \end{aligned} 
\]
On the path $(u_0, u_1, u_2, \dots, u_\ell, u_{\ell+1})$ (with $u_0:=v_i$, $u_{\ell+1}:=v_{i+1}$), we apply this lower bound with $t'-t=4^\mu$ along each edge, with rates $r(u_{h-1}, u_{h})= \lambda/4^{\mu}$ for all $h\in\{2, \dots, \ell\}$, and $r(u_0, u_1)=r(u_\ell, u_{\ell+1})=\lambda/(2K)^{\mu}$.  
For \eqref{eq:first-infest}, we recall that $s_{\ell+2}-s_{\ell+1}=1$, so here \eqref{eq:star-i} directly applies, hence
\begin{equation*}
\begin{aligned}
\P(J_j \mbox{ successful}\mid \CA_2(J_j))
&\ge \e^{-1}(1-\e^{-c_1 \la K^{1-\mu}}) \\
&\qquad \cdot \Big(\e^{-4^\mu} \big(1-\e^{-4^\mu\lambda/(2K)^{\mu}}\big)\Big)^2\prod_{h=1}^{\ell}\e^{-4^\mu} \Big(1-\e^{-4^\mu\cdot\lambda/4^{\mu}}\Big).
\end{aligned}
\end{equation*}
Then we may apply that $1-\e^{-x}\ge x/2$ for all $x<1/2$ to arrive at
\begin{equation}\label{eq:qK}
\begin{aligned}
\P(J_j \mbox{ successful}\mid \CA_2(J_j) ) &\ge (1-\e^{-c_1 \la K^{1-\mu}}) e^{-4^\mu(\ell+2)-1} (\la/2)^{\ell} (2K)^{-2\mu}\\
&\ge c(c_2 \lambda)^{\ell}K^{-2\mu}=:q_{K},
\end{aligned}
\end{equation}
for some constant $c>0$ and $c_2:=\e^{-4^\mu}/2$, as long as $(1-\e^{-c_1 \la K^{1-\mu}})\ge 1/2$ which is ensured since we already assumed $K\ge K_0(\delta)$ at \eqref{eq:localtime}.
Since the time-intervals are disjoint, on $\CA_3(v_i)$ from \eqref{eq:e3}, by the strong Markov property, the indicators of the events $\{J_j \mbox{ successful}\} $ stochastically dominate $m_K/4$ independent trials (with $m_K$ from \eqref{eq:tk}), each with success probability $q_{K}$ from \eqref{eq:qK}. Let $\CA_4(v_i)$ be the event that at least one of the intervals is successful. Then
 \begin{equation}\label{eq:E2}
\P(\CA_4(v_i)\mid \CA_1(v_i)\cap \CA_3(v_i))  \ge 1- (1-q_K)^{m_K/4} \ge 1- \e^{-m_Kq_K/4},
 \end{equation}
 where we used that $1-x\ge \e^{-x/2}$ for all $x<1/4$, which is applicable since $q_K$ in \eqref{eq:qK} tends to $0$ with $K$. 
 We now analyze the exponent $m_Kq_K$ as a function of $K$ on the rhs of \eqref{eq:E2}.The assumption in this claim is that $\ell(K)=o(K^{1-2\mu})$ (in contrast to \eqref{eq:ellK} which is more specific). So, we may assume wlog that $\ell(K)$ can be written in the form 
 \begin{equation}\label{eq:ellK-new}
 \ell(K):=\wt g(K) K^{1-2\mu} \qquad \mbox{for}  \qquad \wt g(K)\to 0 \mbox{ as } K\to \infty.
 \end{equation}  
Recalling from \eqref{eq:star-iii} that $T_K=\exp(c_1 \lambda^2 K^{1-2\mu})$, and $m_K\ge cT_K/\ell(K)$ from \eqref{eq:tk}, as well as \eqref{eq:qK}, we obtain using that $1/\ell(K)\ge K^{-(1-2\mu)}$:
 \begin{equation*} 
\begin{aligned}
 m_Kq_K&\ge  c  (T_K/\ell) \cdot  (c_2 \lambda)^{\ell}K^{-2\mu} = c\exp\bigg( c_1 \lambda^2 K^{1-2\mu} + \wt g(K) K^{1-2\mu} \log(c_2\lambda)\bigg) K^{-1}\\
 &=c \exp\bigg( \lambda^2 K^{1-2\mu} (c_1 - \wt g(K) |\log(c_2\la)|/\lambda^2 ) - \log (K)\bigg).
 \end{aligned}
 \end{equation*}
 We now argue that for any small fixed $\lambda>0$ we can choose $K$ sufficiently large so that the rhs tends to infinity. 
 First choose $K(g, \lambda)$ so large that for all $K\ge K(g,\lambda)$ the inequality
 \[\wt g(K)  |\log(c_2\lambda)|/\lambda^2 < c_1/2\]
 holds. This is doable since $\wt g(K)\to 0$. For all  $K>K(g,\lambda)$ we thus have 
 \[ m_Kq_K \ge C \exp\bigg( \lambda^2 K^{1-2\mu} c_1/2 - \log (K)\bigg). \]
 We now further increase $K(g, \lambda)$ if necessary so that 
 $ m_Kq_K\ge \log(8/\delta)/c.$ (We comment that by wlog assuming a monotonically decreasing $\wt g$, the minimal $K({g, \lambda})$ can be chosen  as a constant multiple of
 \[\wt g^{(-1)}\big( c_1 \lambda^2/ (2|\log (c_2\la)|)\big)\vee \lambda^{-(2+\varepsilon)/(1-2\mu)}\]
 for some $\eps>0$, for all $\lambda$ sufficiently small.) Returning to \eqref{eq:E2}, and using \eqref{eq:localtime}, \eqref{eq:e3} we see that 
 \[ \P(\CA_4(v_i)\mid \CA_3(v_i)\cap \CA_1(v_i))\ge 1-\delta/8 \quad \mbox{and}\quad \P(\CA_4(v_i)\cap \CA_3(v_i)\cap \CA_1(v_i)) \ge 1-\delta/2.\]
On the event $\CA_4(v_i)\cap \CA_3(v_i)\cap \CA_1(v_i)$, at least one $J_j$ is successful, and by \eqref{eq:successful-def}, that means that at least $\lambda K^{1-2\mu}/(4e)$ leaves in the star of $v_{i+1}$ are infected at some time in the interval $[t_0, t_0+T_K]$. Using the strong Markov property, and applying now \eqref{eq:star-iii}, $v_{i+1}$ stays $\lambda K^{-\mu}$ infested during the rest of the time interval $[t_0, t_0+T_K]$ with probability $1-\exp(-c_1\lambda^2K^{1-2\mu})\ge 1-\delta/4$ by our initial assumption that $K\ge K_0(\delta)$. This finishes the proof.  
\end{proof}

\subsubsection{Local survival through renormalization}\label{sec:renorm}
Having established Lemma \ref{lem:embedded_stars} and Claim \ref{claim:spread_between_stars} we are in a position to prove Theorems \ref{thm:prod_GW} and \ref{thm:max_GW}(a) by showing that the embedded structure $H_{K,\ell(K)}$ sustains the infection (locally) indefinitely with positive probability. This is formalized below in Lemma \ref{lem:survival_on_the_line}. The proof of this has two steps. The first step is a time-renormalization. Based on the results of Claim \ref{claim:spread_between_stars}, we prove that on $H_{K, \ell(K)}$ the infection moves between neighboring centers with large enough probability on a specified discrete time-scale, leading to a renormalized version of the contact process on $\N$. The second step is to establish a relationship between this renormalized contact process, and a certain oriented percolation model, which then can be analyzed by techniques from percolation theory, involving a Peierls-type argument. This connection was already used in \cite{DG83} to derive various results for the contact process on $\Z$.

\begin{lemma}\label{lem:survival_on_the_line}
For any fixed $\mu<1/2$ and $\lambda>0$, there is a $K_0(\lambda)$ such that the following holds for all $K>K_0(\lambda)$. Let $H=H_{K, \ell(K)}$ be the graph defined in Definition \ref{def:M-embedding} with $\ell(K)=o(K^{1-2\mu})$ and with $v_1$ being the center of its first star. Consider the penalty function $f(x,y)=(xy)^\mu$. Then both the contact process $\CPf(H,\ind_{v_1})$ and $\BRW(H,\ind_{v_1})$ exhibit local survival with positive probability.
\end{lemma}

\begin{proof} By the stochastic domination between $\CPf(H,\ind_{v_1})$ and $\BRW(H,\ind_{v_1})$ in Lemma \ref{lem:StochDom2}, it is enough to prove the statement for $\CPf(H,\ind_{v_1})$.
For fixed $\lambda>0$, we choose a small $\delta>0$ specified later. Then we choose $K$ large enough such that $K\ge K_{\lambda,\delta}$ as in Claim \ref{claim:spread_between_stars}. Finally, let $T_K=\exp(c_1\lambda^2 K^{1-2\mu})$ as in Claim \ref{claim:star}. Then, Claim \ref{claim:spread_between_stars} yields the following: for any $v_i$ in $H_{K,\ell(K})$, if $v_i$ is $\lambda K^{-\mu}$-infested at some time $t_0$, then $v_{i+1}$ is  $\lambda K^{-\mu}$-infested by $v_i$ at time $t_0+T_K$ with probability at least $1-\delta$, and the same holds for $v_{i-1}$ when $i\ge 2$. (However, these two events are not necessarily independent.) Throughout this proof, the term "infested" will refer to "$\lambda K^{-\mu}$-infested".

Now we construct an oriented percolation model, which we couple with $\CPf(H,\ind_{v_1})$ so that it dominates from below $\CPf(H,\ind_{v_1})$ restricted to the vertices $\{v_1,v_2,\ldots\}$ at times $\{T_K,2T_K,\ldots\}$. Let $\CH$ be an oriented graph on the vertex set
\begin{equation*}
V_{\CH}=\{(x,y)\in \Z^+\times\Z^+: x+y\text{ even}\}
\end{equation*}
with the oriented (equivalently, directed) edge set
\begin{equation}\label{eq:renorm_open}
E_{\CH}=\{((x_1,y_1),(x_2,y_1+1))\in V_{\CH}\times V_{\CH}: |x_2-x_1|=1\}.
\end{equation}
Observe that $\CH$ is isomorphic to a subgraph (a cone) of $\Z^+\times\Z^+$ as a graph but the edges are `diagonal' and have Euclidean length $\sqrt{2}$. In $\CH$, we will refer to the vertex sets $\{(x,1)\}_{x\in\Z^+}$ and $\{(1,y)\}_{y\in\Z^+}$ as the $x$- and $y$-axis, repectively. For every oriented edge $e=((x_1,y_1),(x_2,y_2))$ -- where $y_2=y_1+1$ and $x_2=x_1\pm 1$ by \eqref{eq:renorm_open} --  define the event $\CA_e=\CA_{(x_1,y_1),(x_2,y_2)}$ that either $v_{x_1}$ is not infested at time $y_1 T_K$, or $v_{x_1}$ is infested at time $y_1 T_K$ and it infests $v_{x_2}$ by time $y_2 T_K$ in the sense of Claim \ref{claim:spread_between_stars}. The same claim shows that
\begin{equation}\label{eq:renorm_open_prob}
\P(\CA_e)\ge 1-\delta \quad\text{for every }e\in E_{\CH}.
\end{equation}

Now let $\eta:V_{\CH}\to \{0,1\}$ be a function on the vertices of $\CH$ defined recursively as
\begin{equation}\label{eq:eta-def}
\begin{aligned}
\eta((x,1))&=\ind\{x=1\},\\
\eta((x,y+1))&=
\begin{cases}
1 & \text{ if }\eta(x-1,y)=1 \text{ and }\CA_{(x-1,y),(x,y+1)} \text{ holds, or}\\
  & \quad \eta(x+1,y)=1 \text{ and } \CA_{(x+1,y),(x,y+1)} \text{ holds},\\
0 & \text{ otherwise}.
\end{cases}
\end{aligned}
\end{equation}
Define the event
\begin{equation}\label{eq:renorm_I1}
\CI_1=\{\text{$v_1$ is infested at time $T_K$ in $\CPf(H,\ind_{v_1})$}\},
\end{equation}
which exactly corresponds to $\eta((1,1))=1$.
Then, conditioned on $\CI_1$, $\eta(x,y)=1$ exactly when there is an ``infestation'' path $\pi$ through stars $(v_1=v_{\pi_1}, v_{\pi_2}, \dots, v_{\mathfrak{l}(\pi)}=v_x)$ so that $v_{\pi_j}$ is infested by $v_{\pi_{j-1}}$ at time $jT_K$. So, on $\CI_1$, 
\begin{equation}\label{eq:renorm_stochdom}
\big(\eta(x,y)\big)_{(x,y)\in V_{\CH}}\ {\buildrel d \over \le}\  \big(\xi_{y T_K}(v_x)\big)_{(x,y)\in V_{\CH}}.
\end{equation}
We now define a subgraph of $\CH$. Let us declare each edge $e\in E_{\CH}$ open if and only if $\ind\{\CA_e\}=1$,  closed otherwise, and denote the graph of open edges by $G(\CH)$. This is a percolation model, where the outgoing edges from a vertex $(x,y)$ are dependent, however, the outgoing edges from distinct vertices are independent due to the strong Markov property and Claim \ref{claim:spread_between_stars}.  The open connected component of $(1,1)$ is
\begin{equation}\label{eq:renorm_CCdef}
\CC_{(1,1)}=\{(x,y)\in V_{\CH}: \text{there is an oriented path of open edges from $(1,1)$ to $(x,y)$}\}.
\end{equation}
Then, comparing $\CC_{(1,1)}$ to $\{(x,y):\eta(x,y)\}=1\}$ in \eqref{eq:eta-def}, which is defined recursively as precisely those vertices that are accessible from $(1,1)$ via an oriented path of open edges in $\CH$, we obtain that $\{(x,y):\eta(x,y)\}=1\}=\CC_{(1,1)}$.

Now we carry out a Peierls-type argument to prove local survival of $\CPf$. Due to the coupling and stochastic domination in \eqref{eq:renorm_stochdom}, and \eqref{eq:renorm_CCdef}, it is enough to show that with positive probability $\CC_{(1,1)}$ contains infinitely many vertices of the form $(1,y)$. This implies for $\CPf$ that $v_1$ is infested at times $yT_K$, for infinitely many $y$, which guarantees local survival.
Let
\begin{equation}\label{eq:renorm_max}
Y_{\text{max}}=\sup\{y\in\Z^+:(1,y)\in\CC_{(1,1)}\}.
\end{equation}
We will prove that for small enough $\delta>0$ in \eqref{eq:renorm_open_prob} it holds that  $\P(Y_{\text{max}}=\infty)>3/4.
$

Assume to the contrary that $\{Y_{\text{max}}=k\}$ for some $k<\infty$. We now construct a path of length $k$, which starts from the $y$-axis next to $(1,k)$, and forms a part of the boundary of $\CC(1,1)$ containing enough closed edges in $\CH$. Define for each edge $e=((x_1,y_1),(x_2,y_2))\in E_{\CH}$ its (unoriented) dual $e'=\{(x_1,y_2),(x_2,y_1)\}$. The dual edges  connect vertices on the dual lattice $\CH':=\{(x,y)\in \Z^+\times\Z^+: x+y \text{ odd}\}$. We declare the dual edge $e'$ closed if $e$ is closed, and open if $e$ is open. We then define the (outer edge-) boundary of $\CC_{(1,1)}$ as the set of dual edges
\begin{equation}\label{eq:renorm_boundary}
\partial \CC_{(1,1)}=\{e':\text{ exactly one of the two endpoints of $e$ is in }\CC_{(1,1)}\}.
\end{equation}
Since $\CH$ is a cone in $\Z^+\times\Z^+$, and $\CC_{(1,1)}$ is connected per definition, $\partial\CC_{(1,1)}$ is a union of connected contours in $\CH'$, which along with (parts of) the $x$- and $y$-axes encircle $\CC_{(1,1)}$. Assume now that the event $\{Y_{\text{max}}=k\}$ occurs. This implies that  $(1,k)\in\CC_{(1,1)}$ and $(1,k+2)\notin\CC_{(1,1)}$. So, define the edges and their duals
\begin{align*}
\hat{e}_{k,1}&=((1,k),(2,k+1)),\qquad\quad \hat e_{k,1}'=\{(1,k+1), (2, k)\},\\
\hat{e}_{k,2}&=((2,k+1),(1,k+2)), \qquad \hat e_{k,2}'=\{(1,k+1), (2,k+2)\}.
\end{align*}
 Now, if $(2,k+1)\notin\CC_{(1,1)}$, then since $(1,k)\in \CC_{(1,1)}$, the dual edge $\hat{e}_{k,1}'\in\partial\CC_{(1,1)}$ (and  $\hat{e}_{k,2}'\notin\partial\CC_{(1,1)}$). In this case, define $\hat{e}_k=\hat{e}_{k,1}$. On the other hand, if $(2,k+1)\in\CC_{(1,1)}$, then since $(1,k+2)\notin \CC_{(1,1)}$, the dual edge  $\hat{e}_{k,2}'\in\partial\CC_{(1,1)}$ (and $\hat{e}_{k,1}'\in\partial\CC_{(1,1)}$). In this case, define $\hat{e}_k=\hat{e}_{k,2}$. In both of these cases, the vertex $(1,k+1)$ is the starting point of the dual $\hat{e}_k'$, which is in $\partial\CC_{(1,1)}$, and the other dual edge with endpoint $(1,k+1)$ is not in $\partial\CC_{(1,1)}$. Then we start exploring $\partial\CC_{(1,1)}$, starting from $e'_1:=\hat{e}_k'$ by following the dual edges in this connected component of $\partial\CC_{(1,1)}$. That is, the next dual edge in the path, denoted by $e'_2$, is incident to $(2,k)$ if $e'_1=\{(1,k+1),(2,k)\}$ and to $(2,k+2)$ if $e'_1=\{(1,k+1),(2,k+2)\}$. Then we continue from the other endpoint of $e'_2$, and so on. We continue this exploration process either indefinitely (if $\CC_{(1,1)}$ is infinite), or until we reach the $x$-axis (if $\CC_{(1,1)}$ is finite). As we explain next, these are the only two possible outcomes. For an example of the second outcome, see Figure \ref{fig:oriented}.

\begin{figure}[h]
\begin{center}
\includegraphics[scale=0.5]{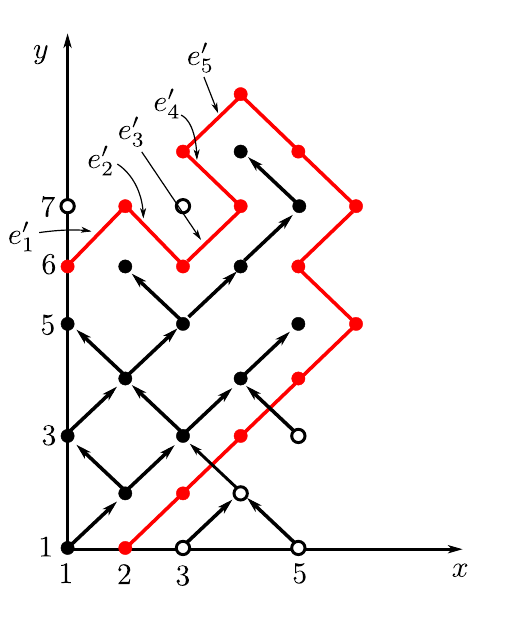}
\caption{This example shows a finite oriented cluster of the origin $\CC_{(1,1)}$: filled black circles are vertices in $\CC_{(1,1)}$ while empty black circles are vertices that do not belong to $\CC_{(1,1)}$. The oriented, black edges are open in $\CH$, while the closed edges  of $\CH$ are not drawn. The red contour and red vertices belong to the dual lattice $\CH'$.
Since $Y_{\text{max}}=5$, the dual contour $\pi_\partial$ starts from $(1,6)$, and follows the closed dual edges colored red, ending at $(2,1)$. Edges of $\CH$ pointing out of $\CC_{(1,1)}$ are all closed (not drawn), whereas edges pointing into $\CC_{(1,1)}$ may be open -- such as the edge $((5,3),(4,4))$ -- or closed.}
\label{fig:oriented}
\end{center}
\end{figure}
 
 Denote by $\pi_\partial=(e'_1,e'_2,\ldots)$ the path (as a sequence of dual edges) obtained this way. It is possible that $\pi_\partial$ visits the $y$-axis above $(1,k+1)$ (say at $(1,y')$ with $y'>k$), but since $Y_{\text{max}}=k$, this can only happen when $(2,y')\in \CC_{(1,1)}$ and $(1,y'+1)\notin \CC_{(1,1)}$, and then we can always continue the path $\pi_\partial$ by traversing the dual edge $\{(1,y'),(2,y'+1)\}$. 
 However, $\pi_\partial$ cannot visit the $y$-axis below $(1,k)$, since then we would have encircled the entire $\CC_{(1,1)}$, starting from $(1,k+1)$, without containing $(1,1)$, a contradiction. Hence, one of the two remaining cases happens. We either find an infinite path $\pi_{\partial}$ in $\partial\CC_{(1,1)}$, and then we set $\pi_\partial(k)$ to be the sequence of its first $k$ edges. Or, we find a finite path $\pi_{\partial}$ that reaches the $x$-axis, in particular, the dual vertex $(2, 1)$. This path has length at least $k$, since the path starts at $(1,k+1)$, and the $y$ coordinate only  changes by $\pm1$ between consecutive vertices on the path. In this case we again set $\pi_\partial(k)$ to be the sequence of the first $k$ edges of $\pi_\partial$.
 
We now categorize edges of $\pi_\partial(k)$ (all are in $\partial \CC_{(1,1)}$) as follows. 
Recall that edges of $\CH$ in \eqref{eq:renorm_open} are oriented (directed), and recall 
\eqref{eq:renorm_boundary}. Given $\CC_{(1,1)}$ let us call the dual edge $e'\in \partial \CC_{(1,1)}$ an \emph{outward} dual edge, if for the edge $e=((x_1,y_1),(x_2,y_2))$ it holds that $(x_1, y_1)\in \CC_{(1,1)}$ and $(x_2, y_2)\notin \CC_{(1,1)}$ and let us call $e'$ an \emph{inward} dual edge if $(x_1, y_1)\notin \CC_{(1,1)}$ and $(x_2, y_2)\in \CC_{(1,1)}$. Per definition of $\CC_{(1,1)}$ in \eqref{eq:renorm_CCdef}, the outward edges and their duals are all closed. However, $\CC_{(1,1)}$ does not determine the status of inward dual edges. 

We now prove that for any realization of $\CC_{(1,1)}$, at least half of the edges of $\pi_\partial(k)$ are outward dual edges, and hence closed.
Let us introduce the notation $\pi_0=(1,k+1), \pi_1, \pi_2, \dots, \pi_k, \dots$ for the vertices of the path $\pi_\partial$ in order, and define the directed edge $\underline{e}'_i=(\pi_{i-1}, \pi_i)$ for all $i\ge 1$ (the directed version of $e'_i$). Then for all outward dual edges $e'_i\in\pi_\partial$, $\underline{e}'_i$ is pointing to the right (in the direction of increasing $x$ coordinate), and for all inward dual edges $e'_i\in\pi_\partial$, $\underline{e}'_i$ is pointing to the left (in the direction of decreasing $x$ coordinate).
Since $\pi_\partial(k)$ starts from $(1,k+1)$, which is part of the $y$-axis, and remains in the positive quadrant, at least half of its dual edges have to be directed to the right, thus, duals of outward edges. Hence, at least $k/2$ dual edges in $\pi_\partial(k)$ are closed. Further, since every vertex in $V_{\CH}$ has at most two outgoing (non-dual) edges, in every possible realization $(e_1', e_2', \dots, e_k')$ of $\pi_{\partial}(k)$ we can find $k/4$ edges that are all closed and that their oriented non-dual edges in $\CH$ all start from different vertices.

By \eqref{eq:renorm_open_prob}, the probability that a given edge (and its dual) is closed is at most $\delta$. As mentioned before \eqref{eq:renorm_CCdef}, the status of different edges are not independent, however, $\CA_{(x_1,y_1),(x_2,y_2)}$ is independent of $\CA_{(x'_1,y'_1),(x'_2,y'_2)}$ if $(x_1,y_1)\ne(x'_1,y'_1)$. That is, two edges $e_1,e_2\in E_\CH$ are open or closed independently if their starting points are distinct. 

We call a given connected path $(e'_1,\ldots,e'_k)$ of dual edges eligible if it is a possible realization of $\pi_\partial(k)$ (of which one requirement is that one of the endpoints of $e_1'$ is $(1,(k+1))$). Then, for all such $(e_1', e_2', \dots, e_k')$, 
\begin{equation}\label{eq:renorm_path_prob}
\P\big(\pi_\partial(k)=(e'_1,\ldots,e'_k)\big)\le \delta^{k/4}.
\end{equation}
Next, we upper bound the number of eligible paths $(e'_1,\ldots,e'_k)$. Since $(e'_1,\ldots,e'_k)$ is a path starting from $(1,k+1)$ on the dual lattice $\CH'$ isomorphic to a quadrant of $\Z^2$, each of the $k$ steps in the exploration of $\pi_\partial(k)$ can be taken in one of at most three directions. This yields that the number of possible trajectories is at most $3^k$. Therefore, by a union bound, 
\begin{equation}\label{eq:renorm_Yk_prob}
\P(Y_{\text{max}}=k)\le \P\left(\bigcup_{(e_1', \dots, e_k') \text{ eligible}} \{\pi_\partial(k)=(e_1', \dots, e_k')\}\right)\le 3^k\delta^{k/4}.
\end{equation}
Then \eqref{eq:renorm_Yk_prob} implies that
\begin{equation}\label{eq:renorm_Ak_prob}
\P(Y_{\max}<\infty)=\sum_{k=1}^\infty\P(Y_{\text{max}}=k)\le \sum_{k=1}^\infty 3^k\delta^{k/4}<1/4,
\end{equation}
whenever $\delta\in(0, (1/15)^4)$.
Consequently, 
$\P(Y_{\text{max}}=\infty)>3/4.$
Finally, recalling $\CI_1$ from \eqref{eq:renorm_I1},  $\P(\CI_1)>1/3$ for large enough $K$ by Claim \ref{claim:star}.  By the stochastic dominance in \eqref{eq:renorm_stochdom}, it follows from a union bound that
\begin{equation*}
\P\big(\CPf(H,\ind_{v_1})\text{ survives locally at }v_1\big)\ge \P\big(\CI_1\cap \{Y_{\text{max}}=\infty\}\big) \ge 1-2/3-1/4>0. 
\end{equation*}
 This proves local survival of $\CPf(H,\ind_{v_1})$ with positive probability.
\end{proof}

\begin{proof}[Proof of Theorem \ref{thm:prod_GW}]
Lemma \ref{lem:embedded_stars} states that for some $M\ge 1$ there exists $K_1$ such that and for $K>K_1$ and $\ell(K)$ as in \eqref{eq:ellK} $H_{K,\ell(K)}$ can be $M$-embedded in $\CT$ almost surely. Set $\bar\lambda=\lambda/M^{2\mu}$ and let $K_0(\bar\lambda)$ be given by Lemma \ref{lem:survival_on_the_line}. Now let
$K>\max(K_0(\bar\lambda),K_1)$. Then Lemma \ref{lem:embedded_stars} yields that $H_{K,\ell(K)}$ can be $M$-embedded in $\CT$ almost surely. Let $v_1$ be the center of the first star in the embedded $H_{K,\ell(K)}$. Recalling \eqref{eq:factor-M} in Definition \ref{def:M-embedding}, we observe that the process $\CPf(\CT,\ind_{v_1})$ restricted to the vertices of the embedded $H_{K,\ell(K)}$ is stochastically dominated from below by the process $\mathrm{CP}_{f,\bar\lambda}(H_{K,\ell(K)},\ind_{v_1})$ on a standalone copy of $H_{K,\ell(K)}$. Combining this with Lemma \ref{lem:survival_on_the_line} and $K\ge K_0(\bar\lambda)$ implies that $\mathrm{CP}_{f,\bar\lambda}(H_{K,\ell(K)},\ind_{v_1})$ survives locally with positive probability. This, along with the fact that with positive probability, $\CPf(\CT,\ind_\varnothing)$ infects $v_1$ at some point in time finishes the proof.
\end{proof}

\begin{proof}[Proof of Theorem \ref{thm:max_GW}(a)]
This is an easy consequence of Theorem \ref{thm:prod_GW} by stochastic domination, noting that $\max(d_u,d_v)^\mu\le (d_u d_v)^\mu$.
\end{proof}

%% file: proofs_CM.tex
\section{The configuration model: k-cores sustain the infection when stars do not}\label{sec:proofsCM}

In this section we will prove part (b) of Theorem \ref{thm:max_CM_survival}. 
A crucial difference between the classical contact process and the degree-dependent version in this regime is that star-graphs do not sustain the infection, in fact they heal quickly when $\mu>1/2$, by Claim \ref{claim:star}. However, we know from Section \ref{sec:max_GW_weak} that the approximating Galton-Watson tree shows global survival (only), which suggests long survival on the configuration model. So we set out to find a new structure -- a subgraph --  embedded in the configuration model that sustains the degree-dependent contact process for a long time. 
To find such a subgraph, 
we need to take into account that vertices that have either too low or too high degree cannot sustain the infection, either because the penalty $f$ on them is too high or because $\lambda$ is assumed to be close to $0$. The subgraph we found is the $k$-core -- a maximal subgraph of the configuration model where each vertex has degree at least $k$ inside the same subgraph -- but with a twist: in the original configuration model with fat tailed degrees, the $k$-core contains vertices of very high degree (e.g. polynomials of $n$). However, the degree-dependent CP near these vertices would have too high penalty $f$, so we need to exclude them from the $k$-core. 

As a result we look at the $k$-core of \emph{not} the original configuration model, but the subgraph obtained after \emph{removing all vertices of degree above} a threshold value $M$, where now $M$ is a constant depending only on $k$ but not on the total number of vertices $n$. It is a priori unclear whether such a low truncation value even produces a connected graph, let alone contains a linear-sized $k$-core.  So our first step is to study the dependence between $M=M(k)$ and $k$ so that a linear sized $k$-core still exists in the configuration model where all vertices of degree above $M$ are removed. 

After solving this issue, and finding the linear-sized $k$-core on vertices of degree at most $M(k)$, we show that $\CPf$ survives on this $k$-core. For this step, our proof is a non-trivial adaptation of the proof of \cite[Theorem 1.2,  part (b)]{mourratvalesin2016}, which shows long survival in the original contact process model on $(d+1)$-regular random graphs, when $\lambda$ is above the lower critical $\lambda_1(\mathbb{T}^d)$ on $d$-regular trees needed for global survival \cite{Pem92}. However, in our case we have $\lambda$ arbitrarily close to $0$. Fortunately, we \emph{can} choose $k$ as a function of $\lambda$ that makes the process locally supercritical. This also makes the proof different from that in  \cite{mourratvalesin2016} even beyond finding the $k$-cores. 

First, we define the $k$-core of a graph.

\begin{definition}\label{def:kcore}
Let $G$ be any simple, finite graph. For a fixed positive integer $k$, the $k$-core of $G$ is the largest induced subgraph $\Core_k(G)$ of $G$ such that every vertex in $\Core_k(G)$ has degree at least $k$ within $\Core_k(G)$.
\end{definition}

It is not hard to see that the $k$-core $\Core_k(G)$ in Definition \ref{def:kcore} is well-defined -- but may be empty -- by the following algorithm producing it. First, delete all vertices of $G$ that have degree less than $k$ along with their incident edges. Then do the same with the resulting graph, repeatedly, until no new vertex is deleted. The output of this algorithm is the unique largest induced subgraph of $G$ with all degrees at least $k$. Note that the $k$-core of a graph might be empty or may contain more than one component.

\subsection{The subgraph spanned on low-degree vertices contains a \texorpdfstring{$k$}{k}-core}
Our first goal is to prove the existence of a $k$-core in the configuration model after we remove all vertices with too high degrees. Throughout this section, we work with the configuration model $\mathrm{CM}(\underline d_n)=:G$ in Definition \ref{def:CM} on the degree sequence 
$\underline d_n=(d_1, \dots, d_n)$ that satisfies the regularity assumptions in Assumption \ref{assu:regularity}, and the weak power-law empirical degrees of Assumption \ref{assu:empirical-power-law} with exponent $\tau>2$ and error $\varepsilon>0$.

We now set up the procedure of removing all vertices above some degree $M$ and the edges attached to those in $\mathrm{CM}(\underline d_n)$. This is often called a \emph{targeted attack} on the graph. Because the graph is formed by a random matching, and the half-edges that have one endpoint at a vertex with degree larger than $M$ and another endpoint at a vertex with degree at most $M$ are also removed,  the degrees in the remaining graph are random.
\begin{definition}[Configuration model under targeted attack]\label{def:cm-attack}
 Consider the configuration model $\mathrm{CM}(\underline d_n)$ in Definition \ref{def:CM}  on the degree sequence 
$\underline d_n=(d_1, \dots, d_n)$. Fix some value $M\ge 0$. Denote 
\begin{equation}\label{eq:M-vertices}
\begin{aligned}
\CV_{\le M}&:=\{i\le n: d_i \le M\}, \quad V_{\le M}:= |\CV_{\le M}|, \qquad H_{\le M}:=\sum_{i=1}^n d_i \ind_{\{ d_i\le M\}}, \\
\CV_{> M}&:=\{i\le n: d_i > M\}, \quad V_{> M}:= |\CV_{\ge M}|, \qquad H_{> M}:=\sum_{i=1}^n d_i \ind_{\{ d_i> M\}},
\end{aligned}
\end{equation}
Let $G_n[\CV_{\le M}]$ denote the (random) subgraph of $\mathrm{CM}(\underline d_n)$ that is spanned on the vertex set $\CV_{\le M}$. For any $v\in \CV_{\le M}$, we denote the \emph{random} degree of $v$ in $G_n[\CV_{\le M}]$ by $\widetilde d_v$, and we write $\widetilde n_i$ for the number of vertices with degree $i$ in $G_n[\CV_{\le M}]$. For any $z\ge 0$ define
\begin{equation}\label{eq:empirical-truncated}
\widetilde F_{n,M}(z):=\frac{1}{V_{\le M}}\sum_{v\in \CV_{\le M}} \ind_{\{\widetilde d_i\le z\}}= \frac{\sum_{i\le z}\widetilde n_i}{V_{\le M}},
\end{equation}
and let $\widetilde D_{n,M}$ denote a random variable with cdf $\widetilde F_{n,M}(z)$.
  \end{definition} 
  Observe that $\widetilde F_{n,M}(z)$ is the new empirical distribution of the degrees, after the targeted attack. This distribution is random, caused by the random matching that generated the graph before the attack.
  The quantities in \eqref{eq:M-vertices} all depend on $n$, which we suppress in notation.
We are ready to state the existence of the $k$-core in the configuration model under attack.

\begin{theorem}\label{prop:CMcore}
Consider the configuration model $\mathrm{CM}(\underline d_n)=:G_n$ in Definition \ref{def:CM} on the degree sequence 
$\underline d_n=(d_1, \dots, d_n)$ that satisfies the regularity assumptions in Assumption \ref{assu:regularity}, and the weak power-law empirical degrees of Assumption \ref{assu:empirical-power-law} with exponent $\tau\in(2,3)$ and error $\varepsilon>0$. 
Fix a large enough positive integer $k$ and define
\begin{equation}\label{eq:CMcore_eta}
\eta_{\min}:=\frac{(3-\tau)}{(3-\tau)-\varepsilon(\tau-1)}\cdot \frac{1+\varepsilon}{1-\varepsilon}-1.
\end{equation}
Assume $\tau, \varepsilon$ are such that $\eta_{\min}\in[0,\infty)$, and for any $\eta>\eta_{\min}$ let $M:=M_{k,\eta}=k^{(1+\eta)/(3-\tau)}$. Let $G_n[\CV_{\le M}]$ be the configuration model under attack in Definition \ref{def:cm-attack}, and denote by $\Core_{k}(G_n[\CV_{\le M}])$ its k-core. Then there exists some $\rho=\rho(k)>0$ such that
\begin{equation}\label{eq:CoreSize}
\lim_{n\to \infty} \P \Big( |\Core_{k}(G_n[\CV_{\le M}])| \ge \rho n \Big) =1, 
\end{equation}
Further, conditioned on its vertex set and degree sequence,  $\Core_{k}(G_n[\CV_{\le M}])$ is itself a configuration model.
\end{theorem}
\begin{remark}[Asymptotics of $\rho$]\label{rem:rho}
The proof of Theorem \ref{prop:CMcore} shows that there exists a constant $c'>0$ such that
\begin{equation}\label{eq:rho} 
\rho(k)> c' k^{-\tfrac{(\tau-1)(1+\varepsilon)}{(2-(\tau-1)(1+\varepsilon))}}.
\end{equation}
\end{remark}
Note that in this lower estimate only the lower bound exponent in Assumption \ref{assu:empirical-power-law} appears.
We comment that $\eta_{\min}\in [0, \infty)$ implies that $\varepsilon<(3-\tau)/(\tau-1)$, which is exactly the condition that the lower bound on the tail-exponent, $(\tau-1)(1+\varepsilon)$ in \eqref{eq:empirical-power-law}, stays strictly below $2$.
Hence a $k$-core exists for all $k$ when the estimates on the empirical power law are so that the tail is always heavier than a power law with infinite variance. Without the truncation at $M$, i.e., for pure power laws, such a result is already known, see \cite{janson2007simple} and \cite{fernholz2003giant}. Here we specify the truncation value $M$ for which the result stays valid. We comment that when $\varepsilon =0$ in Assumption \ref{assu:empirical-power-law}, then our proof can be strengthened so that $M=\Theta(k^{1/(3-\tau)})$ guarantees the existence of a large  $k$-core after the targeted attack. 

We will prove Theorem \ref{prop:CMcore} below using the following two lemmas and the results of Janson and Luczak \cite{janson2007simple} that we will state soon. The first lemma says that the random empirical distribution of $G_n[\CV_{\le M}]$ converges in probability, assuming the regularity assumptions on the original degrees. We use notation from Definition \ref{def:cm-attack}. Given the degree sequence $\underline d_n$, $D_n$ stands for the random variable that follows the empirical distribution $F_n$ of $\underline d_n$ in \eqref{eq:empirical-degree}, and $D$ is the random variable following the limiting distribution in Assumption \ref{assu:regularity}. Define then
\begin{equation}\label{eq:q-n-m}
q_{n,M}:=\E[D_n \ind_{\{D_n\le M\}}]/\E[D_n], \qquad q_M:=\E[D \ind_{\{D\le M\}}]/\E[D],\\
\end{equation}
and we collect the errors below $M$ between the $n$-dependent degree distribution $D_n$ and the limit $D$ as follows:
\begin{equation}\label{eq:delta-n-m}
\begin{aligned}
\delta_n:=\max\Big\{&|q_{n,M}/q_M-1|,\ |(1-q_{n,M})/(1-q_M)-1|, \\
&\ \max_{i\le M,  \P(D=0) } \P(D_n=i),  \max_{i\le M, \P(D=i)\neq 0}|\P(D_n=i)/\P(D=i)-1|\Big\},
\end{aligned}
\end{equation}
with $\delta_n\to 0$ when Assumption \ref{assu:regularity} holds. Typically, for $M$ large $q_M$ is close to $1$ so the relative error of $1-q_{M,n}$ to $1-q_M$ is driving the maximum in the first row, while the second row is only over values $i\le M$.
\begin{lemma}[Degree distribution of CM under attack]\label{lem:CM-attack-degrees}
Consider the configuration model $\mathrm{CM}(\underline d_n)$ in Definition \ref{def:CM}  on the degree sequence 
$\underline d_n=(d_1, \dots, d_n)$ that satisfies the regularity assumptions in Assumption \ref{assu:regularity}. Fix any $M>0$ constant, and let $q_{n,M}, q_M, \delta_n$ as in \eqref{eq:q-n-m}.
Define the following random variable $\widetilde D_M$: for all $i\le M$, let
\begin{equation}\label{eq:binom-thinning} 
\begin{aligned}
p_{M}(i)&:=\P(\widetilde D_M=i) =\sum_{j=i}^M \frac{\P(D=j)}{\P(D\le M)} \binom{j}{i}q_M^i(1-q_M)^{j-i}\\
&= \P( \mathrm{Bin}(D, q_M)=i \mid D\le M). 
\end{aligned}
\end{equation}
Let $X_{n,M}(i):=\widetilde n_i/n=\P(\widetilde D_{n,M}=i\mid G_n[\CV_{\le M}])$ be the random empirical degree distribution of $G_n[\CV_{\le M}]$.
Then for all $\varepsilon_n>0$ that satisfies $\varepsilon_n \gg \max\{\delta_n, 1/\sqrt{n}\}$, 
\begin{equation}\label{eq:convergence-in-P-of-mass}
\P\Big(\sup_{i\le M}\big| X_{n,M}(i)-p_M(i)\big| \ge \varepsilon_n \Big) = O\bigg(\frac{M^3}{n\varepsilon^{2}_n}\bigg)\to 0.
\end{equation}
Further, $\lim_{n\to \infty}\E[\widetilde D_{n,M}\mid G_n[\CV_{\le M}]]=\E[\widetilde D_M]$ in probability. So, the empirical degree distribution $\widetilde D_{n,M}$ of $G_n[\CV_{\le M}]$ satisfies Assumption \ref{assu:regularity} with probability tending to $1$. Furthermore, $G_n[\CV_{\le M}]$ is itself a configuration model on its vertex set, conditioned on the degrees of its vertices.
\end{lemma}
The second lemma proves that the limiting degree distribution of $G_n[\CV_{\le M}]$ is a truncated weak power law with truncation close to $M$ if the original degree distribution satisfied the weak power law assumption. 
\begin{lemma}[Truncated power laws after targeted attack]\label{lem:truncated-power-law}
Consider the configuration model $\mathrm{CM}(\underline d_n)$ in Definition \ref{def:CM}  on the degree sequence 
$\underline d_n=(d_1, \dots, d_n)$ that satisfies the regularity assumptions in Assumption \ref{assu:regularity}, and the power-law empirical degrees of Assumption \ref{assu:empirical-power-law} with exponent $\tau$ and exponent-error $\varepsilon\ge 0$. Let $M>0$ be a constant (i.e., not depending on $n$, but it may depend on $\varepsilon$), and let  
\begin{equation}\label{eq:zmax-def}
\widetilde z_{\max}(M):=2^{-1}(c_\ell/(2c_u))^{\tfrac{1}{(\tau-1)(1+\varepsilon)}} M^{(1-\varepsilon)/(1+\varepsilon)}.
\end{equation}
Consider the limiting degree distribution $\widetilde F_M(z)=:\sum_{i\le z} p_M(i)$ in \eqref{eq:binom-thinning} of $G_n[\CV_{\le M}]$ in Lemma \ref{lem:CM-attack-degrees}. 
Then there exist constants $\widetilde c_\ell, \widetilde c_u, M_0$, such that whenever $M\ge M_0$,  for all $z\in [z_0, \widetilde z_{\max}(M)]$, it holds that
\begin{equation}\label{eq:truncated-power-law}
 \frac{\widetilde c_\ell}{z^{(\tau-1)(1+\varepsilon)}} \le 1-\widetilde F_{M}(z)  \le \frac{\widetilde c_u}{z^{(\tau-1)(1-\varepsilon)}}.
\end{equation} 
\end{lemma}
The proof shows that $\widetilde c_\ell = c_\ell 2^{-(\tau-1)(1+\varepsilon)-2}$ and $\widetilde c_u=2c_u$ are valid choices (although they may not be optimal).
Since the proofs of Lemmas \ref{lem:CM-attack-degrees} and \ref{lem:truncated-power-law} are fairly standard, we provide them in the Appendix on pages \pageref{proof:CM-attack-degrees} and \pageref{proof:truncated-power-law}.

With these lemmas at hand, the proof of Theorem \ref{prop:CMcore} relies on the result of Janson and Luczak \cite{janson2007simple}, describing the $k$-core of the configuration model. To state this result, we introduce some notation.

For a random variable $D$ and $p\in[0,1]$, we let $X_{D,p}$ denote a random variable with Binomial($D,p$) distribution. That is,
\[\P(X_{D,p}=r)=\sum_{l=r}^\infty \P(D=l)\binom{l}{r}p^r(1-p)^{l-r}.\]
We then define the following functions:
\begin{equation}\label{eq:hdp}
h(D,p):=\E[X_{D,p}\ind\{X_{D,p}\ge k\}], \qquad
h_1(D,p):=\P(X_{D,p}\ge k).
\end{equation}

Note that both $h$ and $h_1$ are increasing in $p$, and $h(D,0)=h_1(D,0)=0$. Moreover, $h(D,1)=\E[D\ind\{D\ge k\}]\le \E[D]$, and $h_1(D,1)=\P(D\ge k)\le 1$.

Then the theorem of Janson and Luczak is as follows. They use the same regularity Assumption \ref{assu:regularity} as we do.

\begin{theorem}[Theorem 2.3 in \cite{janson2007simple}]\label{thm:JansonLuczak}
Consider the configuration model $G_n:=\mathrm{CM}(\underline d_n)$ in Definition \ref{def:CM} on the degree sequence 
$\underline d_n=(d_1, \dots, d_n)$ that satisfies the regularity assumptions in Assumption \ref{assu:regularity}. 
For $k\ge 2$ be fixed, let $\Core_k:=\Core_k(G_n)$ be the $k$-core of $G_n$. Let
\begin{equation}\label{eq:p-hat} \hat p:= \max\{ p\le 1: \E[D] p^2=h(D,p) \}.\end{equation}
Then, if $\hat{p}>0$ and $\E[D] p^2<h(D,p)$ for $p$ in some interval $(\hat{p}-\epsilon,\hat{p})$, then $\Core_k(G_n)$ is non-empty whp, and 
\begin{equation}\label{eq:linear-core-Janson}
|\CV(\Core_k)|/n\ {\buildrel \mathbb P \over \longrightarrow}\  h_1(D,\hat{p}), \qquad |\CE(\Core_k)|/n {\buildrel \mathbb P \over \longrightarrow}\ h(D,\hat{p})/2=\E[D] \hat{p}^2/2.
\end{equation}
\end{theorem}
We first need a small extension of this theorem.
\begin{claim}\label{claim:analytic-fixpoint}
Suppose there is a value $p_-$ where $\E[D]p_-^2<h(D,p_-)$ holds (see below \eqref{eq:p-hat}). Then, there is a non-zero fixed point $p_\star>p_-$ of \eqref{eq:p-hat} so that in the interval $(p_\star, p_\star-\eps)$ the inequality $\E[D]p^2<h(D,p)$ holds. Then, $\mathrm{Core_k}(G_n)$  is non-empty and 
\begin{equation}\label{eq:k-core-minimal}
\begin{aligned}
\P&\Big(|\CV(\mathrm{Core_k})|/n \le h_1(D, p_-) (1-\eps)\Big)\\
&\le \P\Big(|\CV(\mathrm{Core_k})|/n \le h_1(D, p_\star) (1-\eps)\Big) \to 0 
\end{aligned}
\end{equation}
as $n\to \infty$.
\end{claim}
\begin{proof}[Sketch of proof]
The first statement, namely that $p^\star$ exists, follows from the continuity of the function $\E[D]p^2-h(D,p)$.  The second statement, that the size of the $k$-core is at least $h_1(D, p_\star)(1-\eps)$, follows from the proof of \cite[Theorem 2.3]{janson2007simple}. Namely, the only case where the proof of \cite[Theorem 2.3]{janson2007simple} does not apply directly is where the function  $f(p)=E[D]p^2-h(D,p)$ does not cross the $0$-line at its maximal $0$ but rather, it touches it. Nevertheless, if one finds a smaller value $p_-$ where the function $f(p)$ is in the negative, it implies that there must be zero-point $p_\star$ of $f$ where the function crosses the $0$ level line. In this case the proof there yields that the density of the $k$-core is at least $h_1(D, p_\star)$, hence \eqref{eq:k-core-minimal} holds. This can be found on \cite[page 56-57]{janson2007simple}, where a value $p_-$ for which $f(p_-)<0$ implies the upper bound on the stopping time of a pruning algorithm generating the $k$-core for \eqref{eq:linear-core-Janson}. That is, the continuous time pruning algorithm of sequentially removing vertices of degrees at least $k$ and their outgoing edges is guaranteed to stop by time $t=-\log (p_-)$, i.e., one can set $t_2=-\log(p_-)$ at the bottom \cite[page 9]{janson2007simple}. An upper bound on the stopping of the pruning algorithm gives a lower bound on the number of remaining vertices forming the $k$-core. In our case, by not knowing whether $p_-$ is adjacent to the maximal fixed point and whether $f$ touches or crosses $0$ there, we lose the upper bound on the $k$-core size.
\end{proof}

With Lemmas \ref{lem:CM-attack-degrees} and \ref{lem:truncated-power-law} at hand, we are ready to prove Theorem \ref{prop:CMcore} by checking the conditions of Theorem \ref{thm:JansonLuczak} and Claim \ref{claim:analytic-fixpoint}.

\begin{proof}[Proof of Theorem \ref{prop:CMcore}]
First, we prove \eqref{eq:CoreSize} holds: we will check that the conditions of Theorem \ref{thm:JansonLuczak}  hold for $G_n[\CV_{\le M}]$ with probability tending to $1$. First, Lemma \ref{lem:CM-attack-degrees} implies that $G_n[\CV_{\le M}]$ is a configuration model (conditioned on its vertices and their degrees), and its (random) degree sequence satisfies Assumption \ref{assu:regularity} with probability tending to $1$. We will use the notations of Lemmas \ref{lem:CM-attack-degrees} -- \ref{lem:truncated-power-law}, so, $\widetilde D_M$ denotes the limiting degree distribution of $G_n[\CV_{\le M}]$.
Since $h(\widetilde D_M,p)$ in \eqref{eq:hdp} is a continuous function of $p$, it is enough for us to find a particular choice of $p$ with $\E[\widetilde D_M]p^2<h(\widetilde D_M,p)$. Based on the tail probabilities of $\widetilde D_M$ from \eqref{eq:truncated-power-law}, in particular the exponent $\tau\in(2,3)$ and the constant $\widetilde c_\ell$ in the lower bound, which holds for $z\in[z_0,\widetilde z_{\max}(M)]$ with $\widetilde z_{\max}(M)$ defined in \eqref{eq:zmax-def}, our goal is to find two positive constants $a_-< a_+$ and $\xi>3-\tau$ and an interval 
\begin{equation}\label{eq:CMcorep}
I_p:=[p_-, p_+]:=\Big[ a_- k^{-(\xi/(3-\tau)-1)}, a_+ k^{-(\xi/(3-\tau)-1)}\Big].
\end{equation}
We will show that when $p\in I_p$, then $\E[\widetilde D_M]p^2<h(\widetilde D_M,p)$. Using \eqref{eq:hdp}, and that $X_{l_1, p}$ stochastically dominates $X_{l_2, p}$ when $l_2>l_1$, we estimate, for some constant $\beta$ and exponent $\xi>3-\tau$ to be chosen later,  \begin{align}
h(\widetilde D_M,p)&=\sum_{l=k}^M\P(\widetilde D_M=l)\sum_{r=k}^l r\!\cdot\!\P(X_{l,p}=r)\nonumber\\
&\ge\sum_{l=\beta k^{\xi/(3-\tau)}}^M\P(\widetilde D_M=l)\E\big[X_{l,p}\ind\{X_{l,p}\ge k\}\big]\nonumber\\
&\ge\P\big(\widetilde D_M\ge \beta k^{\xi/(3-\tau)}\big)\E\big[X_{ \beta k^{\xi/(3-\tau)},p}\ind\{X_{\beta k^{\xi/(3-\tau)},p}\ge k\}\big].\label{eq:CMcore}
\end{align}
We bound the first factor on the rhs of \eqref{eq:CMcore} first. 
Recalling that the tail-bound on $\P(\widetilde D_M>z)$ in \eqref{eq:truncated-power-law}, we get
\begin{equation}\label{eq:CMcore1}
\P\big(\widetilde D_M\ge \beta k^{\xi/(3-\tau)}\big)\ge \widetilde c_\ell \left(\beta k^{\xi/(3-\tau)}\right)^{-(\tau-1)(1+\varepsilon)},
\end{equation}
on the condition that  $\beta k^{\xi/(3-\tau)}\le \widetilde z_{\mathrm{max}}(M)$ which we now check. (This is the place where we need that the truncation point $M$ is high enough.) We expand $\widetilde z_{\mathrm{max}}(M)$ in  \eqref{eq:zmax-def} as a function of $k$ using that $M=k^{(1+\eta)/(3-\tau)}$ below \eqref{eq:CMcore_eta}. We write $\widetilde C$ for the prefactor in \eqref{eq:zmax-def} that only depends on $c_\ell, c_u, \tau,\varepsilon$:
\begin{equation}\label{eq:zmax2}
\widetilde z_{\mathrm{max}}(M)=\widetilde C M^{(1-\varepsilon)/(1+\varepsilon)}=\widetilde C k^{\left(\frac{1+\eta}{3-\tau}\right)(1-\varepsilon)/(1+\varepsilon)}.
\end{equation}
Treating $\beta, \widetilde C$ as constants while $k$ can be chosen arbitrarily large, the rhs of \eqref{eq:zmax2} is larger than $\beta k^{\xi/(3-\tau)}$ for all sufficiently large $k$ when
\begin{equation}\label{eq:eta-xi-relation}
\xi<(1+\eta)(1-\varepsilon)/(1+\varepsilon),
\end{equation}
which shall lead to the assumption that $\eta>\eta_{\min}$ in \eqref{eq:CMcore_eta} shortly.
Next, we bound the second factor on the rhs of \eqref{eq:CMcore}. For any variable $X$ it holds that 
$\E[X \ind_{\{X\ge k\}}]= \E[X] -\E[X \ind_{\{X< k\}}]\ge  \E[X] - k \P(X<k)$.
In  \eqref{eq:CMcore} $X\sim\Bin( \beta k^{\xi/(3-\tau)},p)$, and with the choice $a_-:=2/\beta$, we can lower bound its mean using that $p>p_-$ in \eqref{eq:CMcorep} as $\beta k^{\xi/(3-\tau)}p>2k^{\xi/(3-\tau)}k^{-(\xi/(3-\tau)-1)}=2k$. Hence, a Chernoff bound applies and we obtain that
\begin{equation}\label{eq:CMcoreChernoff}
k\P\big(X_{ \beta k^{\xi/(3-\tau)},p}<k\big)\le k\exp\left(-\beta k^{\xi/(3-\tau)}p/8\right)\le k\exp\left(-k/4\right),
\end{equation}
for all $p>p_-$ in \eqref{eq:CMcorep}. Using again that $p>p_-$ implies $\beta k^{\xi/(3-\tau)}p\ge 2k$, the second factor in \eqref{eq:CMcore} can be bounded from below for all sufficiently large $k$ as
\begin{align}
\E[X_{ \beta k^{\xi/(3-\tau)},p}\ind\{X_{\beta k^{\xi/(3-\tau),p}
}\ge k\}]&\ge \beta k^{\xi/(3-\tau)}p-k\exp\left(-k/4\right)\ge\beta k^{\xi/(3-\tau)}p/2.
\label{eq:CMcore4}
\end{align}
Substituting \eqref{eq:CMcore1} and \eqref{eq:CMcore4} into \eqref{eq:CMcore} gives, for all $\beta$, $p>p_-$ in \eqref{eq:CMcorep} and all $\xi>3-\tau>0$ that
\begin{equation}\label{eq:CMcore5}
h(\widetilde D_M,p)\ge (\widetilde c_\ell/2)\cdot  \beta^{1-(\tau-1)(1+\varepsilon)}k^{(1-(\tau-1)(1+\varepsilon))\xi/(3-\tau)}p.
\end{equation}
Thus, $h(\widetilde D_M,p)>\E[\widetilde D_M]p^2$ holds when
\begin{equation}\label{eq:CMcore6}
(\widetilde c_\ell/2)  \beta^{1-(\tau-1)(1+\varepsilon)}k^{(1-(\tau-1)(1+\varepsilon))\xi/(3-\tau)}>\E[\widetilde D_M]p.
\end{equation}
At this point we still have the freedom of choosing $\beta$ and $\xi>3-\tau$ provided that the relation between $\eta,\xi$ in \eqref{eq:eta-xi-relation} holds.
Since $p<a_+k^{-(\xi/(3-\tau)-1)}$ in \eqref{eq:CMcorep}, first we compare the powers of $k$ on both sides. The inequality \eqref{eq:CMcore6} holds for all sufficiently large $k$ if
\[
(1-(\tau-1)(1+\varepsilon))\xi/(3-\tau)\ge-(\xi/(3-\tau)-1).
\]
After elementary computations, the smallest $\xi$ that satisfies this inequality, and hence the threshold $\eta$ for \eqref{eq:eta-xi-relation} is
\begin{equation}\label{eq:xi-min}
\xi\ge \xi_{\min}:=\frac{3-\tau}{3-\tau-\varepsilon(\tau-1)}, \qquad \eta>\eta_{\min}= \frac{\xi_{\min}(1+\varepsilon)}{1-\varepsilon} -1,
\end{equation}
which equals $\eta_{\min}$ in \eqref{eq:CMcore_eta}.
 Comparing now constants on the two sides of \eqref{eq:CMcore6} yields that
 \[a_+:=(\widetilde c_\ell/2\E[\widetilde D_M])  \beta^{1-(\tau-1)(1+\varepsilon)}.\]
 Solving the inequality $a_-=2/\beta<a_+$ gives that the interval $I_p$ is non-empty whenever
\begin{equation*}
\beta>\left(4\E[\widetilde D_M]/\widetilde c_\ell\right)^{1/[2-(\tau-1)(1+\varepsilon)]}.
\end{equation*}
Summarizing, we have found that whenever $\beta$ satisfies this inequality, and $p$ is in the interval
\begin{equation*}
I_p=\Big[(2/\beta)\cdot k^{-(1/(3-\tau-\varepsilon(\tau-1)) -1)}, (\widetilde c_\ell/2\E[\widetilde D_M])  \beta^{1-(\tau-1)(1+\varepsilon)} k^{-(1/(3-\tau-\varepsilon(\tau-1)) -1)}\Big],
\end{equation*}
then the required inequality for the existence of the $k$-core holds. This implies that $\hat{p}>p_+$, and we can estimate the asymptotic proportion of the $k$-core \eqref{eq:linear-core-Janson}, $h_1(\widetilde D_M, \hat p)\ge h_1(\widetilde D_M, p_+)$ following similar steps as in \eqref{eq:CMcore}:
\begin{align*}
h_1(\widetilde D_M, \hat p) &\ge \P(\widetilde D_M\ge k^{\xi/(3-\tau)}) \P\big( X_{k^{\xi/(3-\tau)}, p_+}\ge k\big)\\
&\ge \widetilde c_\ell k^{-\xi(\tau-1)(1+\varepsilon)/(3-\tau)} (1-\exp(-k/4)),
\end{align*}
using the same $\xi=\xi_{\min}$ and Chernoff bound as in \eqref{eq:xi-min} and in \eqref{eq:CMcoreChernoff}, yielding \eqref{eq:rho} in Remark \ref{rem:rho}.

Finally, we need to check that conditioned on its vertex set and degree sequence, $\Core_{k}(G_n[\CV_{\le M}])$ is itself a configuration model. This follows from the fact that every matching of half-edges within $\Core_{k}(G_n[\CV_{\le M}])$, given its degree sequence, has equal probability by the construction of the configuration model.
\end{proof}

This finishes the first combinatorial part, i.e., the existence of a large $k$-core. We now (slowly) transition to studying the contact process on the $k$-core.  
The proof of Theorem \ref{thm:max_CM_survival}, part (b), is based on a structural property of $\Core_{k}(G_n[\CV_{\le M}])$, that we define next. This structural property guarantees that an infected set of vertices can pass the infection to many other vertices in a unit time step.  
\begin{definition}[$(\delta,k)$-expansion]\label{def:DeltaGood}
Fix any $\delta\in(0,1)$ and an even positive integer $k$. We say that a (multi)graph $G$ on $n$ vertices is $(\delta,k)$-good if 
for every set $\{v_1,\ldots,v_{\lfloor\delta n\rfloor}\}$ of $\lfloor\delta n\rfloor$ vertices in $G$, we can choose a subset $\CI_g$ of the indices of size  $|\CI_g|\ge \lfloor \delta n\rfloor /8$ such that each $v_i: i\in \CI_g$ has $k/2$ neighbors $w_{i,1},\ldots,w_{i,k/2}$ in $G$ such that the vertices $v_i, i\in \CI_g$ and $w_{i,j}, i\in \CI_g,  j\le k/2$ are all distinct.
\end{definition}
A graph being $(\delta,k)$-good is somewhat stronger than requiring that the $1$-neighborhood of any $\lfloor\delta n\rfloor$ many vertices expands by a factor $k/16$, since we need enough individual vertices that expand to $k/2$ different vertices.
 The following lemma proves that $\Core_{k}(G_n[\CV_{\le M}])$ has the $(\delta,k)$-good property for small enough $\delta>0$.

\begin{lemma}\label{lem:DeltaGood}
Consider the configuration model $\mathrm{CM}(\underline d_n)=:G_n$ in Definition \ref{def:CM} on the degree sequence $\underline d_n=(d_1, \dots, d_n)$ so that for an even integer $k>1$ and constant $\zeta>1$, $d_i\in[k, k^{\zeta}]$ holds for all $i\in [n]$.
Then there exists some $\delta_0=\delta_0(k,\zeta)>0$ independent of $n$, such that for all $\delta<\delta_0$,
\begin{equation}\label{eq:DeltaGood}
\P(G_n\text{ is }(\delta,k)\text{-good})>1-e^{-n\delta \log (1/\delta) /8}.
\end{equation}
\end{lemma}
\begin{proof}

Let $v_1,\ldots,v_{\lfloor\delta n\rfloor}$ be distinct fixed vertices in $G_n$. We will explore, i.e., gradually reveal the neighbors of these vertices, as follows. In the first exploration step, we reveal the first $k$ edges adjacent to $v_1$ (according to an arbitrary ordering), one by one. When revealing an edge, we say that a \emph{collision} happens at $v_1$ if the revealed edge either leads to one of $v_1,v_2,\ldots,v_{\lfloor\delta n\rfloor}$, or is parallel to an edge revealed earlier (note that we allow self-loops and multiple edges in $G_n$). During this first step, as soon as the number of collisions at $v_1$ reaches two, we stop revealing the connections of $v_1$ and color $v_1$ red. If the number of collisions does not reach two by the end of step $1$, we color $v_1$ green, and we assign the revealed distinct neighbors of $v_1$, outside the set $\{v_1,\ldots,v_{\lfloor\delta n\rfloor}\}$, the labels $w_{1,1},w_{1,2},\ldots,w_{1,n_1}$. Here, $n_1\in[k-1,k]$, since there was at most one collision.

In the second step we reveal the first $k$ edges adjacent to $v_2$, one by one, \emph{including} the potential edges (at most two) that lead to $v_1$ and have already been revealed. Now we say that a collision happens at $v_2$ if a revealed edge either leads to one of $v_1,v_2,\ldots,v_{\lfloor\delta n\rfloor}$, (except when it was already revealed starting from $v_1$, and thus the collision happened at $v_1$ in which case we do not count it as a new collision), or it leads to one of $w_{1,1},w_{1,2},\ldots,w_{1,n_1}$ (in case $v_1$ was colored green), or is parallel to an edge already revealed at $v_2$. Again, as soon as the number of collisions at $v_2$ reaches two during this step, we stop revealing the edges of $v_2$ and color $v_2$ red. If the number of collisions at $v_2$ does not reach two by the end of the step, we color $v_2$ green, and assign the revealed distinct neighbors of $v_2$, outside the set $\{v_1,\ldots,v_{\lfloor\delta n\rfloor}, w_{1,1},w_{1,2},\ldots,w_{1,n_1}\}$ the labels $w_{2,1},w_{2,2},\ldots,w_{2,n_2}$. Here, $n_2\in[k-3,k]$, since at most three edges caused collisions at either $v_1$ (these can connect to $v_2$) or $v_2$.

We then continue this procedure, in each step revealing the first $k$ connections of $v_3,\ldots,v_{\lfloor\delta n\rfloor}$ analogously to the above, with one modification: if at the beginning of step $i$, when starting to reveal the neighbors of vertex $v_i$ ($i\ge 2$), $v_i$ already has at least $k/4$ adjacent revealed edges coming from the already processed vertex set $\{v_1,v_2,\ldots,v_{i-1}\}$, then we do not reveal any new connections at $v_i$, but color it blue, and continue to the next step $i+1$, with $v_{i+1}$.

After all the $\lfloor \delta n\rfloor$ steps are done, let $\CI_g:=\{i_1, \dots, i_g\}$ denote the indices and $\{v_{i_1},\ldots,v_{i_g}\}$ be the set of green vertices (subset of $\{v_1,\ldots,v_{\lfloor\delta n\rfloor}\}$). We will prove that with probability at least $1-\exp(-Cn)$,
$|\CI_g|\ge  \lfloor \delta n\rfloor/8 $ and $n_i\ge k/2$ for all $i\in \CI_g$. So, the green vertices along with their revealed neighbors $\{w_{i,j}: i\in\CI_g, j\le n_i\}$ demonstrate the $(\delta,k)$-good property of $G_n$ in Definition~\ref{def:DeltaGood}.  

Later, we take a union bound over all subsets of size $\lfloor \delta n \rfloor$, but for now we fix a choice of  $\{v_1,\ldots,v_{\lfloor \delta n\rfloor}\}$. First, we bound the number of blue vertices. When at step $j$, we reveal at most two edges that connect $v_j$ to some $v_{j'}$, for $j'>j$. Hence, we reveal at most $2\lfloor\delta n\rfloor$ edges with both endpoints in the set $\{v_1,\ldots,v_{\lfloor\delta n\rfloor}\}$, which we call \emph{internal} edges. These involve at most $4\lfloor\delta n\rfloor$ half-edges at $\{v_1,\ldots,v_{\lfloor\delta n\rfloor}\}$. Since more than $16\lfloor\delta n\rfloor/k$ vertices adjacent to at least $k/4$ internal edges would involve more than $4\lfloor\delta n\rfloor$ half-edges, by the pigeonhole principle, for all $k\ge 2$:
\begin{equation}\label{eq:discard0}
\begin{aligned}
|\text{Blue vertices}| &= |\{i\in[ \lfloor\delta n\rfloor]:\text{ $v_i$ is adjacent to at least $k/4$ internal edges}\}|\\
&\le 4\lfloor\delta n\rfloor / (k/4)=16\lfloor\delta n\rfloor/k\le \lfloor\delta n\rfloor/8.
\end{aligned}
\end{equation}
Hence, the exploration reveals the neighborhood of at least $7\lfloor\delta n\rfloor/8$ and at most $\lfloor\delta n\rfloor$ vertices that can be either red or green.
Next, we bound the number of red vertices. Here we use that  $G_n$ is a configuration model, with all degrees in the interval $[k, k^{\zeta}]$. Thus we can carry out the exploration process above by matching the first (at most) $k$ half-edges of each vertex under consideration. After revealing the $j$th edge, for $j\le k\lfloor\delta n\rfloor -1$, we have discovered at most $j$ new vertices and so half-edges attached to at most $\lfloor\delta n\rfloor+j$ vertices can cause a collision when matching the $j+1$th half-edge. And, there are at least $n k-2j-1$ remaining unmatched half-edges to choose from. Let us denote by $\CF_{j}$ the $\sigma$-algebra generated by the outcome of the matching of the first $j$ half-edges.  Then, for all $k\ge 2$ and sufficiently small $\delta=\delta(k)>0$, and for any realization in $\CF_j$
\[
\P(\text{collision at $j+1^{\text{st}}$ edge} \mid \CF_{j})\le\frac{(\lfloor\delta n\rfloor+j)k^{\zeta}}{nk-2j-1}\le \frac{(\delta n+\delta n k)k^{\zeta}}{(1-2\delta)nk}\le 2 \delta k^{\zeta}.
\]
Let $Y_j=1$ if revealing the $j^{\text{th}}$ edge causes a collision and $Y_j=0$ otherwise. Then $(Y_1,Y_2,\ldots)$ is dominated by a sequence of i.i.d. Bernoulli variables with parameter $2\delta k^{\zeta}$. 
We color $v_i$ red if at least two collisions happen at step $i$, i.e, if at least $2$ of the $Y_j$ variables corresponding to the at most $k$ revealed edges at $v_i$ are $1$. So, with $X_{n,p}$ a binomial variable as before, independently across different $v_i$,
\begin{equation}\label{eq:red}
\P(v_i \text{ red})\le \P(X_{k, 2 \delta k^{\zeta}}\ge 2) \le k^2 4 \delta^2 k^{2\zeta}  = 4 \delta^2 k^{2+2\zeta}.
\end{equation}
Combining \eqref{eq:discard0} and \eqref{eq:red} yields that the number of red vertices is stochastically dominated by a Binomial random variable with parameters $\lfloor\delta n\rfloor$ and $4 \delta^2 k^{2+2\zeta}=:q$. 
Hence, by a crude upper bound on the binomial coefficients,
\begin{align}
\P( |i: v_i \text{ red}| \ge 3 \lfloor\delta n\rfloor/4) &\le \P(X_{ \lfloor\delta n\rfloor, q}  > 3\lfloor\delta n\rfloor/4)=\sum_{r>3\lfloor \delta n\rfloor/4}\binom{\lfloor\delta n\rfloor}{r}q^{r}(1-q)^{\lfloor\delta n\rfloor-r}\nonumber\\
&
\le \lfloor\delta n\rfloor 2^{\lfloor\delta n\rfloor}q^{3\lfloor\delta n\rfloor/4}=\lfloor\delta n\rfloor 2^{\lfloor\delta n\rfloor}\big(4\delta^2  k^{2+2\zeta}\big)^{3\lfloor\delta n\rfloor/4}.\label{eq:discard}
\end{align}
after substituting the value of $q$. After elementary rewrite we obtain for small enough $\delta=\delta(k)>0$,
\begin{equation}\label{eq:discard2}
\begin{aligned}
\P&( |i: v_i \text{ red}| \ge 3\lfloor\delta n\rfloor/4)\\
&\le  \lfloor\delta n\rfloor\exp\big((3/2)\log(\delta)\lfloor\delta n\rfloor+(5/2)\log(2)\lfloor\delta n\rfloor+(3/4)\log(k^{2+2\zeta})\lfloor\delta n\rfloor\big)\\
&\le C\exp\big(-(5/4)\log(1/\delta)\delta n\big).
\end{aligned}
\end{equation}
We bound the number of ways to choose the $\lfloor\delta n\rfloor$ vertices $S=\{v_1,\ldots,v_{\lfloor\delta n\rfloor}\}$:
\begin{equation}\label{eq:binomUB}
\binom{n}{\lfloor\delta n\rfloor}\le\frac{n^{\lfloor\delta n\rfloor}}{(\lfloor\delta n\rfloor)!}\le \frac{n^{\delta n}}{\exp\big(\delta n \log(\delta n)-\delta n\big)} = \exp\big(\delta n (1+\log(1/\delta))\big).
\end{equation}
Combining \eqref{eq:discard2} and \eqref{eq:binomUB}, we obtain for some positive constant $C$ that for all small enough $\delta=\delta(k)>0$, 
\begin{align*}
\P&(\exists S\subset G: |S|=\lfloor \delta n\rfloor, \text{ at least } 3\lfloor\delta n\rfloor/4 \text{ red vertices in } S)\\
&\le\exp\big(\delta n (1+\log(1/\delta))-(5/4)\log(1/\delta)\big)<\exp\big(-n \delta\log (1/\delta)/8\big).
\end{align*} 
 Combining this with \eqref{eq:discard0}, we obtain that with probability at least $1-\exp(-Cn)$, for any choice of $v_1,\ldots,v_{\lfloor\delta n\rfloor}$, there are at least $\lfloor\delta n\rfloor/8$ green vertices among $v_1,\ldots,v_{\lfloor\delta n\rfloor}$. The green vertices, per design, have at most one collision, among of their at least $3k/4$ revealed edges.
Hence, each green vertex has at least $k/2$ neighbors in $G_n$, that are all distinct from each other and from $v_1,\ldots,v_{\lfloor\delta n\rfloor}$, demonstrating the $(\delta,k)$-good property. This finishes the proof.
\end{proof}

The next lemma studies a contact process with lower infection rate than $\CPf$ with $f=\max(x,y)^\mu$ on a $(\delta,k)$-good graph and shows that when $\lfloor\delta n\rfloor$ vertices are infected, their neighborhood sustains the infection for a unit of time:

\begin{lemma}\label{lem:SurvivalStepOnCore}
Fix some $\lambda>0$, $\mu\in(1/2,1)$ and $\zeta>1$ satisfying $\mu\zeta<1$. Then there exists constants $C'>0$ and $k_0=k_0(\lambda,\mu,\zeta)$ so that for all $k>k_0$ even, the following holds. Let $G_n$ be any multi-graph with degree sequence $\underline d_n=(d_1, \dots, d_n)$ satisfying $d_i\in[k, k^{\zeta}]$ for all $i\in [n]$, so that $G_n$ is $(\delta,k)$-good for some fixed $\delta>0$. Let  $(\underline{\tilde\xi}_t)_{t\ge 0}$ be a contact process $\mathrm{CP}_{f_+,\lambda}$ with $f_+(x,y)\equiv k^{\zeta\mu}$ on $G_n$. Then, for all sufficiently large $n$, and any $t\ge 0$,
\begin{equation}\label{eq:SurvivalStepOnCore-0}
\P\left(|\underline{\tilde\xi}_{t+1}|\ge\lfloor \delta n\rfloor\  \right|\left.\ |\underline{\tilde\xi}_t|\ge\lfloor \delta n\rfloor\right)\ge 1-\exp(-n\delta/(193e)).
\end{equation}
\end{lemma}

By \eqref{eq:domination-between-two-f} in Corollary \ref{cor:StochDom1}, the process $\mathrm{CP}_{f_+,\lambda}$ on $G_n$ dominates from below the contact process $\CPf$ with $f(x,y)=\max(x,y)^\mu$, since $f(d_u, d_v)=\max(d_u, d_v)^\mu\le k^{\zeta\mu}= (\max_{i\le n} d_i)^{\mu}$.

\begin{proof}
We shall fix $k>400$. 
Since $|\underline{\tilde\xi}_t|\ge\lfloor \delta n\rfloor$ in the conditioning in \eqref{eq:SurvivalStepOnCore-0}, denote the first $\lfloor\delta n\rfloor$ infected vertices by $S_t:=\{v_1,\ldots,v_{\lfloor\delta n\rfloor}\}$. Since $G_n$ is $(\delta,k)$-good, choose the index set $\CI_g$ with size $|\CI_g|\ge \lfloor \delta n\rfloor/8$ guaranteed by the $(\delta, k)$-good property in Definition \ref{def:DeltaGood} and write $w_{i,1},\ldots,w_{i,k/2}$ for the distinct neighbors of each $v_i, i\in \CI_g$. For each $i\in \CI_g$ define the event $\CA(v_i)$ as
\begin{equation}\label{eq:SurvivalOnCore_Ai}
\begin{aligned}
\CA(v_i):=\{&\text{$v_i$ infects at least 87 vertices among }w_{i,1},\ldots,w_{i,k/2}\\
&\text{ that stay infected by time }t+1\}.
\end{aligned}
\end{equation}
We will prove that
\begin{equation}\label{eq:SurvivalOnCore}
\P(\CB):=\P\Big(\sum_{i\in \CI_g} \ind_{\CA(v_i)} \ge \lfloor\delta n\rfloor/(32e)\Big)\ge 1-\exp(-\delta n/(193 e)).
\end{equation}
Then, on the event $\CB$, at least $87\lfloor \delta n \rfloor/(32e)$ vertices among $\{w_{i,j}\}_{1\le i\le \lfloor \delta n \rfloor, 1\le j\le k/2}$ are infected at time $t+1$, and since $32e\approx 86.98$, this implies that $|\underline{\tilde\xi}_{t+1}|\ge\lfloor \delta n\rfloor$ holds in \eqref{eq:SurvivalStepOnCore-0}, proving the lemma.

For \eqref{eq:SurvivalOnCore}, we first give a lower bound on $\P(\CA(v_i))$ in \eqref{eq:SurvivalOnCore_Ai}. The probability that $v_i$ does not heal in the time interval $[t,t+1]$ is $1/e$. Given that $v_i$ does not heal, it infects each of $w_{i,1},\ldots,w_{i,k/2}$, in the time interval $[t,t+1]$, with probability at least $1-\exp(-\lambda k^{-\mu\zeta})$, as the infection rate $r(v_i,w_{i,j})$ from $v_i$ to $w_{i,j}$ is $\lambda k^{-\mu\zeta}$. A given $w_{i,j}$ infected in the time interval $[t,t\!+\!1]$ stays infected until $t+1$ with conditional probability at least $1/e$. So, given that $v_i$ does not heal until time $t\!+\!1$, the number of infected vertices among $w_{i,1},\ldots,w_{i,k/2}$ at time $t+1$ is stochastically dominated from below by a Binomial random variable with parameters $k/2$ and $(1/e)(1-\exp(-\lambda k^{-\mu\zeta}))\ge (1/e)(\lambda k^{-\mu\zeta}/2):=p$. This lower bound holds whenever $k\ge \lambda^{-1/\mu\zeta}$, which holds for all $k\ge 2$ when $\lambda<1$ and for all sufficiently large $k$ when $\lambda>1$.
Hence, 
\begin{align}
\P(\CA(v_i))&\ge \P(\tilde \xi_t(v_i)=1\  \forall s\in[t,t+1] )\cdot\P(X_{k/2,p}\ge 87)
\ge \e^{-1}\cdot\P(X_{k/2,p})\ge 87)\nonumber.
\end{align}
The mean $\E[X_{k,p}]= \lambda k^{1-\mu\zeta}/(4e)$ and since $\mu\zeta<1$, this quantity grows with $k$, and we can choose $k$ large enough so that $\E[X_{k,p}]\ge 2\cdot87$. Then, by a Chernoff bound,
\begin{align*}
\P(\CA(v_i))&\ge e^{-1}\cdot\P(X_{k/2,p}\ge 87) \le \e^{-1} (1-\e^{-2\cdot 87/12})\ge 1/(2e).
\end{align*}
Now we use Corollary \ref{cor:StochDom1} to obtain $\ind_{\CA_i}, i\in \CI_g$ is stochastically dominated from below by independent events with success probability $1/(2e)$. Thus, another Chernoff bound finishes the proof of \eqref{eq:SurvivalOnCore}:
\begin{align*}
\P(\CB)\ge\P(X_{\lceil \lfloor\delta n\rfloor/8\rceil,1/(2e)}\ge \lfloor\delta n\rfloor/(32e))
\ge 1-\exp\big(-\lfloor\delta n\rfloor/(16\cdot 12e)\big),
\end{align*}
completing the proof of the lemma with $C':=1/ (193 e)$ where we increased $16\cdot12=192$ by one to compensate for dropping the integer part.
\end{proof}

With Theorem \ref{prop:CMcore}, and Lemmas \ref{lem:DeltaGood} and \ref{lem:SurvivalStepOnCore} at hand, we are ready to prove Theorem \ref{thm:max_CM_survival}, part (b).

\begin{proof}[Proof of Theorem \ref{thm:max_CM_survival}, part (b)]
Observe that in \eqref{eq:CMcore_eta} in Theorem \ref{prop:CMcore},
\begin{equation}\label{eq:zeta-min}
\zeta_{\min}:=\frac{\eta_{\min}+1}{3-\tau} = \frac{1}{3-\tau-\varepsilon(\tau-1)}\cdot\frac{1+\varepsilon}{1-\varepsilon}.
\end{equation}
The inequality \eqref{eq:max_CM_survival_error}, i.e., that $\mu<(3-\tau-\varepsilon(\tau-1))(1+\varepsilon)/(1-\varepsilon)$ and \eqref{eq:zeta-min} together imply that for all $\mu$ satisfying \eqref{eq:max_CM_survival_error} one can choose $\zeta>\zeta_{\min}$ so that $\zeta\mu<1$ also holds. Fix such a $\zeta$. Then, Theorem \ref{prop:CMcore} states that for all sufficiently large but fixed $k$ even,
a linear sized $k$-core of $\mathrm{CM}(\underline d_n)$ exists after removing all vertices of degree larger than $M=k^{(1+\eta)/(3-\tau)}=:k^\zeta$, i.e., for all $\epsilon'>0$, for all sufficiently large $n$,
\begin{equation}\label{eq:SurvivalOnCore1}
\P(\CA_n):=\lim_{n\to \infty} \P \Big( |\Core_{k}(G_n[\CV_{\le k^\zeta}])| \ge \rho n \Big) =1-\epsilon'/3, 
\end{equation}
and conditioned on its vertex set and degree sequence, $\Core_{k}(G_n[\CV_{\le k^\zeta}])$ is itself a configuration model.
Applying Lemma \ref{lem:DeltaGood} on $\Core_{k}(G_n[\CV_{\le k^\zeta}])$ then yields that for all small enough $\delta>0$
\begin{equation}\label{eq:SurvivalOnCore2}
\begin{aligned}
\P(\CB_n\ |\ \CA_n)&:=\P(\Core_{k}(G_n[\CV_{\le k^\zeta}])\text{ is $(\delta,k)$-good}\ |\ \CA_n)\\
&>1-e^{-n\rho(k) \delta/8}>1-\epsilon'/4.
\end{aligned}
\end{equation}
Consider the process $(\underline{\xi}_t)_{t\ge 0}\sim\CPf$ with $f(x,y)=\max(x,y)^\mu$ on $G_n$. For any $t\ge 0$ define the event
\begin{equation*}
\CI_t:=\{\text{at least $\delta\rho n$ vertices of $\Core_{k}(G_n[\CV_{\le k^\zeta}])$ are infected at time $t$}\}.
\end{equation*}
On the event $\CA_n\cap\CB_n$, all vertices in $H_n:=\Core_{k}(G_n[\CV_{\le k^\zeta}])$ have original degrees in the interval $[k,k^\zeta]$ within $G_n$, hence $\CPf$ restricted to $H_n$ is dominated from below by a contact process on $H_n$ with $f(x,y)=k^{\zeta\mu}$, exactly as in Lemma \ref{lem:SurvivalStepOnCore}. Hence, Lemma \ref{lem:SurvivalStepOnCore} applies for $H_n:=\Core_{k}(G_n[\CV_{\le k^\zeta}])$, 
\begin{equation}\label{eq:SurvivalStepOnCore}
\P(\CI_{t+1} | \ \CA_n\cap\CB_n\cap\CI_t)>1-e^{-n\rho(k)\delta/(193e)}.
\end{equation}
whenever $k$ is larger then $k_0=k_0(\lambda,\mu,\eta)$. This latter condition dictates our choice of $k$. Starting from the all-infected initial condition, \eqref{eq:SurvivalStepOnCore} implies that on the event $\CA_n\cap\CB_n$ the extinction time of $\CPf$ is dominated from below by a geometric random variable with success probability $\exp(-C' n)$. Hence, the process survives until time $\exp(nC'/2)$ with probability at least $1-\epsilon'/3$. Combining this with \eqref{eq:SurvivalOnCore1} and \eqref{eq:SurvivalOnCore2} yields that the process $\CPf(G_n,\underline 1_{G_n})$ exhibits long survival, finishing the proof.
\end{proof}

%% file: proofs_CM_stars.tex
\section{The configuration model: survival through a network of stars}\label{sec:embedded_stars_GW} 

The proof of Theorem \ref{thm:prod_CM} (a) (which then implies Theorem \ref{thm:max_CM_survival}(a)) follows the proof of Theorem 4 in \cite{Sly19}, i.e., the proof of the exponentially long survival of the classical contact process on the configuration model with subexponentially tailed degree distributions. We need some modifications to adapt the proof there to the degree-penalized model. Since the proof in \cite{Sly19} is rather lengthy, we only provide an outline of the main steps, and we focus on explaining the necessary modifications for the degree-penalized version. We direct the interested reader to \cite[Section 6, 7]{Sly19} for a full proof.

A common way to show exponentially long survival for the classical contyact process is to find $\Theta(n)$ many embedded star-graphs in the configuration model with paths of bounded degree vertices connecting them, (similarly as we proved local survival on Galton Watson trees in Section \ref{sec:survival_GW} here for the penalized version). The exact structure in this case, corresponding to  \cite[Definition 5.1]{Sly19}, is an embedded expander-graph.

For a graph $H$, and a subset of vertices $A\subset V(H)$ we denote by $\CN_H(A, r)$ the set  of vertices at most distance $r$ from $A$. For some $m\ge 1$, let also $\deg_{G, \le m}(u)$ denote the number of neighbors of vertex $u$ in $G$ that have degree at most $m$.

\begin{definition}\label{def:embdedded-expander}
 We say that a graph $H_{W_0}=(W_0, E(H_{W_0}))$ on a subset of vertices $W_0$ is embedded in $G$ if for each edge  $\{u,v\}\in E(H_{W_0})$, there is an associated path $\pi_{u,v}^{G}$ in $G$ between $u$ and $v$. For $\alpha,R,m,j>0$,
we say that $H_{W_0}$ is an embedded $(\alpha,R,m,j)$-expander in $G$ if for every subset $A\subset W_0$ with $|A|\le\alpha|W_0|$, we have
\begin{align}
&|\CN_H(A, 1)| \ge 2 |A|, \label{eq:expander-1}\\
&|\pi_{u,v}^{G}|\le R \mbox{ for all } u,v \in E(H_{W_0}),  \label{eq:NAR}\\
&\deg_{G}(w)\in[2,m] \quad \mbox{for all  } w\in \pi^G_{u,v}\setminus\{u,v\}, \ \ u,v \in W_0, \label{eq:low-deg-path} \\
&\deg_G(u) \in [j,2j], \ \mbox{and}\  \deg_{G, \le m}(u)\ge j/2 \quad \mbox{for all  } u \in W_0. \label{eq:high-deg-w}
\end{align}
\end{definition}
Observe that \eqref{eq:expander-1} is the expansion property of $H_0$, while \eqref{eq:high-deg-w} ensures that the embedded vertices of $W_0$   serve as star-graphs in $G$, i.e., they have sufficiently high degree. Meanwhile, \eqref{eq:NAR} and \eqref{eq:low-deg-path} ensure that the paths corresponding to each edge of $H_0$ are fairly short and occur on low-degree vertices, so that even the degree-penalized contact process can pass through them with good probability. 
Next, we prove the following structural lemma, corresponding to \cite[Lemma 6.1]{Sly19}.

\begin{lemma}\label{lem:double-truncated-expander}
Consider the configuration model $G_n:=\mathrm{CM}(\underline d_n)$ in Definition \ref{def:CM} on the degree sequence 
$\underline d_n=(d_1, \dots, d_n)$ that satisfies the regularity assumptions in Assumption \ref{assu:regularity}. Further suppose that its limiting degree distribution $D$ has heavier tails than stretched-exponential with stretch-exponent $\zeta$, for some $\zeta>0$, in the sense of Definition \ref{def:stretched-heavy}. Then, for any sufficiently large $m>0$ there exists a $j_0>m$ such that whenever $j>j_0$ then there exists $\alpha, \beta, R>0$ with $R\le o (j^{\zeta})$ such that the following holds whp. The graph $G_n$ contains an $(\alpha, R, m, j)$-embedded expander $H_{W_0}$ on the vertex set $W_0$ with $|W_0|\ge \beta n$.
\end{lemma}

\begin{proof}
We choose $m$ so high that 
\begin{equation}\label{eq:choice-m}
\begin{aligned}
\bar b:=\frac{\E[D (D-1) \ind_{\{D\le m\}}]}{\E[D\ind_{\{D\le m\}}] }&\ge \frac{\E[D(D-1)]}{\E[D]} (1-\eps)\quad\text{and}\\
\E[D\ind_{\{D\le m\}}] &\ge (1-\eps)\E[D].
\end{aligned}
\end{equation}
The proof is similar to the proof of \cite[Lemma 6.1]{Sly19}, and consists of the following steps.

\emph{Step 1. Targeted attack.} Recall the configuration model under targeted attack from Definition \ref{def:cm-attack}. Here we carry out the attack above degree $2j$ (considering \eqref{eq:high-deg-w} and \eqref{eq:low-deg-path}), and we denote the remaining graph by $G_n[\CV_{\le 2j}]$, and the degree of a vertex $v$ in $G_n[\CV_{\le 2j}]$ by $\tilde d_v$. This ensures that all remaining degrees are at most $2j$.

The second criterion in \eqref{eq:choice-m} ensures that each half-edge of a vertex with degree in the interval $[j, 2j]$ is matched to a vertex with degree below $m$ with probability at least $1-\eps$. Hence, denoting by $u_j:=\P(D \in [j, 2j])$, a 
 Chernoff bound similar as in \cite[Lemma 7.1 part (4) and Claim 7.2]{Sly19} ensures that there are at least $\varepsilon n u_j$ many vertices  have $\deg_{G}(u)\in [j, 2j]$ and $\deg_{G, \le m}(u) \ge j/2$, as required in \eqref{eq:high-deg-w}.

\emph{Step 2. Exploration.}
Let $W:=\{v\in G_n: \deg_{G_n}(v)\in [j, 2j], \deg_{G_n,\le m}(v)\in [j/2, 2j]\}$. We find the vertex set $W_0$ of $H_{W_0}$ as a subset of $W$. 
We explore the $R$-neighborhood in $G_n[\CV_{\le 2j}]$ of each vertex $w\in W$ simultaneously, always discarding vertices that have degree (within $G_n$) higher than $m$. The criteria in \eqref{eq:choice-m} implies that the exploration can be approximated by a supercritical branching process with mean offspring $\bar b$ from the first generation on, and also that the total number of half-edges in $G_n[\CV_{\le 2j}]\ge (1-\eps)n\E[D]$ (since $j>m$). Here we introduce a new parameter $r$ that (contrary to usual notation for radius) is controlling the number of allowed overlaps between neighborhoods of vertices in $W_0$. Given $j, R$, we choose the value of the integer $r$ so that for a $w\in W$ the expected number of vertices of $W$ that lie in the neighborhood $\CN_{G_n[\CV_{\le 2j}]}(w,2R)$ is small compared to $r$.  Then it will be unlikely that different neighborhoods $\CN_{G_n[\CV_{\le 2j}]}(w,R)$ intersect in more than $r$ vertices. To determine $r$, we estimate the size of the $(2R-1)$-st generation of the branching process and then we sample the degrees in generation $2R$ according to size-biased distribution $D_j^\star$ of $D\ind_{D\le 2j}$:
\begin{equation}\label{eq:stars_CM00}
\E[|\CN(v, 2R)\cap W_0|]\approx \E[\partial \CN(v, 2R-1)]\P(D^\star_j\in [j, 2j])\approx j\bar{b}^{2R-1} \cdot \frac{u_j j}{d}
\end{equation}
where  $u_j=\P(j\le D\le 2j)$ and $d=\E[D]$. Hence we set the requirement that
\begin{equation}\label{eq:stars_CM1}
\frac{\bar{b}^{2R-1}j^2 u_j}{d}\le\frac{r}{10}.
\end{equation}

\emph{Step 3. Graph contraction.}
We carry out a graph contraction on $G_n[\CV_{\le 2j}]$ as follows: We associate a vertex $v_w$ to each of the neighborhoods $\CN_{G_n[\CV_{\le 2j}]}(w,R)$, $w\in W$, forming the (contracted) vertex set $V'$. We associate to $v_w\in V'$ as many half-edges as there are unmatched half-edges adjacent to any vertex in $\CN_{G_n[\CV_{\le 2j}]}(w,R)$ after the exploration process in step 2 ends. Furthermore, let $V''$ be the set of vertices of $G_n[\CV_{\le 2j}]$ that have not been touched in the exploration process, i.e., vertices that belong to none of the neighborhoods $\cup_{w\in W}\CN_{G_n[\CV_{\le 2j}]}(w,R)$. Then the graph $G'_n$ is obtained by matching the half-edges of the vertex set $V'\cup V''$ uniformly at random. About the degrees of vertices in $V'$ in $G'_n$, i.e., the number of unmatched half-edges in each $\CN_{G_n[\CV_{\le 2j}]}(w,R)$, \cite[Lemma 7.5]{Sly19} proves the following:

There exists positive constant $\varepsilon', \varepsilon'', R_0$, depending only on the degree sequences $(\underline{d}_n)_{n\ge 1}$, such that for all bounded positive numbers $R_1, R, r$ satisfying
\begin{align}
R_0&\le \min\{R_1,R-R_1\}, & 800r&\le \varepsilon'^2(\bar{b}(1-\varepsilon''))^{R_1-1}j,\label{eq:stars_CM2}\\
\frac{\bar{b}^{2R_1-1}j^2 u_j}{d}&\le\frac{1}{10^4}, &\frac{\bar{b}^{2R-1}j^2 u_j}{d}&\le\frac{r}{10},\label{eq:stars_CM3}
\end{align}
the number of vertices in $V'$ with degree at least $M$ is at least $(\varepsilon'/2)|V'|$ whp, where
\begin{equation*}
M=\frac{\varepsilon'^3(\bar{b}(1-\varepsilon''))^{R-1}j}{8}.
\end{equation*}
Note that $M$ grows with $j$, i.e., most contracted vertices have high degree.
In fact all $R, R_1, r$ are dependent on and growing with $j$, while $R_0, \varepsilon, \varepsilon''$ are not.

\emph{Step 4.}
Given that the conditions \eqref{eq:stars_CM2}-\eqref{eq:stars_CM3} are satisfied, \cite{Sly19} proves the existence of a high degree core in $G'_n$, which is an $(\alpha,R,m,j)$-embedded expander in $G_n$. Here we mean core in the sense of Definition \ref{def:kcore}. \cite{Sly19} chooses $r, R, R_1$  as the solution to the following equations:
\begin{align}
\frac{\bar{b}^{2R-1}j^2u_j}{d}&=\frac{r}{10}\label{eq:stars_CM_def1},\\
\epsilon'^2(\bar{b}(1-\epsilon''))^{R_1-1}j&=800r,\label{eq:stars_CM_def2}\\
\bar{b}^{2R_1-1}&=\frac{d}{10^4 j^2 u_j}.\label{eq:stars_CM_def3}
\end{align}
It is relatively easy to check that for large $j$ this set of choices satisfies then  \eqref{eq:stars_CM2}-\eqref{eq:stars_CM3}.
We also set $r, R, R_1$ given by \eqref{eq:stars_CM_def1}--\eqref{eq:stars_CM_def3}, and now compute the value of $R$:  Combining \eqref{eq:stars_CM_def1} and \eqref{eq:stars_CM_def2} gives the relation between $R_1$ and $R$:
\begin{align}
2R-1&=(R_1-1)\cdot\frac{\log(\bar{b}(1-\epsilon''))}{\log(\bar{b})}+\frac{\log\left(\frac{d\epsilon'^2}{8000ju_j}\right)}{\log(\bar{b})}.\label{eq:stars_CM4}
\end{align}
Next, we note that \eqref{eq:stars_CM_def3} yields
\begin{align}\label{eq:stars_CM5}
2R_1-1&=\frac{\log\left(\frac{d}{10^4 j^2 u_j}\right)}{\log(\bar{b})}.
\end{align}
Since we assume that $\E[D^2]<\infty$, it holds that $\lim_{j\to \infty} j^2 u_j=\lim_{j\to \infty} j^2 \P(D\in [j,2j])=0$. So $R_1$ can be chosen arbitrarily large by increasing $j$.
Using this in \eqref{eq:stars_CM4} yields
\begin{align}
2R-1&\approx\frac{\log\left(\frac{d}{10^4 j^2 u_j}\right)}{2\log(\bar{b})}\cdot\frac{\log(\bar{b}(1-\epsilon''))}{\log(\bar{b})}+\frac{\log\left(\frac{d\epsilon'^2}{8000ju_j}\right)}{\log(\bar{b})},\nonumber\\
R&\approx \frac{\log\left(\frac{d}{10^4 j^2 u_j}\right)}{4\log(\bar{b})}+\frac{\log\left(\frac{d\epsilon'^2}{8000ju_j}\right)}{2\log(\bar{b})}.\label{eq:stars_CM6}
\end{align}
In \cite[Theorem 4]{Sly19}, the degree distribution of $G_n$ is subexponential, that is, $u_j=e^{-o(j)}$. Then, the rhs of \eqref{eq:stars_CM6} is $o(j)$. In our case, the degree distribution has heavier tails than stretched-exponential with stretch-exponent $\zeta$, that is, $u_j=e^{-o(j^{\zeta})}$. Therefore, the rhs of \eqref{eq:stars_CM5} is $o(j^{\zeta})$, finishing the proof.
\end{proof}

\begin{proof}[Proof of Theorem \ref{thm:prod_CM} (a), outline]
With Lemma \ref{lem:double-truncated-expander} at hand, the proof can be word-by-word adapted from the proof of \cite[Theorem 4]{Sly19} with the difference that we use Claim \ref{claim:spread_between_stars} for the degree-penalized process, in place of \cite[Lemma 6.2]{Sly19}. Both \cite[Lemma 6.2]{Sly19} and our Claim \ref{claim:spread_between_stars} ensure that given that a star is infested, the infection reaches the next star at most $2R$ away with probability close to $1$. For us, $2R=o(j^{1-2\mu})$ is necessary for Claim \ref{claim:spread_between_stars}, hence the assumption of heavier than stretched exponential decay with exponent $1-2\mu$ for the degree-penalized process. In comparison, in \cite{Sly19}, $R=o(j)$ is necessary for \cite[Lemma 6.2]{Sly19}, which leads to the assumption of subexponential tails there. We note that $j$ depends on the infection rate $\lambda$. 
\end{proof}

\begin{proof}[Proof of Theorem \ref{thm:max_CM_survival}(a)]
This is an easy consequence of Theorem \ref{thm:prod_CM}(a) by stochastic domination, noting that $\max(d_u,d_v)^\mu\le (d_ud_v)^\mu$.
\end{proof}

\begin{remark}\normalfont
Here we highlight the difference between the expander that \cite[Theorem 4]{Sly19} uses vs. what we describe in Lemma \ref{lem:double-truncated-expander} and the reason for the choice of difference. 
In Section \ref{sec:prod_GW_strong}, we have seen the following: a star-graph of degree $j=j(\lambda)$ that survives until $\exp(c j^{1-2\mu})$ long time can transfer the infection along a path of length $o(j^{1-2\mu})$ \emph{if} the path contains only constant degree vertices (say, at most degree $m$, neither depending on $j$ nor on $\lambda$). If we would allow the path to contain vertices of any degree up to $j$, the penalty along the path increases and along such a path whp transmission within $\exp(c j^{1-2\mu})$ long time only happens up to distance $o(j^{1-2\mu}/\log j)$, which can be seen by adapting the proof of Claim \ref{claim:spread_between_stars}. Thus, to obtain a sharp result, in Definition \ref{def:embdedded-expander}, in addition to the constraints \eqref{eq:expander-1}, \eqref{eq:NAR}, \eqref{eq:high-deg-w} that are all already present in \cite{Sly19}, we have added \eqref{eq:low-deg-path}, that restricts the embedded paths connecting the stars of degree $j$ to contain only low-degree vertices of degree at most $m$. 
Without the restriction in \eqref{eq:low-deg-path}, the proof in \cite{Sly19} word-by-word carries through for the degree-penalized CP as well, but gives a weaker result: $R$ can only be set in the proof to $R=o(j^{1-2\mu}/\log j)$, which then, by \eqref{eq:stars_CM6}, would result
in the slightly stronger assumption on the degree distribution
\begin{equation}\label{eq:req-stronger}
\P(D= K)\ge\exp\{-g(K)K^{1-2\mu}/\log(K)\}
\end{equation}
along an infinite subsequence $(K_i)_{i\ge 1}$ and
with some function $g$ such that $g(x)\to 0$ as $x\to\infty$. For limiting degree distributions satisfying \eqref{eq:req-stronger}, the proof of  \cite[Theorem 4]{Sly19} goes through for the degree-penalized version without any modifications. The modification \eqref{eq:low-deg-path} thus eliminates the extra $1/\log(K)$ factor in the tail-requirement on $D$ in \eqref{eq:req-stronger} so that the same assumption as for GW trees, Definition \ref{def:stretched-heavy} with $\zeta=1-2\mu$ is enough.
\end{remark}

%% file: appendix.tex
\section{Proofs of technical lemmas}

\subsection{Proof of the statement in Example \ref{ex:iid-degrees}}\label{s:proof-ex-iid-degrees}

Assumption \ref{assu:regularity} is a consequence of the law of large numbers.
To prove Assumptions \ref{assu:empirical-power-law}, and \ref{assu:empirical-power-law-2}, we also need to consider $n$-dependent values for $\nu_n(z)$ and $1-F_n(z)$ which makes the statement non-trivial. 
We bound the maximum degree first, this immediately gives \eqref{eq:max-degree} in Assumption \ref{assu:empirical-power-law-2}. Here we use that $ \P(D \ge z ) \le 1/z^{\alpha-\varepsilon'}\le 1/z^{\alpha(1-\eps')}$ holds for all $\eps'$ to estimate that the probability that \eqref{eq:max-degree} fails to hold for given $n,C_u,\varepsilon_1>0$ is
\begin{align}
\P\left(\max_{i\le n}D_{n,i}>n^{1/(\alpha(1-\eps_1))}\right)
&\le n \P\left(D> n^{1/(\alpha(1-\eps_1))}\right)\nonumber\\
&\le  n n^{-(\alpha(1-\varepsilon')/((\alpha(1-\varepsilon_1))}=n^{1-(1-\eps')/(1-
\eps_1)}.\label{eq:max-deg-bound}
\end{align}
For any fixed $\eps_1>0$, choose $\eps':=\eps_1/2$ and then the exponent of $n$ is negative. Hence, with probability tending to $1$, we have $\max d_i\le n^{1/(\alpha(1-\eps_1))}=:z_{\max}(\eps_1)$ for \emph{any} fixed $\eps_1$. We can rewrite the exponent to obtain that $\alpha=\tau-1$.
This means that  $\nu_n(z)$  has discrete support on $[0, n^{1/\alpha(1-\eps_1)}]$ with probability tending to $1$, hence it is enough to consider $z\in \N$ in this range. 
We now recall that for any Binomial random variable with parameters $n$ and $p$, and any $c> 1$,
\begin{equation} \label{eq:bin-rate}
\P(\mathrm{Bin}(n,p)\ge c np ) \le \exp( - np (c \log c +1-c) ) = \exp( - npc ( \log c +1/c-1)).
\end{equation}
Clearly the right-hand side is tending to $0$ as long as $npc\to\infty$ and $\log c\to \infty$ both hold.
We start by estimating the upper tail for Assumption \ref{assu:empirical-power-law}, so that we prove \eqref{eq:assu-11-holds}.
Our goal is to show that for all $z\le n^{1/(\alpha(1-\eps_1))}=z_{\max}$,  for some $\eps_2>0$ that is still arbitrarily small, whp
\begin{equation} \label{eq:goal-empirical-tail-1}
\P\big(\forall z \in [z_0(\eps_2/2), z_{\max}(\eps_1)]: 1-F_n(z)  \le  z^{-\alpha(1-\eps_2)}\big) \to 1.
\end{equation}
Note that $n(1-F_n(z))$ is the number of vertices with degree above $z$, which has $\mathrm{Bin}(n, \P(D> z))$ distribution. For all $z\ge z_0(\eps')$ this is stochastically dominated from above by a $\mathrm{Bin}(n, z^{-\alpha(1-\eps')})$ distribution.
Hence,
\[ \P\Big(1-F_n(z)  \ge   z^{-\alpha(1-\eps_2)}\Big) \le \P\Big( \mathrm{Bin}(n, z^{-\alpha(1-\eps')}) \ge n z^{-\alpha(1-\eps_2)}\Big).\]
Now we apply \eqref{eq:bin-rate} with $p=z^{-\alpha(1-\eps')}$ and $c= z^{-\alpha(1-\eps_2)+\alpha(1-\eps')}=z^{\alpha(\eps_2-\eps')}$. 
The exponent of $z$ is positive whenever $\eps_2> \eps'$ which we already we may safely assume since $\eps'$ can be chosen arbitrarily, hence $\log c\to \infty$. Further,  $npc=n z^{-\alpha(1-\eps_2)}$ tends to $\infty$ exactly when $z=o( n^{1/(\alpha(1-\eps_2))})$
which can be made always true in the range $[1, n^{1/(\alpha(1-\eps_1))}]$ of the empirical distribution by choosing $\eps_2\ge \eps_1 \ge \eps':=\eps_2/2$, but all of them arbitrarily small. At $z_{\max}(\eps_1)$ the exponent in \eqref{eq:bin-rate} becomes minimal and is at least constant times $npc=n z_{\max}^{-\alpha(1-\eps_2)}=n n^{-(1-\eps_2)/(1-\eps_1)}=n^{+\delta}$.
Taking a union bound over all $z\in [1, n^{1/(\alpha(1-\eps_1))}]$ and the bound in \eqref{eq:max-deg-bound} we obtain that 
\[
\begin{aligned}
1-\P(\forall z\ge 1: 1-F_n(z)  \le z^{-\alpha(1-\eps_2)})&= \P(\exists i\le n: D_i > z_{\max}(\eps_1)) \\
&\qquad+ \P\big(\exists z \le z_{\max}(\eps_1): 1-F_n(z)  \le z^{-\alpha(1-\eps_2)}\big) \\
&\le  n^{1-(1-\eps')/(1-
\eps_1)}+n^{1/(\alpha(1-\eps_1))} \exp(- n^{\delta}) \to 0.
\end{aligned}
 \]
This finishes the proof of \eqref{eq:goal-empirical-tail-1} and the upper bound in Assumption \ref{assu:empirical-power-law}.
To prove the lower bound we need the opposite direction, i.e., for all $c\le 1/2$,
\begin{equation}\label{eq:bin-rate-lower}
\P(\mathrm{Bin}(n,p)\le c np ) \le \exp( - np/ 8),
\end{equation}
as long as $np\to \infty$. We now estimate $n(1-F_n(z))=\mathrm{Bin}(n, \P(D>z))$ stochastically from below by $\mathrm{Bin}(n, z^{-(\alpha+\eps')})\ge \mathrm{Bin}(n, z^{-\alpha(1+\eps')})$ which is true for all fixed $\eps'$ and all $z>z_0(\eps')$, since the lower bound here is coming from \eqref{eq:heavier-power-law}. So let us set $z_{\max}^{(\ell)}(\eps, n)$ in Assumption \ref{assu:empirical-power-law} to be $n^{1/(\alpha(1+\eps))}$, 
and then the mean $n(z_{\max}^{(\ell)}(\eps,n))^{-(\alpha(1+\eps'))}=n^{1-(1+\eps')/(1+\eps)}$ tends to infinity whenever $\eps'< \eps$. 
Further, if $\eps'<\eps$ then also $n z^{-\alpha(1+\eps)} \le n z^{-\alpha(1+\eps')}/2$ for all $z\le z_{\max}^{(\ell)}(\eps,n)$, and so \eqref{eq:bin-rate-lower} applies with $p=z^{-\alpha(1+\eps')}$. 
By a union bound then
\[ \P\Big(\exists z\in[z_0,z_{\max}^{(\ell)}(\eps,n)]: 1-F_n(z) \le z^{-\alpha(1+\eps)} \Big) \le n^{1/(\alpha(1+\eps))} \exp(- n^{1-(1+\eps')/(1+\eps)}/8 ),\] 
which tends to $0$. It remains to prove \eqref{eq:point-mass} in Assumption \ref{assu:empirical-power-law-2}, and here we can use the extra assumption \eqref{eq:powerlaw-mass}. Namely, following the bound on the maximum in \eqref{eq:max-deg-bound}. We want to prove that
\begin{equation*}
\P\Big(\exists z \in[z_0, z_{\max}(\eps_1)]: \nu_n(z)\ge z^{-\tau(1-\eps)}\Big) \to 0.
\end{equation*}
In this case $n\nu_n(z)=n_z=\Bin(n, \P(D=z))$ which is stochastically dominated by $\Bin(n, z^{-\tau(1-\eps')})$, and returning to \eqref{eq:bin-rate}, now $c=z^{\tau(\eps-\eps')}$ tends to infinity whenever $\eps>\eps'$, and now $nz^{-\tau(1-\eps)}$ takes the role of $npc$. This tends to infinity whenever $z=o(n^{1/(\tau(1-\eps))})$, which is for \emph{small} $\eps>0$ much less than the maximum degree $z_{\max}(\eps_1)=n^{1/(\tau-1)(1-\eps_1)}$.
Nevertheless, we can set a reasonable $\eps$, namely, whenever
we set $\eps>1/\tau$, e.g. set $\eps:=1/\tau+\delta$,  then $nz^{-\tau(1-\eps)}= n  z^{-\tau (1-1/\tau -\delta)}= nz^{-(\tau-1-\delta)}$, and so for $z_{\max}=n^{1/(\tau-1)(1-\eps_1)}$ this is $n n^{-(\tau-1-\delta)/(\tau-1)(1-\eps_1)}$, which has a positive exponent whenever $\delta>\eps_1(\tau-1)$. Since $\eps_1$ was arbitrarily small, $\delta$ is thus also arbitrarily small.
This, together with a union bound with \eqref{eq:max-deg-bound} finishes the proof of \eqref{eq:iid-with-1/tau}:
\[
\begin{aligned}
1-\P(\forall z\ge 1: \nu_n(z)  \le z^{-\tau(1-1/\tau+\delta)})&= \P(\exists i\le n: D_i > z_{\max}(\eps_1)) \\
&\qquad+ \P\big(\exists z \le z_{\max}(\eps_1): \nu_n(z)  \le z^{-\tau(1-1/\tau+\delta)}\big) \\
&\le  n^{1-(1-\eps')/(1-
\eps_1)}+n^{1/(\tau-1)(1-\eps_1))} \exp(- n^{\delta}) \to 0.
\end{aligned}
 \]
 If one considers truncated power-law distributions with maximal degree $z_{\max,\mathrm{tr}}=o(n^{1/\tau})$, then $n z_{\max, \mathrm{tr}}^{-\tau(1-\eps)}\to \infty$ for all possible values $z$, hence the proof above works with $\eps>0$ arbitrary.

\subsection{Proof of long survival on stars}

\begin{proof}[Proof of Claim \ref{claim:star}]
Denote the neighbors of $v$ by $w_1,\ldots,w_K$. Define
\begin{align*}
    \CA_1&=\{\xi^v_t(v)=1\text{ for all }t\in[0,1]\},\\
    \CA_2&=\left\{\left|\left\{i:\xi^v_1(w_i)=1\right\}\right|\ge\lambda K^{1-\mu}/(4e)\right\}.
\end{align*}
Since $v$ recovers at rate $1$, $\P(\CA_1\mid  \xi_0(v)=1)=1/e$. Conditioning on $\CA_1$, $v$ infects each of $w_i$ during $[0,1]$ with rate $\lambda K^{-\mu}$, independently of each other. Hence, for $i=1,\ldots,K$,
\begin{equation*}
\P(v\text{ infects $w_i$ at some }t\in[0,1])=1-e^{-\lambda K^{-\mu}}\ge\lambda K^{-\mu}/2,
\end{equation*}
using that $\lambda K^{-\mu}<1$. Each $u_i$ that becomes infected during $[0,1]$ is still infected at time $1$ with conditional probability at least $1/e$. Hence
\begin{equation*}
    \left|\left\{i:\xi^v_1(w_i)=1\right\}\right| \mid \CA_1 \succcurlyeq X\sim \mathrm{Bin}\left(K, \lambda K^{-\mu}/(2e)\right),
\end{equation*}
where $\succcurlyeq$ stands for stochastic domination. By a standard Chernoff bound, this yields
\begin{align}
\P(\CA_2\mid \CA_1)&\ge\P\left(X\ge \lambda K^{1-\mu}/(4e)\right)\ge 1-e^{-\lambda K^{1-\mu}/(16e)}.\label{eq:Chernoff}
\end{align}
Therefore,
\begin{equation*}
    \P(\CA_2)\ge \P(\CA_1)\cdot\P(\CA_2\mid \CA_1)\ge \left(1- e^{-\lambda K^{1-\mu}/(16e)}\right)/e,
\end{equation*}
finishing the proof of \eqref{eq:star-i} in Claim \ref{claim:star}.

We now turn to proving \eqref{eq:star-ii} and \eqref{eq:star-iii}. Starting from time 0, we declare each unit time-interval $[s,s+1]$ for $s\in\N$ \emph{successful} if the following events jointly occur:
\begin{equation}
\begin{aligned}
    \CB^1_s&=\left\{\left|\left\{i:\xi_s(w_i)=1\right\}\right|\ge\lambda K^{1-\mu}/(8e)\right\},\\
    \CB^2_s&=\left\{\left|\left\{i:\xi_t(w_i)=1\text{ for all }t\in[s,s+1]\right\}\right|\ge \lambda K^{1-\mu}/(16e^2)\right\},\\
    \CB^3_s&=\left\{\int_s^{s+1}\xi_t(v)\,\mathrm{d}t\ge1/2\right\},\\
    \CB^4_s&=\left\{\left|\left\{i:\xi_{s+1}(w_i)=1\right\}\right|\ge \lambda K^{1-\mu}/(8e)\right\}.\label{eq:Fs}
\end{aligned}
\end{equation}
Here $\CB^1_s$ is the event that a large enough number of leaves of the star are infected at the beginning of the time interval $[s,s+1]$, which will be enough to sustain the infestation during the whole period, while $\CB^4_s$ is the corresponding event for the end of the time interval. $\CB^2_s$ is the event that the star is infested during $[s,s+1]$, while $\CB^3_s$ is the event that the center is infected during at least half of the time interval $[s,s+1]$. Note that $\CB^4_{s}=\CB^1_{s+1}$ for all $s$ and that $\CB_0^1$ holds by the condition of the Lemma. We now fix some $s\in\mathbb{N}$ and bound the conditional probabilities of each of these events given the previous ones. First, any leaf of the star that is infected at time $s$ will stay infected during the whole interval $[s,s+1]$ with conditional probability at least $1/e$, conditioned on any trajectory of the process on the other vertices. Formally, for any $i=1,\ldots,K$,
\begin{equation*}
\inf_{\eta}\P\left(\begin{array}{l}\xi_{t}(u_i)=1\text{ for all }t\in[s,s+1] \mid \xi_s(w_i)=1,\\ \xi\equiv\eta \text{ on $[s,s+1]$ on all vertices apart from $u_i$}\end{array}\right)\ge 1/e.
\end{equation*}
Hence, by a Chernoff bound similar to \eqref{eq:Chernoff},
\begin{equation}\label{eq:F2}
    \P(\CB^2_s \mid \CB^1_s)\ge 1-e^{-\lambda K^{1-\mu}/(64e^2)}.
\end{equation}
We will now give a bound on $\P(\CB_s^3\mid \CB_2^1\cap \CB_s^2)$, using that an infested star has enough leaves infected at every time to send back the infection to the center frequently enough to keep it infected for at least half of the time. Formally, given $\CB^1_s\cap \CB^2_s$, $(\xi_t(v))_{t\in[s,s+1]}$ is a Markov process on the state space $\{0,1\}$ with transition rates
\[Q_{01}\ge\lambda K^{1-\mu}\cdot\lambda K^{-\mu}/(16e^2)=\lambda^2 K^{1-2\mu}/(16e^2),\quad\quad Q_{10}=1,\]
and some starting state $\xi_s(v)$. Let us introduce auxiliary Markov processes $(Y_t)_{t\ge0}$, $(Y'_t)_{t\ge0}$ and $(Y''_t)_{t\ge0}$ on $\{0,1\}$, all starting from the same initial state $\xi_s(v)$, with transition rates
\begin{alignat*}{3}
    q_{01}&=\lambda^2 K^{1-2\mu}/(16e^2),\quad\quad & q_{10}&=1\quad\quad&\text{of $Y$},\\
    q'_{01}&=1,\quad\quad& q'_{10}&=16e^2/(\lambda^2 K^{1-2\mu})\quad\quad&\text{of $Y'$},\\
    q''_{01}&=1,\quad\quad& q''_{10}&=1/2\quad\quad&\text{of $Y''$},
\end{alignat*}
respectively. Note that $Y'$ is a time-changed (slowed-down) version of $Y$, and $Y''$ is stochastically dominated by $Y'$ when $(16e^2)/(\lambda^2 K^{1-2\mu})<1/2$. Then, recalling \eqref{eq:Fs}, we have
\begin{align}\nonumber
    \P(\CB_s^3\mid \CB_s^2)&\ge\P\left(\int_0^1 Y_t\,\mathrm{d}t\ge\frac{1}{2}\right)=\P\left(\frac{16e^2}{\lambda^2 K^{1-2\mu}}\int_0^{\frac{\lambda^2 K^{1-2\mu}}{16e^2}}Y'_t\,\mathrm{d}t\ge\frac{1}{2}\right)\\\label{eq:local_time}
    &\ge\P\left(\frac{16e^2}{\lambda^2 K^{1-2\mu}}\int_0^{\frac{\lambda^2 K^{1-2\mu}}{16e^2}}Y''_t\,\mathrm{d}t\ge\frac{1}{2}\right).
\end{align}
Note that the stationary distribution of $Y''$ is $(\pi_0,\pi_1)=(1/3,2/3)$. The large deviation principle for Markov chains (see for example \cite{DemZei}) yields that the time average of $Y''_t$ on the right-hand side of \eqref{eq:local_time} is close to $\pi_1$ with large probability, as $K\to\infty$:
\begin{equation}\label{eq:LDPMarkov}
\P\left(\frac{16e^2}{\lambda^2 K^{1-2\mu}}\int_0^{\frac{\lambda^2 K^{1-2\mu}}{16e^2}}Y''_t\,\mathrm{d}t\ge\frac{1}{2}\right)\ge 1-\exp\{-c\lambda^2 K^{1-2\mu}\}.
\end{equation}
Combining \eqref{eq:local_time} and \eqref{eq:LDPMarkov} gives
\begin{equation}\label{eq:F3}
    \P(\CB^3_s \mid \CB^1_s\cap \CB^2_s)\ge 1-\exp\{-c\lambda^2 K^{1-2\mu}\}
\end{equation}
for some $c>0$.

Given $\CB^3_s$, during $[s,s+1]$, $v$ spends at least $1/2$ time in total in state $1$, during which it infects all the leaves with rate $\lambda K^{1-\mu}$. Each leaf infected this way is still infected at $s+1$ with conditional probability at least $1/e$. Hence, for $\CB_s^4$ given by \eqref{eq:Fs}, another Chernoff bound, similar to \eqref{eq:Chernoff}, yields
\begin{equation}\label{eq:F4}
    \P(\CB^4_s \mid \CB^1_s\cap \CB^2_s\cap \CB^3_s)\ge 1-e^{-\lambda K^{1-\mu}/(32e)}.
\end{equation}
Combining \eqref{eq:F2}, \eqref{eq:F3} and \eqref{eq:F4} yields
\begin{equation}\label{eq:success_prob}
    \P(\CB^1_s\cap \CB^2_s\cap \CB^3_s\cap \CB^4_s \mid \CB^1_s)\ge 1-\exp\{-c'\lambda^2K^{1-2\mu}\}
\end{equation}
for some $c'>0$. Hence, the number of consecutive successful time-intervals stochastically dominates a Geometric random variable with parameter $\exp\{-c'\lambda^2 K^{1-2\mu}\}$. This implies both \eqref{eq:star-ii} and \eqref{eq:star-iii} in the claim.
\end{proof}

\subsection{Proofs about the degrees in the configuration model}
\begin{proof}[Proof of Lemma \ref{lem:CM-attack-degrees}]\label{proof:CM-attack-degrees}
We will consider 
$\widetilde n_i:=\sum_{j=1}^n \ind_{\{\widetilde d_j=i, d_j\le M\}}$ which gives the number of vertices with degree $i$ in $G_n[\CV_{\le M}]$. Then the (random) empirical mass function of $\widetilde F_{n,M}$ can be written as
\begin{equation}\label{eq:empirical-mass-1}
\P(\widetilde D_{n,M}=i\mid G_n[\CV_{\le M}]) = \frac{\widetilde n_i}{V_{\le M}} = \Big(\frac{V_{\le M}}{n}\Big)^{-1} \cdot \frac{\widetilde n_i}{n}.
\end{equation}
We can now analyze both factors on the rhs separately. The first factor is already given by \eqref{eq:M-vertices}, and can be exactly described using $D_n$ with cdf in \eqref{eq:empirical-degree}
\begin{equation}\label{eq:empirical-mass-2}
\frac{V_{\le M}}{n} = \frac{n F_n(M)}{n}=F_n(M)=\P(D_n\le M). 
\end{equation}
By the definition of $\delta_n$ in \eqref{eq:delta-n-m}, this falls in the range  $\P(D\le M)\pm M\delta_n$. Turning to the second factor $\widetilde n_i/n$ in \eqref{eq:empirical-mass-1}, we introduce $n_\ell:=|\CV_{\ell}|=\sum_{j=1}^n\ind_{\{d_j=\ell\}}$ the number of degree-$\ell$ vertices in the original graph.
Then we can carry out a first and second moment method, i.e., we take expectation over the realization of the matching and hence the graph $G_n[\CV_{\le M}]$. We start with the first moment:
\begin{equation}\label{eq:expect-ni}
\frac{1}{n}\E[\widetilde n_i]=\frac{1}{n}\sum_{\ell=i}^M \sum_{v\in \CV_{\ell}} \E[\ind_{\{\widetilde d_v=i\}}\mid d_v=\ell] =  \frac{1}{n}\sum_{\ell=i}^M \sum_{v\in \CV_{\ell}} \P(\widetilde d_v=i \mid d_v=\ell).
\end{equation}
To analyze $\P(\widetilde d_v=i \mid d_v=\ell)$, we first deal with self-loops at $v\in\CV_{\ell}$. Labeling the half-edges of $v$ as $h_1, h_2, \dots, h_\ell$, the number of self-loops at $v$ is $S_v:=\sum_{1\le s,t\le \ell} \ind_{\{h_s \leftrightarrow h_t\}}$, with $\leftrightarrow$ standing for the event that the two half-edges are matched to each other. We denote the total number of half-edges in the graph by $H_n=\E[D_n]n$, and then a first moment method yields
\begin{equation}\label{eq:loops-at-dv}
\P( S_v\ge 1 ) \le \E[S_v] = \binom{\ell}{2} \frac{1}{H_n-1}\le \frac{M^2}{\E[D_n] n}.
\end{equation}
Recall from  \eqref{eq:M-vertices} in Definition \ref{def:cm-attack} that we denote by $H_{\le M}$ and $H_{>M}$ the number of half-edges attached to vertices of degree at most $M$ and larger than $M$, respectively.
Partition now the $\ell$ half-edges of $v$ into (arbitrary) two groups of size $i$ and $\ell-i$, respectively:  $h_{s_1}, \dots, h_{s_i}$ and $h_{s_{i+1}}, \dots, h_{s_\ell}$, and let us write informally 
\begin{equation*}
\CA_{\{s_1, \dots, s_i\}}:=\Big\{ \{h_{s_1}, \dots, h_{s_i}\} \leftrightarrow \CV_{\le M}, \{h_{s_{i+1}}, \dots, h_{s_\ell}\}\leftrightarrow \CV_{>M}, S_v=0 \Big\}
\end{equation*}
for the event that the half-edges $h_{s_1}, \dots, h_{s_i}$ are all matched to half-edges belonging to vertices in $\CV_{\le M}$, the half-edges $h_{s_{i+1}}, \dots, h_{s_\ell}$ are all matched to half-edges belonging to vertices in $\CV_{> M}$, and there is no self-loop created among $h_{s_1}, \dots, h_{s_i}$.
Then, matching half-edges one by one, we come to
\[ 
\begin{aligned}
\P(\CA_{\{s_1, \dots, s_i\}}) 
&=\prod_{a=0}^{i-1} \frac{H_{\le M}-\ell-a}{H_n-2a-1} \cdot \prod_{b=0}^{\ell-i-1} \frac{H_{> M}-b}{H_n-2(i+b)-1}.
\end{aligned}
\]
Observe that per definition $H_{\le M}=n\E[D_n\ind_{\{D_n\le M\}}]$ and $H_n=n\E[D_n]$ so one can compute, using also that $\ell\le M$, that each factor in the first product is $q_{n,M}(1+O(M/n))$ and each factor in the second product is $(1-q_{n,M})(1+O(M/n))$. Considering all the possible partitions of the half-edges into two groups of $i$ and $\ell-i$ half-edges, and using that there are $\ell\le M$ factors in the two products together, we arrive at
\begin{equation}\label{eq:prob-dv-i-lower}
\begin{aligned}
\P(\widetilde d_v=i\mid d_v=\ell)  &\ge \sum_{\{s_1, \dots, s_i\}\subset [\ell] } \P\big(\CA_{\{s_1, \dots, s_i\}}\big)\\
&=  \binom{\ell}{i} q_{n,M}^i (1-q_{n,M})^{\ell-i} \big(1+O\big(\tfrac{M^2}{n}\big)\big).
\end{aligned}
\end{equation}
A similar upper bound holds: we account for the error caused by the event that there might be self-loops at $v$ in \eqref{eq:loops-at-dv},
\begin{equation}\label{eq:prob-dv-i-upper}
\begin{aligned}
\P(\widetilde d_v=i\mid d_v=\ell) & \le \Pv(S_v\ge 1) + \!\!\!\!\!\sum_{\{s_1, \dots, s_i\}\subset [\ell] } \P\big(\CA_{\{s_1, \dots, s_i\}}\big)\\
&= O\big(\tfrac{M^2}{n}\big) + \binom{\ell}{i} q_{n,M}^i (1-q_{n,M})^{\ell-i} \big(1+O\big(\tfrac{M^2}{n}\big)\big).
\end{aligned}
\end{equation}
Using these bounds in \eqref{eq:expect-ni}, and that $|\CV_\ell|/n=\P(D_n=\ell)$, we arrive at
\begin{equation*}
\begin{aligned}
\frac{1}{n}\E[\widetilde n_i]&=  \sum_{\ell=i}^M\P(D_n=\ell)\bigg(O\big(\tfrac{M^2}{n}\big) + \binom{\ell}{i} q_{n,M}^i (1-q_{n,M})^{\ell-i} \big(1+O\big(\tfrac{M^2}{n}\big)\big)\bigg)\\
&=O\big(\tfrac{M^2}{n}\big) \P(D_n\le M)+ \big(1+O\big(\tfrac{M^2}{n}\big)\big) \sum_{\ell=i}^M\P(D_n=\ell) \binom{\ell}{i} q_{n,M}^i (1-q_{n,M})^{\ell-i}.
\end{aligned}
\end{equation*}
Combining this with \eqref{eq:empirical-mass-1} and \eqref{eq:empirical-mass-2}, we obtain that 
\begin{equation}\label{eq:empirical-mean-last}
\begin{aligned}
\P(\widetilde D_{n,M}=i)&=\frac{1}{V_{\le M}}\E[\widetilde n_i]\\
&=O\big(\tfrac{M^2}{n}\big) + \big(1+O\big(\tfrac{M^2}{n}\big)\big) \sum_{\ell=i}^M\frac{\P(D_n=\ell)}{\P(D_n\le M)} \binom{\ell}{i} q_{n,M}^i (1-q_{n,M})^{\ell-i}.
\end{aligned}
\end{equation}
We can here observe that the rhs gives the probability $\P(\mathrm{Bin}(D_n, q_{n,M})=i \mid D_n\le M)$.
Since $\P(D_n\le M)\to \P(D\le M)$ and $q_{n,M}\to q_{M}$ by Assumption \ref{assu:regularity}, the rhs of \eqref{eq:empirical-mean-last} tends to 
\begin{equation}\label{eq:limit-mass-i}
\begin{aligned}
 \sum_{\ell=i}^M&\frac{\P(D=\ell)}{\P(D\le M)} \binom{\ell}{i} q_{M}^i (1-q_{M})^{\ell-i}\\
 &\qquad= \P(\mathrm{Bin}(D, q_M)=i \mid D\le M)=:\P(\widetilde D_M =i), 
 \end{aligned}
 \end{equation}
and we have just proved that the mean of the (random) empirical distribution $\widetilde F_{n,M}$ converges pointwise for each $i\le M$, and identified the limit random variable in \eqref{eq:binom-thinning}. Further, recalling the definition of $\delta_n$ from \eqref{eq:delta-n-m} we may write $q_{n,M}=q_M(1\pm\delta_n)$, $1-q_{n,M}=(1-q_M)(1\pm\delta_n)$ and similarly we can use that $\P(D_n=\ell)/\P(D=\ell)\in (1-\delta_n,1+\delta_n)$ when the limit $\P(D=i)$ is non-zero and otherwise $\P(D_n=i)\le \delta_n$. So, we subtract \eqref{eq:empirical-mean-last} from \eqref{eq:limit-mass-i} to obtain after elementary error-estimates that
\begin{equation}\label{eq:limit-difference}
\begin{aligned}
\big|\P(\widetilde D_{n,M}=i) - \P(\widetilde D_M=i)\big| &\le O\big(\tfrac{M^2}{n}\big) + O\big(\tfrac{M^2}{n}+ M \delta_n\big) \P(\widetilde D_M=i)\\
&=O\big(\tfrac{M^2}{n}+\delta_n M\big).
\end{aligned}
\end{equation}
This finishes comparing the first moments.  We now turn to the variance in  \eqref{eq:empirical-mass-1}. Clearly
\begin{equation}\label{eq:var-dnM}
\mathrm{Var}\Big(\P(\widetilde D_{n,M}=i\mid G_n[\CV_{\le M}])\Big) = \mathrm{Var}\Big(\frac{\widetilde n_i}{V_{\le M}}\Big) = \Big(\frac{V_{\le M}}{n}\Big)^{-2} \cdot \frac{\mathrm{Var}(\widetilde n_i)}{n^2}.
\end{equation}
The first factor on the rhs is $\P(D_n\le M)^{-2}$. Using the indicator representation of $\widetilde n_i$, we compute using the covariance formula that
\begin{equation}\label{eq:variance-ni}
\begin{aligned}
\frac{\mathrm{Var}(\widetilde n_i)}{n^2}=\frac{1}{n^2}\sum_{\ell, \ell'=i}^M  \sum_{\substack{v\in \CV_{\ell}, \\u\in \CV_{\ell'}}}&\Big(\P\big(\widetilde d_v=i, \widetilde d_u=i \mid d_v=\ell, d_u=\ell'\big)\\
&\qquad - \P\big(\widetilde d_v=i, \mid d_v=\ell\big) \P\big(\widetilde d_u=i \mid  d_u=\ell'\big)\Big).
\end{aligned}
\end{equation}
When $u=v$, the two vertices are the same,  the (co)variance is at most $1$, and the summation contains only at most $n$ terms, hence the error coming from coinciding $u,v$ is $O(1/n)\P(D_n\le M)$ when summed also over $\ell=\ell'$. Now we treat the case when $u\neq v$.
For $\P\big(\widetilde d_v=i \mid  d_v=\ell\big)$ and $\P\big(\widetilde d_u=i \mid  d_u=\ell'\big)$ we can use the bounds in \eqref{eq:prob-dv-i-lower} and \eqref{eq:prob-dv-i-upper}. Similarly to there, we compute the first term $\P\big(\widetilde d_v=i, \widetilde d_u=i \mid d_v=\ell, d_u=\ell'\big)$ as well. Let $S_{u,v}$ denote the number of self-loops at the two vertices $u, v$ together plus the number of edges between $u$ and $v$.
Then a first moment method yields 
\begin{equation}\label{eq:loops-at-duv}
\P( S_{u,v}\ge 1 ) \le \E[S_{u,v}] = \frac{1}{H_n-1}\bigg(\binom{\ell}{2}  + \binom{\ell'}{2}  + \ell \ell'\bigg)\le \frac{2M^2}{\E[D_n] n}.
\end{equation}
Now we label the half-edges $h_1^{\sss{(v)}}, \dots, h_\ell^{\sss{(v)}}$ and $h_1^{\sss{(u)}}, \dots, h_{\ell'}^{\sss{(u)}}$ of $v$ and $u$, respectively, and partition them into  two subsets each, defined by the index sets $\{s_1, \dots, s_i\}, \{ s_{i+1}, \dots, s_\ell\}\subset[\ell]$ and $\{t_1, \dots, t_i\}, \{t_{i+1}, \dots, t_{\ell'}\} \subset [\ell']$.
We introduce the event 
\begin{equation*}
\begin{aligned}
\CA_{\{s_1, \dots, s_i, t_1, \dots, t_i\}}:=\Big\{& \{h^{\sss{(v)}}_{s_1}, \dots, h_{s_i}^{\sss{(v)}}, h^{\sss{(u)}}_{t_1}, \dots, h_{t_i}^{\sss{(u)}}\} \leftrightarrow \CV_{\le M}, \\
&\{h_{s_{i+1}}^{\sss{(v)}}, \dots, h_{s_\ell}^{\sss{(v)}}, h_{t_{i+1}}^{\sss{(u)}}, \dots, h_{t_{\ell'}}^{\sss{(u)}} \}\leftrightarrow \CV_{>M}, S_{u,v}=0 \Big\},
\end{aligned}
\end{equation*}
 the event that the half-edges $h^{\sss{(v)}}_{s_1}, \dots, h_{s_i}^{\sss{(v)}}, h^{\sss{(u)}}_{t_1}, \dots, h_{t_i}^{\sss{(u)}}$ are all matched to half-edges belonging to vertices in $\CV_{\le M}$, the half-edges $h_{s_{i+1}}^{\sss{(v)}}, \dots, h_{s_\ell}^{\sss{(v)}}, h_{t_{i+1}}^{\sss{(u)}}, \dots, h_{t_{\ell'}}^{\sss{(u)}} $ are all matched to half-edges belonging to vertices in $\CV_{> M}$, and there is no self-loop and edge created at and between $u$ and $v$.
 Then
 \begin{equation*} 
\begin{aligned}
\P(\CA_{\{s_1, \dots, s_i, t_1, \dots, t_i\}}) 
&=\prod_{a=0}^{2i-1} \frac{H_{\le M}-\ell-\ell'-a}{H_n-2a-1} \cdot \prod_{b=0}^{\ell+\ell'-2i-1} \frac{H_{> M}-b}{H_n-2(2i+b)-1}.
\end{aligned}
\end{equation*}
Using that $\ell, \ell'\le M$, one can compute that each factor in the first product is $q_{n,M}(1+O(M/n))$ and each factor in the second product is $(1-q_{n,M})(1+O(M/n))$. Hence, similarly to \eqref{eq:prob-dv-i-lower} and   \eqref{eq:prob-dv-i-upper}, summing over all possible partitions, we obtain the lower bound
\begin{equation*}
\begin{aligned}
\P\big(\widetilde d_v=i, &\widetilde d_u=i\mid d_v=\ell, d_u=\ell'\big)\\
&\ge  \binom{\ell}{i} q_{n,M}^i (1-q_{n,M})^{\ell-i} \binom{\ell'}{i} q_{n,M}^i (1-q_{n,M})^{\ell'-i}\big(1+O(\tfrac{M^2}{n})\big).
\end{aligned}
\end{equation*}
and the upper bound is the same as the rhs with an additive $O\big(\tfrac{M^2}{n}\big)$ coming from \eqref{eq:loops-at-duv}. We see that this is the same bound as the one in \eqref{eq:prob-dv-i-upper}, multiplied together for $u$ and $v$. Hence, returning to \eqref{eq:variance-ni}, when we take the difference of the two terms, the summand $1$ in the $(1+O\big(\tfrac{M^2}{n}\big))$ factor cancels, and each summand can be upper bounded as
\begin{equation*}
\begin{aligned}
\big|\P\big(\widetilde d_v=i, &\ \widetilde d_u=i\mid d_v=\ell, d_u=\ell'\big) - \P\big(\widetilde d_v=i, \mid d_v=\ell\big)\P\big(\widetilde d_u=i \mid  d_u=\ell'\big)\big| \\
&\le O\big(\tfrac{M^2}{n}\big) + O\big(\tfrac{M^2}{n}\big)\binom{\ell}{i} q_{n,M}^i (1-q_{n,M})^{\ell-i} \binom{\ell'}{i} q_{n,M}^i (1-q_{n,M})^{\ell'-i}.
\end{aligned}
\end{equation*}
We account the $O(1/n)\P(D_n\le M)$ error coming from $u=v$, and use $|\CV_{\ell}|/n=\P(D_n=\ell)$ and $|\CV_{\ell'}|/n=\P(D_n=\ell')$, then we obtain in \eqref{eq:variance-ni} that 
\begin{equation*}
\begin{aligned}
\frac{\mathrm{Var}(\widetilde n_i)}{n^2}&\le O(\tfrac{1}{n})\P(D_n\le M) + O\big(\tfrac{M^2}{n}\big)\sum_{\ell, \ell'=i}^M  \P(D_n=\ell)\P(D_n=\ell')\\ 
&\quad\cdot\left(1+ \binom{\ell}{i} q_{n,M}^i (1-q_{n,M})^{\ell-i} \binom{\ell'}{i} q_{n,M}^i (1-q_{n,M})^{\ell'-i}\right)\\
&\le O(\tfrac{M^2}{n}) \Big( \P(D_n\le M) +  2\P(D_n\le M)^2 \Big),
\end{aligned}
\end{equation*}
where the last row is a far from sharp upper bound.

Wlog we may assume $M$ is large enough for $\P(D_n\le M)\ge 1/2$ to hold. Using the previous inequality in \eqref{eq:var-dnM}, and that the first factor there is $1/\P(D_n\le M)^2$, we come to
\begin{equation}\label{eq:variance-ni3}
\mathrm{Var}\Big(\P(\widetilde D_{n,M}=i\mid G_n[\CV_{\le M}])\Big)\le O(\tfrac{M^2}{n}) \Big( 2+ 1/\P(D_n\le M)\Big) = O(\tfrac{M^2}{n}),
\end{equation} 
which is true \emph{uniformly} in $i$, i.e., the factor $O(M^2/n)$ is not depending on $i$.
We now study $X_{n,M}(i):=\P(\widetilde D_{n,M}=i\mid G_n[\CV_{\le M}])$ and then $\P(\widetilde D_{n,M}=i)=\E[X_{n,M}(i)]$ computed in \eqref{eq:empirical-mean-last}, and $\P(\widetilde D_M=i):=p_M(i)$ in \eqref{eq:limit-mass-i} is $\lim_{n\to \infty}\E[X_{n,M}(i)]$. By \eqref{eq:limit-difference}, $|\E[X_{n,M}(i)]-p_M(i)|\le O(M^2/n+M\delta_n)$. 
Using the triangle inequality, for any $\varepsilon_n>0$ it is true that
\begin{equation*}
\begin{aligned}
\Big\{\big| X_{n,M}(i)&-p_M(i)\big| \ge \varepsilon_n \Big\}\\
& \subseteq \Big\{\big| X_{n,M}(i)-\E[X_{n,M}(i)]\big| \ge \varepsilon_n/2 \Big\} \cup  \Big\{\big| \E[X_{n,M}(i)]-p_M(i)\Big| \ge \varepsilon_n/2 \Big\}. 
\end{aligned}
\end{equation*}
The second event in the rhs deterministically does not hold whenever $\epsilon_n\gg O(M^2/n+M\delta_n)$, and simultaneously for all $i\le M$, whenever $n$ is sufficiently large, by taking the maximum error over $i\le M$.
So we may compute using a union bound followed by Chebyshev's inequality for all $i$ that
\begin{equation}\label{eq:chebyshev-final}
\begin{aligned}
\P\Big(\sup_{i\le M}\big| X_{n,M}(i)-p_M(i)\big| \ge \varepsilon_n \Big) &\le  \sum_{i\le M}\P\Big(\big| X_{n,M}(i)-\E[X_{n,M}(i)]\big| \ge \varepsilon_n/2 \Big)  \\
&\le \sum_{i\le M} 4\varepsilon_n^{-2} \mathrm{Var}(X_{n,M}(i)) = O\Big(\tfrac{M^3}{n\varepsilon^{2}_n}\Big),
\end{aligned}
\end{equation}
where we used \eqref{eq:variance-ni3} and summed over $i$ to obtain the last bound.  The rhs tends to zero as $n\to \infty$ by the assumption that $\varepsilon \gg 1/\sqrt{n}$ implying $n\varepsilon^2_n\to \infty$. This finishes the proof of \eqref{eq:convergence-in-P-of-mass}.
We compute the (random) mean of the empirical distribution $\widetilde F_{n,M}$ of $G_{n}[\CV_{\le M}]$ as 
\begin{equation*}
\E\big[ \widetilde D_{n,M} \mid G_{n}[\CV_{\le M}]\big] = \sum_{i=1}^M i X_{n,M}(i)\  {\buildrel \P \over \longrightarrow }\ \sum_{i=1}^M ip_M(i) =\E\big[\widetilde D_M\big]
\end{equation*}
by \eqref{eq:chebyshev-final}. Finally, the fact that $\widetilde D_{n,M}$ is a configuration model, conditioned on its vertices and their degrees, follows from the fact that every matching of its half-edges have equal probability under the law of the configuration model $G_n$. This finishes the proof of Lemma \ref{lem:CM-attack-degrees}.
\end{proof}

\begin{proof}[Proof of Lemma \ref{lem:truncated-power-law}]\label{proof:truncated-power-law}
We analyze now the limiting distribution in \eqref{eq:binom-thinning} under the assumption that the original empirical distribution sequence $(F_n)_{n\ge 1}$ satisfies both Assumptions \ref{assu:regularity} and \ref{assu:empirical-power-law}. In particular, Assumption \ref{assu:empirical-power-law} implies that the cdf of the limiting distribution $F_D$ of $D_n$ satisfies \eqref{eq:empirical-power-law}
for all $\varepsilon>0$  such that  for all $n\ge n_0(\varepsilon)$, and for all $z\ge z_0$ that
\begin{equation}\label{eq:empirical-power-law-limit}
\frac{c_\ell}{z^{(\tau-1)(1+\varepsilon)}}\le 1-F_D(z) \le \frac{c_u}{z^{(\tau-1)(1-\varepsilon)}}.
\end{equation}
and $\E[D]<\infty$ by assumption. We observe first that $\widetilde D_M$ in \eqref{eq:binom-thinning} is a binomial thinning of $(D| D\le M)$, hence
$\widetilde D_M$ is stochastically dominated from above by $(D| D\le M)$. So, by the definition of stochastic domination,
\begin{equation}\label{eq:pm-bounds}
 1-\widetilde F_{M}(z)=\P(\widetilde D_M>z) \le \P(D > z \mid D\le M) = \frac{(1-F_D(z))-(1-F_D(M))}{F_D(M)}. 
 \end{equation}
Using now \eqref{eq:empirical-power-law-limit}, estimating the numerator from above and the denominator from below, assuming that $M$ is such that $c_uM^{-(\tau-1)/2}\le 1/2$, for all $\varepsilon\in (0, (\tau-1)/2]$ it holds  for all $z\in[z_0, M]$ that
\begin{equation}\label{eq:limit-upper-bound}
1-\widetilde F_{M}(z)\le \frac{c_u z^{-(\tau-1)(1-\varepsilon)} }{1-c_uM^{-(\tau-1)(1-\varepsilon)}}\le  \frac{c_u z^{-(\tau-1)(1-\varepsilon)} }{1-c_uM^{-(\tau-1)/2}}\le 2c_u z^{-(\tau-1)(1-\varepsilon)}
\end{equation}
which finishes the proof of the upper bound in \eqref{eq:truncated-power-law} with $\widetilde c_u=2c_u$. 
For the lower bound in \eqref{eq:truncated-power-law} we will also need a lower bound on $\P(D > z \mid D\le M)$. Using the rhs of \eqref{eq:pm-bounds}, estimating the denominator by at most $1$, and the numerator using \eqref{eq:empirical-power-law-limit},  we obtain
\begin{align}
\P(D > z \mid D\le M)&\ge c_\ell z^{(\tau-1)(1+\varepsilon)} - c_u M^{-(\tau-1)(1-\varepsilon)} \nonumber\\
&=  c_\ell z^{-(\tau-1)(1+\varepsilon)} \big(1- (z^{(\tau-1)(1+\varepsilon)}/M^{(\tau-1)(1-\varepsilon)}) \cdot (c_u/c_\ell)\big).\label{eq:limit-bound}
\end{align}
Here we require that the second factor is at least, say, $1/2$, which leads to
\begin{equation}\label{eq:lower-without-binom}
\begin{aligned}
\P(D > z \mid D\le M)&\ge (c_\ell/2) z^{-(\tau-1)(1+\varepsilon)}\\
\mbox{for all} \quad  z &\le (c_\ell/(2c_u))^{\tfrac{1}{(\tau-1)(1+\varepsilon)}} M^{1- \tfrac{2\varepsilon}{1+\varepsilon}}=: \widetilde z'_{\max}(M).
\end{aligned}
\end{equation}
Observe that even without considering the binomial thinning in \eqref{eq:binom-thinning}, one cannot hope to prove a lower bound for  $z$ too close to  $M$. Nevertheless, $\widetilde z'_{\max}(M)$ is growing with $M$ for all $\varepsilon< 1$, and it gets closer to $\Theta(M)$ as $\varepsilon$ is smaller, which intuitively means that the sharper bound one has on the tail of $D$, the sharper bound we can also get on probabilities of $D$ falling in given intervals. Nevertheless, even for $\varepsilon=0$, we must require $z\le c_2 M$ for some constant $c_2\le 1$.

Now we  compute the thinning probability in \eqref{eq:binom-thinning}:
\begin{equation} \label{eq:bounds-qm}
\begin{aligned}
1-q_M &= \E[D]^{-1}\E[D\ind_{\{D>M\}}] =   \E[D]^{-1} \Big( M \P(D> M) + \sum_{j=M}^\infty \P(D> j)\Big) \\
&\le \E[D]^{-1} (c_{u} M^{1-(\tau-1)(1-\varepsilon)} + \sum_{j=M}^{\infty} c_u j^{-(\tau-1)(1-\varepsilon)})  \le c_{u,1} M^{1-(\tau-1)(1-\varepsilon)},
\end{aligned} 
\end{equation}
for some constant  $c_{u,1}>0$ (that does not depend on $M$) and a similar lower bound holds $1-q_M\ge c_{\ell,1} M^{1-(\tau-1)(1+\varepsilon)}$.
Then, using the Binomial representation in \eqref{eq:binom-thinning}, and then stochastic domination of $\mathrm{Bin}(j, q)$ by $\mathrm{Bin}(j^\star, q)$ when $j\le j^\star$, we obtain that for all $j^\star\ge z$,
\begin{equation}\label{eq:dm-larger-z}
\begin{aligned}
\P(\widetilde D_M\ge z) &=\sum_{j=z}^M \frac{\P(D=j)}{\P(D\le M)} \P\big(\mathrm{Bin}(j,q_M) \ge z\big)\\
&\ge \P\big(\mathrm{Bin}(j^\star, q_M)\ge z \big) \P\big(D\ge j^\star \mid D\le M\big),
\end{aligned}
\end{equation}
and we can optimise the value $j^\star=j^\star(z)\ge z$ to obtain a sharp enough bound. For the second factor on the rhs we may use \eqref{eq:lower-without-binom}. Moving to the `complement' binomial, we estimate the first factor in \eqref{eq:dm-larger-z} as
\begin{equation}\label{eq:complement-binom}
\begin{aligned}
\P(\mathrm{Bin}(j^\star, q_M)\ge z) &= \P(\mathrm{Bin}(j^\star, 1-q_M) \le j^\star-z) \\
&= 1- \P(\mathrm{Bin}(j^\star, 1-q_M) > j^\star-z).
\end{aligned}
\end{equation}
We observe that the thinning probability $1-q_M$ in \eqref{eq:bounds-qm} tends to zero with $M$. 
So, when $z\le \widetilde z'_{\max}/2$, we may take $j^\star(z):=2z$, and use Markov's inequality on the rhs in \eqref{eq:complement-binom} to obtain
\begin{equation*}
 1-\P\big(\mathrm{Bin}(2z, 1-q_M)> z \big) \ge 1-\frac{2z(1-q_M)}{z}=1- 2(1-q_M) \ge 1/2
 \end{equation*}
for all $M$ large enough so that $c_{u,1}M^{1-(\tau-1)(1-\varepsilon)}<1/4$. Using this bound in \eqref{eq:dm-larger-z}, along with \eqref{eq:lower-without-binom}, we obtain for all $z<\widetilde z_{\max}/2$ that
\[ 
\P(\widetilde D_M\ge z)\ge \P\big(D\ge 2z \mid D\le M\big)/2 \ge (c_\ell/4) (2z)^{-(\tau-1)(1+\varepsilon)},
\]
which finishes the proof by choosing
\[\widetilde c_\ell:=2^{-(\tau-1)(1+\varepsilon)-2}c_\ell\quad\text{ and }\quad\widetilde z_{\max}(M):=2^{-1}(c_\ell/(2c_u))^{\tfrac{1}{(\tau-1)(1+\varepsilon)}} M^{1- \tfrac{2\varepsilon}{1+\varepsilon}}.\]
\end{proof}